\documentclass[12pt,letter]{article}

\usepackage{color}
\definecolor{green}{rgb}{0.834580, 0.893016, 0.718697}

\usepackage{bbm}
\usepackage{natbib}
\usepackage{amstext}
\usepackage{amsthm}
\usepackage{epsfig}
\usepackage{pdfpages}
\usepackage{amsmath,amssymb,amsthm}
\usepackage{graphicx}
\usepackage{fancyhdr}
\usepackage{calc}
\usepackage{dsfont}
\usepackage{subfigure}
\usepackage{rotating}
\usepackage{multirow}
\usepackage{amsmath}
\usepackage{lscape}
\usepackage{tabularx}
\usepackage{bm}
\usepackage{caption}
\usepackage{color}
\usepackage{enumerate,array}
\usepackage{verbatim}
\usepackage{fancyhdr}
\usepackage{stackengine}
\usepackage{enumitem}

\makeatletter
\DeclareRobustCommand\check[1]{{\mathpalette\@widecheck{#1}}}
\def\@widecheck#1#2{%
    \setbox\z@\hbox{\m@th$#1#2$}%
    \setbox\tw@\hbox{\m@th$#1%
       \widehat{%
          \vrule\@width\z@\@height\ht\z@
          \vrule\@height\z@\@width\wd\z@}$}%
    \dp\tw@-\ht\z@
    \@tempdima\ht\z@ \advance\@tempdima2\ht\tw@ \divide\@tempdima\thr@@
    \setbox\tw@\hbox{%
       \raise\@tempdima\hbox{\scalebox{1}[-1]{\lower\@tempdima\box
\tw@}}}%
    {\ooalign{\box\tw@ \cr \box\z@}}}
\makeatother

\usepackage{hyperref}
\definecolor{cornellred}{rgb}{0.7, 0.11, 0.11}
\hypersetup{    
	colorlinks=true,
	linkcolor=blue,
	citecolor=blue,
	urlcolor=black
}

\renewcommand{\baselinestretch}{1.5}

\numberwithin{equation}{section}

\renewcommand{\Re}{{\rm Re}}
\renewcommand{\Im}{{\rm Im}}

\allowdisplaybreaks

\newtheorem{theorem}{Theorem}[section]
\newtheorem{lemma}{Lemma}[section]
\newtheorem{proposition}{Proposition}[section]

\newtheorem{corollary}{Corollary}[section]

\theoremstyle{definition}
\newtheorem{definition}{Definition}[section]
\newtheorem{remark}{Remark}[section]
\newtheorem{algorithm}{Algorithm}[section]

\newtheorem{assumption}{Assumption}

\newtheorem{innercustomthm}{Condition}
\newenvironment{newcondition}[1]
  {\renewcommand\theinnercustomthm{#1}\innercustomthm}
  {\endinnercustomthm}

\newcommand{\cov}{{\rm cov}}

\newcommand{\E}{{\rm E}}

\newcommand{\ISE}{\textrm{ISE}}
\renewcommand{\H}{\mathcal{H}}
\newcommand{\HH}{\mathbb{H}}

\newcommand{\one}{\mathbbm{1}}

\def\D{\mathfrak{D}}
\def\BL{{\rm BL}}
\def\G{\mathbb{G}}
\def\bc{\mathbf{c}}
\def\bZ{\mathbf Z}

\newcommand{\var}{{\rm var}}

\newcommand{\e}{\epsilon}

\newcommand{\EE}{\mathcal{E}}

\newcommand{\C}{\mathcal{C}}
\newcommand{\F}{\mathcal{F}}

\newcommand{\X}{\mathcal{X}}

\newcommand{\converged}{\overset{d}{\longrightarrow}}
\newcommand{\weakconverge}{\rightsquigarrow}

\newcommand{\iidsim}{\overset{\small{\text{iid}}}{\sim}}

\renewcommand{\l}{\langle}
\renewcommand{\r}{\rangle}
\renewcommand{\i}{{\rm i}}
\renewcommand{\phi}{\varphi}
\renewcommand{\L}{\mathfrak L}

\renewcommand{\tilde}{\widetilde}
\renewcommand{\hat}{\widehat}
\renewcommand{\epsilon}{\varepsilon}
\renewcommand{\P}{{\rm P}}

\def\boxit#1{\vbox{\hrule\hbox{\vrule\kern6pt  \vbox{\kern6pt#1\kern6pt}\kern6pt\vrule}\hrule}}

\def\boxit#1{\vbox{\hrule\hbox{\vrule\kern6pt
          \vbox{\kern6pt#1\kern6pt}\kern6pt\vrule}\hrule}}

\newtheorem{innercustomgeneric}{\customgenericname}
\providecommand{\customgenericname}{}
\newcommand{\newcustomtheorem}[2]{%
  \newenvironment{#1}[1]
  {%
   \renewcommand\customgenericname{#2}%
   \renewcommand\theinnercustomgeneric{##1}%
   \innercustomgeneric
  }
  {\endinnercustomgeneric}
}
\newcustomtheorem{customthm}{Assumption}
\newcustomtheorem{customcor}{Corollary}

\renewcommand{\baselinestretch}{1.4}

\makeatletter
\def\singlespace{\deltaf\baselinestretch{1}\@normalsize}

  \renewenvironment{thebibliography}[1]{%
    \begin{oldthebibliography}{#1}%
      \setlength{\parskip}{0.3ex}%
      \setlength{\itemsep}{0ex}%
  }%
  {%
    \end{oldthebibliography}%
  }

\begin{document}

\begin{center}

{\bf \Large Statistical inference for function-on-function linear regression 
} 

\end{center}

\baselineskip=13pt
\begin{center}

Holger Dette, Jiajun Tang\\
Fakult\"at f\"ur Mathematik, Ruhr-Universit\"at Bochum, Bochum, Germany
\end{center}

\baselineskip=20pt
\vspace{.5cm}

\noindent {\bf Abstract:} 
We propose  a reproducing kernel Hilbert space approach to estimate the slope in a function-on-function 
linear regression via penalised least squares, regularized by the 
thin-plate spline smoothness penalty. 
In contrast to most of the work on functional linear regression, 
our main focus is on  statistical inference  with respect to the sup-norm. 
This point of view is motivated by the fact that 
slope (surfaces) with rather different shapes may still be  identified as similar when the 
difference is measured by 
an $L^2$-type norm. However,  in applications it is often desirable to use metrics reflecting the visualization of the objects in the statistical analysis.

We prove the weak convergence of the  slope surface  estimator as a process 
in the space of all continuous functions. This  allows us the construction of simultaneous confidence 
regions for the  slope surface  and simultaneous prediction bands.
As a further consequence,  we derive new tests 
 for the  hypothesis that the maximum deviation between the ``true'' slope surface and a given surface is less or equal than a given threshold. In other words: we are  not trying to test for exact equality (because in many applications this hypothesis is hard to justify), but rather for pre-specified deviations under the null hypothesis.
 To ensure practicability, non-standard bootstrap procedures are developed  addressing
 particular features that arise in  these testing problems. 
  
As a by-product, we also derive several new results and statistical inference tools
for the function-on-function linear regression model, such as  minimax optimal convergence rates and
likelihood-ratio tests. 
 We also demonstrate that  the new methods have good finite sample properties by means of a simulation study and illustrate their practicability 
 by analyzing a  data example.

\medskip

\medskip

\noindent {\bf Keywords:}  function-on-function linear regression, minimax optimality,
simultaneous confidence regions,
relevant hypotheses,  bootstrap, reproducing kernel Hilbert space, maximum deviation
\smallskip

\noindent {\bf AMS Subject Classification:}  62R10, 62F03, 62F25, 46E22

\section{Introduction}\label{key}

Over the past decades, new measurement  technologies
provide enormous amounts of data with complex structure.
A popular and extremely successful  approach to   model  high-dimensional data  on a dense grid
 exhibiting a  certain degree of smoothness is functional data analysis (FDA), which considers the  observations as discretized functions. Meanwhile, numerous   practical and theoretical aspects in FDA have been discussed 
\citep[see, for example, the monographs][among others]{bosq2000,ramsay2005,FerratyVieu2010,horvath2012,hsingeubank2015}.
A large portion of  the literature uses  dimension reduction techniques such as (functional) principal components.  On the other hand, as argued in \cite{AueRiceSonmez2015}, 
there are numerous  applications, where   it is reasonable to assume that the functions are at least continuous, and in  such cases dimension reduction techniques can incur a loss of information and 
fully functional methods can prove advantageous.

Because of its simplicity and good interpretability,  the  {\it scalar-on-function regression model}  
\begin{align}\label{eq:scalar}
Y_i= \alpha_0 + 
\int_0^1\beta_0(s)\,X_i(s)\,ds+\epsilon_i\,,\quad1\leq i\leq n\,.
\end{align}
has found considerable attention 
\citep[see, for exemple, ][among manny others]{james2002,cardot2003,muller2005,yao2005,hall2007,yuancai}.
Here, the $Y_i$ and  the  (centred) errors $\varepsilon_i$
are scalar variables,  the predictors $X_i$ are functions (typically of time or location) defined on the interval $[0,1]$, and 
the scalar $\alpha_0$ and the function  $\beta_0$ are the unknown parameters to be estimated. 
On the other hand, there also  exist many applications, where both,  the predictor and the response, are functions, and in recent years the 
{\it  function-on-function regression model}
\begin{align}\label{model}
Y_i(t)=\alpha_0(t)+\int_0^1\beta_0(s,t)\,X_i(s)\,ds+\epsilon_i(t)\,,\quad t\in[0,1]\,,\ 1\leq i\leq n\,,
\end{align}
 has gained increasing attention \citep[see][]{lian2007,lian2015,scheipl2016,benatia2017,luo2017,sun2018}. 
Here $\alpha_0$, $Y_i$, $X_i,$ $\varepsilon_i$ are functions defined on the 
interval $[0,1]$ and the slope parameter 
$\beta_0$ is a function defined on the square $[0,1]^2$,
which we call slope surface throughout this paper in order to distinguish it from the slope function in model \eqref{eq:scalar}.

The slope $\beta_0$  quantifies the strength of the dependence between the  predictor and  the response,  
and is the main object of statistical inference in this context.
Many methods, such as estimation, testing, confidence regions, have been 
developed in the last decades for the  scalar-on-function linear regression model
 \eqref{eq:scalar}, which  are often based on the $L^2$ metric 
\citep[see, for example,][among many others]{hall2007,horvath2012}. 
A popular estimation tool is  functional principle component (FPC) analysis, which  provides a series representation of the function  $\beta_0$ in the
corresponding  $L^2$ space  \citep[see, for example,][]{yao2005}. 
Other authors  proposed reproducing kernel Hilbert space (RKHS) approaches  
to estimate the slope parameter
in a functional linear regression model. For example, 
\cite{yuancai} used the  RKHS framework to 
construct a minimax optimal  estimate in  the scalar-on-function linear regression, and 
  \cite{caiyuan2012} discussed 
minimax properties of their RKHS estimator in terms of prediction accuracy.  
 We also refer to the work of 
\cite{meister2011} who  showed the asymptotic equivalence of the scalar-on-function linear regression and the Gaussian white noise model in the Le Cam's sense.
Besides estimation, the problem of  testing  the hypotheses
\begin{equation} \label{class}
  H_0:    \beta_0  =  \beta_*\ 
  \mbox{~~~versus ~~}   H_1: \beta_0  \not =   \beta_*~,
\end{equation}
for a prespecified function $\beta_*$ in the scalar-functional linear regression model has been discussed intensively  
 \citep[see][among others]{cardot2003b,cardot2004,hilgertetal2010,lei2014,kong2016,qu2017}. There also exist several proposals to construct $L^2$-based confidence regions \citep[see][among others]{muller2005,imaizumi2019}.

Non-linear and semiparametric scalar-on-function regression models, such as generalized   linear  models and the Cox model,
have been studied by   \cite{shang2015}, \cite{li2020}  and  \cite{hao2021}.
For the  function-on-function model \eqref{model}, the literature is more scarce.
\cite{lian2015} studied the minimax prediction rate in an   RKHS, where regularization of the estimator is only performed in one argument, while \cite{scheipl2016} investigated 
 a penalized B-spline approach.   \cite{benatia2017} used  Tikhonov regularization, \cite{luo2017} proposed a so-called signal compression approach and  \cite{sun2018} considered  a tensor product    RKHS approach 
to estimate the slope surface and the achieved the minimax prediction risk.

This list of references is by no means complete, but a
common feature of most of the work in this context consists in the fact that 
statistical methodology is developed in a Hilbert space framework (often the space or a subspace of the square-integrable functions on an interval), which means that the 
statistical properties of estimators, tests and confidence regions for the slope parameter are usually described in terms of a   norm  corresponding
to a Hilbert space. While this is convenient from a theoretical point of view and also reflects the mathematical structure of the (integral) operator of the functional linear model, it has some drawbacks from a practical perspective. In applications, using a metric that reflects the visualization of the curve/surface is usually more desirable, since functions/surfaces 
with a small  difference with respect to an $L^2$-type distance  can differ significantly in terms of maximum deviation. For example, a  confidence region of the slope function/surface based on an $L^2$-type distance 
 is often hard
to visualize   and does not  give much information about the shape of the curve or surface.

The choice of the metric also matters  if one takes a more careful look at the formulation of the hypotheses in \eqref{class}.
We argue that, in many  regression problems, it is very unlikely that the unknown  slope   $\beta_0$ coincides  with a pre-specified function/surface $\beta_*$ 
on its  complete domain,   and as a consequence, 
 testing the null hypothesis in \eqref{class} might be questionable in such cases.
Usually, hypotheses of the form \eqref{class} are formulated with the intention to investigate the question whether the effect 
of the predictor on the response can be approximately
described by the function/surface $\beta_*$, such that 
the difference $\beta_0 - \beta_*$
is in some sense ``small''.  This question can be better answered
by testing the hypotheses  of a {\it relevant difference} 
 \begin{equation} \label{rel}
    H_0:   \| \beta_0  - \beta_*\| \leq   \Delta  ~~\mbox{ 
 versus   ~~}   H_1:   \|  \beta_0 - \beta_*  \|> \Delta\,,
 \end{equation}
  where $\| \cdot \|$ denotes a norm and $\Delta > 0$  defines a threshold. 
 Hypotheses  of this type 
 have recently found some interest in functional data analysis 
 \citep[see, for example,][]{fogarty2014,dette2020aos},
and here the choice of the norm matters, 
 as different norms define different hypotheses.
  One may also view the  choice of the threshold  $\Delta $ as a particular perspective of a bias-variance trade-off,
which depends sensitively on the specific application, and, of course, also on the metric under consideration. In particular, we argue that the specification of the threshold in \eqref{rel}  is more accessible  for 
a  norm which reflects the visualization, 
such as the sup-norm. 

In the present paper, we address these issues and provide new   statistical methodology for the function-on-function linear regression model \eqref{model}
if inference is based on the maximum deviation. We propose  an estimator  for the slope surface $\beta_0$ 
minimizing an integrated   squared error loss
with a  thin-plate spline smoothness penalty functional,
and 
prove its   minimax  optimality 
using an RKHS framework. 
Based on a Bahadur representation, we establish the 
 weak convergence of this 
 estimator  as a process  in the Banach space $C([0,1]^2)$ with a Gaussian limiting process. As the covariance structure of this process 
 is not easily accessible, we develop a multiplier bootstrap to obtain
 quantiles for the distribution of functionals of the limiting process.
 In contrast to the $L^2$-metric based methods, this enables us  to construct 
 simultaneous asymptotic $(1-\alpha)$-confidence regions
for the slope surface $\beta_0$ in model \eqref{model}.
Moreover, we also provide an efficient 
solution  to the problem of testing for a relevant deviation from a given function $\beta_*$
with respect to the sup-norm. Here, we 
combine the developed bootstrap methodology 
with  estimates of the {\it extremal  set} 
of the function $\beta_0 - \beta_*$,  and develop 
an asymptotic level $\alpha$-test for the relevant   hypotheses in \eqref{rel},
where the norm is given by the sup-norm.
 Although we mainly concentrate on the model \eqref{model}, 
it is worth mentioning that, as a special case,
our approach provides also new  methods   for the scalar-on-function linear regression model \eqref{eq:scalar}, which allows
inference with respect to the sup-norm. 

The rest of this article is organized as follows. In Section~\ref{sec:method}, we propose our RKHS methodology of function-on-function linear regression and 
 study the asymptotic properties of our estimator in Section~\ref{sec:asymptotic}.
 Section~\ref{sec:sc} 
discusses several statistical applications of our results and the finite sample properties
of the proposed methodology 
 are illustrated in Section~\ref{sec:finite}. 
Finally, the technical details and proofs of  our theoretical results are given  in the online supplementary material.

\section{Function-on-function linear regression}\label{sec:method}

Suppose that  $(X_1,Y_1),\ldots,(X_n,Y_n)$
 are independent identically distributed random variables defined  by  the  function-on-function regression model in \eqref{model}, 
where $\epsilon_i$ is the centred random noise, and the slope surface $\beta_0$ is defined on $[0,1]^2$. 
For the sake of 
brevity, throughout  this article, we assume that the observed curves, i.e., $X_i$ and $Y_i$ in \eqref{model}, are centred, that is, $\E\{X(s)\}=\E\{Y(t)\}=0$, for any $(s,t)\in[0,1]^2$, so that we may ignore the intercept function $\alpha_0$, since $\alpha_0(t)=\E\{Y(t)\}-\int_0^1\beta_0(s,t)\,\E\{X(s)\}ds$. In this case, the function-on-function linear regression model in \eqref{model} becomes
\begin{align}\label{model0}
Y_i(t)=\int_0^1\beta_0(s,t)\,X_i(s)\,ds+\epsilon_i(t)\,,\quad 1\leq i\leq n\,,
\end{align}
and a similar relation can be derived for the model 
\eqref{eq:scalar}.

In the sequel, we use $L^2([0,1])$ and $L^2([0,1]^2)$ to denote the  space of square-integrable functions on  $[0,1]$ and $[0,1]^2$, respectively, and the corresponding inner product is denoted by $\l\cdot,\cdot\r_{L^2}$. By $C([0,1]^2)$ we  denote the Banach space of continuous functions on $[0,1]^2$ equipped with the supremum norm $\Vert\cdot\Vert_\infty$,  by 
``$\weakconverge$"  we denote  weak convergence in $C([0,1])$ and $C([0,1]^2)$, and ``$\converged$" stays for  convergence in distribution in $\mathbb{R}^k$ (for some positive integer $k$).


We start by proposing a RKHS approach for estimating the slope surface $\beta_0$ in  model \eqref{model0}, and define by 
\begin{align}\label{H}
\H=\Big\{\beta:[0,1]^2\to\mathbb{R}\,\big|\,&\tfrac{\partial^{\theta_1+\theta_2}\beta}{\partial s^{\theta_1}\partial t^{\theta_2}}\text{ is absolutely continuous, for }0\leq\theta_1+\theta_2\leq m-1\,;\notag\\
&\tfrac{\partial^{\theta_1+\theta_2}\beta}{\partial s^{\theta_1}\partial t^{\theta_2}}\in L^2([0,1]^2),\text{ for }\theta_1+\theta_2=m\Big\}
\end{align}
the Sobolev space  of order $m>1$ on $[0,1]^2$. It is known (see, for example, \citealp{whaba1990}) that $\H$ in \eqref{H} is a Hilbert space equipped with the Sobolev norm defined by
\begin{align}\label{snorm}
\Vert\beta\Vert_{\H}^2=\sum_{0\leq\theta_1+\theta_2\leq m-1}{\theta_1+\theta_2\choose \theta_1}\left(\int \frac{\partial^{\theta_1+\theta_2}\beta}{\partial s^{\theta_1}\partial t^{\theta_2}}\right)^2+\sum_{\theta_1+\theta_2=m}{m\choose\theta_1}\int \left(\frac{\partial^m\beta}{\partial s^{\theta_1}\partial t^{\theta_2}}\right)^2\,.
\end{align}
We propose to estimate $\beta_0$ in model \eqref{model0}  by 
\begin{align}\label{hatbeta}
\hat\beta_{n}&=\underset{\beta\in \H}{\arg\min}\ \big\{ L_n(\beta)+(\lambda/2) J(\beta,\beta)\big\}\,,
\end{align}
where
\begin{align}\label{elln}
L_n(\beta)=\frac{1}{2n}\sum_{i=1}^n\int_0^1\left\{Y_i(t)-\int_0^1\beta(s,t)X_i(s)\,ds\right\}^2dt\,
\end{align}
is the integrated squared loss functional, $\lambda > 0$
 is a regularization parameter, and for $m>1$,
\begin{align}\label{jm}
&J(\beta_1,\beta_2)=\sum_{\theta=0}^m{m\choose \theta}\int_0^1\int_0^1\frac{\partial^{m}\beta_1}{\partial s^{\theta}\,\partial t^{m-\theta}}\times\frac{\partial^{m}\beta_2}{\partial s^{\theta}\,\partial t^{m-\theta}}\,ds\,dt\,
\end{align}
is the thin-plate spline smoothness penalty functional (see, for example, \citealp{wood2003}).

In \eqref{hatbeta}, for notational brevity, we suppress the dependence of $\hat\beta_n$ on $\lambda$, and denote by 
\begin{equation}
\label{obj}
    L_{n,\lambda}(\beta)= L_n(\beta)+(\lambda/2) J(\beta,\beta)
\end{equation} 
 the objective function in \eqref{hatbeta}. For $\beta_1,\beta_2\in\H$, we consider the following map $\l\cdot,\cdot\r_K:\H\times\H\to\mathbb{R}$ defined by
\begin{align}\label{inner}
\langle\beta_1,\beta_2\rangle_K=V(\beta_1,\beta_2)+\lambda J(\beta_1,\beta_2)\,,
\end{align}
where 
\begin{align}\label{v}
&V(\beta_1,\beta_2)=\int_{[0,1]^3}C_X(s_1,s_2)\,\beta_1(s_1,t)\,\beta_2(s_2,t)\,ds_1\,ds_2\,dt\,
\end{align}
and  
\begin{align} \label{m2a}
C_X(s_1,s_2)=\cov\{X(s_1),X(s_2)\}
\end{align}
denotes the the covariance function of the predictor. We first make the following mild assumption on $C_X$.
\begin{assumption}\label{a1}
$C_X$ is continuous on $[0,1]^2$. For any $\gamma\in L^2([0,1])$, $\int_0^1C_X(s,s')\gamma(s)ds=0$ implies that $\gamma\equiv0$.  
\end{assumption}

Our first result, which is proved in Section~\ref{app:prop:equinorm}, shows that
the relation  $\langle\cdot,\cdot\rangle_K$ in  \eqref{inner} defines an  inner product on the space
$\H$ in \eqref{H},  and its corresponding norm is
 equivalent to the Sobolev norm $\Vert\cdot\Vert_\H$ given in \eqref{snorm}.

\begin{proposition}\label{prop1}
If  Assumption~\ref{a1} holds,   then  $\langle\cdot,\cdot\rangle_K$ in  \eqref{inner} is a well-defined inner product on $\H$. If $\Vert\cdot\Vert_K$ denotes its corresponding norm, then  $\Vert\cdot\Vert_K$ and $\Vert\cdot\Vert_{\H}$ in \eqref{snorm} are equivalent. Moreover, $\H$ is a reproducing kernel Hilbert space equipped with the inner product $\langle\cdot,\cdot\rangle_K$. 
\end{proposition}

In the sequel, for $s_1,s_2,t_1,t_2\in[0,1]$, let $K\{(s_1,t_1), (s_2,t_2)\}$ denote the reproducing kernel of the reproducing kernel Hilbert space $\H$ equipped with inner product $\langle\cdot,\cdot\rangle_K$. For functions $x,y$ on $[0,1]$, let $x\otimes y$ denote the function defined by $x\otimes y(s,t)=x(s)y(t)$. We use $\sum_{k,\ell}$ to denote the sum $\sum_{k=1}^{\infty}\sum_{\ell=1}^{\infty}$ for abbreviation. Let $\C_X$ denote the covariance operator of $X$ defined by 
 \begin{align}\label{m1a}
 \C_X(x)=\int_0^1 C_X(s,\cdot)\,x(s)\,ds\,, 
 \end{align}
for $x\in L^2([0,1]).$
We assume that there exists a sequence of functions in $\H$ that diagonalizes operators $V$ in  \eqref{v} and $J$ in  \eqref{jm} simultaneously. A concrete example that satisfies the following assumption 
will be provided in Section~\ref{app:aux:cond} of the online supplement.

\begin{assumption}[Simultaneous diagonalization]\label{a201}
There exists a sequence of functions $\phi_{k\ell} =x_{k\ell} \otimes \eta_\ell \in\H$, such that $\Vert \phi_{k\ell}\Vert_{\infty}\leq c (k\ell)^{a}$ for any $k,\ell\geq 1$, and
\begin{align}\label{diag}
V(\phi_{k\ell},\phi_{k'\ell'})=\delta_{kk'}\,\delta_{\ell\ell'}\,,\qquad J(\phi_{k\ell},\phi_{k'\ell'})=\rho_{k\ell}\,\delta_{kk'}\,\delta_{\ell\ell'}\,,
\end{align}
where $a\geq 0$, $c>0$ are constants, $\delta_{kk'}$ is the Kronecker delta and $\rho_{k\ell}$ are constants, such that $\rho_{k\ell}\asymp (k\ell)^{2D}$ for some constant $D>a+1/2$. Furthermore, any $\beta\in\H$ admits the expansion 
$$
\beta=\sum_{k,\ell}V(\beta,\phi_{k\ell})\phi_{k\ell}
$$
with convergence in $\H$ with respect to the norm $\Vert\cdot\Vert_K$.
\end{assumption}

Note that a  similar diagonalization assumption  has been made  in \cite{shang2015} in the context of generalized scalar-on-function model.
For the inner product $\l\cdot,\cdot\r_K$ in \eqref{inner} it follows from Assumption~\ref{a201} that
$$
\l\phi_{k\ell},\phi_{k'\ell'}\r_K=V(\phi_{k\ell},\phi_{k'\ell'})+\lambda J(\phi_{k\ell},\phi_{k'\ell'})=(1+\lambda\rho_{k\ell})\,\delta_{kk'}\,\delta_{\ell\ell'}~~(k,k',\ell,\ell'\geq 1).
$$
Therefore, it follows  $\l\beta,\phi_{k'\ell'}\r_K=\sum_{k,\ell}V(\beta,\phi_{k\ell})\l\phi_{k\ell},\phi_{k'\ell'}\r_K=(1+\lambda\rho_{k\ell})V(\beta,\phi_{k\ell})$  for any $\beta\in\H$, so that
\begin{align}\label{expansion}
\beta(s,t)=\sum_{k,\ell}V(\beta,\phi_{k\ell})\phi_{k\ell}(s,t)=\sum_{k,\ell}\frac{\l\beta,\phi_{k\ell}\r_K}{1+\lambda\rho_{k\ell}}\phi_{k\ell}(s,t)\,.
\end{align}
Recall that $K$ is the reproducing kernel and using the notation  $K_{(s,t)}=K\{(s,t),\cdot\}$ we have $\phi_{k\ell}(s,t)=\l K_{(s,t)},\phi_{k\ell}\r_K$, so that by \eqref{expansion},
\begin{align}\label{kst2}
K_{(s,t)}=\sum_{k,\ell}\frac{\phi_{k\ell}(s,t)}{1+\lambda\rho_{k\ell}}\phi_{k\ell}\,;\qquad  K\{(s_1,t_1), (s_2,t_2)\}=\sum_{k,\ell}\frac{\phi_{k\ell}(s_1,t_1)\phi_{k\ell}(s_2,t_2)}{1+\lambda\rho_{k\ell}}\,.
\end{align}
For $\beta_1,\beta_2\in\mathcal{H}$, let $W_\lambda$ denote a linear self-adjoint operator such that $\l W_\lambda\beta_1,\beta_2\r_K=\lambda J(\beta_1,\beta_2)$. By definition, for the $\{\phi_{k\ell}\}_{k,\ell\geq 1}$ in Assumption~\ref{a201}, we have $\l W_\lambda\phi_{k\ell},\phi_{k'\ell'}\r_K=\lambda J(\phi_{k\ell},\phi_{k'\ell'})=\lambda\rho_{k\ell}\,\delta_{kk'}\,\delta_{\ell\ell'}$, so that in view of \eqref{expansion},
\begin{align}\label{wphi}
W_\lambda\phi_{k\ell}=\sum_{k',\ell'}\frac{\l W_\lambda\phi_{k\ell},\phi_{k'\ell'}\r_K}{1+\lambda\rho_{k'\ell'}}\phi_{k'\ell'}=\frac{\lambda\,\rho_{k\ell}\,\phi_{k\ell}}{1+\lambda\rho_{k\ell}}\,.
\end{align}
For any $z\in L^2([0,1]^2)$ and $\beta\in\H$, $\mathfrak S_{z}(\beta)=\int_0^1\int_0^1\beta(s,t)z(s,t)dsdt$ is a bounded linear functional. By the Riesz representation theorem, there exists a unique element $\tau(z)\in\mathcal{H}$ such that
\begin{align}\label{tau0}
\l\tau(z),\beta\r_K=\mathfrak S_{z}(\beta)=\int_0^1\int_0^1\beta(s,t)\,z(s,t)\,ds\,dt\,.
\end{align}
In particular, $\l\tau(z),\phi_{k\ell}\r_K=\l z,\phi_{k\ell}\r_{L^2}$, so that
\begin{align}\label{tau}
\tau(z)=\sum_{k,\ell}\frac{\l z,\phi_{k\ell}\r_{L^2}}{1+\lambda\rho_{k\ell}}\phi_{k\ell}\,.
\end{align}

\section{Asymptotic properties}\label{sec:asymptotic}

In order to develop statistical methodology for inference on the  slope surface in the function-on-function linear regression model \eqref{model},
we study in this section the asymptotic properties of the  estimator $\hat \beta_n$ defined by \eqref{hatbeta}. 
We first present a Bahadur representation,
which is used to prove weak convergence of the estimator (point-wise and as process 
 in $C([0,1]^2)$). Several statistical applications of the following results will be  given in Section \ref{sec:sc} below.

We begin introducing  several useful quantities. Recalling the notation of  $W_\lambda$ and $\tau$ defined in \eqref{wphi} and \eqref{tau}, respectively, we obtain by direct calculations the first and second order Fr\'echet derivatives of the integrated squared error $ L_n$ in \eqref{elln}
\begin{align}\label{dell}
\mathcal D L_n(\beta)\beta_1&=-\frac{1}{n}\sum_{i=1}^n\int_{0}^1\int_0^1\bigg\{Y_i(t)-\int_{0}^1\beta(s_1,t)\,X_i(s_1)ds_1\bigg\}\,X_i(s_2)\,\beta_1(s_2,t)\,ds_2\,dt\notag\\
&=-\frac{1}{n}\sum_{i=1}^n \Bigg\l\tau\bigg[X_i\otimes\bigg\{Y_i-\int_{0}^1\beta(s,\cdot)\,X_i(s)\,ds\bigg\}\bigg],\,\beta_1\Bigg\r_K\,;\notag\\
\mathcal D^2 L_n(\beta)\beta_1\beta_2&=\frac{1}{n}\sum_{i=1}^n\int_0^1\int_0^1\left\{\int_{0}^1\beta_1(s_1,t)\,X_i(s_1)\,ds_1\right\}X_i(s_2)\,\beta_2(s_2,t)\,ds_2\,dt\notag\\
&=\frac{1}{n}\sum_{i=1}^n \Bigg\l\tau\bigg[X_i\otimes\bigg\{\int_{0}^1\beta_1(s,\cdot)\,X_i(s)\,ds\bigg\}\bigg],\,\beta_2\Bigg\r_K\,.
\end{align}
Therefore, it follows for the function 
$ L_{n,\lambda}$ in \eqref{obj} that 
\begin{align*}
    \mathcal D L_{n,\lambda}(\beta)\beta_1 &=\l S_{n,\lambda}(\beta),\beta_1\r_K ~,\\
    \mathcal D^2 L_{n,\lambda}(\beta)\beta_1\beta_2&=\l \mathcal DS_{n,\lambda}(\beta)\beta_1,\beta_2\r_K~,
\end{align*}
where we use the notations
\begin{equation}
 \label{dsn} 
 \begin{split}     
S_{n,\lambda}(\beta)&=-\frac{1}{n}\sum_{i=1}^n\tau\bigg[X_i\otimes\bigg\{Y_i-\int_{0}^1\beta(s,\cdot)\,X_i(s)\,ds\bigg\}\bigg]+W_\lambda(\beta)\,,\\
\mathcal DS_{n,\lambda}(\beta)\beta_1&=\frac{1}{n}\sum_{i=1}^n\tau\bigg[X_i\otimes\bigg\{\int_{0}^1\beta_1(s,\cdot)\,X_i(s)\,ds\bigg\}\bigg]+W_\lambda(\beta_1)\,.
\end{split}
\end{equation}
In addition,  we have, in view of \eqref{inner}, 
$$
\E\{\mathcal D^2 L_{n,\lambda}(\beta)\beta_1\beta_2\}=V(\beta_1,\beta_2)+\lambda J(\beta_1,\beta_2)=\l\beta_1,\beta_2\r_K.
$$

\begin{remark}
In the unregularized case, where $\lambda=0$, 
one can show that  setting $\E\{S_{n,0}(\beta)\}=0$ yields the common  FPC expansion 
\begin{align}\label{tildebeta}
\beta=\sum_{k,\ell}\frac{\E(\l X,u_k\r_{L^2}\l Y,v_\ell\r_{L^2})}{\E(\l X,u_k\r_{L^2}^2)}\,u_k\otimes v_\ell\,,
\end{align}
where the $u_{k}$'s and the $v_\ell$'s are eigenfunctions of the covariance functions of $X$ and $Y$,
respectively  \citep[see, for example,][]{yao2005}. 
To see this, in view of the definition of $\tau$ in \eqref{tau0}, note that $\tau(0)=0$ and that $\tau(z)=0$ implies $z=0$. Hence, $\E\{S_{n,0}(\beta)\}=\E\big(\tau\big[X\otimes\big\{Y-\int_0^1\beta(s,\cdot)X(s)ds\big\}\big]\big)=0$ implies that $\E(X\otimes Y)=\E\big[X\otimes\big\{\int_0^1\beta(s,\cdot)X(s)ds\big\}\big]$, so that $\cov\{X(s),Y(t)\}=\int_0^1\beta(s',t)C_X(s,s')ds'$. Then, \eqref{tildebeta} can be obtained via the expansion of this equation on both sides  with respect to the eigenfunctions.
\end{remark}

%

We  now state several assumptions required for the asymptotic theory developed in this section.

\begin{assumption}\label{a0}
For any $t_1,t_2\in[0,1]$, $\E\{\epsilon(t_1)|X\}=0$ and $\E\{\epsilon(t_1)\epsilon(t_2)|X\}=C_{\epsilon}(t_1,t_2)$ almost surely. Assume further that $C_\e(t_1,t_2)=\sigma_\e^2\,\delta(t_1,t_2)$, for some $\sigma_\e^2>0$, where $\delta$ is the Dirac-delta function. 
\end{assumption}

\begin{assumption}\label{a:x}
There exist constants $c_X,c_\e>0$ such that $\E\{\exp(c_X\Vert X\Vert_{L^2})\}< \infty$ and $\E\{\exp(c_\e\Vert \e\Vert_{L^2})|X\}< \infty$ almost surely. There exists a constant $c'>0$ such that, for any $\omega\in L^2([0,1])$,
\begin{align}\label{moment}
\E\bigg\{\int_0^1 X(s)\,\omega(s)ds\bigg\}^4\leq c'\,\bigg[\E\bigg\{\int_0^1 X(s)\,\omega(s)ds\bigg\}^2\,\bigg]^2\,.
\end{align}
Moreover,  $\E\{\e( \cdot )\e( \cdot )\e( \cdot )\e( \cdot  )|X\}\in L^2([0,1]^4)$ almost surely.
\end{assumption}

\begin{assumption}\label{a:rate}
The regularization parameter 
$\lambda$ in \eqref{inner} satisfies  $\lambda=o(1)$, $n^{-1}\lambda^{-1/(2D)}=o(1)$ and $n^{-1/2}\lambda^{-\varsigma}\,(\log\log n)^{1/2}=o(1)$   as $n\to\infty$, where $\varsigma=(2D-2a-1)/(4Dm)+(a+1)/(2D)>0$, for the constants $a$ and $D$ in Assumption~\ref{a201}.
\end{assumption}

\begin{remark}  \label{remass}
The reason for postulating a  white-noise error covariance in Assumption~\ref{a0} is that the commonly used $L^2$ loss function defined in \eqref{elln} corresponds to the likelihood function in the case of the Gaussian white noise error process; see, for example, \citealp{wellner2003}. 
It is also notable, that for  the scalar-on-function model in \eqref{eq:scalar}, this assumption  is in fact not necessary (as there is no error function in this model) and   Assumption~\ref{a0}
reduces to $\E(\e|X)=0$ and $\E(\e^2|X)=\sigma_\e^2$ almost surely. 

Assumption~\ref{a:x} requires that $\Vert X\Vert_{L^2}$, and $\Vert \e\Vert_{L^2}$ conditional on $X$, have finite exponential moments. Moreover, condition \eqref{moment} is a common moment assumption in the context of scalar-on function regression, used, for example, in \cite{caiyuan2012} and \cite{shang2015}. Assumption~\ref{a:rate} specifies the condition on the rate in which $\lambda$ tends to zero as  $n\to\infty$.

\end{remark}

The first result of this section  establishes a  Bahadur representation for the estimator \eqref{hatbeta} in the function-on-function linear regression model \eqref{model0}.
It is essential  for deriving  weak convergence of  the estimator $\hat\beta_{n}$, which  serves as the foundation of our statistical analysis in Section~\ref{sec:sc}.
The proof of Theorem~\ref{thm:bahadur} is given in Section~\ref{app:thm:bahadur}.

\begin{theorem}[Bahadur representation]\label{thm:bahadur}
Suppose Assumptions~\ref{a1}--\ref{a:rate} are satisfied. Then,     we have 
\begin{align*}
\Vert\hat\beta_{n}-\beta_0+S_{n,\lambda}(\beta_0)\Vert_K=O_p(v_n), 
\end{align*}
where $S_{n,\lambda}$ is defined in \eqref{dsn}, $\varsigma>0$ is the constant  in Assumption~\ref{a:rate} and
\begin{align} \label{vnn}
v_n = n^{-1/2}\lambda^{-\varsigma}(\lambda^{1/2}+n^{-1/2}\lambda^{-1/(4D)})(\log\log n)^{1/2} \,.
\end{align}
\end{theorem}


Due to the reproducing property of the kernel $K$, we have, for $(s,t)\in[0,1]^2$ fixed,
$$
\hat\beta_{n}(s,t)-\beta_0(s,t)=\l \hat\beta_{n}-\beta_0,K_{(s,t)}\r_K\,,
$$  and, by Theorem~\ref{thm:bahadur}, this 
expression can be linearized to  establish  point-wise asymptotic normality of $\hat\beta_n(s,t)$.
The following theorem gives a rigorous formulation of 
 these heuristic arguments and is proved in Section~\ref{app:thm:pointwise}.
\begin{theorem}\label{thm:pointwise}
Suppose that Assumptions~\ref{a1}--\ref{a:rate} hold. 
Assume $n\lambda^{(2a+1)/(2D)}\{\log(\lambda^{-1})\}^{-4}\to \infty$, $\sqrt{n}v_n=o(1)$, $n\lambda^2=o(1)$, $\sum_{k,\ell}(1+\lambda\rho_{k\ell})^{-2}\phi_{k\ell}^2(s,t)\asymp \lambda^{-(2a+1)/(2D)}$, as $n\to\infty$; $\sum_{k,\ell}\rho_{k\ell}^2\,V^2(\beta_0,\phi_{k\ell})< \infty$. Then,
\begin{align*}
\frac{\sqrt{n}\{\hat\beta_{n}(s,t)-\beta_0(s,t)\}}{\sqrt{\sum_{k,\ell}(1+\lambda\rho_{k\ell})^{-2}\phi_{k\ell}^2(s,t)}}\overset{d.}{\longrightarrow} N(0,1)\,.
\end{align*} 
\end{theorem}


The final result of this section    establishes the weak convergence of the process $\hat\beta_{n}$ in the space $C([0,1]^2)$, which enables us to construct simultaneous confidence regions for the slope surface $\beta_0$ (see Section~\ref{sec:simultaneousband} below). The proof  is given in Section~\ref{app:thm:process}.

\begin{theorem}\label{thm:process}
Suppose that Assumptions~\ref{a1}--\ref{a:rate} hold
and that
$n\lambda^2=o(1)$, $\sqrt{n}v_n=o(1)$, $\lambda=o(n^{-\nu_1})$, $n^{1-\nu_2}\lambda^{(2a+1)/(2D)}\to \infty$ for some constants $\nu_1,\nu_2>0$ as $n\to\infty$. Assume that
$\sum_{k,\ell}\rho_{k\ell}^2\,V^2(\beta_0,\phi_{k\ell})< \infty$, that 
the limit 
\begin{align}\label{cz}
C_{Z}\{(s_1,t_1),(s_2,t_2)\}=\lim_{n\to\infty}\lambda^{(2a+1)/(2D)}\sum_{k,\ell}\frac{\phi_{k\ell}(s_1,t_1)\phi_{k\ell}(s_2,t_2)}{(1+\lambda\rho_{k\ell})^{2}}
\end{align}
exists, and that there exist nonnegative constants $c_0,b,\vartheta$ such that
\begin{align}\label{holder}
\limsup_{\lambda\to0}\,\lambda^{(2b+1)/(2D)}\sum_{k,\ell}\frac{|\phi_{k\ell}(s_1,t_1)-\phi_{k\ell}(s_2,t_2)|^2}{(1+\lambda\rho_{k\ell})^2}\leq c_0\max\{|s_1-s_2|^{2\vartheta},|t_1-t_2|^{2\vartheta}\}\,.
\end{align}
If  the constant $a$ in Assumption~\ref{a201} either satisfies the condition (i) $b< a$, $\vartheta\geq0$ or the condition 
(ii) $b=a$, $\vartheta>1$, then 
\begin{align}\label{hatgn}
\{ \G_n(s,t)\}_{s,t \in [0,1]}=\sqrt{n}\lambda^{(2a+1)/(4D)}\{\hat\beta_{n}(s,t)-\beta_0(s,t)\}_{s,t \in [0,1]} \weakconverge \{Z(s,t)\}_{s,t \in [0,1]}
\end{align}
in $ C([0,1]^2) $,
where ${Z}$ is a mean-zero Gaussian process with  covariance kernel $C_{Z}$ 
in \eqref{cz}.
\end{theorem}

\section{Statistical consequences}\label{sec:sc}

In this section, we study several statistical inference problems regarding the model in \eqref{model0}. We first show in Section~\ref{sec:minimax} that the estimator $\hat\beta_n$ in \eqref{hatbeta} achieves the minimax convergence rate. In Section~\ref{sec:simultaneousband}, we propose point-wise and  simultaneous confidence regions for the slope surface $\beta_0$.
In Section~\ref{sec:ch}, we develop a new test for the 
the classical hypotheses \eqref{class} based on the sup-norm using the duality between confidence regions and hypotheses testing. Moreover, 
we also extend the penalized likelihood-ratio test for scalar-on-function linear regression proposed in \cite{shang2015} to the function-on-function linear regression model (a numerical comparison of both tests can be found in Section \ref{sec:simulation} and shows some superiority of the confidence region approach.)  In Section~\ref{sec:rele}, we study a test for a relevant deviation of the ``true'' slope function and a given function $\beta_*$. Finally, a simultaneous prediction band for the conditional mean curve $\E\{Y(t)|X=x_0\}$ is proposed in Section~\ref{sec:prediction}. The methodology requires knowledge of the constants $a$ and $D$ in Assumption \ref{a201}, and a data driven rule for this choice will be given in Section \ref{sec:numerical}.

We also emphasize that, although we are mainly concentrating on the function-on-function linear regression model,
all results  presented so far also hold for the scalar-on-function linear model (under even weaker assumptions).
As a consequence, we also obtain new powerful 
methodology for the scalar-on-function linear regression model \eqref{eq:scalar} as well, 
and  we briefly illustrate this  fact for the 
problem of testing relevant hypotheses in Section~\ref{sec:scalar}. 

\subsection{Optimality}\label{sec:minimax}

Under Assumption~\ref{a1}, the  operator  $V$ in \eqref{v} defines a norm, say $\Vert\beta\Vert_V^2=V(\beta,\beta)$, on $\H$.
As a by-product of the Bahadur representation in Theorem~\ref{thm:bahadur}, we are able to show the upper bound for the convergence rate of the estimator $\hat\beta_n$  in  \eqref{hatbeta}  with respect to  the $\Vert\cdot\Vert_V$-norm. Moreover, we also prove that this rate is of the same order as the lower bound for  estimating $\beta_0$, which shows that $\hat\beta_n$  is minimax optimal. To be precise, let $\mathcal G$ denote the collection of all estimators from the data $(X_1,Y_1),\ldots,(X_n,Y_n)$, and let $\mathcal F$ denote the collection of the joint distribution $F$ of the $X$ and $Y$ that satisfies Assumptions~\ref{a1}--\ref{a:x}, according to the linear model in \eqref{model0}. 
The following theorem
is proved in  Section~\ref{app:thm:minimax}.

\begin{theorem}[Optimal convergence rate]
\label{thm:minimax}
Suppose Assumptions~\ref{a1}--\ref{a:rate} hold. \\
(i) By taking $\lambda\asymp n^{-2D/(2D+1)}$, we have
\begin{align*}
\lim_{c\to \infty}\,\limsup_{n\to\infty}\,\sup_{\substack{\beta_0\in\mathcal{\H}\\F\in\mathcal{F}}}\,\P\big(\Vert\hat\beta_{n}-\beta_0\Vert_V^2\geq c\,n^{-2D/(2D+1)}\big)=0\,.
\end{align*}
(ii) There exists a constant $c_0>0$ such that
\begin{align*}
\liminf_{n\to\infty}\,\inf_{\tilde\beta\in\mathcal{G}}\,\sup_{\substack{\beta_0\in\mathcal{\H}\\F\in\mathcal{F}}}\,\P\big(\Vert\tilde\beta-\beta_0\Vert_V^2\geq c_0\,n^{-2D/(2D+1)}\big)>0\,.
\end{align*}
\end{theorem}



Theorem~\ref{thm:minimax} shows that the  estimator $\hat\beta_n$ in \eqref{hatbeta}
achieves the  minimax optimal convergence rate $n^{-2D/(2D+1)}$ with respect to the $\Vert\cdot\Vert_V$-norm.  It is of interest to compare  this result 
with the minimax prediction rate obtained in \cite{sun2018}. First, we consider the estimation of the slope surface $\beta_0$ in the Sobolev space $\H$ on the square $[0,1]^2$ defined in \eqref{H}, whereas \cite{sun2018} considered a tensor product RKHS on $[0,1]^2$. Second, \cite{sun2018} showed their minimax properties in terms of {\it excess prediction rate} (hereinafter denoted by EPR), defined by
\begin{align}\label{epr}
{\rm EPR}(\tilde\beta_n)&=\int_0^1\E_{n+1}\bigg\{\left|Y_{n+1}-\int_0^1\tilde\beta_n(s,t)X_{n+1}(s)ds\right|^2-\left|Y_{n+1}-\int_0^1\beta_0(s,t)X_{n+1}(s)ds\right|^2\bigg\}dt\notag\\
&=\int_0^1\E_{n+1}\left|\int_0^1\big\{\tilde\beta_n(s,t)-\beta_0(s,t)\big\} X_{n+1}(s)ds\right|^2dt\,.
\end{align}
Here  $\tilde\beta_n$ is  an estimator from  the data $\{(X_i,Y_i)\}_{i=1}^n$,  $(X_{n+1},Y_{n+1})$ is an  independent future observation and $\E_{n+1}$ is the 
conditional expectation  with respect to $(X_1,Y_1),\ldots,(X_n,Y_n)$ (which means that the  expectation is
taken with respect to  $(X_{n+1},Y_{n+1})$). In fact, we have 
\begin{align*}
{\rm EPR}(\tilde\beta_n)&=\int_0^1\bigg[\int_{0}^1\int_0^1C_X(s_1,s_2)\{\tilde\beta_n(s_1,t)-\beta_0(s_1,t)\} \{\tilde\beta_n(s_2,t)-\beta_0(s_2,t)\}ds_1ds_2\bigg]dt\\
&=V(\tilde\beta_n-\beta_0,\tilde\beta_n-\beta_0)=\Vert\tilde\beta_n-\beta_0\Vert_V^2\,,
\end{align*}
which shows  that the difference between $\tilde\beta_n$ and the true $\beta_0$ in squared $\Vert\cdot\Vert_V$-norm is equivalent to ${\rm EPR}(\tilde\beta_n)$. Therefore, 
it follows from Theorem~\ref{thm:minimax}  that for  the  estimator $\hat\beta_n$ in  \eqref{hatbeta}, ${\rm EPR}(\hat\beta_n)$ achieves the minimax rate $n^{-2D/(2D+1)}$, which is determined by the constant $D>0$ that specifies the growing rate of $J(\phi_{k\ell},\phi_{k\ell})$ in Assumption~\ref{a201}.
In comparison, \cite{sun2018} showed that the EPR of their estimator achieves the minimax rate $n^{-2\breve D/(2\breve D+1)}$, where the constant $\breve D>0$ characterises the decay rate of eigenvalues of the kernel $$\Pi\{(s_1,t_1),(s_2,t_2)\}=\int_{[0,1]^3}C_X(s,t)\tilde K^{1/2}\{(s_1,t_1),(s,t)\}\tilde K^{1/2}\{(s_2,t_2),(s,u)\}\,ds\,dt\,du\,,$$ where $\tilde K$ is the reproducing kernel of their tensor product RKHS. 

\subsection{Confidence regions}
\label{sec:simultaneousband}

The asymptotic normality of the estimator  $\hat\beta_n(s,t)$ in Theorem~\ref{thm:pointwise} enables us to construct a  point-wise $(1-\alpha)$-confidence interval of $\beta_0(s,t)$, for fixed $(s,t)\in[0,1]^2$, since
\begin{align*}
\lim_{n\to\infty}\P\Big\{\beta_0(s,t)\in\big[\hat\beta_{n}(s,t)-\mathcal Q_{1-\alpha/2}\,\sigma_\tau(s,t)\,,\hat\beta_{n}(s,t)+\mathcal Q_{1-\alpha/2}\,\sigma_\tau(s,t)\big]\Big\}=1-\alpha\,,
\end{align*}
where $\sigma_\tau(s,t)=\big\{\sum_{k,\ell}(1+\lambda\rho_{k\ell})^{-2}\phi_{k\ell}^2(s,t)\big\}^{1/2}$ and $\mathcal Q_{1-\alpha/2}$ is the $(1-\alpha/2)$-quantile of the standard normal distribution.

On the other hand, the construction of simultaneous confidence regions  based on the sup-norm
for the slope surface $\beta_0$ is more complicated.
In principle, this is possible using 
Theorem~\ref{thm:process} and the continuous mapping theorem,
which give 
\begin{align}\label{t}
\sqrt{n}\lambda^{(2a+1)/(4D)}\sup_{(s,t)\in[0,1]^2}|\hat\beta_{n}(s,t)-\beta_0(s,t)|\converged T=\max_{(s,t)\in[0,1]^2}|Z(s,t)|\,,
\end{align}
where  $Z$  is the mean-zero Gaussian process  defined in \eqref{hatgn}.
Thus, if  $\mathcal Q_{1-\alpha}(T)$ denotes the $(1-\alpha)$-quantile of 
the distribution of  $T$ and 
\begin{align*}
\hat\beta_{n}^{\pm}(s,t)=\hat\beta_{n}(s,t)\pm \frac{\mathcal Q_{1-\alpha}(T)}{\sqrt{n}\lambda^{(2a+1)/(4D)}}\, ,
\end{align*}
then the set $\mathfrak{C}_n(\alpha)=\big\{\beta:\hat\beta_n^-(s,t)\leq\beta(s,t)\leq \hat\beta_n^+(s,t)\big\}$ defines a simultaneous asymptotic $(1-\alpha)$-confidence region for $\beta_0$, i.e., $$
\lim_{n\to\infty}\P\{\beta_0\in\mathfrak{C}_n(\alpha)\}=1-\alpha.
$$
However, the quantiles of the distribution of
$T$ depend on the  covariance function $C_Z$ in \eqref{cz}  of the Gaussian process $Z$, 
and is rarely available in practice.
In order to circumvent this difficulty,  we propose the following bootstrap procedure to approximate $\mathcal Q_{1-\alpha}(T)$.


\begin{algorithm}[Bootstrap simultaneous confidence region for the slope surface $\beta_0$]~\label{algo:cb}
\begin{enumerate}[nolistsep]

\item Generate i.i.d.~bootstrap weights $\{M_{i,q}\}_{1\leq i\leq n,1\leq q\leq Q}$ independent of the data $\{(X_i,Y_i)\}_{i=1}^n$ from a two-point distribution: taking $1-1/\sqrt{2}$ with probability $2/3$ and taking $1+\sqrt{2}$ with probability $1/3$, such that $\E(M_{i,q})=\var(M_{i,q})=1$.

\item Compute  $\hat\beta_{n}$ in \eqref{hatbeta}; for each $1\leq q\leq Q$, compute the bootstrap estimator
\begin{align}\label{hatbetastar}
\hat\beta_{n,q}^*&=\underset{\beta\in \H}{\arg\min}\bigg[\frac{1}{2n}\sum_{i=1}^nM_{i,q}\int_0^1\bigg\{Y_i(t)-\int_0^1\beta(s,t)X_i(s)ds\bigg\}^2dt+\frac{\lambda}{2}\, J(\beta,\beta)\bigg]\,.
\end{align}

\item For $1\leq q\leq Q$, let
\begin{align}\label{gnqstar}
\G_{n,q}^*(s,t)& =\sqrt{n}\lambda^{(2a+1)/(4D)}\{\hat\beta_{n,q}^*(s,t)-\hat\beta_{n}(s,t)\}\,;  \nonumber \\
 \hat T_{n,q}^* & =\sup_{(s,t)\in[0,1]^2}|\G_{n,q}^*(s,t)|\,.
\end{align}
Compute the empirical $(1-\alpha)$-quantile of the  sample $\{\hat T_{n,q}^*\}_{q=1}^Q$, denoted by $\mathcal Q_{1-\alpha}(\hat T_{n,Q}^*)$.

\item Let $\hat\beta_{n,Q}^{*^\pm}(s,t)=\hat\beta_{n}(s,t)\pm \mathcal Q_{1-\alpha}(\hat T^*_{n,Q})/\{\sqrt{n}\lambda^{(2a+1)/(4D)}\}$. Define the set 
\begin{equation}
\label{det1}
\mathfrak{C}_{n,Q}^*(\alpha)=\big\{\beta: [0,1]^2 \to \mathbb{R}  : ~ \hat\beta_{n,Q}^{*^-}(s,t)\leq\beta(s,t)\leq \hat\beta_{n,Q}^{*^+}(s,t)\big\}
\end{equation} 
as the simultaneous $(1-\alpha)$ confidence region for
the slope surface $\beta_0$ in model \eqref{model0}.
\end{enumerate}
\end{algorithm}

The following theorem, which is proved in Section~\ref{app:thm:bootstrap}, provides a theoretical justification of the above bootstrap procedure and establishes the consistency of the simultaneous confidence region in Algorithm~\ref{algo:cb}.

\begin{theorem}\label{thm:cb:boot}

Under the conditions of Theorem~\ref{thm:process} we have 
\begin{equation}
\label{conset}
\{\G^*_{n,q}(s,t)\}_{s,t\in[0,1]}\weakconverge \{Z(s,t)\}_{s,t\in[0,1]}   \text{ ~~in } C([0,1]^2)
\end{equation}
conditionally on the data $\{(X_i,Y_i)\}_{i=1}^n$, 
where $Z$ is the Gaussian process in Theorem~\ref{thm:process}.
\\
In particular, the set $\mathfrak C_n^*(\alpha)$ in Algorithm~\ref{algo:cb} defines a simultaneous asymptotic $(1-\alpha)$ confidence region for the slope surface $\beta_0$ in model \eqref{model0}, that is 
\begin{equation}
\label{consetboot}
\lim_{Q\to\infty}\lim_{n\to\infty}\P\{\beta_0\in\mathfrak{C}^*_{n,Q}(\alpha)\}=1-\alpha\,.
\end{equation}
\end{theorem}

\subsection{Classical hypotheses}\label{sec:ch}

For a given surface $\beta_*$ on $[0,1]^2$, consider the  ``classical'' hypotheses 
\begin{align}\label{eq:ch}
H_0:\,\beta_0=\beta_*\qquad&\text{versus}\qquad H_1:\,\beta_0\neq\beta_*\, .
\end{align}
 In the special case where $\beta_*\equiv0$, \eqref{eq:ch} becomes $H_0:\beta_0=0$ versus $H_1:\beta_0\neq0$, which is the conventional hypothesis for linear effect; we refer to \cite{Tekbudak} for a review in the scalar-on-function regression context.


In order to construct a test for \eqref{eq:ch}, we may utilize the duality between   hypotheses thesing and confidence regions \citep[see, for example,][]{aitchison1964}. Specifically,  recall from Section~\ref{sec:simultaneousband} that  we are able to construct a simultaneous confidence region $\mathfrak C_{n,Q}^*(\alpha)$ for $\beta_0$ using Algorithm~\ref{algo:cb}, such that
$\P\{\beta_0\in\mathfrak{C}^*_{n,Q}(\alpha)\}\to1-\alpha$ as $n,Q\to \infty$. Then, the decision rule, which rejects 
the null hypothesis, whenever
\begin{align}\label{eq:chrule}
\beta_*\notin\mathfrak C_{n,Q}^*(\alpha)\,,
\end{align}
defines an asymptotic level $\alpha$
test for the classical hypotheses in \eqref{eq:ch}.


An alternative approach to construct a test for these classical hypotheses  is to extend the penalized likelihood ratio test (hereinafter denoted by PLRT), proposed in \cite{shang2015} for the scalar-on-function regression context, to the functional response context. Specifically, for the objective function $L_{n,\lambda}$
in \eqref{obj}, consider the penalized likelihood ratio test statistic defined by
\begin{align}\label{lrts}
\L_n(\beta_*)= L_{n,\lambda}(\beta_*)- L_{n,\lambda}(\hat\beta_{n})\,.
\end{align} 
In order to find the asymptotic distribution of $\L_n(\beta_*)$ under the null hypothesis,
we  define the sequences 
\begin{align}\label{usigma}
u_n=\frac{\big\{\sum_{k,\ell}(1+\lambda\rho_{k\ell})^{-1}\big\}^2}{\sum_{k,\ell}(1+\lambda\rho_{k\ell})^{-2}}\,,\qquad\sigma_n^2=\frac{\sum_{k,\ell}(1+\lambda\rho_{k\ell})^{-1}}{\sum_{k,\ell}(1+\lambda\rho_{k\ell})^{-2}}\,,
\end{align}
and obtain the following result, which is proved in Section \ref{app:thm:lrt}.

\begin{theorem}\label{thm:lrt}
Let Assumptions~\ref{a1}--\ref{a:rate} be satisfied. Assume that as $n\to\infty$, $n\lambda^{(2D+1)/(2D)}=o(1)$, $nv_n^2=o(1)$ and $nv_n\lambda^{(D+1)/(2D)}=o(1)$, where $v_n$ is
defined in \eqref{vnn}. Then, under the null hypothesis in \eqref{eq:ch}, 
\begin{align*}
\frac{1}{\sqrt{2u_n}}\{2n\sigma_n^2\,\L_n(\beta_*)-u_n\}\converged N(0,1)\,,
\end{align*}
where  $u_n$ and $\sigma_n^2$ are given  in \eqref{usigma}.
\end{theorem}

Then, the PLRT  at nominal level $\alpha$ 
rejects the null hypothesis in 
\eqref{eq:ch}, whenever
\begin{align}\label{eq:plrt}
2n\sigma_n^2\,\L_n(\beta_*)-u_n\geq\sqrt{2 u_n}\,\mathcal Q_{1-\alpha}\,,
\end{align}
where $\mathcal Q_{1-\alpha}$ is the $(1-\alpha)$-quantile of the standard normal distribution. 
We compare the test \eqref{eq:chrule} and PLRT  \eqref{eq:plrt} for the classical hypotheses \eqref{eq:ch} through simulated data in Section~\ref{sec:simulation}.

\subsection{Relevant hypotheses}\label{sec:rele}

It turns out that the construction of an asymptotic level $\alpha$ test  for  relevant hypotheses as formulated
in \eqref{rel} is substantially more difficult.  Recall that we are interested in testing whether the maximum deviation between  a given surface $\beta_*$ 
and the unknown ``true'' slope surface $\beta_0$ exceeds a given value $\Delta\geq0$, and note  that with the notation  $d_\infty=\sup_{(s,t)\in[0,1]^2}|\beta_0(s,t)-\beta_*(s,t)|$ the  relevant hypothesis in \eqref{rel} can be rewritten as
\begin{align}\label{rele}
H_0:\,d_\infty\leq\Delta\quad&\text{versus}\quad H_1:\, d_\infty>\Delta~.
\end{align}
Therefore, a reasonable decision rule is to reject the null hypothesis for large values of the statistic
\begin{equation} \label{estsup}
    \hat d_\infty=\sup_{(s,t)\in[0,1]^2}|\hat\beta_{n}(s,t)-\beta_*(s,t)|~.
\end{equation}
When $\Delta=0$, the above relevant hypothesis reduces to the classical hypotheses
in \eqref{eq:ch}. In this case, under the null
hypothesis  $H_0:\beta_0=\beta_*$, there exists only  one function-on-function linear model,
which simplifies the  asymptotic analysis of  the corresponding test statistics substantially, because
basically the asymptotic distribution 
can be obtained from  Theorem \ref{thm:process} via continuous mapping (see also the discussion in Section \ref{sec:ch}). 
On the other hand, if $\Delta >0$, there
appear additional nuisance parameters in the asymptotic distribution of the difference 
$ \hat d_\infty -   d_\infty$, which makes the analysis of a decision rule more intricate.

For a precise description of the asymptotic distribution of $\hat d_\infty$ in  the case $\Delta >0$,  let
\begin{align}\label{epm}
\EE^{\pm}=\{(s,t)\in[0,1]^2:\beta_0(s,t)-\beta_*(s,t)=\pm d_\infty\}\,,
\end{align}
denote the set of points, where the surface   $\beta_0-\beta_*$ attains it sup-norm (the set  $\EE^{+}$)
or its negative sup-norm  (the set $\EE^{-}$).
Here we take the convention that $\EE^{+}=\EE^{-}=[0,1]^2$ if $d_\infty=0$
and denote by  
$\EE=\EE^+\cup\EE^-$ the 
 set of {\it  extremal points}
 of the difference $\beta_0-\beta_*$.
The following result describes 
the asymptotic properties of $\hat d_\infty$ 
and is crucial for constructing a test for the relevant hypothesis. It 
is proved in Section~\ref{app:thm:extreme}.

\begin{corollary} \label{cordet}
If the assumptions of Theorem~\ref{thm:process} are satisfied, then 
\begin{align}\label{te}
\sqrt{n}\lambda^{(2a+1)/(4D)}(\hat d_\infty-d_\infty)\converged T_{\EE}= \max\Big \{\sup_{(s,t)\in\EE^+}Z(s,t),\sup_{(s,t)\in\EE^-}\{-Z(s,t)\}\Big \}\, , 
\end{align}
where $Z$  is  the mean-zero Gaussian process defined in \eqref{hatgn}.
\end{corollary}

Note that the distribution of $T_{\EE}$ 
depends on the covariance structure of the limiting process $Z$ in \eqref{hatgn} and implicitly through 
the sets of extremal points $\EE^+$ and $\EE^-$ 
on the ``true'' (unknown) difference $\beta_0-\beta_*$. In order to motivate the final test, assume for the moment
the quantile, say $\mathcal Q_{1-\alpha}(T_{\EE})$, of this distribution would be  available (we will soon provide an estimate for it),
then we will show in Section~\ref{app:thm:extreme}
that 
\begin{align}\label{eq:rele}
\lim_{n\to\infty}\P\bigg\{\hat d_\infty>\Delta+\frac{\mathcal Q_{1-\alpha}(T_{\EE})}{\sqrt{n}\lambda^{(2a+1)/(4D)}}\bigg\}=
\left\{\begin{array}{ll}
\vspace{-0.5em}0 & \quad \text{if}\ d_\infty<\Delta\,\\
\vspace{-0.5em}\alpha &\quad \text{if}\ d_\infty=\Delta\,\\
1 &\quad \text{if}\ d_\infty>\Delta\,
\end{array}\right. ~.
%
\end{align}
Here the first two lines correspond to the null hypothesis $d_\infty \leq \Delta$ and the third line to the alternative in \eqref{rele}.  

This yields, in principle, a consistent asymptotic level $\alpha$ test for the  relevant hypotheses \eqref{rele}.
To implement such  a test  we need to approximate the quantiles of the random variable $T_\EE$ in \eqref{te}.
While the covariance structure of the process $Z$ can be again estimated by the multiplier bootstrap (see the discussion below), the estimation of the extremal sets is a little more tricky. For this purpose   
we propose
to estimate the sets $\EE^{+}$ and $\EE^-$ by
\begin{align}\label{extremal}
&\hat{\EE}^{+}=\Big\{(s,t)\in[0,1]^2:\hat\beta_{n}(s,t)-\beta_*(s,t)\geq\hat d_\infty-c\,\frac{\log n}{\sqrt{n}}\Big\}\,,\notag\\
&\hat{\EE}^{-}=\Big\{(s,t)\in[0,1]^2:\hat\beta_{n}(s,t)-\beta_*(s,t)\leq -\hat d_\infty+c\,\frac{\log n}{\sqrt{n}}\Big\}\,,
\end{align}
respectively, where we use a term $c\log n/\sqrt{n}$ in the cut-off values, for some tuning parameter $c>0$. Then, the 
random variable $T_{\EE}$ in \eqref{te} can be approximated  by
\begin{align}\label{hatte}
\hat T_{\EE}= \max\bigg\{\sup_{(s,t)\in\hat{\EE}^+}Z(s,t)\,,\sup_{(s,t)\in\hat{\EE}^-}\{-Z(s,t)\}\bigg\}\,.
\end{align}
In view of \eqref{eq:rele}, the null hypothesis should be rejected at nominal level $\alpha\in(0,1)$, if 
\begin{align}\label{d}
\hat d_\infty=\sup_{(s,t)\in[0,1]^2}|\hat\beta_{n}(s,t)-\beta_*(s,t)|>\Delta+\frac{\mathcal Q_{1-\alpha}(\hat T_{\EE})}{\sqrt{n}\lambda^{(2a+1)/(4D)}}\,,
\end{align}
where $\mathcal Q_{1-\alpha}(\hat T_{\EE})$ denotes the $(1-\alpha)$-quantile of $\hat T_{\EE}$.
Now, we still need to approximate the quantile $\mathcal Q_{1-\alpha}(\hat T_{\EE})$ of $\hat T_\mathcal E$. Since the asymptotic distribution of $\hat T_\EE$ depends on the unknown covariance function $C_Z$ in \eqref{cz}, we propose to combine a multiplier  bootstrap  similar to the ones introduced in Section~\ref{sec:simultaneousband} with the estimation of the extremal sets. Specifically, for $1\leq q\leq Q$ and the process $\G_{n,q}^*(s,t)$ defined in \eqref{gnqstar}, let 
\begin{align}\label{hattnq}
\hat T_{\EE,n,q}^*=\max\bigg\{\sup_{(s,t)\in\hat{\EE}^+}\G_{n,q}^*(s,t)\,,\sup_{(s,t)\in\hat{\EE}^-}\{-\G_{n,q}^*(s,t)\}\bigg\}\,,
\end{align}
where $\hat\EE^{\pm}$ are the estimated extremal sets defined in \eqref{extremal}. Then, the quantile of $\hat T_\EE$ can be approximated by the quantiles of the bootstrap extremal value estimates $\{\hat T_{\EE,n,q}^*\}_{q=1}^Q$. We summarize the bootstrap procedures for the relevant hypothesis in \eqref{rele} at nominal level $\alpha$ in the following algorithm.

\begin{algorithm}[Bootstrap for relevant hypotheses]~\label{algo:rt}
\begin{enumerate}[nolistsep]

\item Generate i.i.d.~bootstrap weights $\{M_{i,q}\}_{1\leq i\leq n,1\leq q\leq Q}$ and compute the bootstrap process     $\G_{n,q}^*(s,t)$  in \eqref{gnqstar}.

\item Compute the extremal sets $\hat\EE^{\pm}$ in \eqref{extremal}. For $1\leq q\leq Q$, compute $\hat T_{\EE,n,q}^*$ in \eqref{hattnq} and obtain the empirical $(1-\alpha)$-quantile of the  sample $\{\hat T_{\EE,n,q}^*\}_{q=1}^Q$, denoted by $\mathcal Q_{1-\alpha}(\hat T_{\EE,n,Q}^*)$.

\item Reject the null hypothesis in \eqref{rele} at nominal level $\alpha$, if
\begin{align}\label{d2}
\hat d_\infty=\sup_{(s,t)\in[0,1]^2}|\hat\beta_{n}(s,t)-\beta_*(s,t)|>\Delta+\frac{\mathcal Q_{1-\alpha}(\hat T_{\EE,n,Q}^*)}{\sqrt{n}\lambda^{(2a+1)/(4D)}}\,.
\end{align}
\end{enumerate}
\end{algorithm}

The following theorem, which is proved in Section~\ref{app:proof:rt}, provides a theoretical justification of the test  \eqref{d2}.
\begin{theorem}\label{thm:rt:boot}
Suppose the conditions of Theorem~\ref{thm:process} hold. Then, the decision rule \eqref{d2} defines 
a consistent and asymptotic level $\alpha$ test for the hypotheses \eqref{rele}, that is
\begin{align}\label{eq:rele:boot}
\lim_{Q\to\infty}\lim_{n\to\infty}\P\bigg\{\hat d_\infty>\Delta+\frac{\mathcal Q_{1-\alpha}(\hat T_{\EE,n,Q}^*)}{\sqrt{n}\lambda^{(2a+1)/(4D)}}\bigg\}=
\left\{\begin{array}{ll}
\vspace{-0.5em}0 & \quad \text{if}\ d_\infty<\Delta\\
\vspace{-0.5em}\alpha &\quad \text{if}\ d_\infty=\Delta\\
1 &\quad \text{if}\ d_\infty>\Delta
\end{array}\right. ~.
%
\end{align}
\end{theorem}

\subsection{Simultaneous prediction bands}\label{sec:prediction}

Based on the estimator $\hat\beta_{n}$ in \eqref{hatbeta}, we can construct a simultaneous confidence region of the conditional mean $\mu_{x_0}(t)=\E\{Y(t)|X=x_0\}=\int_0^1\beta_{0}(s,t)x_0(s)ds$, using the consistent estimator 
\begin{equation}\label{m3}
\hat\mu_{x_0}(t)=\int_0^1\hat\beta_{n}(s,t)x_0(s)ds. 
\end{equation}
The following theorem establishes the weak convergence of the process $\hat\mu_{x_0}$ in the space $C([0,1])$, which enables us to construct simultaneous confidence regions for the function $\mu_{x_0}$. The proof   is given in Section~\ref{app:thm:prediction:band}.

\begin{theorem}[Simultaneous prediction band]\label{thm:prediction:band}

Suppose that the conditions of Theorem~\ref{thm:process} are satisfied. Then,
$$
\sqrt{n}\lambda^{(2a+1)/(4D)}
\big \{\hat\mu_{x_0}(t)-\mu_{x_0}(t)
\big  \}_{t \in [0,1]} \weakconverge  \big \{ Z_{x_0}(t) \big \}_{t \in [0,1]} 
~~~~\text{in } C([0,1])\,,
$$  
where $Z_{x_0}$ is a mean zero Gaussian process with covariance function 
\begin{align}\label{czx0}
C_{Z,{x_0}}(t_1,t_2)=\int_0^1\int_0^1C_{Z}\{(s_1,t_1),(s_2,t_2)\}\,x_0(s_1)\,x_0(s_2)\,ds_1\,ds_2\,
\end{align}
and  $C_Z$ is defined in \eqref{cz},
Moreover, 
$$
\sqrt{n}\lambda^{(2a+1)/(4D)}\sup_{t\in[0,1]}|\hat\mu_{x_0}(t)-\mu_{x_0}(t)|\converged R_{x_0}=\max_{t\in[0,1]}|Z_{x_0}(t)|\,.
$$
\end{theorem}

As the quantiles of the distribution of 
$R_{x_0}$ depend in a complicate way on the covariance structure of the process $Z_{x_0}$, 
    we propose the following bootstrap procedure for a simultaneous asymptotic $(1-\alpha)$ prediction band for
    the finction $ t \to \mu_{x_0}(t)=\E\{Y(t)|X=x_0\}$.

\begin{algorithm}[Bootstrap simultaneous prediction band]~\label{algo:pb}
\begin{enumerate}[nolistsep]

\item Generate i.i.d.~weights $\{M_{i,q}\}_{1\leq i\leq n,1\leq q\leq Q}$ and compute the bootstrap estimators $\{\hat\beta_{n,q}^*\}_{q=1}^Q$ in \eqref{hatbetastar}. Compute $\hat\beta_{n}$ in \eqref{hatbeta} and  $\hat\mu_{x_0}(t)$ in \eqref{m3}.

\item For $1\leq q\leq Q$, compute $\mathbb L_{x_0,q}^*(t)=\sqrt{n}\lambda^{(2a+1)/(4D)}\int_0^1\{\hat\beta_{n,q}^*(s,t)-\hat\beta_{n}(s,t)\}x_0(s)ds$ and define $\hat R_{x_0,q}^*=\sup_{t\in[0,1]}|\mathbb L_{x_0,q}^*(t)|$. Compute the empirical $(1-\alpha)$-quantile of the bootstrap sample $\{\hat R_{x_0,q}^*\}_{q=1}^Q$, denoted by $\mathcal Q_{1-\alpha}(\hat R_{x_0,Q}^*)$.

\item Let $\hat\mu_{x_0,Q}^{*^\pm}(t)=\hat\mu_{x_0}(t)\pm \mathcal Q_{1-\alpha}(\hat R_{x_0,Q}^*)/\{\sqrt{n}\lambda^{(2a+1)/(4D)}\}$. Define the set 
\begin{equation} \label{m4}
\mathfrak{B}_{n,Q}^*(\alpha)=\big\{\mu:\hat\mu_{x_0,Q}^{*^-}(t)\leq\mu(t)\leq \hat\mu_{x_0,Q}^{*^+}(t)\big\}
\end{equation}
as simultaneous $(1-\alpha)$ prediction band for the function $\mu_{x_0}$.
\end{enumerate}
\end{algorithm}

The following theorem provides a formal justification of the bootstrap procedure  in Algorithm~\ref{algo:pb}, the proof   uses similar arguments as given in 
the proof of Theorem~\ref{thm:cb:boot}  and is therefore omitted.

\begin{theorem}[Bootstrap simultaneous prediction band]\label{thm:pb}
Suppose the assumptions in Theorem~\ref{thm:process} are satisfied. Then, the  set $\mathfrak B_{n,Q}^*(\alpha)$ in \eqref{m4} defines a simultaneous asymptotic $(1-\alpha)$ prediction band for the function $\mu_{x_0}$, that is
$$
\lim_{Q\to\infty}\lim_{n\to\infty}\P\{\mu_{x_0}\in\mathfrak{B}^*_{n,Q}(\alpha)\}=1-\alpha\,.
$$
\end{theorem}


\subsection{Scalar response}\label{sec:scalar}

The results presented so  far provide also new inference tools for the scalar-on-function linear model
\begin{align}\label{sof}
Y_i=\int_0^1\beta_0(s)\,X_i(s)\,ds+\epsilon_i\,,\quad1\leq i\leq n\, , 
\end{align}
which can be considered as a special case of  model \eqref{model}, where the response $Y$ is a scalar variable. In this setting, the estimator defined in \eqref{hatbeta} becomes
\begin{align*}
\hat\beta_{n}&=\underset{\beta\in \H_{\rm s}}{\arg\min}\bigg[\frac{1}{2n}\sum_{i=1}^n\left\{Y_i-\int_0^1\beta(s)X_i(s)\,ds\right\}^2+\frac{\lambda}{2}\, J_{\rm s}(\beta,\beta)\bigg]\,,
\end{align*}
where $\H_{\rm s}=\big\{\beta:[0,1]\to\mathbb{R}\,|\,\beta,\beta',\ldots,\beta^{(m-1)}\text{ are absolutely continuous};\,\beta^{(m)}\in L^2([0,1])\big\}$ is the Sobolev space on $[0,1]$ of order $m$, and $J_{\rm s}(\beta,\beta)=\int_0^1\{\beta^{(m)}(s)\}^2ds$. A direct consequence of Theorem~\ref{thm:process} in the function-response setting is the  weak convergence
\begin{align}\label{gnsalar}
\{ \G_n(s) \}_{s \in [0,1]} =\sqrt{n}\lambda^{(2a+1)/(4D)} \big \{\hat\beta_{n}(s)-\beta_0(s) \big \}_{s \in [0,1]} \weakconverge \{ Z(s)\}_{s \in [0,1]}
\end{align}
in $C([0,1]),$
where $Z$ is a mean-zero Gaussian process.
Hence, the methodology proposed in Section~\ref{sec:sc}, namely the bootstrap procedures for simultaneous confidence regions, relevant hypothesis tests and simultaneous prediction bands, carries over naturally to the scalar response case. 

Exemplary, we consider (for a given constant $\Delta\geq 0$) the problem of constructing a test for 
the  relevant hypotheses 
\begin{align}\label{h0scalar}
H_0:\,\sup_{s\in[0,1]}|\beta_0(s)-\beta_*(s)|\leq\Delta\quad&\text{versus}\quad H_1:\, \sup_{s\in[0,1]}|\beta_0(s)-\beta_*(s)|>\Delta\,,
\end{align}
in model \eqref{sof}, which is more challenging in nature to tackle.
 A consistent estimator of the maximum deviation $d_\infty=\sup_{s\in[0,1]}|\beta_0(s)-\beta_*(s)|$ is $\hat d_\infty=\sup_{s\in[0,1]}|\hat\beta_{n}(s)-\beta_*(s)|$, so that the null hypothesis in \eqref{h0scalar} should be rejected for large values of $\hat d_\infty$. 
As analog of  Algorithm~\ref{algo:rt}, we obtain 
 the following bootstrap test for the relevant hypotheses in \eqref{h0scalar}.

\begin{algorithm}[Bootstrap   test for relevant hypotheses in the scalar-on-function mdoel]~\label{algo:rt:scalar}
\begin{enumerate}[nolistsep]

\item Generate i.i.d.~bootstrap weights $\{M_{i,q}\}_{1\leq i\leq n,1\leq q\leq Q}$ as in Algorithm~\ref{algo:cb} and for each $1\leq q\leq Q$, compute the bootstrap estimator
\begin{align*}
\hat\beta_{n,q}^*=\underset{\beta\in \H_{\rm s}}{\arg\min}\bigg[\frac{1}{2n}\sum_{i=1}^nM_{i,q}\bigg\{Y_i-\int_0^1\beta(s)\,X_i(s)\,ds\bigg\}^2+\frac{\lambda}{2}\, J_{\rm s}(\beta,\beta)\bigg]
\end{align*}
and $\G_{n,q}^*(s)=\sqrt{n}\lambda^{(2a+1)/(4D)}\{\hat\beta_{n,q}^*(s)-\hat\beta_{n}(s)\}$.

\item Compute the extremal sets 
$$
\hat{\EE}^{\pm}=\Big\{s\in[0,1]:\pm\{\hat\beta_{n}(s)-\beta_*(s)\}\geq\hat d_\infty-c\,\frac{\log n}{\sqrt{n}}\Big\}\,.
$$
\item For $1\leq q\leq Q$, compute 
$$
\hat T_{\EE,n,q}^*=\max\bigg\{\sup_{s\in\hat{\EE}^+}\G_{n,q}^*(s)\,,\sup_{s\in\hat{\EE}^-}\{-\G_{n,q}^*(s)\}\bigg\}\,,
$$
and obtain the empirical $(1-\alpha)$-quantile of  the bootstrap sample $\{\hat T_{\EE,n,q}^*\}_{q=1}^Q$, denoted by $\mathcal Q_{1-\alpha}(\hat T_{\EE,n,Q}^*)$.

\item Reject the null hypothesis in \eqref{h0scalar} at nominal level $\alpha$, if
\begin{align} \label{testscalar} 
\sup_{s\in[0,1]}|\hat\beta_{n}(s)-\beta_*(s)|>\Delta+\frac{\mathcal Q_{1-\alpha}(\hat T_{\EE,n,Q}^*)}{\sqrt{n}\lambda^{(2a+1)/(4D)}}\,.
\end{align}
\end{enumerate}
\end{algorithm}
 
 It can be shown by similar arguments as given in the proof of Theorem \ref{thm:rt:boot} that the test \eqref{testscalar}
 is a consistent and asymptotic level $\alpha$ test.  The details are omitted for the sake of brevity.

\section{Finite sample properties}\label{sec:finite}

\subsection{Implementation}\label{sec:numerical}

Because the  estimators $\hat\beta_{n}$   in \eqref{hatbeta} and  its bootstrap analog $\hat\beta_{n,q}^*$ in \eqref{hatbetastar} are defined as the solution of a (penalized) minimization problem on  an infinite dimensional function space, exact solution are inaccessible. In this section, we introduce finite-sample methods to circumvent this difficulty, and propose a method to choose the regularization parameter $\lambda$. We shall only present our approach  for computing  the bootstrap estimator $\hat\beta_{n,q}^*$ in \eqref{hatbetastar}, since  the estimator \eqref{hatbeta} can be viewed as a special case of \eqref{hatbetastar} by taking $M_{i,q}=1$ for any $1\leq i\leq n$ and $1\leq q\leq Q$.

We start by deducing from Assumption~\ref{a201} that $J(x_{k\ell}\otimes\eta_\ell,x_{k'\ell'}\otimes\eta_{\ell'})=\rho_{k\ell}\,\delta_{kk'}\,\delta_{\ell\ell'}$, so that for $\beta(s,t)=\sum_{k,\ell}b_{k\ell}\,\phi_{k\ell}(s,t)\in\H$ and for $b_{k\ell}\in\mathbb{R}$, we have $J(\beta,\beta)=\sum_{k,\ell}b_{k\ell}^2\,\rho_{k\ell}$. We consider the Sobolev space on $[0,1]^2$ of order $m=2$. In this case, the penalty functional in \eqref{hatbeta} is $J(\beta,\beta)=\int_0^1\int_0^1(\beta_{ss}^2+2\beta_{st}^2+\beta_{tt}^2)\,ds\,dt$, where $\beta_{st} = \frac{\partial^2 \beta  }{\partial s \partial t}$.  For the choice of the basis, we use Proposition~\ref{prop:1.3} in Section~\ref{app:cond:prop1.3} of the online supplement. More precisely,
$
\eta_1(t) \equiv 1$,
$\eta_\ell(t) = \sqrt{2}\cos\{(\ell-1)\pi t\}$   ($ \ell=  2,3, \ldots $) and for 
$\ell\geq 1$ the functions  $\{\tilde x_{k\ell}\}_{k\geq1}$ are the eigenfunctions of  integro-differential equation
\begin{align}\label{eq}
\left\{ 
\begin{array}{ll}
\displaystyle\rho_\ell\int_{0}^1C_X(s_1,s_2)\,\tilde x(s_2)\,ds_2=\tilde x^{(4)}(s_1)-2(\ell-1)^2\pi^2\tilde x^{(2)}(s_1)+(\ell-1)^4\pi^4\\
\tilde x^{(\theta)}(0)=\tilde x^{(\theta)}(1)=0\,,\qquad\text{ for }\theta=3\text{ and }4\,
\end{array}\right.
\end{align}
with corresponding eigenvalues $\{\rho_{k\ell} \}_{k\geq1}$.
In order to find the eigenvalue and the eigenfunction of \eqref{eq}, we use \texttt{Chebfun}, an efficient open-source Matlab add-on package, available at \nolinkurl{https://www.chebfun.org/}. We substitute the covariance function $C_X$ in \eqref{eq} by its empirical version $\hat{C}_X$, and find the eigenvalues $\hat\rho_{k\ell}$ and the normalized eigenfunctions $\hat x_{k\ell}$.
Observing that the functions $\{\eta_\ell\}_{\ell\geq1}$
are given by the cosine basis, we take the empirical eigenfunctions $\hat\phi_{k\ell}=\hat x_{k\ell}\otimes\eta_\ell$. Now, we approximate the space $\H$ by a finite dimensional subspace spanned by $\{\hat\phi_{k\ell}\}_{1\leq k,\ell\leq v}$, defined by $\tilde\H=\big\{\sum_{1\leq k,\ell\leq v}b_{k\ell}\,\hat \phi_{k\ell}\big\}$, where $v$ is a truncation parameter that depends on the sample size $n$.

For $1\leq q\leq Q$, for the $q$-th bootstrap estimator $\hat\beta_{n,q}^*$ and for the bootstrap weights $\{M_{i,q}\}_{i=1}^n$ in Algorithms~\ref{algo:cb}--\ref{algo:pb}, let $\tilde M_q={\rm diag}(M_{1,q},\ldots,M_{n,q})$ denote an $n\times n$ diagonal matrix. For $1\leq i\leq n$ and $1\leq k,\ell\leq v$, let $\omega_{ik\ell}=\int_0^1X_i(s)\hat x_{k\ell}(s)ds$ and let $\Omega_\ell=(\omega_{ik\ell})$ denote a $n\times v$ matrix; let $\hat\Lambda_{\ell}={\rm diag}\big\{\hat\rho_{1\ell},\ldots,\hat\rho_{v\ell}\big\}$ denote a $v\times v$ diagonal matrix; let $\tilde Y_{i\ell}=\l Y_i,\eta_\ell\r_{L^2}$ and let $\tilde Y_\ell=(\tilde Y_{1\ell},\ldots,\tilde Y_{n\ell})^{\rm T}\in\mathbb{R}^v$. If we write $\tilde\beta_{n,q}=\sum_{k=1}^v\sum_{\ell=1}^v\tilde b_{k\ell}^{(q)}\,\hat\phi_{k\ell}\in\tilde\H$, then, in order to approximate $\hat\beta_{n,q}^*$ in \eqref{hatbetastar}, we find the $\tilde b_{k\ell}^{(q)}$'s by solving the following optimization problem
\begin{small}
\begin{align}\label{lambda}
\{\tilde b_{k\ell}^{(q)}\}&=\underset{\{b_{k\ell}^{(q)}\}}{\arg\min}\left\{\frac{1}{2n}\sum_{i=1}^nM_{i,q}\int_0^1\bigg| Y_i(t)-\sum_{k,\ell=1}^{v}b_{k\ell}^{(q)}\,\eta_\ell(t)\int_0^1 X_i(s)\hat x_{k\ell}(s)ds\bigg|^2dt+\frac{\lambda_q }{2}  \sum_{k,\ell=1}^{v}b_{k\ell}^{(q)^2}\,\hat \rho_{k\ell}\right\}\notag\\
&=\underset{\{b_{k\ell}^{(q)}\}}{\arg\min}\left\{\frac{1}{2n}\sum_{i=1}^n\sum_{\ell=1}^{v}M_{i,q}\bigg|\tilde Y_{i\ell}-\sum_{k=1}^{v}b_{k\ell}^{(q)}\int_0^1 X_i(s)\hat x_{k\ell}(s)ds\bigg|^2+\frac{\lambda_q }{2} \sum_{k,\ell=1}^{v}b_{k\ell}^{(q)^2}\,\hat \rho_{k\ell}\right\}\notag\\
&=\underset{\{b_{\ell}^{(q)}\}}{\arg\min}\Bigg\{\frac{1}{2n}\sum_{\ell=1}^v\big(\tilde Y_\ell-\Omega_{\ell}\,b_{\ell}^{(q)}\big)^{\rm T}\,\tilde M_q\,\big(\tilde Y_\ell-\Omega_{\ell}\,b_{\ell}^{(q)}\big)+\frac{\lambda_q }{2} \sum_{\ell=1}^v b_{\ell}^{(q)^{\rm T}}\,\hat\Lambda_\ell\, b_{\ell}^{(q)}\,\Bigg\}\,,
\end{align}
\end{small}
where we write $b_{\ell}^{(q)}=(b_{1\ell}^{(q)},\ldots,b_{v\ell}^{(q)})^{\rm T}\in\mathbb{R}^v$. By direct calculations, for $1\leq\ell\leq v$, we have 
\begin{align}\label{hatb}
\hat b_{\ell}^{(q)}=(\Omega_\ell^{\rm T}\,\tilde M_q\,\Omega_\ell+n\lambda_q\hat\Lambda_\ell)^{-1}\Omega_\ell^{\rm T}\,\tilde M_q\,\tilde Y_\ell\,,
\end{align}
so that we can approximate $\hat\beta_{n,q}^*$ in \eqref{hatbetastar} by 
\begin{align*}
\tilde\beta_{n,q}^*=\sum_{\ell=1}^v(\hat b_{\ell}^{(q)^{\rm T}}\,\hat x_{\ell})\otimes\eta_\ell\,, 
\end{align*}
if we let $\hat x_\ell=(\hat x_{1\ell},\ldots,\hat x_{v\ell})^{\rm T}$ denote a function-valued $v$-dimensional vector. We propose to use generalized cross-validation (GCV, see, for example, \citealp{whaba1990}) to choose the smoothing parameter $\lambda_q$ in \eqref{lambda}. For the $q$-th bootstrap estimator $\tilde\beta_{n,q}^*$, we choose $\lambda_q$ that minimizes the GCV score
\begin{align*}
\text{GCV}(\lambda_q)=\frac{n^{-1}\sum_{\ell=1}^v\Vert \hat Y_{\ell,q}-\tilde Y_\ell\Vert_{2}^2}{\{1-{\rm{tr}}(H_q)/n\}^2}\,,
\end{align*}
where $\hat Y_{\ell,q}=\Omega_\ell(\Omega_\ell^{\rm T}\,\tilde M_q\,\Omega_\ell+n\lambda_q\hat\Lambda_\ell)^{-1}\Omega_\ell^{\rm T}\,\tilde M_q\,\tilde Y_\ell$ and $H_{q}$ is the so-called hat matrix with ${\rm tr}(H_{q})=\sum_{\ell=1}^v{\rm tr}\{\Omega_\ell(\Omega_\ell^{\rm T}\,\tilde M_q\,\Omega_\ell+n\lambda_q\,\hat\Lambda_\ell)^{-1}\,\Omega_\ell^{\rm T}\,\tilde M_q\}$.

The statistical inference methods in Section~\ref{sec:sc} rely on the parameters $a$ and $D$, and we propose to estimate these two parameters from the data. To achieve this, we make use of the growing rate of the eigenvalues of the integro-differential equation~\eqref{eq}. As indicated by Proposition~\ref{prop:1.3} in Section~\ref{app:cond:prop1.3} of the inline supplement, the $\rho_{k\ell}$'s diverge at a rate of $(k\ell)^{2D}$, so that we exploit the linear relationship between  $\log(\rho_{k\ell})$ and $\log(k\ell)$. Specifically, for the empirical eigenvalues $\{\hat\rho_{k\ell}\}$ of equation \eqref{eq}, we fit a line through the points $\{(\log(k\ell),\log(\hat\rho_{k\ell}))\}_{1\leq k,\ell\leq 2v}$, and take $\tilde D$ to be the value of the slope of this line divided by 2, where we use a total number of $4v^2$ eigenvalues. In the case of $m=2$, by Proposition~\ref{prop:1.3}, $D\geq 3$ and $a=D-2$, so that we take 
\begin{align*}
\hat D=\max\{\tilde D,3\} \text{~~  and ~~} \hat a=\hat D-2\,.
\end{align*}


\subsection{Simulated data}\label{sec:simulation}


For evaluating the functions  $X$ and $Y$  on their domain  $[0,1]$ we take 
$100$ equally spaced time points. For the data generating process (DGP), we used the following three settings:
\begin{enumerate}[]
\item[(1)] Let $f_1(s)\equiv1$, $f_{j+1}(s)=\sqrt{2}\cos(j\pi s)$, for $j\geq 1$, and define
$$
\beta_0(s,t)=f_1(s)f_1(t)+4\sum_{j=2}^{50}(-1)^{j+1}j^{-2}f_j(s)f_j(t)\,.$$
Let $X_i=\sum_{j=1}^{50}j^{-1}Z_{ij} \,f_j$, where $Z_{ij}\iidsim \text{unif}(-\sqrt{3},\sqrt{3})$, for $1\leq i\leq n, 1\leq j\leq 50$.

\item[(2)] Let $\beta_0(s,t)=e^{-(s+t)}$; the $X_i$'s are the same as DGP 1.

\item[(3)]
Let $f_1(s)\equiv1$, $f_{j+1}(s)=\sqrt{2}\cos(j\pi s)$ and $g_{j+1}(s)=\sqrt{2}\{1+\cos(j\pi s)\}$, for $j\geq 1$ and define 
$$
\beta_0(s,t)=f_1(s)f_1(t)+4\sum_{j=2}^{50}(-1)^{j+1}j^{-2}g_j(s)f_j(t) \,;
$$
the $X_i$'s are the same as DGP 1.
\end{enumerate}
The first setting is similar to the ones used in \cite{yuancai} and \cite{sun2018}; the second setting is exactly the same as Scenario 1 in \cite{sun2018}; the third setting is a non-standard setting that involves an asymmetric slope surface 
$\beta_0$. We took $\e$ to be the Gaussian process with the following three covariance settings: 
\begin{itemize}[]
\item[(i)] For $t_1,t_2\in[0,1]$, $C_\e(t_1,t_2)=\sigma_1^2\delta(t_1,t_2)$, where $\sigma_1^2=0.1\times\int \var\{\tilde Y(t)\}dt$ and $\tilde Y(t)=\int_0^1\beta_0(s,t)X(s)ds$, for $t\in[0,1]$.
\item[(ii)] For $t_1,t_2\in[0,1]$, $C_\e(t_1,t_2)=\sigma_2^2(t)\delta(t_1,t_2)$, where $\sigma_2^2(t)=0.1\times\var\{Y(t)\}$, for $t\in[0,1]$.
\item[(iii)] For $t_1,t_2\in[0,1]$,
$C_\e(t_1,t_2)=2\sigma_1^2\delta(t_1,t_2)$, where $\sigma_1^2$ is as in (i). 
\end{itemize}
For each above setting, we simulated $1000$ Monte Carlo samples, each of size $n=30$ or 60, and we took the bootstrap sample size $Q=300$. We compared our method (hereinafter referred to as RK) with the tensor product reproducing kernel Hilbert space method proposed in \cite{sun2018} (hereinafter referred to as TP). For our method, we took the number of components $v=\lceil n^{2/5}\rceil$ in Section~\ref{sec:numerical}. To evaluate the performance of different estimators, we considered the following three criteria. The first criterion is the integrated squared error of $\hat\beta$, defined by 
\begin{align*}
\ISE(\hat\beta)=\int_0^1\int_0^1|\hat\beta(s,t)-\beta_0(s,t)|^2dsdt.
\end{align*}
The second criterion is the excess prediction risk ${\rm EPR}(\hat\beta)$ defined in \eqref{epr}. The third criterion is the maximum deviation, defined by
\begin{align*}
{\rm MD}(\hat\beta)=\sup_{(s,t)\in[0,1]^2}|\hat\beta(s,t)-\beta_0(s,t)|\,.
\end{align*}
In Table~\ref{table:estimation}, we report the three quartiles of ${\rm ISE}$, ${\rm EPR}$ and ${\rm MD}$ of the estimators computed from the $1000$ Monte Carlo samples under the data generating process 1--3 with error processes (i)--(iii), using our method (RK) and \cite{sun2018}'s method (TP). Figure~\ref{fig:setting1} displays the plots of the true slope surface $\beta_0$ and their corresponding estimators using RK and TP, under the data generating processes 1--3 with error (i) and sample size $n=60$.

The results in Table~\ref{table:estimation} indicate that, for DGPs 1 and 3, our method (RK) produces higher estimation accuracy in terms of ISE, EPR and MD compared to \cite{sun2018}'s method (TP), whereas \cite{sun2018}'s produces slightly better estimators in DGP 2. These results are in accordance with the fact that, in contrast to DGPs 1 and 3, the true $\beta_0$ in DGP 2 is multiplicatively separable
and the approach of \cite{sun2018} is based on this assumption. However, it is notable that the loss 
of the RK-method, which does not require this condition, is not substantial.
From Table~\ref{table:estimation}, we also notice that, in some cases, both methods perform better in error setting (ii) than in error setting (i). An explanation for this observation is  that, the point-wise signal-to-noise ratio is $10$ in error setting (i), whereas this value is smaller than 10 for some $t\in[0,1]$ in setting (ii). As for computation, \cite{sun2018}'s method involves computing the inverse or the Cholesky decomposition of matrices, whose size are larger than $n^2$-by-$n^2$, which means their method could be time consuming for sample size $n$ larger than, say, $100$.

We also evaluated the performance of the 
simultaneous confidence region $\mathfrak{C}^*_{n,Q}$ defined in Algorithm~\ref{algo:cb}, using the uniform covering probability 
$$
{\rm UCP}(\mathfrak{C}^*_{n,Q})=\P\big\{\beta_0(s,t)\in\mathfrak{C}^*_{n,Q},\text{ for all }(s,t)\in[0,1]^2\big\}.
$$
In Table~\ref{table:cp} we report the empirical empirical UCP from  $1000$ simulation runs for data generating processes 1--3 with all error setting (i) and nominal level $\alpha=0.10$ and 0.05. We observe a reasonable  approximation of the confidence level in all cases under consideration.
The simultaneous confidence regions for the slope function for the  DGPs 1--3 and the error process (i)  are displayed in Figure \ref{fig:1}.

For the finite sample properties of classical hypothesis tests proposed in Section~\ref{sec:ch}, we consider the following hypothesis:
\begin{align}\label{simu:ch}
H_0:\,\beta_0=0\qquad&\text{versus}\qquad H_1:\,\beta_0\neq0\,,
\end{align}
that is, we put $\beta_*\equiv0$ in \eqref{eq:ch}. We compared the decision rule based on the bootstrap confidence regions defined in \eqref{eq:chrule} (denoted by BT)  and the penalized likelihood ratio test (PLRT) at \eqref{eq:plrt}. Here, for the PLRT, in view of \eqref{lambda} and \eqref{hatb}, substituting $\tilde M_q$ by $I_n$ and observing that $\Omega_\ell^{\rm T}\,\Omega_\ell=nI_v$, the statistic $\L_n(0)=L_{n,\lambda}(0)-L_{n,\lambda}(\hat\beta_n)$ (defined in equation \eqref{lrts}) can be estimated by
\begin{align*}
\tilde \L_n(0)=\frac{1}{2n}\sum_{\ell=1}^v\tilde Y_\ell^{\rm T}\Omega_\ell\,(nI_v+n\lambda\hat\Lambda_\ell)^{-1}\Omega_\ell^{\rm T}\,\tilde Y_\ell\,.
\end{align*}
We took $n=30$ and $60$, and chose the nominal level $\alpha=0.05$ and used DGPs 1--3 with error settings (i)--(iii). For DGPs 1 and 3, the empirical rejection probabilities are all 1.0 for both methods for all settings. Table~\ref{table:ch} displays the empirical rejection probabilities under DGP~2 with error settings (i)--(iii), together with the empirical sizes under $H_0$ (that is, $\beta_0=0$), of both BT and PLRT for the classical hypothesis \eqref{simu:ch} out of 1000 simulation runs. From the results we observe reasonable approximation of both BT and PLRT of the nominal level 0.05 under $H_0$; BT outperforms PLRT in terms of empirical power, and as expected, the empirical powers increases for larger sample sizes.

Next, we study the finite sample properties of the test \eqref{d2} for the relevant hypotheses
\begin{align}\label{simu:rele}
H_0:\sup_{(s,t)\in[0,1]^2}|\beta_0(s,t)|\leq\Delta \quad\text{  versus } \quad H_1:\sup_{(s,t)\in[0,1]^2}|\beta_0(s,t)|>\Delta\,,
\end{align}
(we put $\beta_*\equiv0$ in \eqref{rele}), where the nominal level is chosen as $\alpha=0.05$. We used the data generating processes 1--3 with error setting (i), where the true $\Vert\beta_0\Vert_\infty=6.0$, $1.0$ and $11.0$ for the three DGPs, respectively. We took the cut-off parameter $c=\Vert\hat\beta_n\Vert_\infty/4$ in \eqref{extremal}, which scales according to the magnitude of $\hat\beta_n$. In Figure~\ref{fig:rele}, we display the empirical rejection probabilities of test \eqref{rele} based on $1000$ simulation runs, for the three data generating processes, for different values of $\Delta$ in \eqref{rele}. 
The results shown in Figure~\ref{fig:rele} indicate that, when $\Delta\leq d_\infty$, the empirical rejection probabilities are smaller than $\alpha=0.05$, and when $\Delta>d_\infty$, the rejection probabilities increases towards 1 as $\Delta$ increases, which is consistent with our theory.

\begin{figure}[h!]
\centering
\includegraphics[width=5cm]{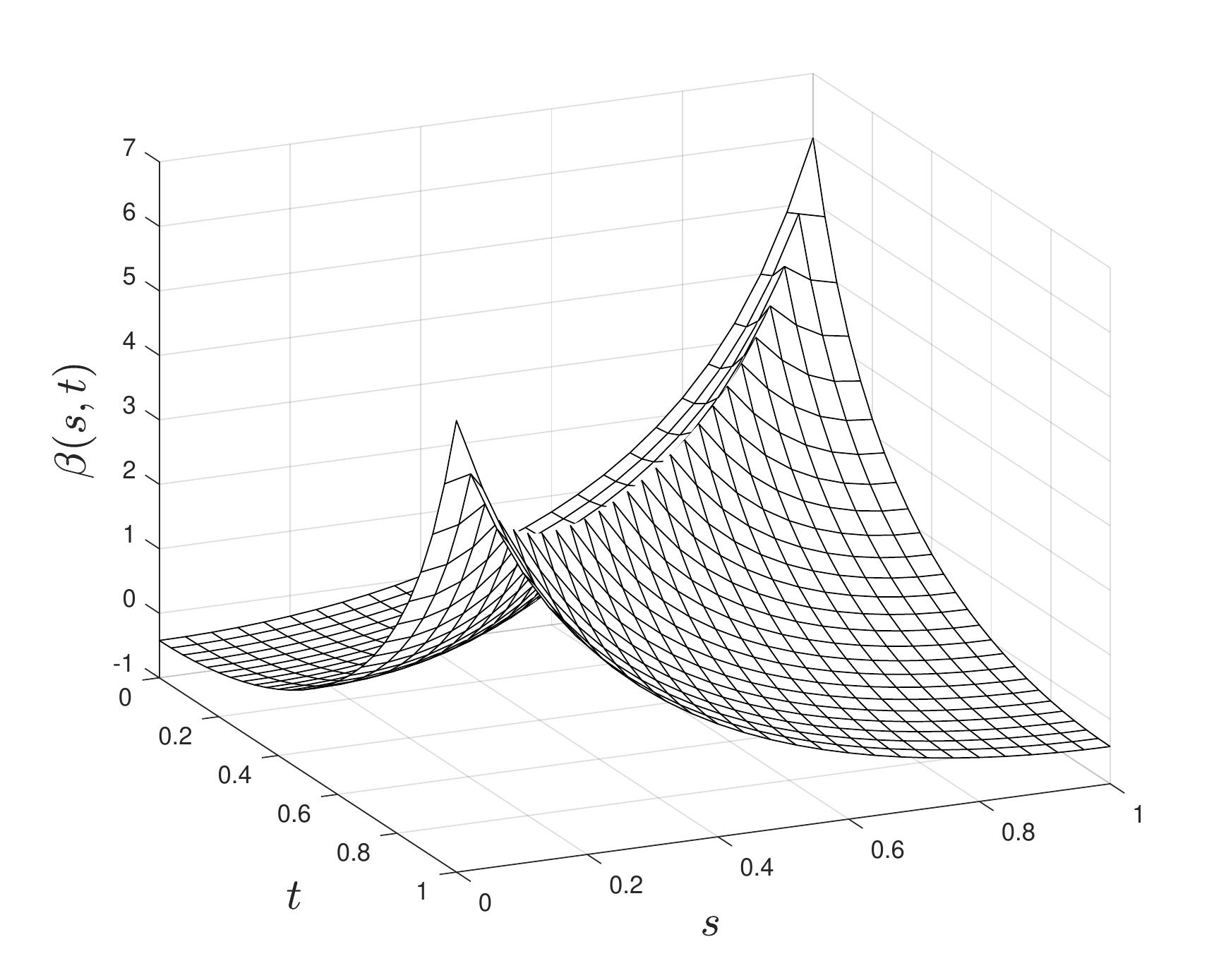}
\hspace{-0.2cm}\includegraphics[width=5cm]{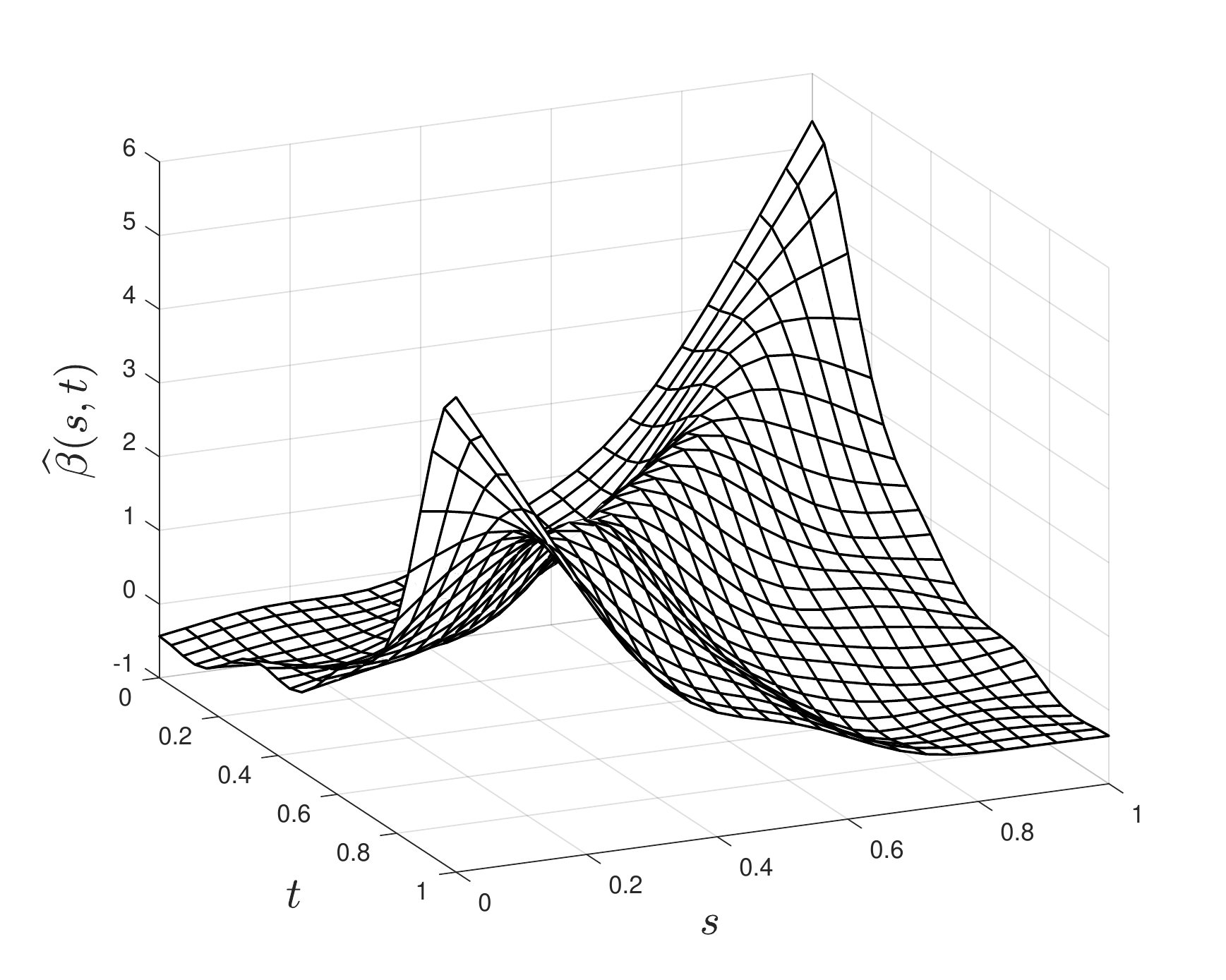}
\hspace{-0.2cm}\includegraphics[width=5cm]{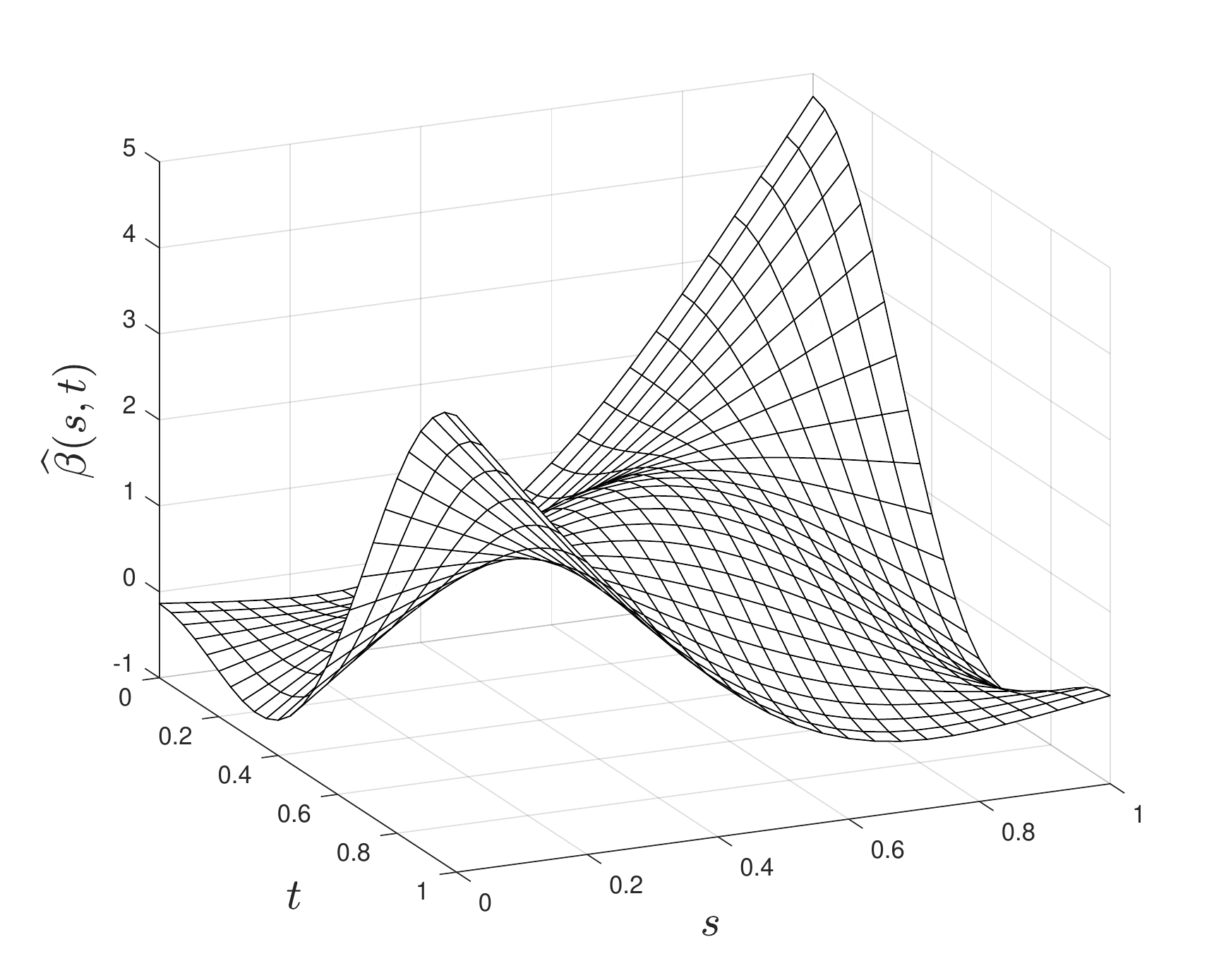}
\hspace{-0.2cm}\includegraphics[width=5cm]{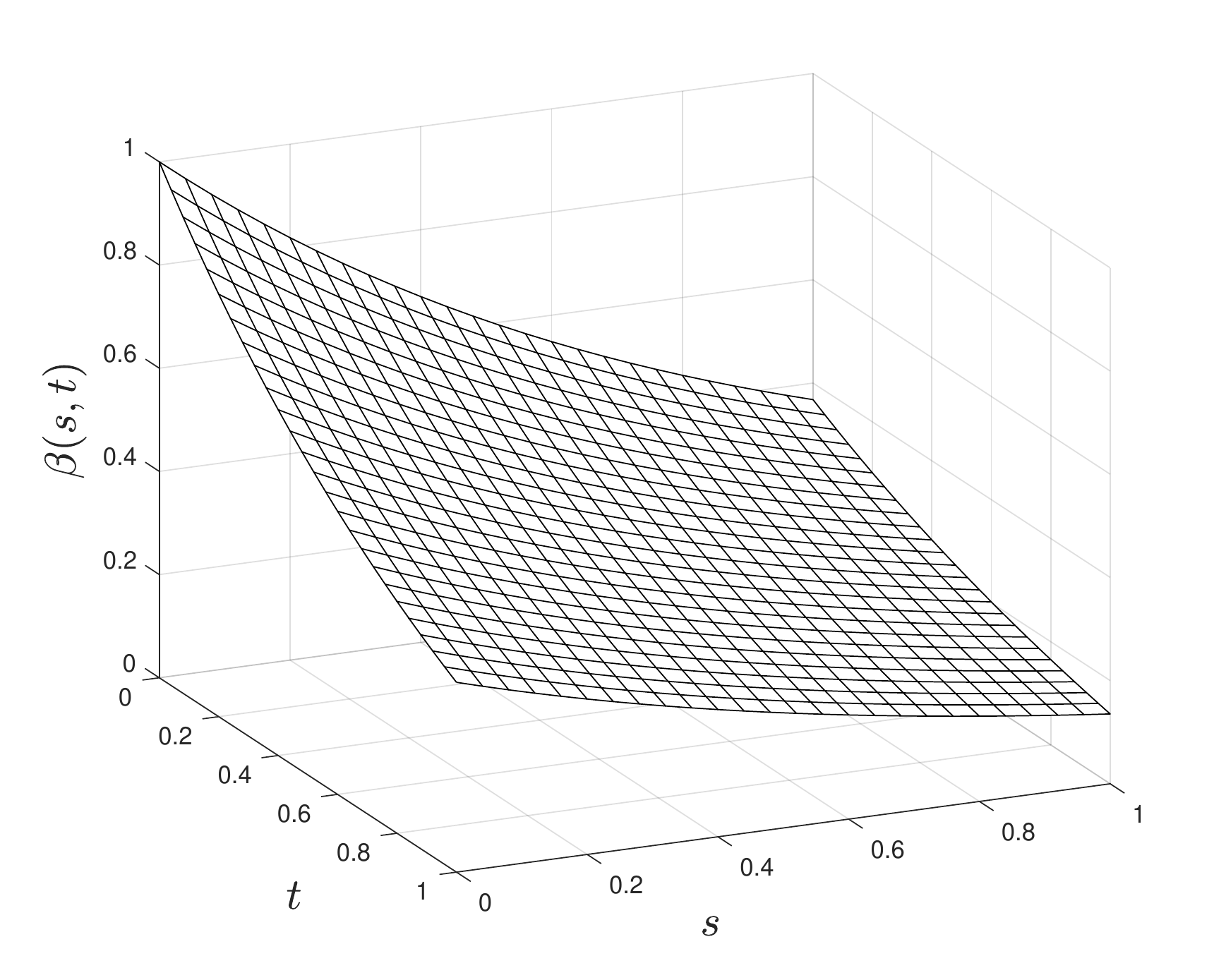}
\hspace{-0.2cm}\includegraphics[width=5cm]{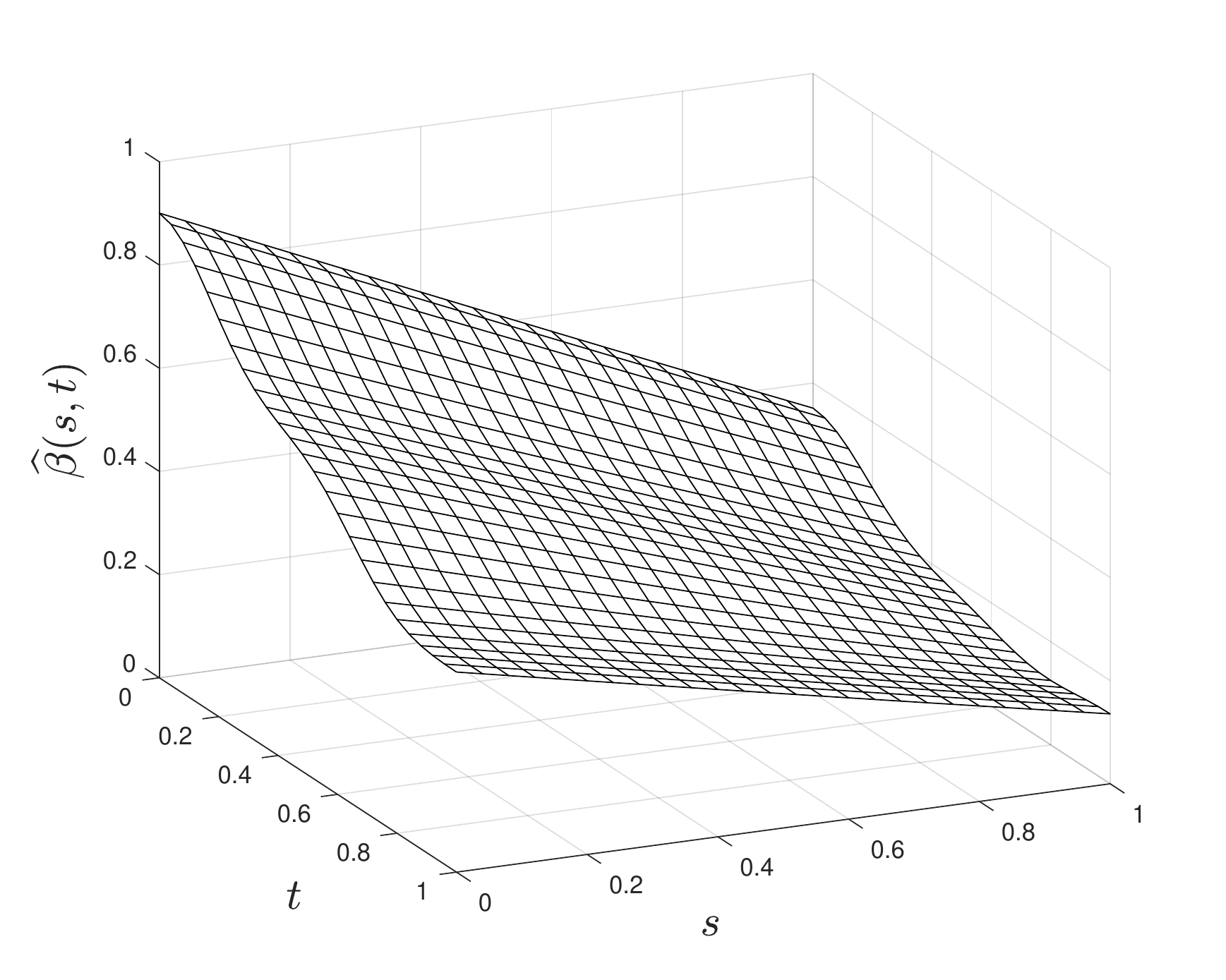}
\hspace{-0.2cm}\includegraphics[width=5cm]{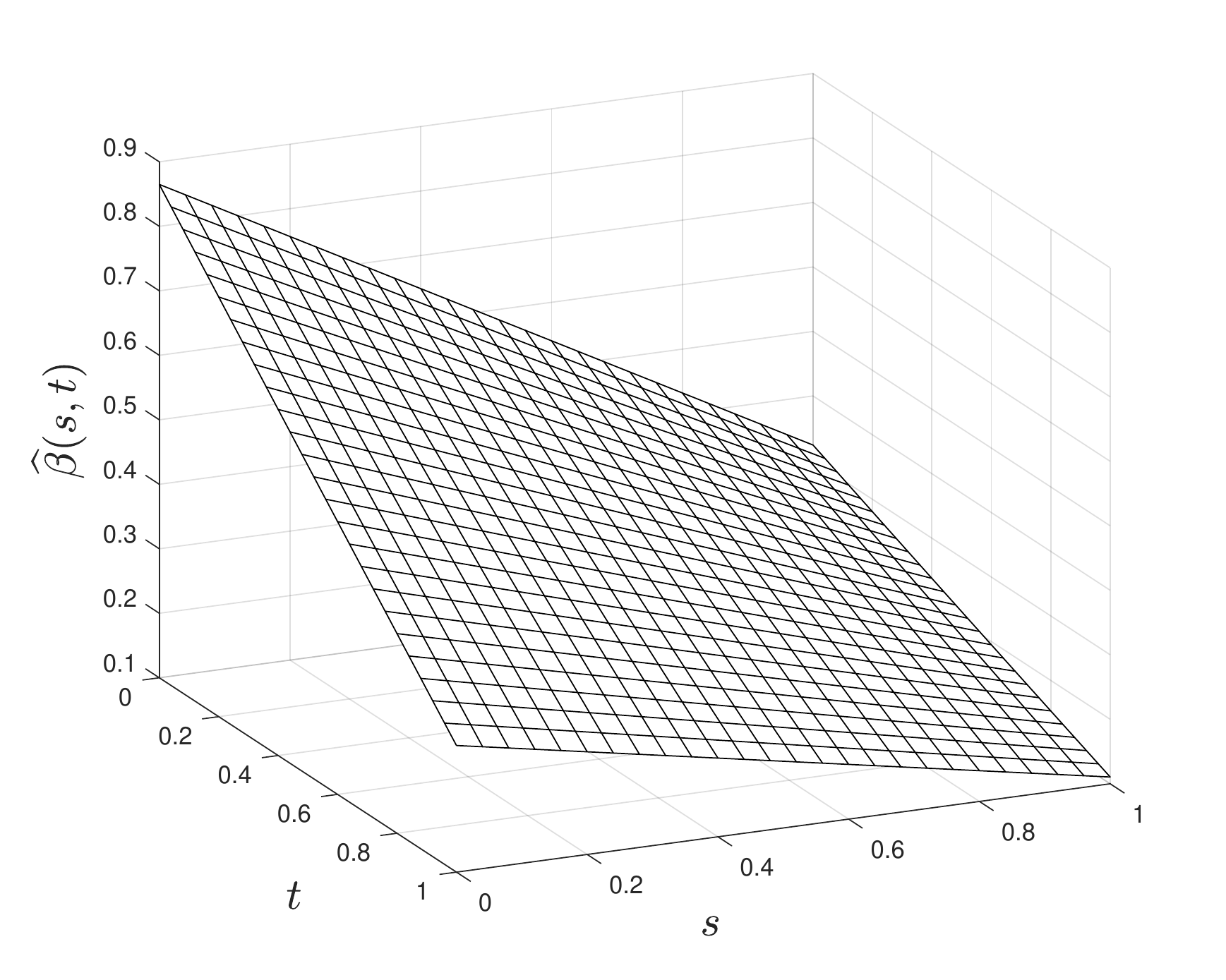}
\hspace{-0.2cm}\stackunder[5pt]{\includegraphics[width=5cm]{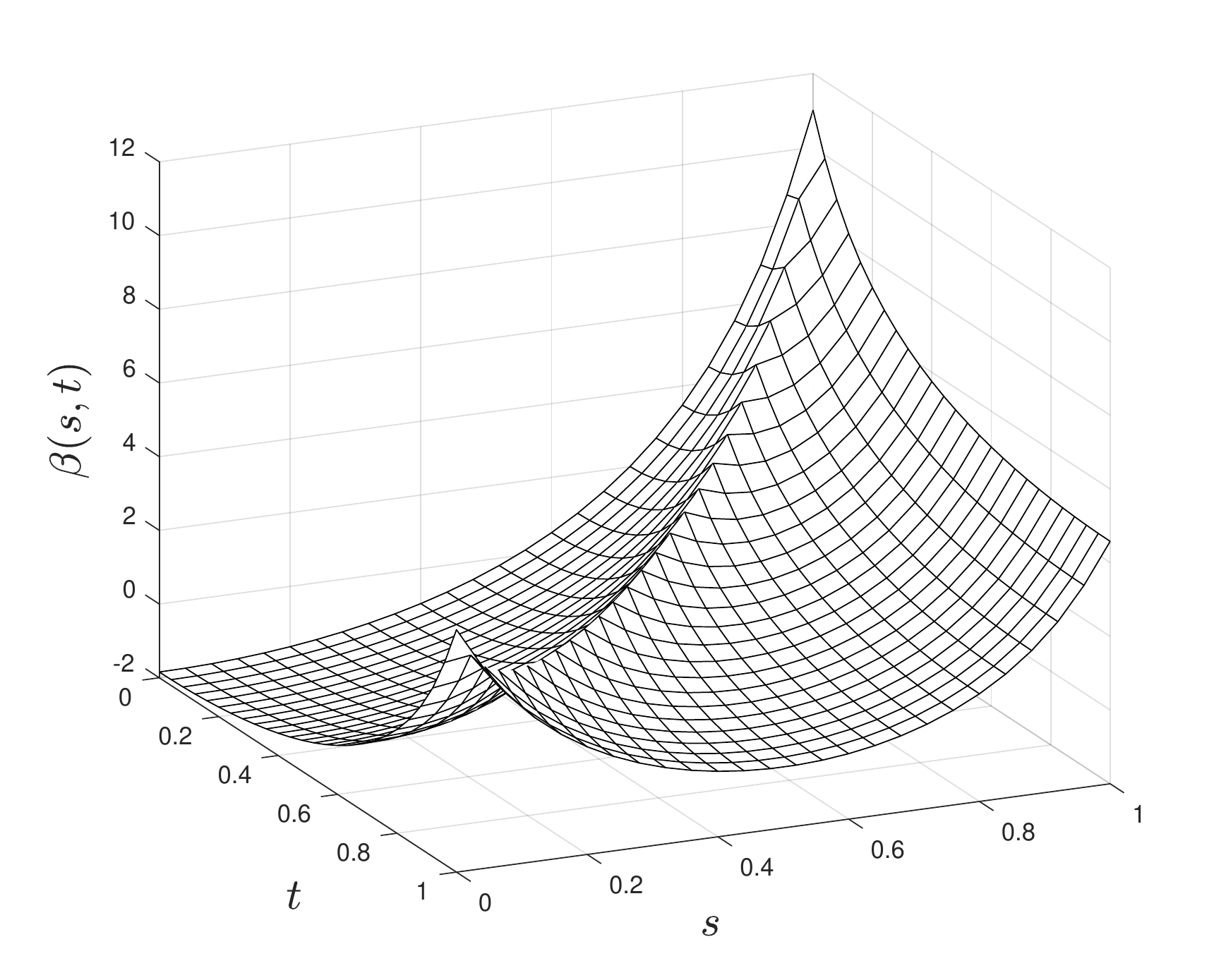}}{True}
\hspace{-0.2cm}\stackunder[5pt]{\includegraphics[width=5cm]{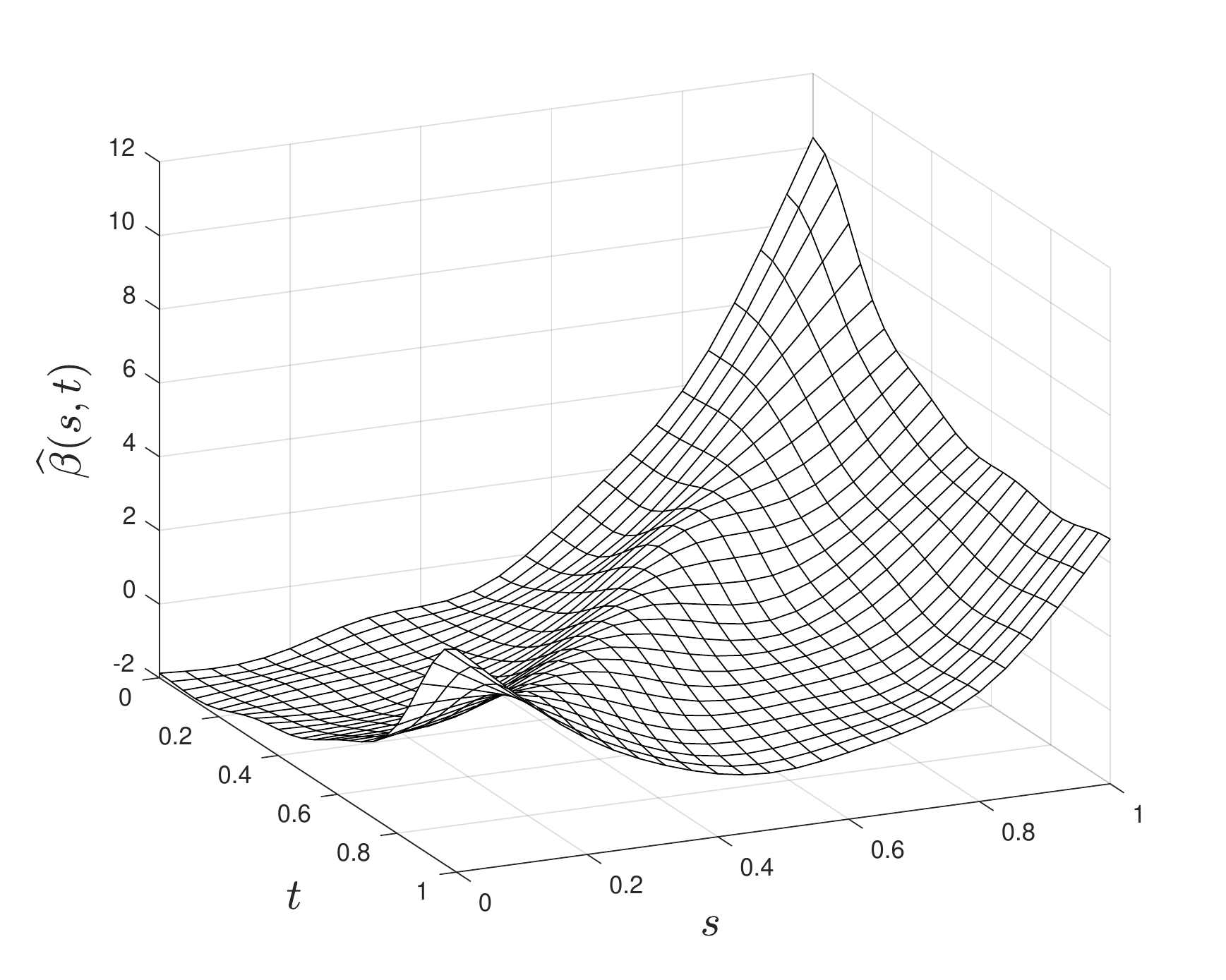}}{RK}
\hspace{-0.2cm}\stackunder[5pt]{\includegraphics[width=5cm]{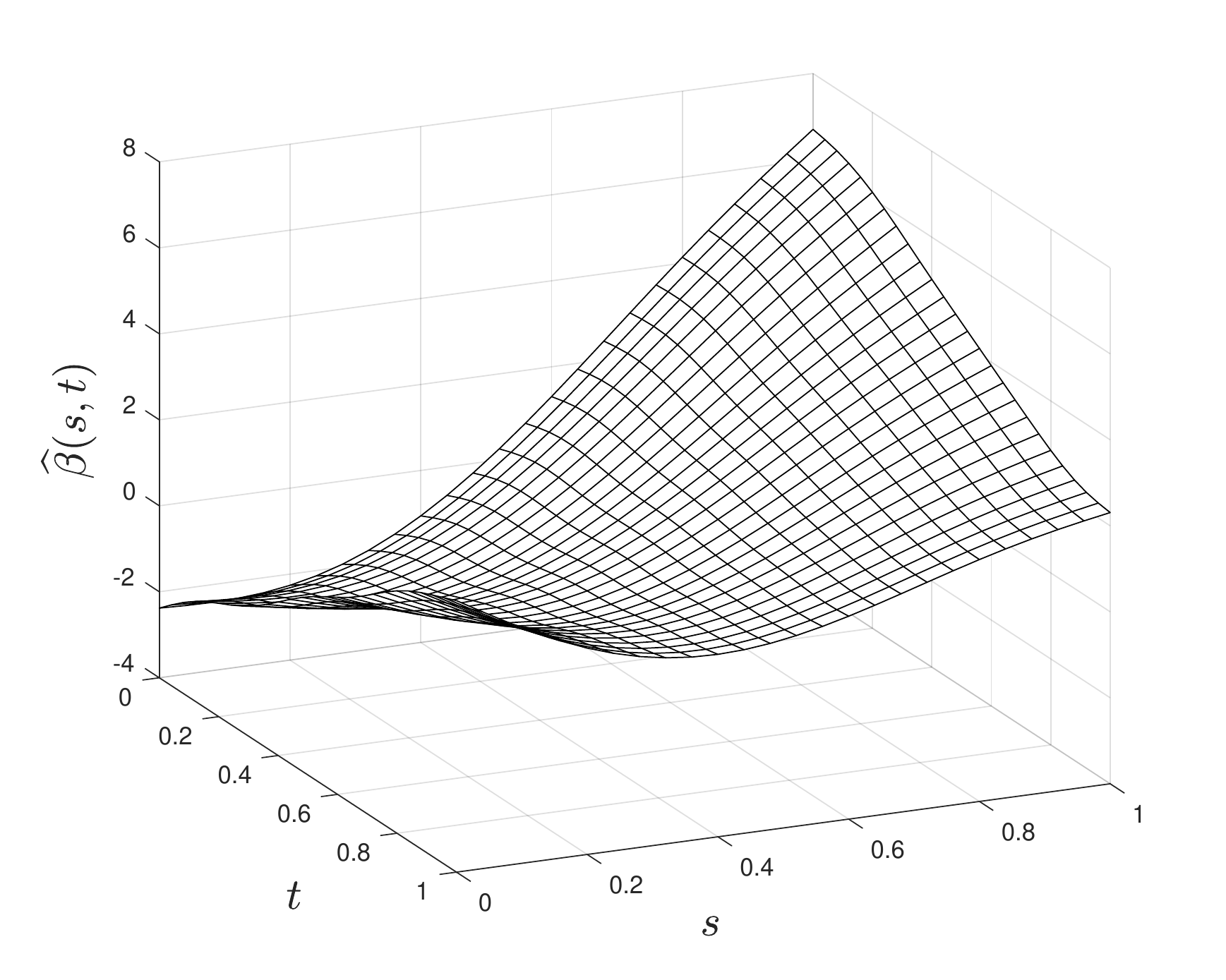}}{TP}
\caption{\it Plots of the ``true'' slope surface $\beta_0$  in model \eqref{model0} and the corresponding estimators using our method (RK) and \cite{sun2018}'s method (TP), under DGP 1--3 (rows 1--3) with error (i) with sample size $n=60$.}\label{fig:setting1}
\end{figure}

\begin{sidewaystable}
\begin{small}
\centering
\begin{tabular}{ccccc|cc|ccc}
\cline{1-9}
\multirow{2}{*}{DGP}&\multirow{2}{*}{Error}&\multirow{2}{*}{Method}&\multicolumn{2}{c|}{${\rm ISE}\ (\times10^{-3})$}&\multicolumn{2}{c|}{${\rm EPR}\ (\times10^{-4})$}&\multicolumn{2}{c}{${\rm MD}\ (\times 10^{-1})$}&\\
\cline{4-9}
&&&$n=30$&$n=60$& $n=30$& $n=60$&$n=30$&$n=60$\\    
\cline{1-9}
\multirow{6}{*}{1}&\multirow{2}{*}{(i)} &{RK} &17.6 [16.4, 20.1]&13.4 [11.6, 16.2]&31.7 [28.0, 35.6]&27.3 [23.7, 31.3]& 11.0 [9.32, 12.2]& 7.21 [6.23, 8.50]\\ 
&&{TP}& 73.8 [73.1, 74.6]&53.1 [52.4, 54.3]&87.8 [81.9, 95.2]&61.7 [56.5, 68.8]& 13.2 [12.2, 14.3]& 11.4 [10.3, 12.5]\\
\cline{2-9}
&\multirow{2}{*}{(ii)}&{RK} &17.0 [16.1, 18.5]&12.6 [10.6, 15.1]&28.3 [24.7, 33.0]&25.3 [21.9, 29.6]& 10.5 [9.27, 12.0]& 6.58 [5.83, 7.59]\\
&&{TP}&68.7 [68.1, 69.6]&33.0 [32.4, 33.7]&85.9 [80.9, 91.5]&38.0 [34.9, 45.3]& 13.0 [12.0, 13.9]& 10.4 [9.13, 11.7]\\
\cline{2-9}
&\multirow{2}{*}{(iii)}&{RK} &23.2 [19.7, 28.6]&21.1 [17.5, 26.0]&44.7 [38.6, 53.9]&39.0 [34.2, 44.3]& 11.7 [9.23, 14.2]& 7.51 [6.57, 9.01]\\
&&{TP}&79.3 [78.0, 80.6]&34.0 [32.8, 37.6]&96.1 [88.9, 112]&57.3 [47.9, 68.4]& 18.9 [15.5, 22.2]& 11.8 [10.4, 13.5]\\
\cline{1-9}
\multirow{6}{*}{2}&\multirow{2}{*}{(i)} &{RK} &1.14 [0.75, 1.74]&0.65 [0.46, 0.88]&7.45 [5.37, 9.76]&4.62 [3.36, 5.61]& 0.86 [0.67, 1.13]& 0.65 [0.54, 0.85]\\ 
&&{TP} & 0.87 [0.67, 1.23]&0.38 [0.41, 0.48]&5.67 [4.59, 7.56]&3.96 [3.19, 4.20]& 0.45 [0.32, 0.49]& 0.34 [0.29, 0.39]\\
\cline{2-9}
&\multirow{2}{*}{(ii)}&{RK} &0.32 [0.23, 0.40]&0.23 [0.17, 0.29]&2.29 [1.83, 2.96]&1.97 [1.68, 2.37]& 0.53 [0.46, 0.59]& 0.51 [0.45, 0.55]\\
&&{TP}&0.22 [0.18, 0.28]&0.18 [0.11, 0.25]&2.18 [2.09, 2.51]&1.05 [1.96, 2.14]& 0.33 [0.32, 0.35]& 0.33 [0.32, 0.35]\\
\cline{2-9}
&\multirow{2}{*}{(iii)}&{RK} &2.39 [1.49, 3.47]&1.08 [0.75, 1.57]&13.1 [9.38, 17.4]&6.96 [5.09, 8.96]& 1.28 [0.95, 1.65]& 0.80 [0.65, 1.14]\\
&&{TP}&1.50 [0.99, 1.95]&0.70 [0.40, 0.86]&7.77 [5.71, 9.87]&5.52 [4.44, 7.52]& 0.47 [0.37, 0.68]& 0.42 [0.35, 0.51]\\
\cline{1-9}
\multirow{6}{*}{3}&\multirow{2}{*}{(i)} &{RK} &28.6 [26.7, 31.6]&15.0 [12.9, 18.6]&37.7 [33.0, 43.6]&31.5 [27.2, 36.2]& 12.8 [11.0, 15.6]& 8.97 [7.88, 10.2]\\ 
&&{TP}&146 [142, 150]&83.8 [81.8, 86.9]&138 [129, 149]&77.7 [71.7, 73.3]& 14.5 [11.8, 17.5]& 11.3 [9.31, 13.5]\\
\cline{2-9}
&\multirow{2}{*}{(ii)}&{RK} &29.1 [27.1, 32.4]&17.7 [14.0, 22.2]&39.7 [33.4, 46.7]&34.6 [29.8, 39.4]& 11.9 [10.2, 14.6]& 8.81 [7.27, 10.0]\\
&&{TP}&150 [147, 153]& 84.4 [82.6, 88.6]&141 [129, 156]&81.8 [72.4, 92.9]& 15.8 [13.1, 19.7]& 11.9 [10.0, 14.6]\\
\cline{2-9}
&\multirow{2}{*}{(iii)}&{RK} &32.5 [29.2, 39.0]&22.5 [18.4, 26.9]&54.3 [44.3, 62.4]&41.0 [36.3, 48.1]& 15.5 [12.9, 18.1]& 9.79 [8.27, 11.6]\\
&&{TP}&170 [166, 174]&86.6 [83.3, 92.5]&148 [134, 163]&94.8 [84.4, 105.1]& 23.5 [22.2, 24.9]& 13.8 [11.6, 16.7]\\
\cline{1-9}
\end{tabular}
\caption{\it The three quartiles $(2^{\rm nd}\, [1^{\rm st},\, 3^{\rm rd}])$ of integrated squared error $({\rm ISE})$, excess prediction rate $({\rm EPR})$ and maximum deviation $({\rm MD})$ of estimators computed from 1000 simulation runs under the data generating processes 1--3 with error processes (i)--(iii), using our method (RK) and \cite{sun2018}'s method (TP).\label{table:estimation}}
\end{small}
\end{sidewaystable}

\begin{figure}[h!]
\centering
\includegraphics[width=5cm]{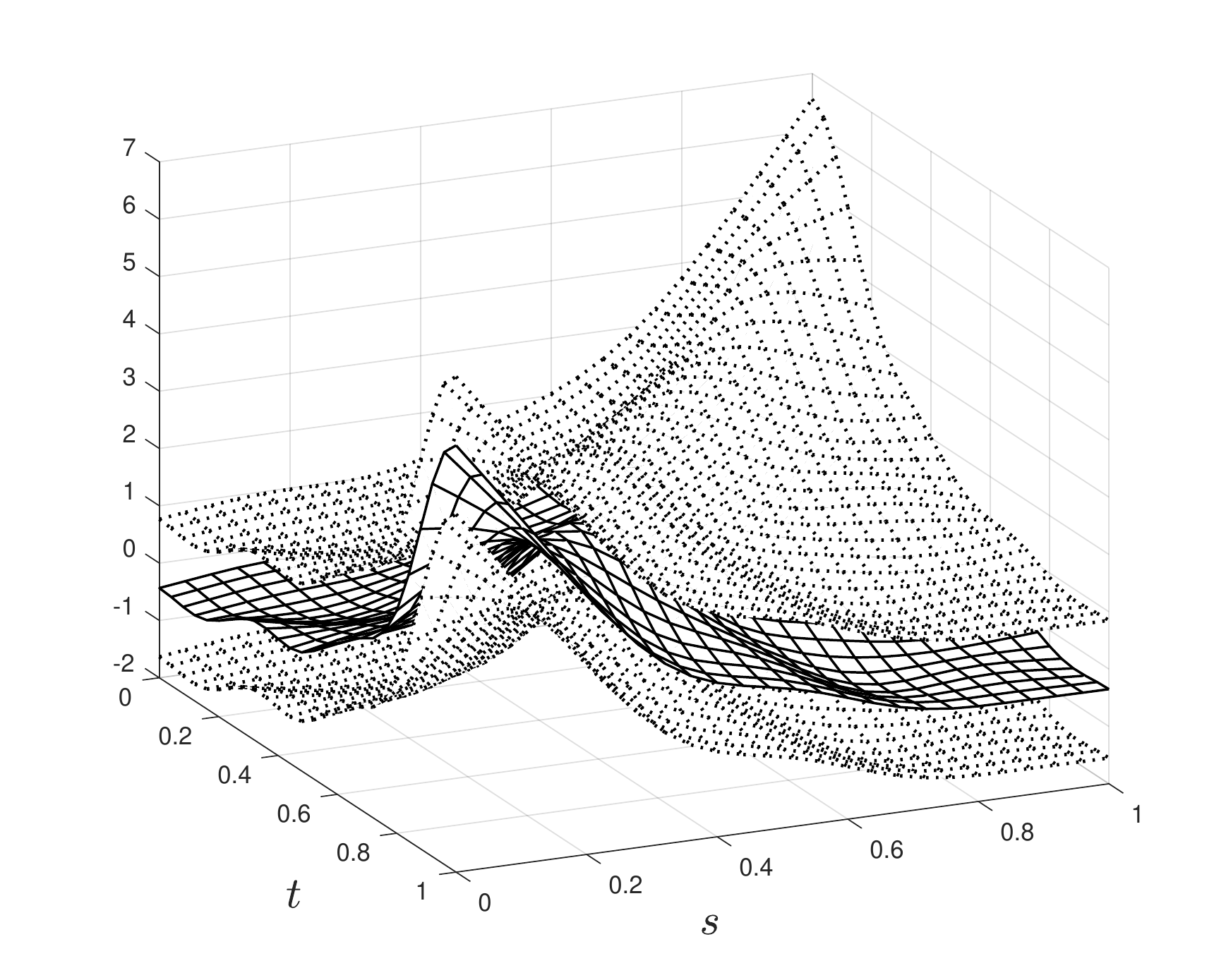}
\includegraphics[width=5cm]{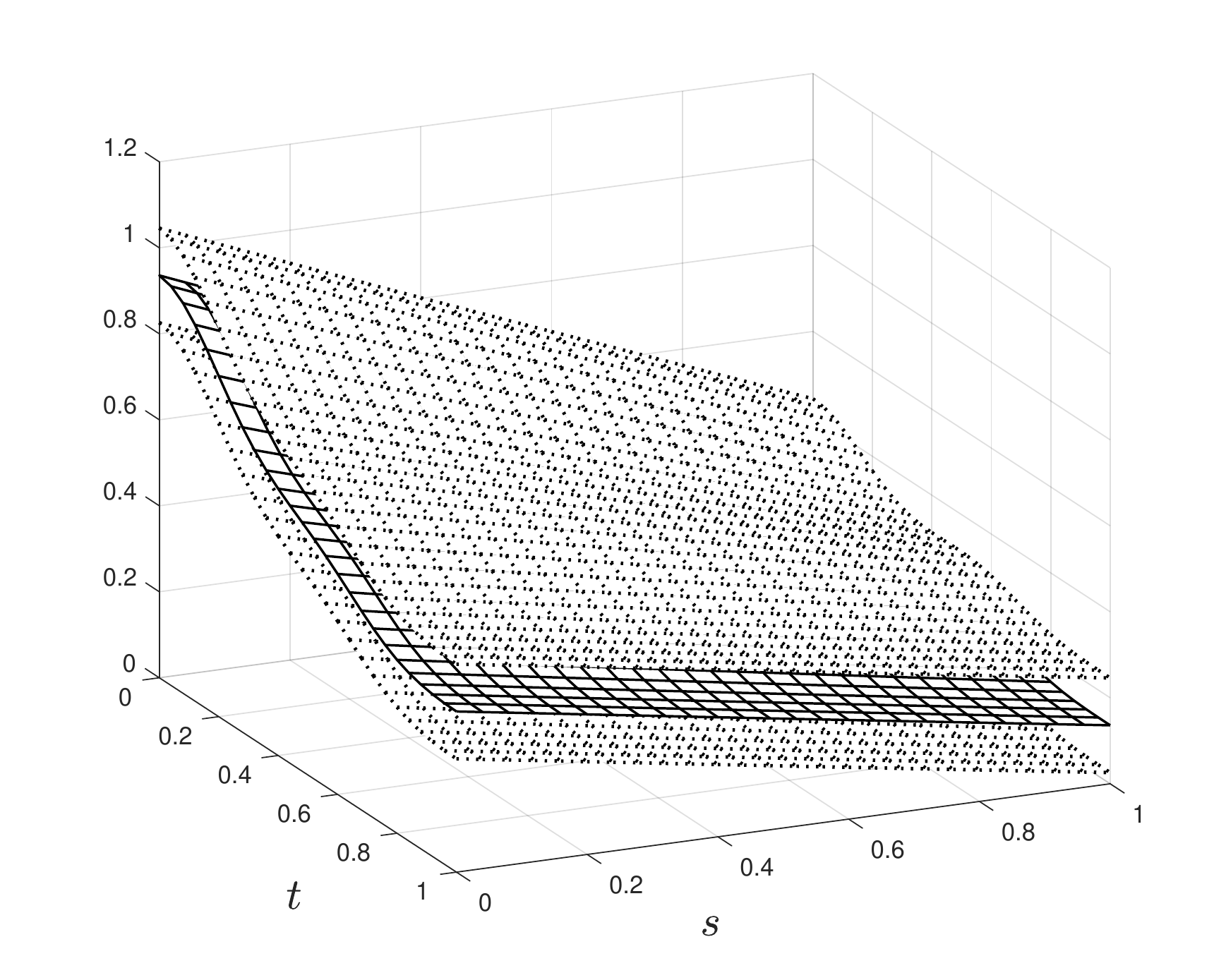}
\includegraphics[width=5cm]{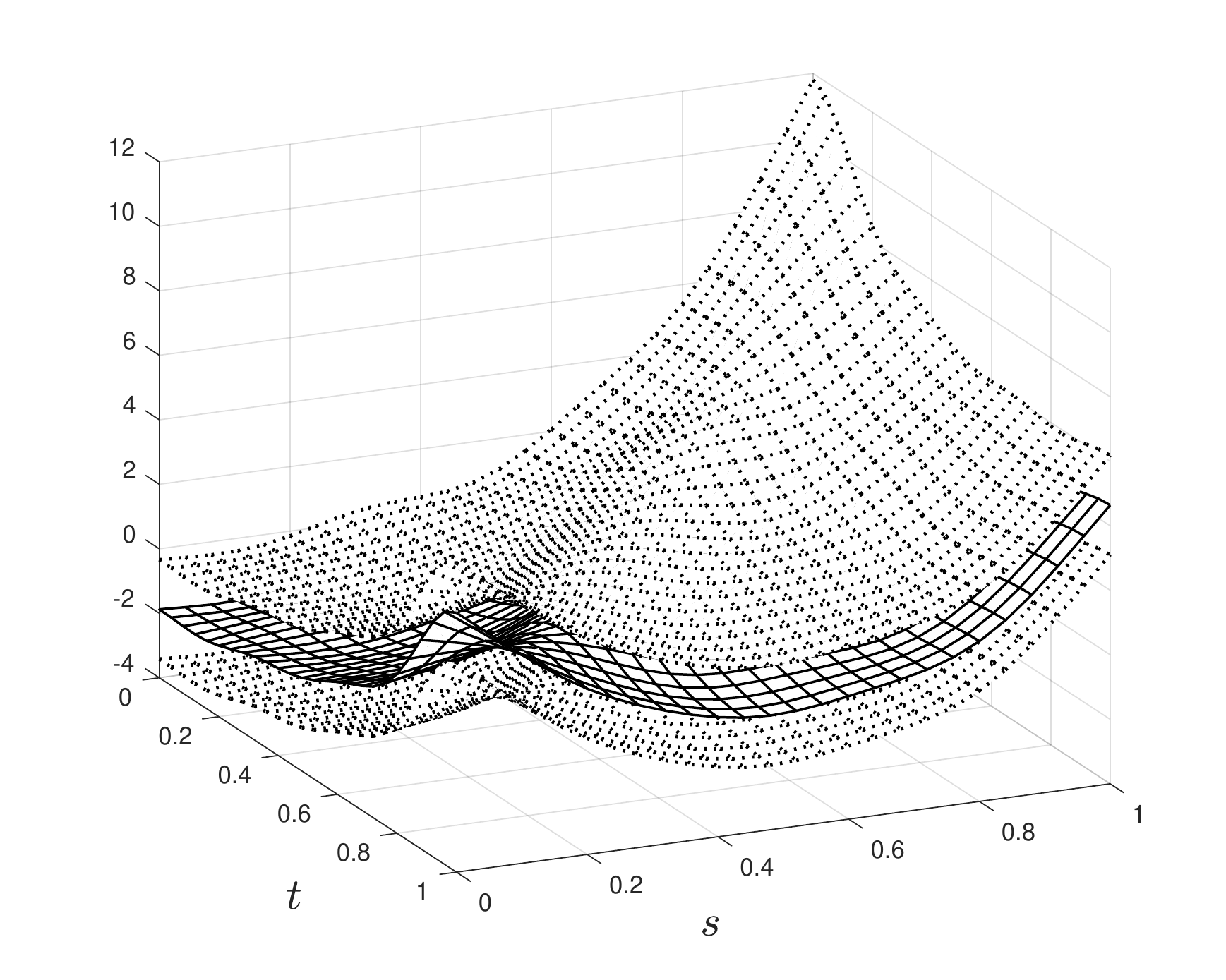}
\caption{\it Simultaneous $0.95$ confidence regions for the slope surface $\beta_0$ under DGPs 1--3 with error (i) and sample size $n=60$.}\label{fig:1}
\end{figure}

\begin{table}[h!]
\centering
\begin{small}
\begin{tabular}{c|cc|cc|cc} 
\cline{1-7}
DGP&\multicolumn{2}{c|}{1}&\multicolumn{2}{c|}{2}&\multicolumn{2}{c}{3}\\
\cline{1-7}
$\alpha$&0.10&0.05&0.10&0.05&0.10&0.05\\
\cline{1-7}
$n=30$&0.863&0.932&0.882&0.963&0.858&0.915\\
$n=60$&0.881&0.944&0.913&0.960&0.870&0.939\\
\cline{1-7}
\end{tabular} \caption{\it Empirical  covering probabilities  of the simultaneous confidence region \eqref{det1} for the slope surface $\beta_0$ under DGPs 1--3 with error setting (i).\label{table:cp}}
\end{small}
\end{table}

\begin{table}[h!]
\centering
\begin{small}
\begin{tabular}{ccc|cc|ccccc} 
\cline{1-7}
&\multicolumn{2}{c|}{(i)}&\multicolumn{2}{c|}{(ii)}&\multicolumn{2}{c}{(iii)}\\
\cline{2-7}
&$n=30$&$n=60$& $n=30$& $n=60$& $n=30$& $n=60$\\    
\cline{1-7}
\multirow{2}{*}{BT} &0.536 & 0.693  &0.328  &0.494 &0.478  &0.546  \\ 
 &(0.061)&(0.037) & (0.012)&(0.028) &(0.063)  &(0.042) \\
\cline{1-7}
\multirow{2}{*}{PLRT}& 0.367& 0.562 & 0.304 &0.418 &  0.330&0.513\\
 &(0.059) &(0.020) & (0.074)&(0.057) &(0.027) &(0.060) \\
\cline{1-7}
\end{tabular} \caption{\it Empirical rejection probabilities under DGP~2, together with empirical sizes (in brackets) of the decision rule based on the bootstrap confidence region $({\rm BT})$ in \eqref{eq:chrule} and the penalized likelihood ratio test $({\rm PLRT})$ in \eqref{eq:plrt} for the classical hypothesis \eqref{simu:ch} with error settings (i)--(iii) at nominal level $\alpha=0.05$.\label{table:ch}}
\end{small}
\end{table}

\begin{figure}[h!]
\centering
\hspace{-0.5cm}\includegraphics[width=6cm]{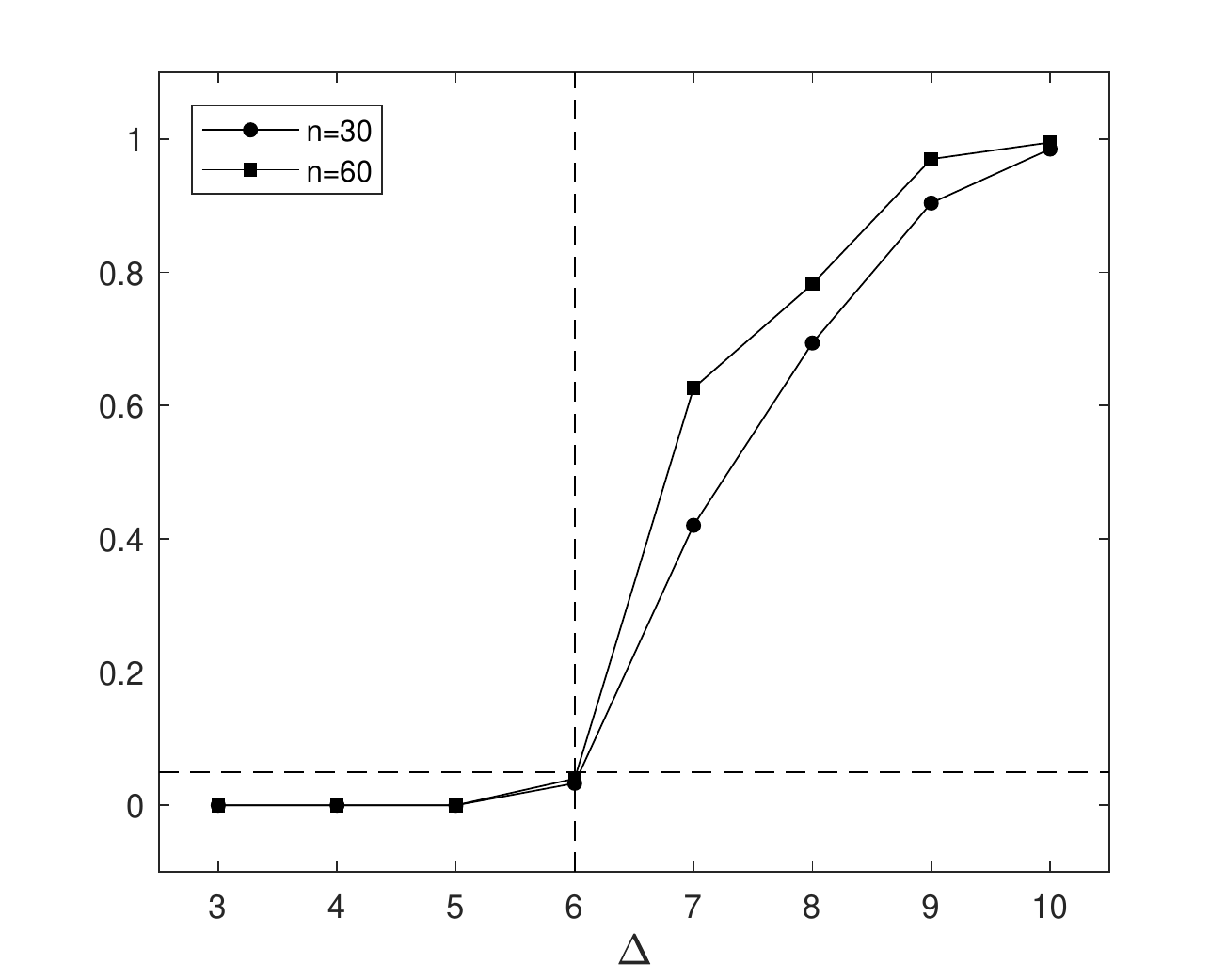}\hspace{-0.5cm}\includegraphics[width=6cm]{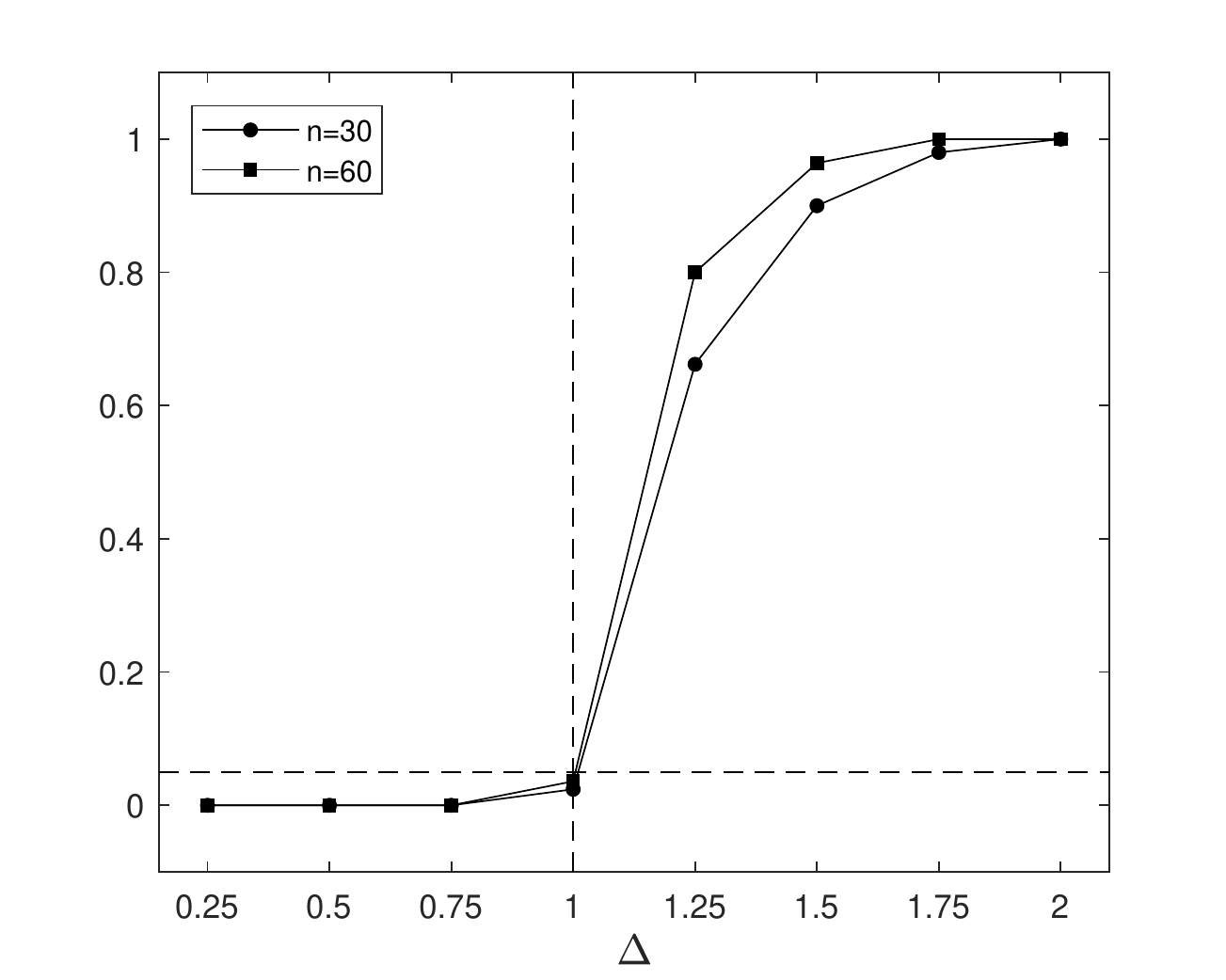}\hspace{-0.5cm}\includegraphics[width=6cm]{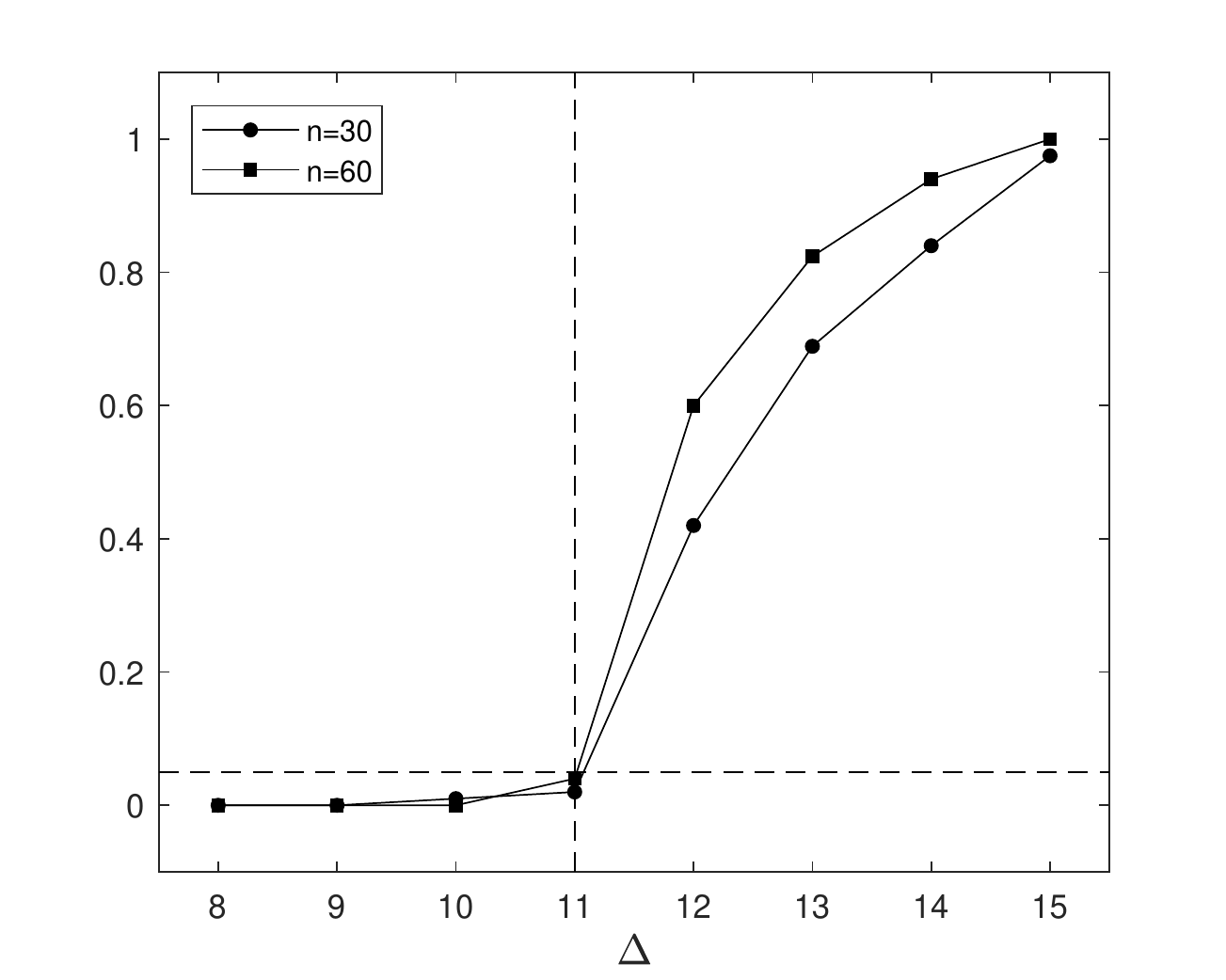}
\caption{\it Empirical rejection probabilities of test \eqref{d2} for the relevant hypothesis in \eqref{simu:rele} at nominal level $\alpha=0.05$, under DGPs 1--3 with error setting (i) and sample size $n=30,60$, for different values of $\Delta$ in \eqref{rele}. The horizontal dashed line is the nominal level 0.05; the vertical dashed line is $\Delta=d_\infty$.\label{fig:rele}}
\end{figure}

\subsection{Real data example}\label{sec:real}


We applied the new methodology to the Canadian weather data in \cite{ramsay2005}, which consists of daily temperature and precipitation at $n=35$ locations in Canada averaged over 1960 to 1994. In this case, for $1\leq i\leq 35$, $X_i$ is the average daily temperature for each day of the year at the $i$-th location, and $Y_i$ is the base 10 logarithm of the corresponding average precipitation; see \cite{ramsay2005}, p.~248. We took the domain of $X$ and $Y$ to be $[0,1]$ with $365$ equality spaced time points. The size of the  bootstrap sample is  $Q=300$ and  the truncation parameter is chosen as  $v=\lceil n^{2/5}\rceil=4$.
In Figure~\ref{fig:weather}, we display the estimated slope function $\beta_0$ and the 0.95 confidence region, using our method RK. In order to evaluate the prediction accuracy, for both our method RK and \cite{sun2018}'s method TP, we computed the integrated squared prediction error (ISPE) and maximum prediction deviation (MPD), for each observation ($1\leq i\leq n$), defined by
\begin{align}\label{m5}
\begin{split}
&{\rm ISPE}_i=\int_0^1\bigg|Y_i(t)-\int_0^1X_i(s)\,\hat\beta_{-i}(s)\,ds\bigg|^2\,dt\,;\\
&{\rm MPD}_i=\sup_{t\in[0,1]}\bigg|Y_i(t)-\int_0^1X_i(s)\,\hat\beta_{-i}(s)\,ds\bigg|\,,
\end{split}
\end{align}
where $\hat\beta_{-i}$ is the estimator of the slope function based on the data with the $i$-th observation removed. 
In Figure~\ref{fig:ispe}, we display the boxplot of $\{\sqrt{\rm ISPE}_i\}_{i=1}^n$ and $\{{\rm MPD}_i\}_{i=1}^n$, for both methods RK and TP. The results in Figure~\ref{fig:ispe} show that, in general, our method performs better in terms of prediction accuracy and robustness, which is indicated by a smaller median, smaller interquartile range in terms of $\sqrt{\rm ISPE}$ and ${\rm MPD}$, and fewer outliers of $\sqrt{\rm ISPE}$. In contrast, \cite{sun2018}'s method achieves a smaller minimum value of both $\sqrt{\rm ISPE}$ and ${\rm MPD}$.

\begin{figure}[h!]
\centering
\includegraphics[width=5cm]{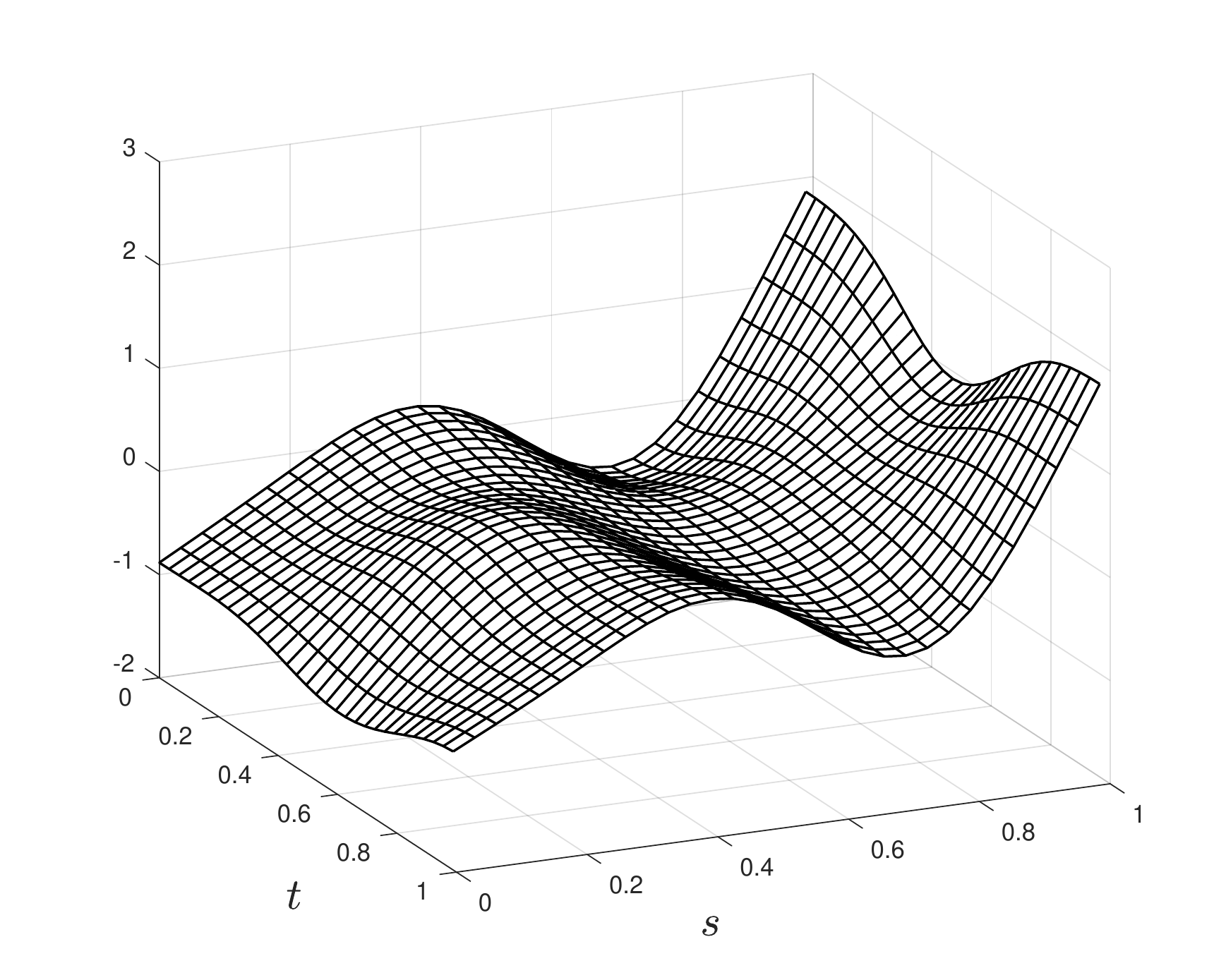}
\includegraphics[width=5cm]{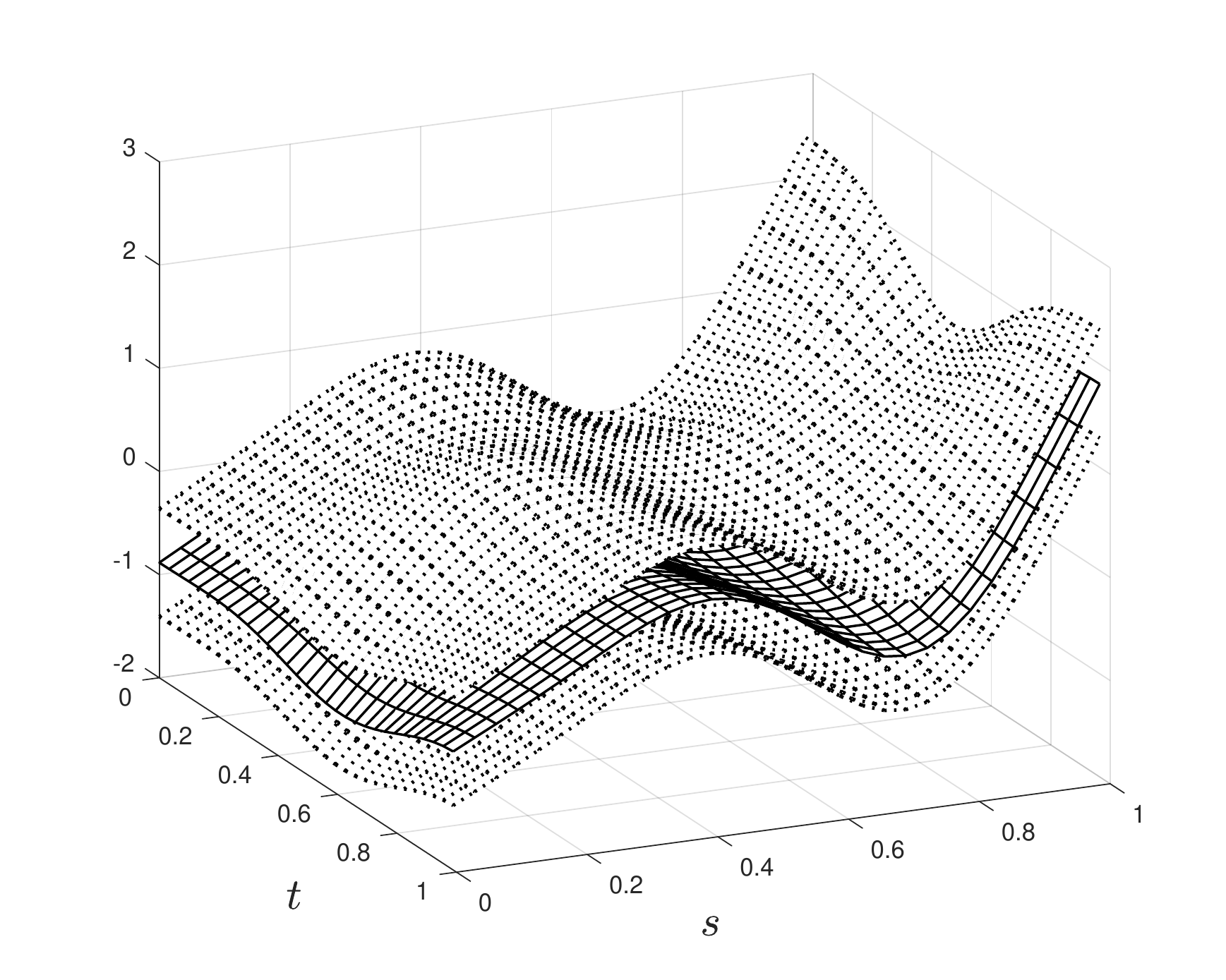}
\caption{\it Estimated slope surface in model \eqref{model0}
(left panel) and its 0.95 simultaneous band (right panel), using the Canadian weather data.}\label{fig:weather}
\end{figure}

\begin{figure}[h!]
\centering
\includegraphics[width=7cm]{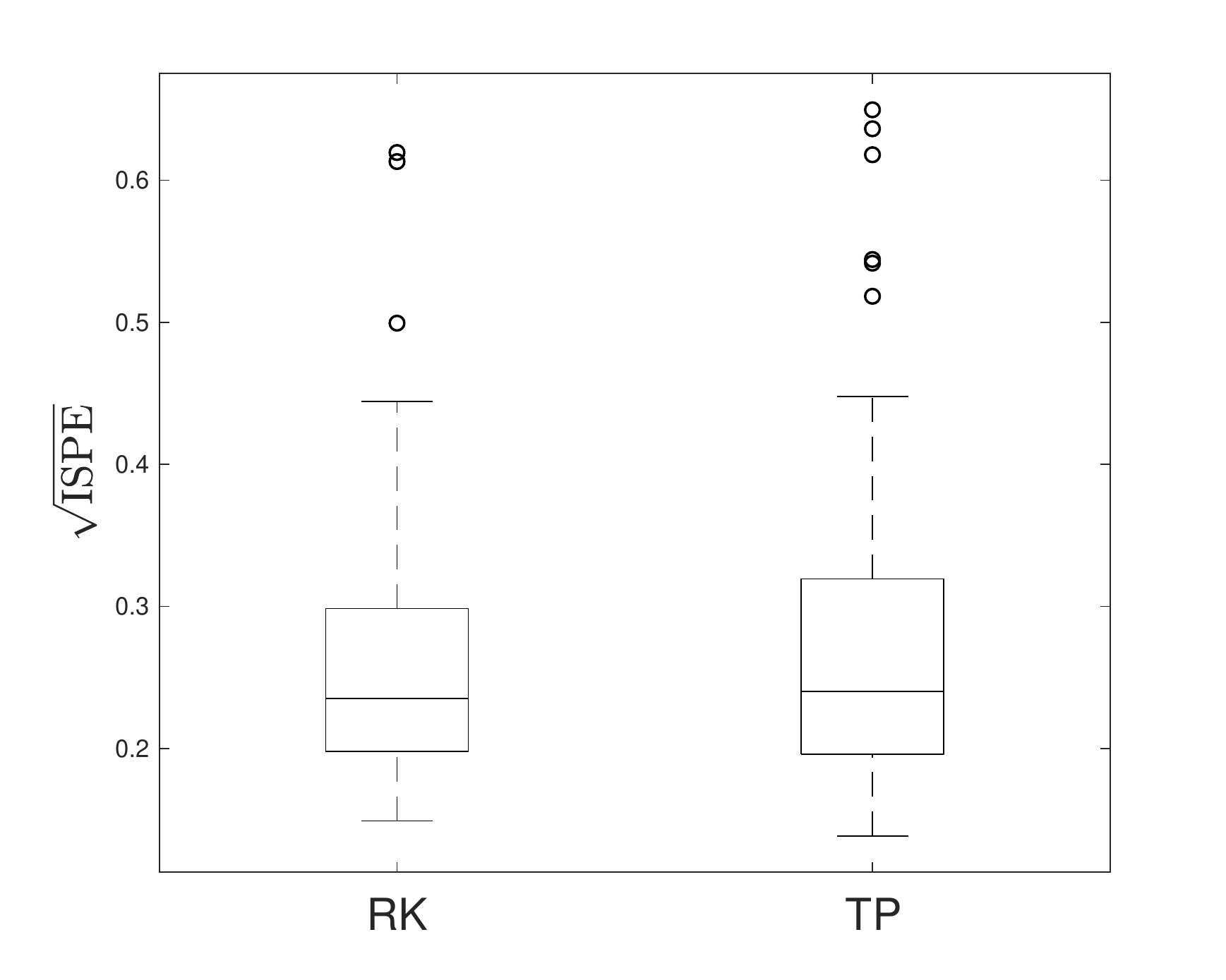}
\includegraphics[width=7cm]{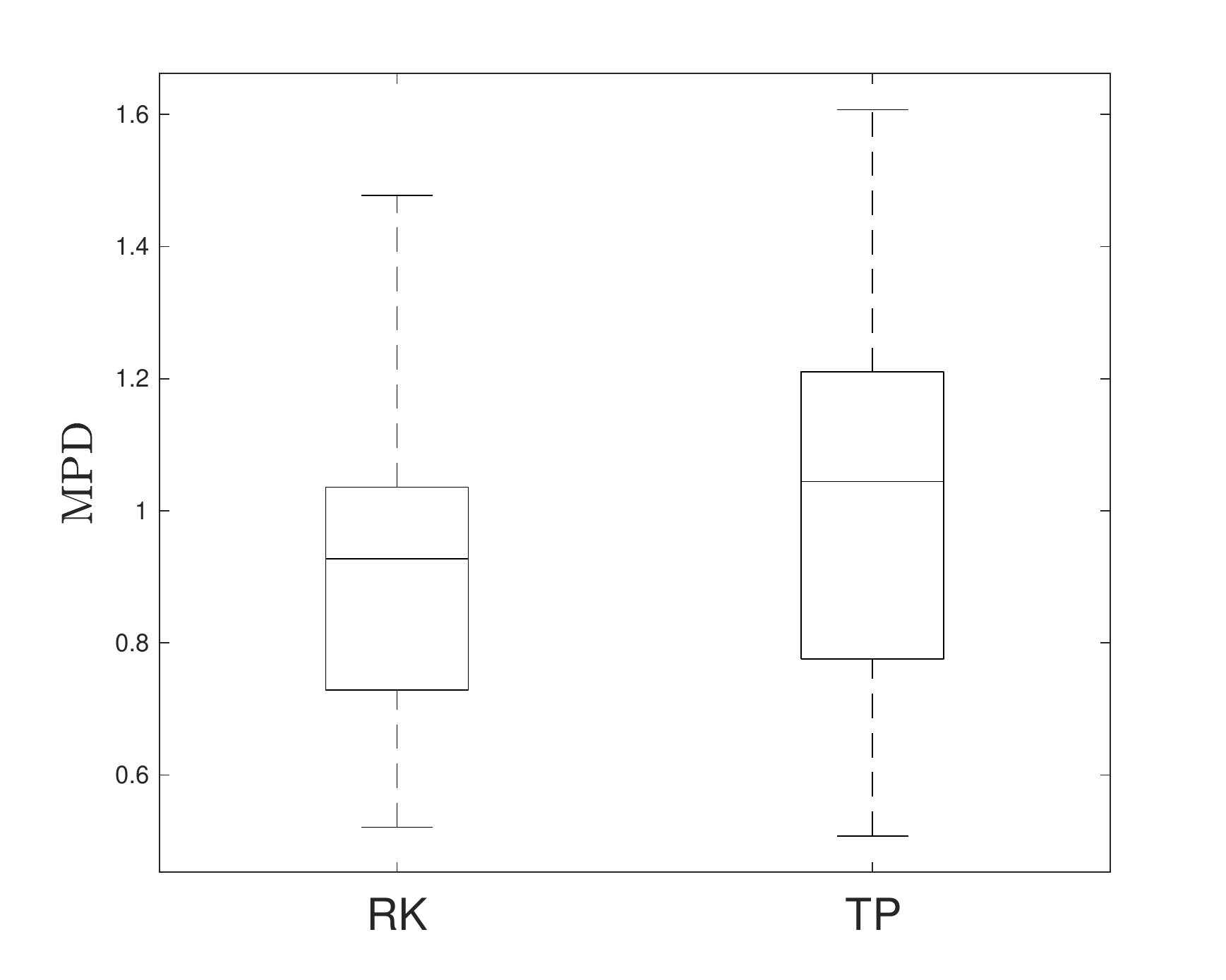}
\caption{\it Boxplot of the square root of the integrated squared prediction error $\sqrt{\rm ISPE}$ (left penal), and the maximum prediction deviation $({\rm MPD})$ (right penal) defined in \eqref{m5} using our method RK and \cite{sun2018}'s method TP.
}\label{fig:ispe}
\end{figure}

\newpage

\baselineskip=13pt
\vspace{1cm}

\begin{center}
{\large\textbf{Acknowledgements}}
\end{center}
\noindent 
This work has been supported in part by the
Collaborative Research Center ``Statistical modeling of nonlinear
dynamic processes'' (SFB 823, Teilprojekt A1,C1) of the German Research Foundation
(DFG). The authors would like to thank Martina Stein, who typed parts of this paper with considerable expertise.
The Canadian weather data are available in \cite{ramsay2014}'s \texttt{fda} R package.

\begingroup
\renewcommand{\section}[2]{\subsection#1{#2}}
{\centering

}

\endgroup

\clearpage

\appendix

\setcounter{page}{1}

\renewcommand{\thesection}{\Alph{section}}

\pagestyle{fancy}
\fancyhf{}
\rhead{\textbf{\thepage}}
\lhead{\textbf{NOT-FOR-PUBLICATION APPENDIX}}

\begin{center}
{\bf \large Supplementary material for ``Statistical inference for function-on-function linear regression"}

\end{center}

\baselineskip=17pt
\begin{center}
Holger Dette, Jiajun Tang\\
Fakult\"at f\"ur Mathematik, Ruhr-Universit\"at Bochum, Bochum, Germany
\end{center}

\vspace{.5cm}

In this supplementary material we provide technical details of our theoretical results. 
In Section~\ref{app:proof} we provide the proofs of our theorems in our main article. In Section~\ref{app:aux:lem} we provide supporting lemmas that are used in the proofs in Section~\ref{app:proof}. Section~\ref{app:aux:cond} provides a concrete example that satisfies Assumption~\ref{a201}.
In the sequel, we use $c$ to denote a generic positive constant that might differ from line to line.


\section{Proofs of main results}\label{app:proof}

\subsection{Proof of Proposition~\ref{prop1}}\label{app:prop:equinorm}

To begin with, by Mercer's theorem and Assumption~\ref{a1}, there exists a series of positive real values $\{ v_j\}_{j=1}^\infty$ and orthogonal basis functions $\{\gamma_j(s)\}_{j=1}^\infty$ of $L^2([0,1])$ such that
\begin{align}\label{ceigen}
C_X(s_1,s_2)=\sum_{j=1}^\infty v_j\gamma_j(s_1)\gamma_j(s_2)\,,
\end{align}
where $ v_{j}\geq  v_{j+1}>0$, for $j\geq 1$. For any $\beta\in\H$ such that $\Vert\beta\Vert_K^2=0$, for the $\gamma_j$ in \eqref{ceigen} and for $t\in[0,1]$, let $c_{\beta,j}(t)=\int_0^1\beta(s,t)\gamma_j(s)ds$, so that $\beta(s,t)=\sum_{j=1}^\infty\gamma_j(s)c_{\beta,j}(t)$. Since $\{\gamma_j\}_{j=1}^\infty$ is the orthogonal basis of $L^2([0,1])$, we have
\begin{align}\label{beta2}
\Vert\beta\Vert_{L^2}^2=\left\Vert\sum_{j=1}^\infty\gamma_j(s)c_{\beta,j}(t)\right\Vert_{L^2}^2= \sum_{j=1}^\infty\Vert\gamma_j\Vert_{L^2}^2\Vert c_{\beta,j}\Vert_{L^2}^2=\sum_{j=1}^\infty\Vert c_{\beta,j}\Vert_{L^2}^2\,.
\end{align}
Moreover,
\begin{align}\label{vbb}
V(\beta,\beta)&=\int_0^1\int_0^1\int_0^1C_X(s_1,s_2)\bigg\{\sum_{j=1}^\infty\gamma_j(s_1)c_{\beta,j}(t)\bigg\}\bigg\{\sum_{j=1}^\infty\gamma_j(s_2)c_{\beta,j}(t)\bigg\}ds_1ds_2dt\notag\\
&=\sum_{j=1}^\infty v_j\,\Vert c_{\beta,j}\Vert_{L^2}^2\,.
\end{align}
By the fact that $ v_j>0$ and Assumption~\ref{a1}, $\Vert\beta\Vert_K^2=0$ implies that, for any $j\geq1$, $ v_j\int_0^1c^2_{\beta,j}(t)dt=V(\beta,\beta)\leq \Vert\beta\Vert_K^2=0$, so that $c_{\beta,j}(t)=0$, and by \eqref{beta2} we have $\Vert\beta\Vert_{L^2}^2=\sum_{j=1}^\infty\Vert c_{\beta,j}\Vert_{L^2}^2=0$, which shows that $\Vert\beta\Vert_K^2=0$ implies $\beta=0$. Also, both $V$ and $J$ in \eqref{inner} are symmetric bilinear operators. Therefore, $\l\cdot,\cdot\r_K$ is an well-defined inner product. 

Next, we  show the equivalence of $\Vert\cdot\Vert_K$ and $\Vert\cdot\Vert_{\H}$. First, for any $\beta\in\H$, by \eqref{beta2} and \eqref{vbb},
\begin{align}\label{vl2}
V(\beta,\beta)&=\sum_{j=1}^\infty v_j\Vert c_{\beta,j}\Vert_{L^2}^2\leq  v_1 \sum_{j=1}^\infty\Vert c_{\beta,j}\Vert_{L^2}^2= v_1 \Vert\beta\Vert_{L^2}^2\leq c v_1 \Vert\beta\Vert_{\H}^2\,,
\end{align}
for some $c>0$. Hence,
\begin{align}\label{0}
\Vert\beta\Vert_K^2=V(\beta,\beta)+\lambda J(\beta,\beta)\leq (c v_1+\lambda)\Vert\beta\Vert_\H^2\,.
\end{align}
We proceed to show that there exists a constant $c_0>0$ such that $\Vert\beta\Vert_\H^2\leq c_0\Vert\beta\Vert_K^2$. To achieve this, recall the definition of    $J$ in \eqref{jm} and note that $J(\beta,\beta)$ is a semi-norm on $\H$. Let $\H_0=\{\beta\in\H:J(\beta,\beta)=0\}$ denote the null space of $J(\beta,\beta)$. It is known that $\H_0$ is a finite-dimensional subspace of $\H$ spanned by the polynomials of total degree $\leq m-1$, and $m_0:={\rm dim}\{\H_0\}=(m+1)m/2$; see \cite{whaba1990}. Let $\{\xi_1,\ldots,\xi_{m_0}\}$ denote an orthonormal basis of $\H_0$. Let $\H_1=\{\gamma_1\in\H:\l\gamma_1,\gamma_0\r_\H=0,\,\forall\gamma_0\in\H_0\}$ denote the orthogonal complement of $\H_0$ in $\H$, such that $\H=\H_0\oplus\H_1$, where ``$\oplus$" stands for the direct sum. That is, for any $\beta\in\H$, there are unique vectors $\beta_0,\beta_1$ such that
\begin{align}\label{beta}
\beta=\beta_0+\beta_1\,,\quad\beta_0\in\H_0\,,\ \beta_1\in\H_1\,,
\end{align}
Here, in view of \eqref{snorm}, $\l\cdot,\cdot\r_{\H}$ is the inner product corresponding to $\Vert\cdot\Vert_\H$ defined by
\begin{align*}
\l\gamma_1,\gamma_2\r_{\H}&=\sum_{0\leq\theta_1+\theta_2\leq m-1}{\theta_1+\theta_2\choose \theta_1}\int_{[0,1]^2} \frac{\partial^{\theta_1+\theta_2}\gamma_1}{\partial s^{\theta_1}\partial t^{\theta_2}}dsdt\times\int_{[0,1]^2} \frac{\partial^{\theta_1+\theta_2}\gamma_2}{\partial s^{\theta_1}\partial t^{\theta_2}}dsdt\\
&\qquad+\sum_{\theta_1+\theta_2=m}{m\choose\theta_1}\int_{[0,1]^2}\frac{\partial^m\gamma_1}{\partial s^{\theta_1}\partial t^{\theta_2}}\times\frac{\partial^m\gamma_2}{\partial s^{\theta_1}\partial t^{\theta_2}} dsdt\,,\qquad \text{for }\gamma_1,\gamma_2\in\H\,.
\end{align*}
Due to the fact that $\xi_k\in \H_0$, for $1\leq k\leq m_0$, we have
\begin{align*}
0=\l\xi_k,\beta_1\r_{\H}=\sum_{0\leq\theta_1+\theta_2\leq m-1}{\theta_1+\theta_2\choose \theta_1}\int_{[0,1]^2} \frac{\partial^{\theta_1+\theta_2}\xi_k}{\partial s^{\theta_1}\partial t^{\theta_2}}dsdt\times\int_{[0,1]^2} \frac{\partial^{\theta_1+\theta_2}\beta_1}{\partial s^{\theta_1}\partial t^{\theta_2}}dsdt\,,
\end{align*}
for $1\leq k\leq m_0$. We deduce from the above result that $\int_{[0,1]^2} \frac{\partial^{\theta_1+\theta_2}\beta_1}{\partial s^{\theta_1}\partial t^{\theta_2}}dsdt=0$, for $0\leq\theta_1+\theta_2\leq m-1$. In fact, the above argument shows that
\begin{align*}
\H_1=\left\{\gamma\in\H:\int_{[0,1]^2}\frac{\partial^{\theta_1+\theta_2}\beta_1}{\partial s^{\theta_1}\partial t^{\theta_2}}dsdt=0,\, 0\leq\theta_1+\theta_2\leq m-1\right\}\,.
\end{align*}
Therefore,
\begin{align}\label{1}
\Vert\beta_1\Vert_\H^2=J(\beta_1,\beta_1)\leq\lambda^{-1} V(\beta,\beta)+J(\beta_1,\beta_1)=\lambda^{-1}\Vert\beta\Vert_K^2\,.
\end{align}

It then suffice to show that $\Vert\beta_0\Vert_\H^2\leq c_0\Vert\beta\Vert_K^2$ for some $c_0>0$. Since $\beta_0\in\H_0$, we have $J(\beta_0,\beta_0)=J(\beta_0,\beta_1)=0$, so that in view of \eqref{1},
\begin{align}\label{3}
\Vert\beta\Vert_K^2&=V(\beta_0+\beta_1,\beta_0+\beta_1)+\lambda J(\beta_1,\beta_1)\notag\\
&=V(\beta_0,\beta_0)+2V(\beta_0,\beta_1)+V(\beta_1,\beta_1)+\lambda\Vert\beta_1\Vert_\H^2\,.
\end{align}
Since $V(\cdot,\cdot)$ is an inner product, by the Cauchy-Schwarz inequality,
\begin{align}\label{4}
|V(\beta_0,\beta_1)|\leq \{V(\beta_0,\beta_0)\}^{1/2}\{V(\beta_1,\beta_1)\}^{1/2}
\end{align}

Next, we examine the connection between $\Vert\beta_1\Vert_\H^2$ and $V(\beta_1,\beta_1)$. It is known that both $\H_0$ and $\H_1$ are reproducing kernel Hilbert spaces with inner product $\l\cdot,\cdot\r_\H$ restricted to $\H_0$ and $\H_1$, respectively. Let $C_1\{(s_1,t_1),(s_2,t_2)\}$ denote the reproducing kernel of $\H_1$. It is known that $C_1$ is continuous and square-integrable on $[0,1]^2\times[0,1]^2$; see, for example, Section~2.4 in \cite{whaba1990} and Section~4.3.2 in \cite{gu2013}. Hence, by Mercer's theorem, $C_1$ admits the following eigen-decomposition:
\begin{align*}
C_1\{(s_1,t_1),(s_2,t_2)\}=\sum_{j=1}^\infty\zeta_{j}\,\chi_j(s_1,t_1)\chi_j(s_2,t_2)\,,
\end{align*}
where $\zeta_j\geq\zeta_{j+1}\geq0$, for $j\geq 1$, $\{\chi_j\}_{j\geq1}$ forms an orthonormal basis of $L^2([0,1]^2)$, and $s_1,s_2,t_1,t_2\in[0,1]$. Note that
\begin{align*}
\l\chi_j,\chi_\ell\r_{L^2}=\delta_{j\ell}\,,\qquad \l\chi_j,\chi_\ell\r_{\H}=\zeta_j^{-1}\delta_{j\ell}\,,
\end{align*}
where $\delta_{j\ell}$ is the Kronecker delta; see, for example, \cite{cucker2001} and \cite{yuancai}. For $\beta_1$ in \eqref{beta}, we have $\beta_1(s,t)=\sum_{j=1}^\infty\l\beta_1,\chi_j\r_{L^2}\chi_j(s,t)$, so that
\begin{align}\label{2}
\Vert\beta_1\Vert_{\H}^2&=\sum_{j=1}^\infty\l\beta_1,\chi_j\r_{L^2}^2\Vert\chi_j\Vert_\H^2=\sum_{j=1}^\infty\zeta_j^{-1}\l\beta_1,\chi_j\r_{L^2}^2\notag\\
&\geq\zeta_1^{-1} \sum_{j=1}^\infty\l\beta_1,\chi_j\r_{L^2}^2=\zeta_1^{-1}\Vert\beta_1\Vert_{L^2}^2\,.
\end{align}
In view of \eqref{beta2} and \eqref{vbb},
\begin{align*}
V(\beta_1,\beta_1)&=\sum_{j=1}^\infty v_j\Vert c_{\beta_1,j}\Vert_{L^2}^2\leq v_1\sum_{j=1}^\infty\Vert c_{\beta_1,j}\Vert_{L^2}^2= v_1\,\Vert\beta_1\Vert_{L^2}^2\,.
\end{align*}
Combining the above equation with \eqref{2} yields that
\begin{align*}
\Vert\beta_1\Vert_{\H}^2\geq\zeta_1^{-1} v_1^{-1} V(\beta_1,\beta_1)\,.
\end{align*}
Therefore, combining the above equation with \eqref{3} and \eqref{4}, we find that
\begin{align}\label{5}
\Vert\beta\Vert_K^2&\geq V(\beta_0,\beta_0)-2|V(\beta_0,\beta_1)|+V(\beta_1,\beta_1)+\lambda\Vert\beta_1\Vert_\H^2\notag\\
&\geq V(\beta_0,\beta_0)-2\{V(\beta_0,\beta_0)\}^{1/2}\{V(\beta_1,\beta_1)\}^{1/2}+\left(1+\frac{\lambda}{ v_1\zeta_1}\right)V(\beta_1,\beta_1)\notag\\
&=\frac{\lambda}{ v_1\zeta_1+\lambda}V(\beta_0,\beta_0)+\left[\sqrt{\frac{ v_1\zeta_1}{ v_1\zeta_1+\lambda}}\{V(\beta_0,\beta_0)\}^{1/2}-\sqrt{\frac{ v_1\zeta_1+\lambda}{ v_1\zeta_1}}\{V(\beta_1,\beta_1)\}^{1/2}\right]^2\notag\\
&\geq\frac{\lambda}{ v_1\zeta_1+\lambda}V(\beta_0,\beta_0)\,.
\end{align}

Next, we examine the connection between $V(\beta_0,\beta_0)$ and $\Vert\beta_0\Vert_{\H}^2$. Since $\beta_0\in\H_0$ and $\{\xi_j\}_{j=1}^{m_0}$ is an orthonormal basis of $\H_0$ under the inner product $\l\cdot,\cdot\r_{\H}$, we have $\beta_0(s,t)=\sum_{j=1}^{m_0}\l\beta_0,\xi_j\r_{\H}\,\xi_j(s,t)$ and $\Vert\beta_0\Vert_{\H}^2=\sum_{j=1}^{m_0}\l\beta_0,\xi_j\r_{\H}^2$. Note that
\begin{align*}
V(\beta_0,\beta_0)=\sum_{j=1}^{m_0}\sum_{\ell=1}^{m_0}\l\beta_0,\xi_j\r_{\H}\,\l\beta_0,\xi_\ell\r_{\H}\,V(\xi_j,\xi_\ell)\,.
\end{align*}
Let $b$ denote an $m_0\times 1$ vector whose $j$-th entry is $\l\beta_0,\xi_j\r_{\H}$, and let $V_*$ denote an $m_0\times m_0$ matrix, whose $(j,\ell)$-th entry is
\begin{align*}
&V(\xi_j,\xi_\ell)=\int_{[0,1]^3}C_X(s_1,s_2)\xi_j(s_1,t)\xi_\ell(s_2,t)\,ds_1ds_2dt\,.
\end{align*}
Now, we have $V(\beta_0,\beta_0)=b^{\rm T}V_*b$ and $\Vert\beta_0\Vert_{\H}^2=\Vert b\Vert_2^2$. Due to Assumption~\ref{a1}, the matrix $V_*$ is a positive definite matrix, and therefore admits a singular value decomposition $V_*=U^{\rm T}DU$, where $U$ is an orthogonal matrix and $W={\rm diag}(d_1,\ldots,d_{m_0})$ is a diagonal matrix with $d_1\geq\ldots\geq d_{m_0}>0$. Therefore,
\begin{align*}
V(\beta_0,\beta_0)=b^{\rm T}U^{\rm T}WUb\geq d_{m_0}\Vert Ub\Vert_2^2=d_{m_0}\Vert b\Vert_2^2=d_{m_0}\Vert\beta_0\Vert_{\H}^2\,.
\end{align*}

Therefore, combining the above result with \eqref{5}, we find
\begin{align*}
\Vert\beta_0\Vert_{\H}^2\leq d_{m_0}^{-1}V(\beta_0,\beta_0)\leq \frac{ v_1\zeta_1+\lambda}{d_{m_0}\lambda}\Vert\beta\Vert_K^2\,.
\end{align*}
Combining the above equation with \eqref{1} yields that
\begin{align}\label{hk}
\Vert\beta\Vert_{\H}^2=\Vert\beta_0\Vert_{\H}^2+\Vert\beta_1\Vert_{\H}^2\leq\frac{ v_1\zeta_1+\lambda+d_{m_0}}{d_{m_0}\lambda}\Vert\beta\Vert_K^2\,.
\end{align}
This together with \eqref{0} completes the proof of the equivalence between $\Vert\cdot\Vert_{\H}$ and $\Vert\cdot\Vert_{K}$. Since $\H$ is a reproducing kernel Hilbert space equipped with $\Vert\cdot\Vert_{\H}$, we therefore deduce that $\H$ equipped with $\Vert\cdot\Vert_{K}$ is a reproducing kernel Hilbert space.

\subsection{Proof of Theorem~\ref{thm:bahadur}}\label{app:thm:bahadur}

We first prove the following lemma,
which is useful for proving Theorem~\ref{thm:bahadur}. In this section, without loss of generality, we assume that $\sigma_\e^2=1$ in Assumption~\ref{a0}.

\begin{lemma}\label{lem:hnbeta}
For any $\beta\in\H$, let
\begin{align}
&g(X_i,\beta)=
\tau\bigg[X_i\otimes\bigg\{\int_0^1\beta(s,\cdot)X_i(s)ds\bigg\}\bigg]\,,\label{gxbeta}\\
&H_n(\beta)=\frac{1}{\sqrt{n}}\sum_{i=1}^n\big[g(X_i,\beta)-\E\{g(X,\beta)\}\big]\,.\label{hnbeta}
\end{align}
For $p_n\geq 1$, let 
\begin{equation} \label{m6}
    \mathcal F_{p_n}=\{\beta\in\H:\Vert\beta\Vert_{L^2}\leq 1,J(\beta,\beta)\leq p_n\}.
    \end{equation}
    Then, under Assumptions~\ref{a1}--\ref{a:x}, as $n\to\infty$,
\begin{align}
&\sup_{\beta\in\mathcal{F}_{p_n}}\,\frac{\Vert H_n(\beta)\Vert_K}{p_n^{1/(2m)}\Vert\beta\Vert_{L^2}^{(m-1)/m}+n^{-1/2}}=O_p\big(\lambda^{-1/(2D)}\log\log n\big)^{1/2}\,.
\end{align}

\end{lemma}

\begin{proof}
We follow the proof of Lemma~3.4 in \cite{shang2015}. For the $\{x_{k\ell}\}_{k,\ell\geq1}$ and $\{\eta_\ell\}_{\ell\geq 1}$ in Assumption~\ref{a201}, and for $1\leq i\leq n$, let
\begin{align}\label{wx}
w(X_i)=\Vert X_i\Vert_{L^2}\bigg(\sum_{k,\ell}\frac{1}{1+\lambda\rho_{k\ell}}\,\l X_i,x_{k\ell}\r_{L^2}^2\,\Vert\eta_\ell\Vert^2_{L^2}\bigg)^{1/2}\,,
\end{align}
and let $\X_n=\{w(X_i)\}_{i=1}^n$. By Lemma~\ref{lem:l2} in Section \ref{app:aux},
\begin{align*}
&\frac{1}{\sqrt{n}}\,\Big\Vert\big[g(X_i,\beta_1)-\E\{g(X_i,\beta_1)\}\big]-\big[g(X_i,\beta_2)-\E\{g(X_i,\beta_2)\}\big]\Big\Vert_K\\
&\leq\frac{1}{\sqrt{n}}\,\bigg\Vert\tau\bigg(X_i\otimes\bigg[\int_0^1\{\beta_1(s,\cdot)-\beta_2(s,\cdot)\}X_i(s)ds\bigg]\bigg)\bigg\Vert_K\\
&\qquad+\frac{1}{\sqrt{n}}\,\E\,\bigg\Vert\tau\bigg(X_i\otimes\bigg[\int_0^1\{\beta_1(s,\cdot)-\beta_2(s,\cdot)\}X_i(s)ds\bigg]\bigg)\bigg\Vert_K\\
&\leq\frac{1}{\sqrt{n}}\,\Vert\beta_1-\beta_2\Vert_{L^2}\times\big[w(X_i)+\E\{w(X_i)\}\big]\,.
\end{align*}
By Theorem 3.5 in \cite{pinelis1994},
\begin{align*}
\P\left\{\Vert H_n(\beta_1)-H_n(\beta_2)\Vert_K\geq x\,\big|\mathcal{X}_n \right\}\leq 2\exp\left(-\frac{x^2}{2\,W_n^2\,\Vert\beta_1-\beta_2\Vert_{L^2}^2}\right)\,,
\end{align*}
where
\begin{align}\label{wn2}
W_n=\frac{1}{\sqrt n}\bigg(\sum_{i=1}^n\big[w(X_i)+\E\{w(X_i)\}\big]^2\bigg)^{1/2}\,.
\end{align}
By Lemma~\ref{lem:wx} in Section \ref{app:aux}, $\E(W_n^2)\leq 4\E|w(X)|^2\leq c\,\lambda^{-1/(2D)}$. Let 
$$\Vert Z\Vert_\Psi=\inf\big\{c>0:\E\{\Psi(|Z|/c)|\mathcal{X}_n\}\leq1\big\}
$$ 
denote the Orlicz norm of a random variable $Z$ conditional on $\mathcal{X}_n$, with $\Psi(x)=\exp(x^2)-1$, then, by Lemma 8.1 in \cite{kosorok2007},
\begin{align*}
\Big\Vert \Vert H_n(\beta_1)-H_n(\beta_2)\Vert_K\Big\Vert_\Psi\leq \sqrt{6}\,W_n\,\Vert\beta_1-\beta_2\Vert_{L^2}\,.
\end{align*}

Let $N(\delta,\mathcal{F}_{p_n},\Vert\cdot\Vert_{L^2})$ denote $\delta$-covering number of the class $\mathcal{F}_{p_n}$ in \eqref{m6} w.r.t.~the $L^2([0,1]^2)$-norm. Since $p_n\geq 1$ for $n$ large enough and $J(p_n^{1/2}\beta,p_n^{1/2}\beta)=p_nJ(\beta,\beta)$, we have $\mathcal{F}_{p_n}\subset p_n^{1/2}\mathcal{F}_{1}$. Hence,
\begin{align*}
\log N(\delta,\mathcal{F}_{p_n},\Vert\cdot\Vert_{L^2})&\leq \log N(\delta,p_n^{1/2}\mathcal{F}_{1},\Vert\cdot\Vert_{L^2})\\
&\leq \log N(p_n^{-1/2}\delta,\mathcal{F}_{1},\Vert\cdot\Vert_{L^2})\leq c\,(p_n^{-1/2}\delta)^{-2/m}\,,
\end{align*}
where in the last step we used the result 
in \cite{birman1967}.
By Lemma 8.2 and Theorem~8.4 in \cite{kosorok2007}, we have
\begin{align*}
&\left\Vert\sup_{\beta_1,\beta_2\in\F_{p_n},\,\Vert\beta_1-\beta_2\Vert_{L^2}\leq\delta}\Vert H_n(\beta_1)-H_n(\beta_2)\Vert_K\right\Vert_{\psi}\\
&\leq c\,W_n\bigg[\int_0^\delta\sqrt{\log\{1+N(\eta,\mathcal{F}_{p_n},\Vert\cdot\Vert_{L^2})\}}\,d\eta+\delta\sqrt{\log\{1+N^2(\delta,\mathcal{F}_{p_n},\Vert\cdot\Vert_{L^2})\}}\,\bigg]\\
&\leq c_1W_n\,p_n^{1/(2m)}\delta^{1-1/m}\,,
\end{align*}
for some absolute constant $c_1>0$. Since $H_n(0)=0$, by Lemma~8.1 in \cite{kosorok2007}, 
\begin{align*}
\P\left\{\sup_{\beta\in\mathcal{F}_{p_n},\,\Vert\beta\Vert_{L^2}\leq\delta}\Vert H_n(\beta)\Vert_K\geq x\,\big|\mathcal{X}_n\right\}\leq 2\exp\big(-c_1^{-2}W_n^{-2} p_n^{-1/m}\delta^{-2+2/m}x^2\big)\,.
\end{align*}
Taking $\gamma=1-1/m$, $b_n=\sqrt{n}\,p_n^{1/(2m)}$, $\theta_n=b_n^{-1}$, $Q_n=\lceil-\log_2\theta_n+\gamma-1\rceil$ and $T_n=c_2(\lambda^{-1/(2D)}\log\log n)^{1/2}$, for some constant $c_2>0$ to be specified below, yields that
\begin{align}\label{p{}}
&\P\left\{\sup_{\beta\in\mathcal{F}_{p_n},\,\Vert\beta\Vert_{L^2}\leq2}\,\frac{\sqrt{n}\Vert H_n(\beta)\Vert_K}{b_n\Vert\beta\Vert_{L^2}^{\gamma}+1}\geq T_n\,\big|\mathcal{X}_n\right\}\notag\\
&\leq\P\left\{\sup_{\beta\in\mathcal{F}_{p_n},\,\Vert\beta\Vert_{L^2}\leq\theta_n^{1/\gamma}}\sqrt{n}\Vert H_n(\beta)\Vert_K\geq T_n\,\big|\mathcal{X}_n\right\}\notag\\
&\hspace{1cm}+\sum_{j=0}^{Q_n}\P\left\{\sup_{\beta\in\mathcal{F}_{p_n},\,(\theta_n2^{j})^{1/\gamma}\leq\Vert\beta\Vert_{L^2}\leq(\theta_n2^{j+1})^{1/\gamma}}\,\frac{\sqrt{n}\Vert H_n(\beta)\Vert_K}{b_n\Vert\beta\Vert_{L^2}^{\gamma}+1}\geq T_n\,\big|\mathcal{X}_n\right\}\notag\\
&\leq\P\left\{\sup_{\beta\in\mathcal{F}_{p_n},\,\Vert\beta\Vert_{L^2}\leq\theta_n^{1/\gamma}}\sqrt{n}\Vert H_n(\beta)\Vert_K\geq T_n\,\big|\mathcal{X}_n\right\}\notag\\
&\hspace{1cm}+\sum_{j=0}^{Q_n}\P\left\{\sup_{\beta\in\mathcal{F}_{p_n},\,\Vert\beta\Vert_{L^2}\leq(\theta_n2^{j+1})^{1/\gamma}}\,\sqrt{n}\Vert H_n(\beta)\Vert_K\geq (b_n\theta_n2^j+1)T_n\,\big|\mathcal{X}_n\right\}\notag\\
&\leq 2\exp\big(-c_1^{-2}W_n^{-2} p_n^{-1/m}\theta_n^{(-2+2/m)/\gamma}n^{-1}T_n^2\big)\notag\\
&\hspace{1cm}+2\sum_{j=0}^{Q_n}\exp\Big\{-c_1^{-2}W_n^{-2} p_n^{-1/m}(\theta_n2^{j+1})^{(-2+2/m)/\gamma}(b_n\theta_n2^j+1)^2n^{-1}T_n^2\Big\}\notag\\
&\leq2\exp\big(-c_1^{-2}W_n^{-2}T_n^2\big)+2(Q_n+1)\exp\big(-c_1^{-2}W_n^{-2}T_n^2/4\big)\notag\\
&\leq 2(Q_n+2)\exp\big(-c_1^{-2}W_n^{-2}T_n^2/4\big)\,.
\end{align}

For the $W_n^2$ in \eqref{wn2}, denote the event $
\mathcal{A}_n=\{W_n^2\leq c_3 \lambda^{-1/(2D)}\}$ for some constant $c_3>0$. Since $\E(W_n^2)\leq c\lambda^{-1/(2D)}$, we have that, for $c_3$ large enough, $\P(\mathcal{A}_n)$ tends to one. On the event $\mathcal{A}_n$, by taking $c_2>2c_1c_3^{-1/2}$, as $n\to\infty$,
\begin{align*}
2(Q_n+2)\exp\big(-c_1^{-2}W_n^{-2}T_n^2/4\big)\leq 2(Q_n+2)\exp\big(-c_1^{-2}c_2^2c_3\log\log n/4\big)=o(1)\,,
\end{align*}
which together with \eqref{p{}} completes the proof of Lemma \ref{lem:hnbeta}.

\end{proof}

Next, we prove the following lemma regarding the convergence rate of $\hat\beta_n$, which is useful to prove Theorem~\ref{thm:bahadur}. 


\begin{lemma}\label{lem:hatbeta}
Under Assumptions~\ref{a1}--\ref{a:rate}, for any $\beta_0\in\H$, we have $$
\Vert\hat\beta_{n}-\beta_0\Vert_K=O_p(\lambda^{1/2}+n^{-1/2}\lambda^{-1/(4D)}).
$$
\end{lemma}

\begin{proof}

For $ L_n$ and $S_{n,\lambda}$ defined in \eqref{elln} and \eqref{dsn}, let
\begin{align}\label{s}
S_n(\beta)=\mathcal D L_n(\beta)\,;\quad S(\beta)=\E\{\mathcal D L_n(\beta)\}\,;\quad S_\lambda(\beta)=\E\{S_{n,\lambda}(\beta)\}\,.
\end{align}
In view of \eqref{dsn}, $S_\lambda(\beta)=S(\beta)+W_\lambda(\beta)$. We first show that there exists a unique element $\beta_\lambda\in\H$ such that $S_\lambda(\beta_\lambda)=0$, and then we prove  the upper bound for $\Vert\beta_\lambda-\beta_0\Vert_K$. Since $$
\mathcal D^2L_{\lambda}(\beta_0)\beta_1\beta_2=\l \mathcal DS_{\lambda}(\beta_0)\beta_1,\beta_2\r_K=\l\beta_1,\beta_2\r_K,
$$
we have $\mathcal DS_\lambda(\beta_0)= id $, where $ id $ is the identity operator. Since $\mathcal D^2S_\lambda$ vanishes, we deduce that $\mathcal DS_{\lambda}(\beta)= id $ for any $\beta\in\H$. Hence, by the mean value theorem, for any $\beta\in\H$, $S_\lambda(\beta)=\beta+S_\lambda(\beta_0)-\beta_0$. Therefore, letting $\beta_\lambda=\beta_0-S_\lambda(\beta_0)$, we find $S_\lambda(\beta_\lambda)=0$, and $\beta_\lambda$ is the unique solution to the estimating equation $S_\lambda(\beta)=0$. Moreover, since $S(\beta_0)=0$, we have $S_\lambda(\beta_0)=S(\beta_0)+W_\lambda(\beta_0)=W_\lambda(\beta_0)$. By the Cauchy-Schwarz inequality, for $J$ in \eqref{jm},
\begin{align}\label{deltalambdabeta}
&\Vert\beta_\lambda-\beta_0\Vert_K=\Vert W_\lambda(\beta_0)\Vert_K=\sup_{\Vert\gamma\Vert_K=1}|\l W_\lambda(\beta_0),\gamma\r_K|=\sup_{\Vert\gamma\Vert_K=1}\lambda |J(\beta_0,\gamma)|\notag\\
&\leq \sup_{\Vert\gamma\Vert_K=1}\big\{\sqrt{\lambda J(\beta_0,\beta_0)}\sqrt{\lambda J(\gamma,\gamma)}\big\}\leq\sup_{\Vert\gamma\Vert_K=1}\big\{\sqrt{\lambda J(\beta_0,\beta_0)}\,\Vert\gamma\Vert_K\big\}=\sqrt{\lambda J(\beta_0,\beta_0)}\,.
\end{align}

Since $\Vert\hat\beta_{n}-\beta_0\Vert_K\leq\Vert\beta_\lambda-\beta_0\Vert_K+\Vert\hat\beta_{n}-\beta_\lambda\Vert_K$, we then proceed to show the rate of $\Vert\hat\beta_{n}-\beta_\lambda\Vert_K$. Let $F_n(\beta)=\beta-S_{n,\lambda}(\beta_\lambda+\beta)$. Recall that $\mathcal DS_\lambda(\beta_\lambda)\beta=\beta$, so that, for $\mathcal DS_{n,\lambda}$ in \eqref{dsn},
\begin{align}\label{th}
F_n(\beta)=I_{1,n}(\beta)+I_{2,n}(\beta)-S_{n,\lambda}(\beta_\lambda)\,,
\end{align}
where
\begin{align}
&I_{1,n}(\beta)=-\{S_{n,\lambda}(\beta_\lambda+\beta)-S_{n,\lambda}(\beta_\lambda)-\mathcal DS_{n,\lambda}(\beta_\lambda)\beta\}\,,\notag\\
&I_{2,n}(\beta)=-\{\mathcal DS_{n,\lambda}(\beta_\lambda)\beta-\mathcal DS_\lambda(\beta_\lambda)\beta\}\,.\label{i1}
\end{align}

First, for $I_{1,n}(\beta)$ in \eqref{i1}, in view of $S_{n,\lambda}$ and $\mathcal DS_{n,\lambda}$ defined in \eqref{dsn},
\begin{align}\label{zero}
I_{1,n}(\beta)&=\frac{1}{n}\sum_{i=1}^n\tau\bigg(X_i\otimes\bigg[Y_i-\int_0^1\{\beta_\lambda(s,\cdot)+\beta(s,\cdot)\}X_i(s)ds\bigg]\bigg)\notag\\
&\quad-\frac{1}{n}\sum_{i=1}^n\tau\bigg[X_i\otimes\bigg\{Y_i-\int_0^1\beta(s,\cdot)X_i(s)ds\bigg\}\bigg]\notag\\
&\quad+\frac{1}{n}\sum_{i=1}^n\tau\bigg[X_i\otimes\bigg\{\int_0^1\beta_\lambda(s,\cdot)X_i(s)ds\bigg\}\bigg]=0\,.
\end{align}
For $I_{2,n}(\beta)$ in \eqref{i1}, in view of \eqref{dsn},
\begin{align}\label{i1nnorm}
\Vert I_{2,n}(\beta)\Vert_K&=\big\Vert \mathcal DS_{n,\lambda}(\beta_\lambda)\beta-\mathcal DS_\lambda(\beta_\lambda)\beta\big\Vert_K\notag\\
& =\big\Vert \mathcal DS_{n}(\beta_\lambda)\beta-\mathcal DS(\beta_\lambda)\beta\big\Vert_K
%
%
=\frac{1}{\sqrt{n}}\Vert H_n(\beta)\Vert_K\,,
\end{align}
where $H_n(\beta)$ is defined in \eqref{hnbeta} in Lemma~\ref{lem:hnbeta}. For $a,D$ in Assumption~\ref{a201} and $c_K$ in Lemma~\ref{lem:norm}, let $p_n=c_K^{-2}\lambda^{(2a+1)/(2D)-1}$. In order to apply Lemma~\ref{lem:hnbeta}, we shall rescale $\beta$ such that the $L^2$-norm of its rescaled version is bounded by $1$. For the constant $c_K$ in Lemma~\ref{lem:norm}, let
\begin{align}\label{tildeb}
\tilde\beta=\left\{
\begin{array}{ll}
\big(c_K\lambda^{-(2a+1)/(4D)}\Vert\beta\Vert_K\big)^{-1}\beta&\text{ if }\beta\neq0\,,\\
0&\text{ if }\beta=0\,.
\end{array}\right.
\end{align}
We have $\Vert\tilde\beta\Vert_{L^2}\leq c_K\lambda^{-(2a+1)/(4D)}\Vert\tilde\beta\Vert_K\leq 1$, since $\Vert\tilde\beta\Vert_K\leq (c_K\lambda^{-(2a+1)/(4D)})^{-1}$ by Lemma~\ref{lem:norm}. In addition, in view of \eqref{inner}, $J(\tilde\beta,\tilde\beta)\leq\lambda^{-1}\Vert\tilde\beta\Vert_K^2\leq c_K^{-2}\lambda^{(2a+1)/(2D)-1}=p_n$, which implies $\tilde\beta\in\mathcal{F}_{p_n}$. By Lemma~\ref{lem:hnbeta}, since $n^{-1/2}=o(p_n^{1/(2m)})$ by Assumption~\ref{a:rate}, we find that, for some constant $c>0$ large enough, with probability tending to one,
\begin{align*}
\Vert H_n(\tilde\beta)\Vert_K\leq c\big(p_n^{1/(2m)}+n^{-1/2}\big)\big(\lambda^{-1/(2D)}\log\log n\big)^{1/2}\leq c\,p_n^{1/(2m)}\lambda^{-1/(4D)}(\log\log n)^{1/2}\,.
\end{align*}
Therefore, in view of \eqref{tildeb}, we deduce from the above inequality that, for some constant $c>0$ large enough, with probability tending to one,
\begin{align*}
\Vert H_n(\beta)\Vert_K&\leq \big(c_K\lambda^{-(2a+1)/(4D)}\Vert\beta\Vert_K\big)\Vert H_n(\tilde\beta)\Vert_K\leq c\,p_n^{1/(2m)}\lambda^{-(a+1)/(2D)}(\log\log n)^{1/2}\Vert\beta\Vert_K\,.
\end{align*}
Recall that $p_n=O(\lambda^{(2a+1)/(2D)-1})$. Therefore, for $I_{2,n}(\beta)$ in \eqref{i1}, in view of \eqref{i1nnorm}, we deduce from the equation result that, for some constant $c>0$ large enough, with probability tending to one,
\begin{align}\label{i1nrate}
\Vert I_{2,n}(\beta)\Vert_K\leq n^{-1/2}\Vert H_n(\beta)\Vert_K&\leq c\,n^{-1/2}p_n^{1/(2m)}\lambda^{-(a+1)/(2D)}(\log\log n)^{1/2}\Vert\beta\Vert_K\notag\\
&\leq c\,n^{-1/2}\lambda^{-\varsigma}(\log\log n)^{1/2}\,\Vert\beta\Vert_K=o(1)\,\Vert\beta\Vert_K\,,
\end{align}
where we used Assumption~\ref{a:rate} in the last step. 

For estimating the remaining term $-S_{n,\lambda}(\beta_\lambda)$ in \eqref{th}, we recall the definition of  $\tau$  in \eqref{tau} and define  $O_i=\tau\big[X_i\otimes\big\{Y_i-\int_0^1\beta_\lambda(s,\cdot)X_i(s)ds\big\}\big]$, for $1\leq i\leq n$. Since $S_{\lambda}(\beta_\lambda)=0$, 
we obtain observing \eqref{dell}  that  
$$
-S_{n,\lambda}(\beta_\lambda)=-\{S_{n,\lambda}(\beta_\lambda)-S_{\lambda}(\beta_\lambda)\}=n^{-1}\sum_{i=1}^n\{O_i-\E(O_i)\}.
$$ Let $\Delta_\lambda\beta=\beta_0-\beta_\lambda$, so that from \eqref{deltalambdabeta} we obtain $\Vert\Delta_\lambda\beta\Vert_K^2\leq c\lambda$ for some constant $c>0$. We notice that
\begin{align}\label{t2}
&\E\Vert S_{n,\lambda}(\beta_\lambda)\Vert_K^2=n^{-1}\E\Vert O_i-\E(O_i)\Vert_K^2
\leq n^{-1}\E\Vert O_i\Vert_K^2\notag\\
&=n^{-1}\E\bigg\Vert\tau\bigg[X_i\otimes\bigg\{Y_i-\int_0^1\beta_\lambda(s,\cdot)X_i(s)ds\bigg\}\bigg]\bigg\Vert_K^2\notag\\
&=n^{-1}\E\bigg\Vert\tau(X_i\otimes\e_i)+\tau\bigg[X_i\otimes\bigg\{\int_0^1\Delta_\lambda\beta(s,\cdot)X_i(s)ds\bigg\}\bigg]\bigg\Vert_K^2\notag\\
&\leq 2n^{-1}\E\big\Vert\tau(X_i\otimes\e_i)\big\Vert_K^2+2n^{-1}\E\bigg\Vert\tau\bigg[X_i\otimes\bigg\{\int_0^1\Delta_\lambda\beta(s,\cdot)X_i(s)ds\bigg\}\bigg]\bigg\Vert_K^2\,.
\end{align}
By Lemma~\ref{lem:etau2} in Section \ref{app:aux}, we have
\begin{align}\label{t3}
&\E\bigg\Vert\tau\bigg[X_i\otimes\bigg\{\int_0^1\Delta_\lambda\beta(s,\cdot)X_i(s)ds\bigg\}\bigg]\bigg\Vert_K^{2}\leq c\,\lambda^{-1/D} \,\Vert\Delta_\lambda\beta\Vert_K^2\,,
\end{align}
and  Lemmas~\ref{lem:sum},  \ref{lem:etk} and Assumption~\ref{a201} give for the first term in \eqref{t2}
\begin{align*}
\E\big\Vert\tau(X_i\otimes\e_i)\big\Vert_K^2=\sum_{k,\ell}\frac{1}{1+\lambda\rho_{k\ell}}\leq c\sum_{k,\ell}\frac{1}{1+\lambda({k\ell})^{2D}}\leq c\,\lambda^{-1/(2D)}\,.
\end{align*}
Combining this equality with \eqref{deltalambdabeta}, \eqref{t2}, \eqref{t3} and Lemma~\ref{lem:etk} yields
\begin{align}\label{snlambda}
\E\Vert S_{n,\lambda}(\beta_\lambda)\Vert_K^2&\leq c\,n^{-1}\E\big\Vert\tau(X_i\otimes\e_i)\big\Vert_K^2+c\,n^{-1}\lambda^{-1/D} \,\Vert\Delta_\lambda\beta\Vert_K^2\notag\\
&\leq c\,n^{-1}\lambda^{-1/(2D)}+c\,n^{-1}\lambda^{1-1/D}\leq c\,n^{-1}\lambda^{-1/(2D)}\,.
\end{align}

Let $q_n=2c_0n^{-1/2}\lambda^{-1/(4D)}$ and denote by $\mathcal{B}(r)=\{\gamma\in\mathcal{H},\,\Vert\gamma\Vert_K\leq r\}$  the $\Vert\cdot\Vert_K$-ball with radius $r>0$ in $\H$. In view of \eqref{i1nrate}, for any $\beta\in\mathcal{B}(q_n)$, with probability tending to one, $\Vert I_{2,n}(\beta)\Vert_K\leq \Vert\beta\Vert_K/2\leq q_n/2$. Therefore, observing \eqref{th} and \eqref{zero}, we obtain for the term $F_n(\beta)$  in \eqref{th}, with probability tending to one, for any $\beta\in\mathcal{B}(q_n)$,
\begin{align*}
\Vert F_n(\beta)\Vert_K\leq\Vert I_{2,n}(\beta)\Vert_K+\Vert S_{n,\lambda}(\beta_\lambda)\Vert_K\leq c_0\,n^{-1/2}\lambda^{-1/(4D)}+q_n/2\leq q_n\,,
\end{align*}
which implies that $F_n\{\mathcal{B}(q_n)\}\subset\mathcal{B}_n(q_n)$ with probability converging to one. Note that for any $\beta_1,\beta_2\in\mathcal B(q_n)$, $F_n(\beta_1)-F_n(\beta_2)=I_{2,n}(\beta_1)-I_{2,n}(\beta_2)$. Due to \eqref{i1nrate}, with probability tending to one,
\begin{align*}
&\Vert F_n(\beta_1)-F_n(\beta_2)\Vert=\Vert I_{2,n}(\beta_1)-I_{2,n}(\beta_2)\Vert_K\leq \Vert\beta_1-\beta_2\Vert_K/2\,,
\end{align*}
which indicates that $F_n$ is a contraction mapping on $\mathcal{B}(q_n)$. By the Banach contraction mapping theorem with probability converging to one, there exists a unique element $\beta^*\in\mathcal{B}_n$ such that $\beta^*=F_n(\beta^*)=\beta^*-S_{n,\lambda}(\beta_\lambda+\beta^*)$. Letting $\hat\beta_{n}=\beta_\lambda+\beta^*$, we have $S_{n,\lambda}(\hat\beta_{n})=0$, which indicates that $\hat\beta_{n}$ is the estimator defined by \eqref{hatbeta}. In view of \eqref{deltalambdabeta},
\begin{align*}
\Vert\hat\beta_{n}-\beta_0\Vert_K&\leq\Vert\beta_{\lambda}-\beta_0\Vert_K+\Vert\hat\beta_{n}-\beta_\lambda\Vert_K=O_p(\lambda^{1/2}+q_n)=O_p\big(\lambda^{1/2}+n^{-1/2}\lambda^{-1/(4D)}\big)\,.
\end{align*}

\end{proof}

Finally, we conclude the proof of Theorem~\ref{thm:bahadur} using Lemmas~\ref{lem:hnbeta} and \ref{lem:hatbeta}.

\begin{proof}[Proof of Theorem~\ref{thm:bahadur}]

Let $\beta_\Delta=\hat\beta_{n}-\beta_0$.
Since $\mathcal D^2S_\lambda$ vanishes and $\mathcal DS_\lambda(\beta_0)=id$, we have $S_\lambda(\hat\beta_{n})-S_\lambda(\beta_0)=\mathcal DS_\lambda(\beta_0)\beta_\Delta=\beta_\Delta$. For $S_n$ and $S$ defined in \eqref{s}, since $S_{n,\lambda}(\hat\beta_{n})=0$,
\begin{align}\label{co}
\hat\beta_{n}-\beta_0+S_{n,\lambda}(\beta_0)&=\beta_\Delta+S_{n,\lambda}(\beta_0)=-S_{n,\lambda}(\hat\beta_{n})+S_{n,\lambda}(\beta_0)+S_\lambda(\hat\beta_{n})-S_\lambda(\beta_0)\notag\\
&=-S_{n}(\hat\beta_{n})+S_{n}(\beta_0)+S(\hat\beta_{n})-S(\beta_0)\,.
\end{align}
Let $r_n=\lambda^{1/2}+n^{-1/2}\lambda^{-1/(4D)}$. For $c_1>0$, denote the event $\mathcal M_n=\big\{\Vert\beta_\Delta\Vert_K\leq c_1r_n\big\}$.  From Lemma~\ref{lem:hatbeta}, we obtain that $\P(\mathcal M_n) $ is arbitrarily close to one except for a finite number of $n$ if   $c_1>0$ is large enough. For the constant $c_K>0$ in Lemma~\ref{lem:norm}, let $q_n=c_1c_K\lambda^{-(2a+1)/(4D)} r_n$ and let $p_n=c_1^2q_n^{-2}\lambda^{-1}r_n^2=c_K^{-2}\lambda^{(-2D+2a+1)/(2D)}$. We have $p_n\geq 1$ for $n$ large enough since $D>a+1/2$ by Assumption~\ref{a201}. In order to apply Lemma~\ref{lem:hnbeta}, we shall rescale $\beta_\Delta$ such that the $L^2$-norm of its rescaled version is bounded by $1$. Let $\tilde\beta_\Delta=q_n^{-1}\beta_\Delta$. By Lemma~\ref{lem:norm}, we have that, on the event $\mathcal M_n$,
\begin{align*}
\Vert\tilde\beta_\Delta\Vert_{L^2}\leq c_K\lambda^{-(2a+1)/(4D)} \Vert\tilde\beta_\Delta\Vert_K\leq c_Kq_n^{-1}\lambda^{-(2a+1)/(4D)} \Vert\beta_\Delta\Vert_K\leq c_1c_Kq_n^{-1}\lambda^{-(2a+1)/(4D)} r_n\leq 1\,.
\end{align*}
In addition, since $J(\beta_\Delta,\beta_\Delta)\leq\lambda^{-1}\Vert\beta_\Delta\Vert_K^2$, we have
\begin{align*}
J(\tilde\beta_\Delta,\tilde\beta_\Delta)\leq q_n^{-2}J(\beta_\Delta,\beta_\Delta)\leq q_n^{-2}\lambda^{-1}\Vert\beta_\Delta\Vert_K^2\leq c_1^2\,q_n^{-2}\lambda^{-1}r_n^2=p_n\,. 
\end{align*}
Hence, we have shown that $\tilde\beta_\Delta\in\mathcal{F}_{p_n}$, where the set $\mathcal{F}_{p_n}$ is defined in \eqref{m6}.

Recalling the notations \eqref{dsn}, \eqref{hnbeta} and the identity \eqref{co}, we have
\begin{align}\label{coo}
\big\Vert\hat\beta_{n}-\beta_0+S_{n,\lambda}(\beta_0)\big\Vert_K&=\big\Vert-S_{n}(\hat\beta_{n})+S_{n}(\beta_0)+S(\hat\beta_{n})-S(\beta_0)\big\Vert_K\notag\\
&=n^{-1/2}\Vert H_n(\beta_\Delta)\Vert_K\,.
\end{align}
Since $\tilde\beta_\Delta\in\mathcal{F}_{p_n}$, by Lemma~\ref{lem:hnbeta},
\begin{align*}
\Vert H_n(\tilde\beta_\Delta)\Vert_K&=O_p\big\{\big(p_n^{1/(2m)}+n^{-1/2}\big)\big(\lambda^{-1/(2D)}\log\log n\big)^{1/2}\big\}\\
&=O_p\big\{p_n^{1/(2m)}\lambda^{-1/(4D)}(\log\log n)^{1/2}\big\}\,.
\end{align*}
since $n^{-1/2}=o(p_n^{1/(2m)})$ by Assumption~\ref{a:rate}. Therefore, we deduce from the above equation that, for some constant $c>0$ large enough, with probability tending to one,
\begin{align}\label{hn}
n^{-1/2}\Vert H_n(\beta_\Delta)\Vert_K&\leq n^{-1/2}q_n\,\Vert H_n(\tilde\beta_\Delta)\Vert_K\leq c\,n^{-1/2}q_n\,p_n^{1/(2m)}\lambda^{-1/(4D)}(\log\log n)^{1/2}\notag\\
&\leq c\,n^{-1/2}(\lambda^{-(2a+1)/(4D)}\, r_n)\,\lambda^{(-2D+2a+1)/(4Dm)}\,\lambda^{-1/(4D)}(\log\log n)^{1/2}\notag\\
&=c\,n^{-1/2}\lambda^{-\varsigma}(\lambda^{1/2}+n^{-1/2}\lambda^{-1/(4D)})(\log\log n)^{1/2}\,,
\end{align}
where $\varsigma>0$ is the constant in Assumption~\ref{a:rate}. Combining the above result with \eqref{coo} yields that
\begin{align*}
\big\Vert\hat\beta_{n}-\beta_0+S_{n,\lambda}(\beta_0)\big\Vert_K&
=O_p\big\{n^{-1/2}\lambda^{-\varsigma}(\lambda^{1/2}+n^{-1/2}\lambda^{-1/(4D)})(\log\log n)^{1/2}\big\}\,,
\end{align*}
which completes the proof of Theorem~\ref{thm:bahadur}.

\end{proof}

\subsection{Proof of Theorem~\ref{thm:pointwise}}\label{app:thm:pointwise}

Recall that
\begin{align}\label{vn}
v_n=n^{-1/2}\lambda^{-\varsigma}(\lambda^{1/2}+n^{-1/2}\lambda^{-1/(4D)})(\log\log n)^{1/2}\,,
\end{align}
where $\varsigma>0$ is the constant in Assumption~\ref{a:rate}, so that by Theorem~\ref{thm:bahadur}, $\Vert\hat\beta_{n}-\beta_0+S_{n,\lambda}(\beta_0)\Vert_K=O_p(v_n)$. In view of \eqref{dsn}, $S_{n,\lambda}(\beta_0)=-n^{-1}\sum_{i=1}^n\tau(X_i\otimes\e_i)+W_\lambda(\beta_0)$, so that
\begin{align}\label{sn}
\hat\beta_{n}-\beta_0=(\hat\beta_{n}-\beta_0+S_{n,\lambda}\beta_0)-W_\lambda(\beta_0)+\frac{1}{n}\sum_{i=1}^n\tau(X_i\otimes\e_i)\,.
\end{align}

We first denote
\begin{align}\label{m7}
\sigma_\tau(s,t)=\bigg\{\sum_{k,\ell}(1+\lambda\rho_{k\ell})^{-2}\phi_{k\ell}^2(s,t)\bigg\}^{1/2}\,,
\end{align}
so that
\begin{align*}
\frac{\sqrt{n}}{\sigma_\tau(s,t)}\{\hat\beta_{n}(s,t)-\beta_0(s,t)\}=I_{1,n}(s,t)+I_{2,n}(s,t)+I_{3,n}(s,t)\,,
\end{align*}
where
\begin{align*}
&I_{1,n}(s,t)=\frac{\sqrt{n}}{\sigma_\tau(s,t)}\big\l\hat\beta_{n}-\beta_0+S_{n,\lambda}\beta_0, K_{(s,t)}\r_K\,,\\
&I_{2,n}(s,t)=-\frac{\sqrt{n}}{\sigma_\tau(s,t)}W_\lambda\beta_0(s,t)\,, \\
&I_{3,n}(s,t)=\frac{1}{\sqrt{n}\,\sigma_\tau(s,t)}\sum_{i=1}^n\tau(X_i\otimes\e_i)(s,t)\,.
\end{align*}

By Theorem \ref{thm:bahadur}, \eqref{kst} and Cauchy-Schwarz inequality,
\begin{align}\label{sn2}
|I_{1,n}(s,t)|&=\frac{\sqrt{n}}{\sigma_\tau(s,t)}\big|\big\l\hat\beta_{n}-\beta_0+S_{n,\lambda}\beta_0, K_{(s,t)}\big\r_K\big|\notag\\
&\leq \frac{\sqrt{n}}{\sigma_\tau(s,t)}\Vert K_{(s,t)}\Vert_K\times\Vert\hat\beta_{n}-\beta_0+S_{n,\lambda}(\beta_0)\Vert_K\notag\\
&\leq\frac{\sqrt{n}}{\sigma_\tau(s,t)}\bigg\{\sum_{k,\ell}\frac{\phi_{k\ell}^2(s,t)}{1+\lambda\rho_{k\ell}}\bigg\}^{1/2}\times\Vert\hat\beta_{n}-\beta_0+S_{n,\lambda}(\beta_0)\Vert_K\notag\\
&\leq O(\sqrt{n}\lambda^{(2a+1)/(4D)})\times O_p(v_n)\times\bigg\{\sum_{k,\ell}\frac{\Vert\phi_{k\ell}\Vert^2_\infty}{1+\lambda\rho_{k\ell}}\bigg\}^{1/2}=O_p(\sqrt{n}v_n)\,,
\end{align}
where we used Assumption~\ref{a201} and Lemma~\ref{lem:sum}, which imply $\sigma_\tau(s,t)\asymp \lambda^{-(2a+1)/(4D)}$.

For the term $I_{2,n}$, in view of \eqref{wphi}, by the assumption that $\sum_{k,\ell}\rho_{k\ell}^2V^2(\beta_0,\phi_{k\ell})< \infty$, Assumption~\ref{a201} and Lemma~\ref{lem:sum},
\begin{align}\label{in2}
&|I_{2,n}(s,t)|=\frac{\sqrt{n}}{\sigma_\tau(s,t)}|W_\lambda\beta_0(s,t)|=\frac{\sqrt{n}\lambda}{\sigma_\tau(s,t)}\bigg|\sum_{k,\ell}\frac{\rho_{k\ell}\,V(\beta_0,\phi_{k\ell})}{1+\lambda\rho_{k\ell}}\,\phi_{k\ell}(s,t)\bigg|\notag\\
&\leq \frac{\sqrt{n}\,\lambda}{\sigma_\tau(s,t)}\bigg\{\sum_{k,\ell}\rho_{k\ell}^2\,V^2(\beta_0,\phi_{k\ell})\bigg\}^{1/2}\bigg\{\sum_{k,\ell}\frac{\Vert\phi_{k\ell}\Vert_\infty^2}{(1+\lambda\rho_{k\ell})^2}\bigg\}^{1/2}\leq c\sqrt{n}\,\lambda=o(1)\,.
\end{align}

Finally, the term $I_{3,n}$ can be estimated as follows. For $1\leq i\leq n$, let 
\begin{align*}
\mathfrak U_i=\frac{\tau(X_i\otimes\e_i)(s,t)}{\sqrt{n}\,\sigma_\tau(s,t)}=\frac{\tau(X_i\otimes\e_i)(s,t)}{\sqrt{n\sum_{k,\ell}(1+\lambda\rho_{k\ell})^{-2}\phi_{k\ell}^2(s,t)}}\,,
\end{align*}
so that $I_{3,n}=\sum_{i=1}^n\mathfrak{U}_i$. We start by noticing that, Assumption~\ref{a0} indicates that
\begin{align*}
&\E\Big\{\big\l \tau(X\otimes\epsilon),\phi_{k\ell}\big\r_{L^2}\Big\}=\E\bigg[\int_{[0,1]^2}\E\{\epsilon(t)|X\}\phi_{k\ell}(s,t)X_i(s)dsdt\bigg]=0\,,
\end{align*}
so that in view of \eqref{tau}, $\E\{\tau(X\otimes\epsilon)(s,t)\}=\sum_{k,\ell}\E\big(\l X\otimes\epsilon,\phi_{k\ell}\r_{L^2}\big)(1+\lambda\rho_{k\ell})^{-1}\phi_{k\ell}(s,t)=0$, so that $\E(\mathfrak{U}_i)=0$. In view of \eqref{tk} in the proof of Lemma~\ref{lem:etk} in Section~\ref{app:aux},
\begin{align*}
&\E|\tau(X\otimes\epsilon)(s,t)|^2=\E\left(\sum_{k,\ell}\frac{\l X\otimes\epsilon,\phi_{k\ell}\r_{L^2}}{1+\lambda\rho_{k\ell}}\phi_{k\ell}(s,t)\right)^2\notag\\
&=\sum_{k,k',\ell,\ell'}\frac{\phi_{k\ell}(s,t)\,\phi_{k'\ell'}(s,t)}{(1+\lambda\rho_{k\ell})(1+\lambda\rho_{k'\ell'})}\E\Big\{\l X\otimes\epsilon,\phi_{k\ell}\r_{L^2}\,\l X\otimes\epsilon,\phi_{k'\ell'}\r_{L^2}\Big\}=\sigma_\tau^2(s,t)\,,
\end{align*}
where $\sigma_\tau$ is defined in \eqref{m7} and we used $\sigma_\tau^2(s,t)$ \eqref{core} in the last step. By the assumption that $\sigma_\tau^2(s,t)\asymp \lambda^{-(2a+1)/(2D)}$, we deduce that $\E|\tau(X\otimes\epsilon)(s,t)|^2=\sigma_\tau^2(s,t)\geq c_0^2\lambda^{-(2a+1)/(2D)}$ for some constant $c_0>0$.

To conclude the proof, we shall check that the triangular array of random variables $
\{\mathfrak U_i\}_{i=1}^n = \{n^{-1/2}\sigma^{-1}_\tau(s,t)\tau(X_i\otimes\e_i)(s,t)\}_{i=1}^n$ satisfies the Lindeberg's condition. By the Cauchy-Schwarz inequality, for any $e>0$,
\begin{align}\label{lind}
&\sum_{i=1}^n\E\big[|\mathfrak{U_i}|^2\times\one\{|\mathfrak{U}_i|>e\}\big]\notag\\
%
%
&=\frac{1}{\sigma^2_\tau(s,t)}\,\E\Big[|\tau(X\otimes\epsilon)(s,t)|^2\times\one\Big\{|\tau(X\otimes\epsilon)(s,t)|\geq e\sqrt{n}\,\sigma_\tau(s,t)\Big\}\Big]\notag\\
&\leq\frac{1}{\sigma^2_\tau(s,t)}\,\Big\{\E|\tau(X\otimes\e)(s,t)|^4\Big\}^{1/2}\times\Big[\P\Big\{|\tau(X\otimes\epsilon)(s,t)|\geq e\sqrt{n}\,\sigma_\tau(s,t)\Big\}\Big]^{1/2}\,.
%
\end{align}
We shall deal with the above moment term and the tail probability separately. In order to find the order of $\E|\tau(X\otimes\e)(s,t)|^4$, in view of \eqref{kst} and Lemma~\ref{lem:sum}, by Assumption~\ref{a201},
\begin{align}\label{supst}
\sup_{(s,t)\in[0,1]^2}\Vert K_{(s,t)}\Vert_K^2=\sum_{k,\ell}\frac{\Vert\phi_{k\ell}\Vert_\infty^2}{1+\lambda\rho_{k\ell}}\leq c\lambda^{-(2a+1)/(2D)}\,.
\end{align}
Therefore, by Lemma~\ref{lem:etk}, we find
\begin{align}\label{supe4}
&\sup_{(s,t)\in[0,1]^2}\E\big|\tau(X\otimes\e)(s,t)\big|^4\leq \sup_{(s,t)\in[0,1]^2}\Vert K_{(s,t)}\Vert_K^4\times\E\Vert\tau(X\otimes\e)\Vert_K^4\leq c\lambda^{-(2a+2)/D}\,.
\end{align}

In order to find the order of the tail probability $\P\big\{|\tau(X\otimes\epsilon)(s,t)|\geq e\sqrt{n}\sigma_\tau(s,t)\big\}$, we first show an upper bound of $|\tau(X\otimes\epsilon)(s,t)|$. To achieve this, in view of \eqref{tk} in the proof of Lemma~\ref{lem:etk} in Section~\ref{app:aux}, by Lemma~\ref{lem:norm},
\begin{align*}
&\Vert\tau(X\otimes\e)\Vert_K=\sup_{\Vert\gamma\Vert_K=1}|\l\tau(X\otimes\e),\gamma\r_K|=\sup_{\Vert\gamma\Vert_K=1}|\l X\otimes\e,\gamma\r_{L^2}|\notag\\
&\leq \sup_{\Vert\gamma\Vert_K=1}\Vert\gamma\Vert_{L^2}\times\Vert X\otimes\e\Vert_{L^2}\leq c_K\lambda^{-(2a+1)/(4D)}\Vert X\Vert_{L^2}\Vert\e\Vert_{L^2}\,.
\end{align*}
By \eqref{supst} and the Cauchy-Schwarz inequality, we deduce from the above equation that
\begin{align}\label{taust}
\sup_{(s,t)\in[0,1]^2}|\tau(X\otimes\epsilon)(s,t)|&=\sup_{(s,t)\in[0,1]^2}|\l K_{(s,t)},\tau(X\otimes\e)\r_K|\leq \sup_{(s,t)\in[0,1]^2}\Vert K_{(s,t)}\Vert_K\times\Vert \tau(X\otimes\e)\Vert_K\notag\\
&\leq c\lambda^{-(2a+1)/(2D)}\Vert X\Vert_{L^2}\Vert\e\Vert_{L^2}\,.
\end{align}
Now, by assumption, $\sigma_\tau(s,t)\geq c_0\lambda^{-(2a+1)/(4D)}$ for some $c_0>0$, hence for any $e>0$, we can choose $c_1>(c_\e D)^{-1}$, for $c_\e>0$ in Assumption~\ref{a:x}, so that
\begin{align}\label{11}
&\P\Big\{|\tau(X\otimes\epsilon)(s,t)|\geq e\sqrt{n}\,\sigma_\tau(s,t)\Big\}\notag\\
&\leq \P\Big(c\lambda^{-(2a+1)/(2D)}\Vert X\Vert_{L^2}\Vert\e\Vert_{L^2}\geq ec_0\sqrt{n}\lambda^{-(2a+1)/(4D)}\Big)\notag\\
&= \P\Big(\Vert X\Vert_{L^2}\Vert\e\Vert_{L^2}\geq ec^{-1}c_0\sqrt{n}\lambda^{(2a+1)/(4D)}\Big)\notag\\
&\leq\P\Big\{\Vert X\Vert_{L^2}\geq ec_1^{-1}c^{-1}c_0\sqrt{n}\lambda^{(2a+1)/(4D)}/{\log(\lambda^{-1})}\Big\}+\P\Big\{\Vert\e\Vert_{L^2}\geq c_1{\log(\lambda^{-1})}\Big\}\notag\\
&\leq \exp\{-c_Xec_1^{-1}c^{-1}c_0\sqrt{n}\lambda^{(2a+1)/(4D)}/\log(\lambda^{-1})\}\E\{\exp(c_X\Vert X\Vert_{L^2})\}\notag\\
&\qquad+\lambda^{c_1c_\e}\E\{\exp(c_\e\Vert\e\Vert_{L^2})\}\notag\\
&=O\big(\lambda^{c_Xec_1^{-1}c^{-1}c_0\{\sqrt{n}\lambda^{(2a+1)/(4D)}/\log^2(\lambda^{-1})\}}\big)+O(\lambda^{c_1c_\e})=o(\lambda^{1/D})\,,
\end{align}
where we used the assumption that $\sqrt{n}\lambda^{(2a+1)/(4D)}/\log^2(\lambda^{-1})\to \infty$ in the last step. Since, by assumption, $\sigma^2_\tau(s,t)\asymp\lambda^{-(2a+1)/(2D)}$, combining the above equation with \eqref{lind} and \eqref{supe4} yields that, for any $e>0$,
\begin{align*}
&\sum_{i=1}^n\E\big[|\mathfrak{U_i}|^2\times\one\{|\mathfrak{U}_i|>e\}\big]\leq\frac{c}{\sigma^2_\tau(s,t)}\times O(\lambda^{-(a+1)/D})\times o(\lambda^{1/(2D)})=o(1)\,.
\end{align*}
Therefore, by Lindeberg's CLT,
\begin{align*}
\frac{1}{\sqrt{n}\,\sigma_\tau(s,t)}\sum_{i=1}^n\tau(X_i\otimes\e_i)(s,t)\overset{d.}{\longrightarrow} N(0,1)\,.
\end{align*}
Combining the above result with \eqref{sn}--\eqref{in2}, we deduce that
\begin{align*}
&\frac{\sqrt{n}}{\sigma_\tau(s,t)}\big\{\hat\beta_{n}(s,t)-\beta_0(s,t)\big\}=\frac{1}{\sqrt{n}\,\sigma_\tau(s,t)}\sum_{i=1}^n\tau(X_i\otimes\e_i)(s,t)+o_p(1)\overset{d.}{\longrightarrow} N(0,1)\,,
\end{align*}
which completes the proof.

\subsection{Proof of Theorem~\ref{thm:process}}\label{app:thm:process}

Recall the definition of the process $\G_n$ in \eqref{hatgn} and note that in view of \eqref{sn},
\begin{align}\label{gn}
\G_n(s,t)=I_{1,n}(s,t)+I_{2,n}(s,t)+I_{3,n}(s,t)\,,
\end{align}
where
\begin{align}\label{in1in2}
&I_{1,n}(s,t)=\sqrt{n}\lambda^{(2a+1)/(4D)}\{\hat\beta_{n}(s,t)-\beta_0(s,t)+S_{n,\lambda}\beta_0(s,t)\}\notag\\
&I_{2,n}(s,t)=-\sqrt{n}\lambda^{(2a+1)/(4D)}W_\lambda\beta_0(s,t)\,,\notag\\
&I_{3,n}(s,t)=n^{-1/2}\lambda^{(2a+1)/(4D)}\sum_{i=1}^n\tau(X_i\otimes\e_i)(s,t)\,.
\end{align}
In view of \eqref{sn2} and Theorem~\ref{thm:bahadur}, for $v_n$ in \eqref{vn},
\begin{align}\label{in1a}
\sup_{(s,t)\in[0,1]^2}|I_{1,n}(s,t)|&\leq \sqrt{n}\lambda^{(2a+1)/(4D)}\bigg\{\sum_{k,\ell}\frac{\Vert\phi_{k\ell}\Vert_\infty^2}{1+\lambda\rho_{k\ell}}\bigg\}^{1/2}\Vert\hat\beta_{n}-\beta_0+S_{n,\lambda}(\beta_0)\Vert_K\notag\\
&=O_p(\sqrt{n}v_n)=o_p(1)\,,\\
\label{in1}
\sup_{(s,t)\in[0,1]^2}|I_{2,n}(s,t)|&=\sqrt{n}\lambda^{1+(2a+1)/(4D)}\sup_{(s,t)\in[0,1]^2}\bigg|\sum_{k,\ell}\frac{\rho_{k\ell}\,V(\beta_0,\phi_{k\ell})}{1+\lambda\rho_{k\ell}}\,\phi_{k\ell}(s,t)\bigg|\notag\\
&\leq \sqrt{n}\lambda^{1+(2a+1)/(4D)}\bigg\{\sum_{k,\ell}\rho_{k\ell}^2\,V^2(\beta_0,\phi_{k\ell})\bigg\}^{1/2}\bigg\{\sum_{k,\ell}\frac{\Vert\phi_{k\ell}\Vert_\infty^2}{(1+\lambda\rho_{k\ell})^2}\bigg\}^{1/2}\notag\\
&=O(\sqrt{n}\lambda)=o(1)\,.
\end{align}
For the term $I_{3,n}$ we use  $\E\big\{\tau(X\otimes\e)(s,t)\big\}=0$
and \eqref{core} to obtain as $n\to\infty$
\begin{align}\label{var}
&\cov\big\{I_{3,n}(s_1,t_1),I_{3,n}(s_2,t_2)\big\}\notag\\
&=\lambda^{(2a+1)/(2D)}\cov\big\{\tau (X\otimes\e)(s_1,t_1),\tau (X\otimes\e)(s_2,t_2)\big\}\notag\\
&=\lambda^{(2a+1)/(2D)}\,\E\Bigg[\bigg\{\sum_{k,\ell}\frac{\l X\otimes\e,\phi_{k\ell}\r_{L^2}}{1+\lambda\rho_{k\ell}}\phi_{k\ell}(s_1,t_1)\bigg\}\bigg\{\sum_{k',\ell'}\frac{\l X\otimes\e,\phi_{k'\ell'}\r_{L^2}}{1+\lambda\rho_{k'\ell'}}\phi_{k'\ell'}(s_2,t_2)\bigg\}\Bigg]\notag\\
&=\lambda^{(2a+1)/(2D)}\sum_{k,\ell}\frac{\phi_{k\ell}(s_1,t_1)\phi_{k\ell}(s_2,t_2)}{(1+\lambda\rho_{k\ell})^2}=C_Z\{(s_1,t_1),(s_2,t_2)\}+o(1)\,.
\end{align}
Note that the results in Section~1.5 in \cite{vaart1996} are valid if the $\ell^\infty([0,1]^2)$ space is replaced by $C([0,1]^2)$. We shall prove the weak convergence in $C([0,1]^2)$ through the following two steps.

\vspace{1em}
\textbf{Step 1}. \emph{Weak convergence of the finite-dimensional marginals of $\G_n$}
\vspace{1em}

In order to prove the weak convergence of the finite-dimensional marginal distributions of $\G_n$, by the Cram\'er-Wold device, we shall show that, for any $q\in\mathbb{N}$, $(c_1,\ldots,c_q)^{\rm T}\in\mathbb{R}^q$ and $(s_1,t_1),\ldots,(s_q,t_q)\in[0,1]^2$,
\begin{align}\label{valid}
\sum_{j=1}^qc_j\G_n(s_j,t_j)\overset{d.}{\longrightarrow}\sum_{j=1}^qc_j Z(s_j,t_j)\,.
\end{align}

For $1\leq i\leq n$, let $\mathfrak U_{i,q}=n^{-1/2}\lambda^{(2a+1)/(4D)}\sum_{j=1}^qc_j\tau(X_i\otimes\e_i)(s_j,t_j)$. In view of \eqref{in1a} and \eqref{in1}, we deduce that
\begin{align*}
\sum_{j=1}^qc_j\G_n(s_j,t_j)&=\sum_{j=1}^qc_jI_{3,n}(s_j,t_j)+\sum_{j=1}^qc_j\{I_{1,n}(s_j,t_j)+I_{2,n}(s_j,t_j)\}=\sum_{i=1}^n\mathfrak U_{i,q}+o_p(1)\,.
\end{align*}
By \eqref{var} and assumption, we find, as $n\to\infty$,
\begin{align*}
\var(\mathfrak U_{i,q})&=n^{-1}\sum_{j_1,j_2=1}^qc_{j_1}c_{j_2}\,\cov\big\{\tau(X_i\otimes\e_i)(s_{j_1},t_{j_1}),\tau(X_i\otimes\e_i)(s_{j_2},t_{j_2})\big\}\\
&=n^{-1}\sum_{j_1,j_2=1}^qc_{j_1}c_{j_2}C_Z\{(s_{j_1},t_{j_1}),(s_{j_2},t_{j_2})\}+o(n^{-1})\,.
\end{align*}
When $\sum_{j_1,j_2=1}^qc_{j_1}c_{j_2}C_Z\{(s_{j_1},t_{j_1}),(s_{j_2},t_{j_2})\}=0$, we have that $\sum_{j=1}^qc_j Z(s_j,t_j)$ has a degenerate distribution with a point mass at zero, so that \eqref{valid} is followed by the Markov's inequality. When $\sum_{j_1,j_2=1}^qc_{j_1}c_{j_2}C_Z\{(s_{j_1},t_{j_1}),(s_{j_2},t_{j_2})\}\neq0$, to prove \eqref{valid}, we shall check that the triangular array of random variables $\{\mathfrak U_{i,q}\}_{i=1}^n$ satisfies Lindeberg's condition. By the Cauchy-Schwarz inequality, we find, for any $e>0$,
\begin{align*}
&\sum_{i=1}^n\E\big\{\mathfrak U_{i,q}^2\,\one(|\mathfrak U_{i,q}|>e)\big\}\\
&=\E\bigg[\lambda^{(2a+1)/(2D)}\bigg|\sum_{j=1}^qc_j\tau(X\otimes\e)(s_j,t_j)\bigg|^2\times\one\bigg\{\lambda^{(2a+1)/(4D)}\bigg|\sum_{j=1}^qc_j\tau(X \otimes\e )(s_j,t_j)\bigg|\geq e\sqrt{n}\bigg\}\bigg]\\
&\leq\lambda^{(2a+1)/(2D)}\bigg\{\E\bigg|\sum_{j=1}^qc_j\tau(X \otimes\e )(s_j,t_j)\bigg|^4\bigg\}^{\frac{1}{2}}\,\P\bigg\{\lambda^{(2a+1)/(4D)}\bigg|\sum_{j=1}^qc_j\tau(X \otimes\e )(s_j,t_j)\bigg|\geq e\sqrt{n}\bigg\}^{\frac{1}{2}}\,.
\end{align*}
In view of \eqref{supe4},
\begin{align*}
&\bigg\{\E\bigg|\sum_{j=1}^qc_j\tau(X \otimes\e )(s_j,t_j)\bigg|^4\bigg\}^{\frac{1}{2}}\leq c\, \bigg\{\sup_{(s,t)\in[0,1]^2}\E|\tau(X \otimes\e )(s,t)|^4\bigg\}^{\frac{1}{2}}\leq c\lambda^{-(a+1)/D}\,.
\end{align*}
Let $\mathfrak s_q=\sum_{j=1}^q|c_j|$. We have $\mathfrak s_q>0$, since
$\sum_{j_1,j_2=1}^qc_{j_1}c_{j_2}C_Z\{(s_{j_1},t_{j_1}),(s_{j_2},t_{j_2})\} \not =0$. In view of \eqref{taust}, by arguments similar to the ones used in \eqref{11},  we find, by taking $c_1>(c_\e D)^{-1}$, for $c_\e>0$ in Assumption~\ref{a:x},
\begin{align*}
&\P\bigg\{\lambda^{(2a+1)/(4D)}\bigg|\sum_{j=1}^qc_j\tau(X \otimes\e )(s_j,t_j)\bigg|\geq e\sqrt{n}\bigg\}\notag\\
&\leq\P\bigg\{\mathfrak s_q\lambda^{(2a+1)/(4D)}\sup_{(s,t)\in[0,1]^2}\big|\tau(X \otimes\e )(s,t)\big|\geq e\sqrt{n}\bigg\}\notag\\
&\leq\P\Big(\Vert X\Vert_{L^2}\Vert\e\Vert_{L^2}\geq ec^{-1} \mathfrak s_q^{-1}\sqrt{n}\lambda^{(2a+1)/(4D)}\Big)\notag\\
&\leq\P\Big\{\Vert X\Vert_{L^2}\geq ec_1^{-1}c^{-1} \mathfrak s_q^{-1}\sqrt{n}\lambda^{(2a+1)/(4D)}/{\log(\lambda^{-1})}\Big\}+\P\Big\{\Vert\e\Vert_{L^2}\geq c_1{\log(\lambda^{-1})}\Big\}\notag\\
&\leq \exp\{-c_Xec_1^{-1}c^{-1} \mathfrak s_q^{-1}\sqrt{n}\lambda^{(2a+1)/(4D)}/\log(\lambda^{-1})\}\E\{\exp(c_X\Vert X\Vert_{L^2})\}\notag\\
&\qquad+\lambda^{c_1c_\e}\E\{\exp(c_\e\Vert\e\Vert_{L^2})\}\notag\\
&=O\big(\lambda^{c_Xec_1^{-1}c^{-1} \mathfrak s_q^{-1}\{\sqrt{n}\lambda^{(2a+1)/(4D)}/\log^2(\lambda^{-1})\}}\big)+O(\lambda^{c_1c_\e})=o(\lambda^{1/D})\,,
\end{align*}

Therefore, for any $e>0$,
\begin{align*}
\sum_{i=1}^n\E\big\{\mathfrak U_{i,q}^2\,\one(|\mathfrak U_{i,q}|>e)\big\}\leq c\lambda^{(2a+1)/(2D)}\times\lambda^{-(a+1)/D}\times o(\lambda^{1/(2D)})=o(1)\,.
\end{align*}
By Lindeberg's CLT,
\begin{align*}
\sum_{j=1}^qc_j\G_n(s_j,t_j)& =\sum_{i=1}^n\mathfrak U_{i,q}+o_p(1) \\
& 
\converged \sum_{j=1}^qc_j Z(s_j,t_j) \sim 
{\cal N} \Big (0,\sum_{j_1,j_2=1}^qc_{j_1}c_{j_2}C_Z\{(s_{j_1},t_{j_1}),(s_{j_2},t_{j_2})\}\Big)
\,.
\end{align*}

\vspace{1em}
\textbf{Step 2}. \emph{Asymptotic tightness of $\G_n$}
\vspace{1em}

Next, we show the equicontinuity of the process $\G_n$ in \eqref{hatgn}. We first focus on the leading term $I_{3,n}$ in \eqref{in1in2}, and recall that
\begin{align}\label{hnst}
I_{3,n}(s,t)=\sum_{i=1}^n\mathfrak U_i(s,t)=n^{-1/2}\lambda^{(2a+1)/(4D)}\sum_{i=1}^n\tau(X_i\otimes\e_i)(s,t)\,.
\end{align}
Let $\Psi(x)=x^2$ and let $\Vert U\Vert_\Psi=\inf\{c>0:\E\{\Psi(|U|/c)\}\leq 1\}$ denote the Orlicz norm for a real-valued random variable $U$. For some metric $d$ on $[0,1]^2$, let $\D(w,d)$ denote the $w$-packing number of the metric space $([0,1]^2,d)$, where $d$ is an appropriate metric specified below. Since $\E\{\tau(X\otimes\e)(s,t)\}=0$ for any $(s,t)\in[0,1]^2$ and $\E\{\l X\otimes\e,\phi_{k\ell}\r_{L^2}\l X\otimes\e,\phi_{k'\ell'}\r_{L^2}\}=\delta_{kk'}\delta_{\ell\ell'}$, for $k,k',\ell,\ell'\geq 1$, by \eqref{holder}, for any $(s_1,t_1),(s_2,t_2)\in[0,1]^2$,
\begin{align}\label{m1}
&\E|I_{3,n}(s_1,t_1)-I_{3,n}(s_2,t_2)|^2\notag\\
&=\lambda^{(2a+1)/(2D)}\,\E\big|\tau(X\otimes\e)(s_1,t_1)-\tau(X\otimes\e)(s_2,t_2)\big|^2\notag\\
&=\lambda^{(2a+1)/(2D)}\,\E\bigg|\sum_{k,\ell}\frac{\l X\otimes\e,\phi_{k\ell}\r_{L^2}}{1+\lambda\rho_{k\ell}}\big\{\phi_{k\ell}(s_1,t_1)-\phi_{k\ell}(s_2,t_2)\big\}\bigg|^2\notag\\
&=\lambda^{(2a+1)/(2D)}\,\sum_{k,\ell}\frac{1}{(1+\lambda\rho_{k\ell})^2}\big|\phi_{k\ell}(s_1,t_1)-\phi_{k\ell}(s_2,t_2)\big|^2\notag\\
%
%
&\leq c\,\lambda^{(a-b)/D}\,\max\{|s_1-s_2|^{2\vartheta},|t_1-t_2|^{2\vartheta}\}\,.
\end{align}
We therefore deduce from \eqref{m1} that 
\begin{align}\label{m2}
\Vert I_{3,n}(s_1,t_1)-I_{3,n}(s_2,t_2)\Vert_\Psi\leq c\,\lambda^{(a-b)/(2D)}\,\max\{|s_1-s_2|^{\vartheta},|t_1-t_2|^{\vartheta}\}\,.
\end{align}

Next, we shall show that, there exists a metric $d$ on $[0,1]^2$ such that, for any $e>0$, 
\begin{align}\label{p}
\lim_{\delta\downarrow0}\,\limsup_{n\to\infty}\,\P\bigg\{\sup_{d\{(s_1,t_1),(s_2,t_2)\}\leq\delta}|I_{3,n}(s_1,t_1)-I_{3,n}(s_2,t_2)|>e\bigg\}=0\,,
\end{align}
where we distinguish the cases: $\vartheta>1$ and $0\leq\vartheta\leq1$.

\vspace{1em}
\emph{Case (i):} $\vartheta>1$
\vspace{1em}

Recall that in the case of $\vartheta>1$, we have assumed $b=a$, and let $d_1\{(s_1,t_1),(s_2,t_2)\}=\max\{|s_1-s_2|^{\vartheta},|t_1-t_2|^{\vartheta}\}$. In view of \eqref{m1}, we have $\Vert I_{3,n}(s_1,t_1)-I_{3,n}(s_2,t_2)\Vert_\Psi\leq c\,d_1\{(s_1,t_1),(s_2,t_2)\}$. Note that the packing number of $[0,1]^2$ with respect to the metric $d_1$ satisfies $\D(\zeta,d_1)\lesssim \zeta^{-2/\vartheta}$. By Theorem~2.2.4 in \cite{vaart1996}, for any $e,\eta>0$,
\begin{align*}
&\P\bigg\{\sup_{d_1\{(s_1,t_1),(s_2,t_2)\}\leq\delta}|I_{3,n}(s_1,t_1)-I_{3,n}(s_2,t_2)|>e\bigg\}\\
&\leq c\,\bigg\Vert\sup_{d_1\{(s_1,t_1),(s_2,t_2)\}\leq\delta}|I_{3,n}(s_1,t_1)-I_{3,n}(s_2,t_2)\bigg\Vert_\Psi\leq c\int_0^{\eta}\sqrt{\D(\zeta,d_1)}d\zeta+\delta\,\D(\eta,d_1)\\
&\leq c\int_0^{\eta}\zeta^{-1/\vartheta}d\zeta+\delta\,w^{-2/\vartheta}=c\,\eta^{(\vartheta-1)/\vartheta}+\delta\,\eta^{-2/\vartheta}\,.
\end{align*}
Using $\eta=\sqrt{\delta}$ and $\vartheta>1$, it follows that \eqref{p} holds by taking the metric $d=d_1$.

\vspace{1em}
\emph{Case (ii):} $0\leq\vartheta\leq1$
\vspace{1em}

In this case, in order to show \eqref{p}, we shall use Lemma~\ref{lem:kley} in Section~\ref{app:aux:lem}, which is a modified version of Lemma~A.1 in \cite{kley2016}. Let 
\begin{align*}
d_2\{(s_1,t_1),(s_2,t_2)\}=\max\{|s_1-s_2|^{2},|t_1-t_2|^{2}\}
\end{align*}
and let
$\overline\eta= \lambda^{(a-b)/(2D-\vartheta D)}$. In view of \eqref{m2}, we have, when $d_2\{(s_1,t_1),(s_2,t_2)\}\geq\overline\eta/2>0$,
\begin{align*}
\Vert I_{3,n}(s_1,t_1)-I_{3,n}(s_2,t_2)\Vert_\Psi\leq c\, \lambda^{(a-b)/(2D)}\big[d_2\{(s_1,t_1),(s_2,t_2)\}\big]^{\vartheta/2}\leq c\, d_2\{(s_1,t_1),(s_2,t_2)\}\,.
\end{align*}
By Assumption~\ref{a:x} and Markov's inequality, by taking $c>c_X^{-1}$, for $c_X$ in Assumption~\ref{a:x}, $\sum_{n=1}^\infty\P\big(\Vert X \Vert_{L^2}\geq c\log n\big)\leq \E\{\exp(c_X\Vert X\Vert_{L^2})\}\sum_{n=1}^\infty n^{-cc_X}< \infty$. By the Borel-Cantelli lemma and applying the same argument to $\Vert\e\Vert_{L^2}$ yields that $\Vert X \otimes\e \Vert_{L^2}=\Vert X \Vert_{L^2}\times\Vert\e \Vert_{L^2}\leq (c\log n)^2$ holds for $n$ large enough almost surely. Hence, by Assumption~\ref{a201} and Lemma~\ref{lem:sum}, for $n$ large enough,
\begin{align}\label{bound}
\sup_{(s,t)\in[0,1]^2}|\mathfrak U_i(s,t)|\leq&\sup_{(s,t)\in[0,1]^2}n^{-1/2}\lambda^{(2a+1)/(4D)}\sum_{k,\ell}\frac{1}{1+\lambda\rho_{k\ell}}|\l X_i\otimes\e_i,\phi_{k\ell}\r_{L^2}|\times\Vert\phi_{k\ell}\Vert_\infty\notag\\
&\leq n^{-1/2}\lambda^{(2a+1)/(4D)}\sum_{k,\ell}\frac{1}{1+\lambda\rho_{k\ell}}\Vert X_i\otimes\e_i\Vert_{L^2}\times\Vert\phi_{k\ell}\Vert_{L^2}\times\Vert\phi_{k\ell}\Vert_\infty\notag\\
&\leq c\,n^{-1/2}\lambda^{(2a+1)/(4D)}(\log n)^2\sum_{k,\ell}\frac{(k\ell)^{2a}}{1+\lambda(k\ell)^{2D}}\notag\\
&\leq c\,n^{-1/2}\lambda^{-(2a+1)/(4D)}(\log n)^2
\end{align}
almost surely. In addition, by \eqref{m1},
\begin{align*}
\sup_{(s_1,t_1),(s_2,t_2)\in[0,1]^2}\E|I_{3,n}(s_1,t_1)-I_{3,n}(s_2,t_2)|^2\leq c\,\lambda^{(a-b)/D}.
\end{align*}
By Bernstein's inequality, combining the above equation with \eqref{bound}, we deduce that, for $n$ large enough, for any $(s_1,t_1),(s_2,t_2)\in[0,1]^2$ and for any $e>0$,
\begin{align}\label{bern}
&\P\Big\{ |I_{3,n}(s_1,t_1)-I_{3,n}(s_2,t_2)|>e/4\Big\}\notag\\
&\leq 2\exp\bigg\{-\frac{e^2/16}{2\lambda^{(a-b)/D}+en^{-1/2}\lambda^{-(2a+1)/(4D)}(\log n)^2/6}\bigg\}\,.
\end{align}

Now, note that $\D(\zeta,d_2)\leq c\,\zeta^{-1}$ and recall that in the case of $0\leq\vartheta\leq1$ we have assumed $b<a$. By Lemma~\ref{lem:kley} and \eqref{bern}, there exists a set $\tilde[0,1]^2\subset [0,1]^2$ that contains at most $\D(\zeta,d_2)$ points, such that, for any $\delta,e>0$ and $\eta>\overline\eta$, as $n\to\infty$,
\begin{align*}
&\P\bigg\{\sup_{d_2\{(s_1,t_1),(s_2,t_2)\}\leq\delta}|I_{3,n}(s_1,t_1)-I_{3,n}(s_2,t_2)|>e\bigg\}\notag\\
&\leq c\,\bigg\{\int_{\overline\eta/2}^\eta\sqrt{\D(\zeta,d_2)}d\zeta+(\delta+2\overline\eta)\,\D(\eta,d_2)\Bigg\}^2\notag\\
&\qquad+\P\Bigg\{\sup_{\substack{d_2\{(s_1,t_1),(s_2,t_2)\}\leq\overline\eta\\(s_1,t_1)\in\tilde[0,1]^2}}|I_{3,n}(s_1,t_1)-I_{3,n}(s_2,t_2)|>e/4\Bigg\}\notag\\
&\leq c\,\bigg\{\int_{\overline\eta/2}^\eta\zeta^{-1/2}d\zeta+(\delta+2\lambda^{(a-b)/(2D-\vartheta D)})\eta^{-1}\bigg\}^2\notag\\
&\qquad+\D(\overline\eta,d_2)\times\sup_{(s_1,t_1),(s_2,t_2)\in\tilde[0,1]^2}\P\Big\{ |I_{3,n}(s_1,t_1)-I_{3,n}(s_2,t_2)|>e/4\Big\}\notag\\
&\leq c\,(\eta+\delta^2\eta^{-2})+c\,\lambda^{-(a-b)/(2D-\vartheta D)}\exp\bigg\{-\frac{e^2/16}{2\lambda^{(a-b)/D}+en^{-1/2}\lambda^{-(2a+1)/(4D)}(\log n)^2/6}\bigg\}\notag\\
&\leq c\,(\eta+\delta^2\eta^{-2})+o(1)\,,
\end{align*}
where in the last step we used $\lambda^{-1}\lesssim n^{2D}$ by Assumption~\ref{a:rate}, and the assumption that $\lambda^{(a-b)/D}=o(n^{-(a-b)\nu_1/D})$ and $n^{-1/2}\lambda^{-(2a+1)/(4D)}=o(n^{-\nu_2})$, for $\nu_1,\nu_2>0$. Therefore, by taking $\eta>0$ small enough, we deduce from the above equation that, when $0\leq\vartheta\leq 1$, for any $e>0$, \eqref{p} holds by taking the metric $d=d_2$. 

As for the remaining processes $I_{1,n}$ and $I_{2,n}$ in \eqref{gn}, in view of \eqref{in1}, for any $e>0$ and for the metric $d$,
\begin{align*}
&\lim_{\delta\downarrow0}\,\limsup_{n\to\infty}\,\P\bigg\{\sup_{d\{(s_1,t_1),(s_2,t_2)\}\leq\delta}|I_{1,n}(s_1,t_1)+I_{2,n}(s_1,t_1)-I_{1,n}(s_2,t_2)-I_{2,n}(s_2,t_2)|>e\bigg\}\\
&\leq \limsup_{n\to\infty}\,\P\bigg\{\sup_{(s,t)\in[0,1]^2}|I_{1,n}(s,t)|+\sup_{(s,t)\in[0,1]^2}|I_{2,n}(s,t)|>e/2\bigg\}=0\,,
\end{align*}
Combining this result  with \eqref{p} proves that the process $\G_n$ is asymptotic uniformly equicontinuous w.r.t.~the metric $d$ in \eqref{p} (that is, $d=d_1$ when $\vartheta>1$, and $d=d_2$ when $0\leq\vartheta\leq 1$), which entails the asymptotic tightness of $\G_n$.

The assertion of the theorem therefore follows from Theorems~1.5.4 and 1.5.7 in \cite{vaart1996}.

\subsection{Proof of Theorem~\ref{thm:minimax}}\label{app:thm:minimax}

By taking $\lambda\asymp n^{-2D/(2D+1)}$, the upper bound in (i) follows from Lemma~\ref{lem:hatbeta} and the fact that $\Vert\beta\Vert_V^2\leq\Vert\beta\Vert_K^2$ for any $\beta\in\H$.

For (ii), we prove the lower bound in the particular case where $\e $ is a mean zero Gaussian white noise process independent of $X$ with $\E\{\e^2(t)\}\equiv\sigma_\e^2>0$. 
By Theorem~2.5 in \cite{tsybakov2008}, in order to show the lower bound, we need to show that, for $M\geq2$, $\H$ contains elements $\beta_0,\ldots,\beta_M$ that satisfy the following two conditions:
\begin{enumerate}[label=(C\arabic*),noitemsep,series=conditionC]
\item\label{c1} $\Vert\beta_{j}-\beta_{k}\Vert_V^2\geq 2c_0n^{-2D/(2D+1)}$, for $0\leq j<k\leq M$;

\item\label{c2} $M^{-1}\sum_{j=1}^M\mathcal{K}(P_j,P_0)\leq\alpha\log M$, where $0<\alpha<1/8$, $\mathcal{K}$ is the Kullback-Leibler divergence, and $P_j$ denotes joint distribution of $(X_{1,j},Y_{1,j}),\ldots,(X_{n,j},Y_{n,j})$, where $Y_{i,j}(t)=\int\beta_{j}(s,t)X_{i,j}(s)ds+\e_{i,j}(t)$, for $1\leq i\leq n$.
\end{enumerate}

For the constant $D$ in Assumption~\ref{a201}, define $\nu_n=\lfloor n^{1/(4D+2)}\rfloor$. For any $\omega=(\omega_{(\nu_n+1,\nu_n+1)},\ldots,\omega_{(2\nu_n,2\nu_n)})\in\{0,1\}^{\nu_n^2}$, let
\begin{align}\label{betaomega}
\beta_{\omega}=c_1n^{-1/2}\sum_{k=\nu_n+1}^{2\nu_n}\,\sum_{\ell=\nu_n+1}^{2\nu_n}\omega_{(k,\ell)}\,\phi_{k\ell}\,,
\end{align}
where $c_1>0$ is a constant independent of $n$ to be specified later. We first verify that the $\beta_\omega$'s are elements in $\H$. Since, by Assumption~\ref{a201}, $\{\phi_{k\ell}\}_{k,\ell\geq 1}$ diagonalizes the operator $J$ defined in \eqref{jm}, we have 
\begin{align*}
\Vert\beta_\omega\Vert_K^2&=c_1^2\,n^{-1}\Bigg\l\sum_{k=\nu_n+1}^{2\nu_n}\,\sum_{\ell=\nu_n+1}^{2\nu_n}\omega_{(k,\ell)}\,\phi_{k\ell}\,,\sum_{k'=\nu_n+1}^{2\nu_n}\,\sum_{\ell'=\nu_n+1}^{2\nu_n}\omega_{(k',\ell')}\,\phi_{k'\ell'}\Bigg\r_K\\
&=c_1^2\,n^{-1}\sum_{k=\nu_n+1}^{2\nu_n}\,\sum_{\ell=\nu_n+1}^{2\nu_n}\omega_{(k,\ell)}^2\,\Vert\phi_{k\ell}\Vert_K^2=c_1^2\,n^{-1}\sum_{k=\nu_n+1}^{2\nu_n}\,\sum_{\ell=\nu_n+1}^{2\nu_n}\omega_{(k,\ell)}^2\,(1+\lambda\rho_{k\ell})\\
&\leq c\,n^{-1}\sum_{k=\nu_n+1}^{2\nu_n}\,\sum_{\ell=\nu_n+1}^{2\nu_n}\{1+\lambda(k\ell)^{2D}\}\,.
\end{align*}
Note that this inequality and the inequality in \eqref{hk} in the proof of Proposition~\ref{prop1} holds for any $\lambda>0$. Therefore, for $\Vert\cdot\Vert_\H$ defined in \eqref{snorm}, combining these two equations, we may take $\lambda=1$ and find that
\begin{align*}
\Vert\beta_\omega\Vert_\H^2\leq c\,\Vert\beta_\omega\Vert_K^2\leq c\,n^{-1}\sum_{k=\nu_n+1}^{2\nu_n}\,\sum_{\ell=\nu_n+1}^{2\nu_n}(k\ell)^{2D}\leq c\,n^{-1}\nu_n^{2+4D}\leq c\,,
\end{align*}
which shows that, for any $\omega\in\{0,1\}^{\nu_n^2}$,  $\beta_\omega$ defined in \eqref{betaomega} is an element of $\H$.

By the Varshamov-Gilbert bound (Lemma~2.9 in \citealp{tsybakov2008}), for $\nu_n^2\geq8$, there exists a subset $\Omega=\{\omega^{(0)},\ldots,\omega^{(M)}\}\in\{0,1\}^{\nu_n^2}$ with $M\geq 2^{\nu_n^2/8}$ such that, $\omega^{(0)}=(0,\ldots,0)$ and for any $0\leq j< j'\leq M$,
\begin{align*}
{\rm H}\big(\omega^{(j_1)},\omega^{(j_2)}\big)\geq\frac{\nu_n^2}{8}\,,
\end{align*}
where ${\rm H}(\cdot,\cdot)$ is the Hamming distance. For $0\leq j\leq M$, let $\omega^{(j)}=(\omega^{(j)}_{(\nu_n+1,\nu_n+1)},\ldots,\omega^{(j)}_{(2\nu_n,2\nu_n)})$. Let $\beta_0,\ldots,\beta_M$ denote the functions defined as in \eqref{betaomega} that corresponds to $\omega^{(0)},\ldots,\omega^{(M)}\in\Omega$. For $0\leq j<j'\leq M$, in view of
\eqref{betaomega},
\begin{align}\label{betadiff}
\beta_{j}-\beta_{j'}=c_1\,n^{-1/2}\sum_{k=\nu_n+1}^{2\nu_n}\,\sum_{\ell=\nu_n+1}^{2\nu_n}\big(\omega^{(j)}_{(k,\ell)}-\omega^{(j')}_{(k,\ell)}\big)\,\phi_{k\ell}\,.
\end{align}
By Assumption~\ref{a201}, since $\{\phi_{k\ell}\}_{k,\ell\geq 1}$ diagonalizes the operator $V$ defined in \eqref{v}, we deduce from \eqref{betadiff} that
\begin{align*}
\Vert\beta_{j}-\beta_{j'}\Vert_V^2&=c_1^2\,n^{-1}\sum_{k=\nu_n+1}^{2\nu_n}\,\sum_{\ell=\nu_n+1}^{2\nu_n}\big(\omega^{(j)}_{(k,\ell)}-\omega^{(j')}_{(k,\ell)}\big)^2\,V(\phi_{k\ell},\phi_{k\ell})\\
&=c_1^2\,n^{-1}\sum_{k=\nu_n+1}^{2\nu_n}\,\sum_{\ell=\nu_n+1}^{2\nu_n}\one\big\{\omega^{(j)}_{(k,\ell)}\neq\omega^{(j')}_{(k,\ell)}\big\}\\
&=c_1^2\,n^{-1}{\rm H}\big(\omega^{(j)},\omega^{(j')}\big)\geq c_1^2\,8^{-1}n^{-1}\nu_n^{2}\geq c_1^2\,8^{-1}n^{-2D/(2D+1)}\,.
\end{align*}
By taking $c_0=c_1^2/16$, the above equation indicates that Condition~\ref{c1} is valid.

For any $0\leq j<j'\leq M$, in view of \eqref{betadiff},
\begin{align*}
\mathcal{K}(P_{j},P_{j'})&=\frac{n}{2\sigma_\e^2}\,\E\int_0^1\bigg[\int_0^1\{\beta_{j}(s,t)-\beta_{j'}(s,t)\}X(s)ds\bigg]^2dt=\frac{n}{2\sigma_\e^2}\Vert\beta_{j}-\beta_{j'}\Vert_V^2\\
&=c_1^2\,\sigma_\e^{-2}\,\sum_{k=\nu_n+1}^{2\nu_n}\sum_{\ell=\nu_n+1}^{2\nu_n}\big(\omega^{(j)}_{(k,\ell)}-\omega^{(j')}_{(k,\ell)}\big)^2\leq c_1^2\,\sigma_\e^{-2}\nu_n^2\,.
\end{align*}
Therefore, for any $0<\alpha<1/8$, by taking $0<c_1<\sigma_\e\sqrt{\alpha\log 2/8}$ in \eqref{betaomega}, we have
\begin{align*}
\frac{1}{M}\sum_{j=1}^M\mathcal{K}(P_j,P_0)\leq c_1^2\,\sigma_\e^{-2}\nu_n^2\leq\frac{\alpha\nu_n^2\log2}{8}\leq\alpha\log M\,,
\end{align*}
which verifies Condition~\ref{c2} and completes the proof.

\subsection{Proof of Theorem~\ref{thm:cb:boot}}\label{app:thm:bootstrap}

Let $\BL_1\{C([0,1]^2)\}$ denote the collection of all functionals $h:C([0,1]^2)\to[-1,1]$ such that $h$ is uniformly Lipschitz: for any $g_1,g_2\in C([0,1]^2)$, $|h(g_1)-h(g_2)|\leq \Vert g_1-g_2\Vert_\infty=\sup_{(s,t)\in[0,1]^2}|g_1(s,t)-g_2(s,t)|$. We shall show that conditionally on the data $\{(X_i,Y_i)\}_{i=1}^n$, the bootstrap process $\G_{n,q}^*$ converges to the same limit as $\G_n$ in \eqref{hatgn}. To achieve this, we shall prove that, for $Z$ in \eqref{hatgn}, as $n\to\infty$,
\begin{align*}
\sup_{h\in \BL_1\{C([0,1]^2)\}}|\E_M\{h(\G_{n,1}^*)\}-\E\{h(Z)\}|=o_p(1)\,,
\end{align*}
where $\E_M$ denote the conditional expectation given the data $\{(X_i,Y_i)\}_{i=1}^n$; see Theorem~23.7 in \cite{vandervaart1998}. Note that the results in Lemma~3.1 in \cite{bucher2019} hold if their $\ell^\infty(T)$ space is replaced by $C(T)$, and therefore, in our case, we shall show that, for any fixed $Q\geq2$, as $n\to\infty$,
\begin{align}\label{gnq}
(\G_n,\G_{n,1}^*,\ldots,\G_{n,Q}^*)\weakconverge (Z,Z_1,\ldots,Z_Q)\quad\text{ in }\{C([0,1]^2)\}^{Q+1}\,,
\end{align}
where $Z_1,\ldots,Z_Q$ are i.i.d.~copies of the process $Z$ in \eqref{hatgn}.

For $1\leq q\leq Q$, define the bootstrap version of $S_{n,\lambda}$ in \eqref{dsn} by
\begin{align*}
S_{n,q}^*(\beta)&=-\frac{1}{n}\sum_{i=1}^nM_{i,q}\,\tau\bigg[X_i\otimes\bigg\{Y_i-\int_0^1\beta(s,\cdot)X_i(s)ds\bigg\}\bigg]+W_\lambda(\beta)\,,
\end{align*}
and let $ L_{n,\lambda}^*(\beta)$ denote the bootstrap objective function in \eqref{hatbetastar}. Direct calculations yields that $L_{n,\lambda}^*(\beta)\beta_1=\l S_{n,q}^*(\beta),\beta_1\r_K$, $\E \{\mathcal D^2 L_{n,\lambda}^*(\beta)\beta_1\beta_2\}=\l \beta_1,\beta_2\r_K$ and $S_{n,q}^*(\beta_0)=-n^{-1}\sum_{i=1}^nM_{i,q}\,\tau(X_i\otimes\e_i)$. Recalling the notation of $I_{n,2}(s,t)$ in \eqref{in1in2}, we have
\begin{align*}
&\G_{n,q}^*(s,t)=\sqrt{n}\lambda^{(2a+1)/(4D)}\{\hat\beta^*_{n,q}(s,t)-\hat\beta_{n}(s,t)\}=\mathfrak J_{n,q}^*(s,t)+H_{n,q}^*(s,t)-I_{n,2}(s,t)\,,
\end{align*}
where
\begin{align*}
\mathfrak J_{n,q}^*(s,t)&=\sqrt{n}\lambda^{(2a+1)/(4D)}\{\hat\beta^*_{n,q}(s,t)-\beta_0(s,t)+S_{n,\lambda}^*\beta_0(s,t)\}\,,\\
H_{n,q}^*(s,t)&=\sqrt{n}\lambda^{(2a+1)/(4D)}\{S_{n,\lambda}\beta_0(s,t)-S^*_{n,q}\beta_0(s,t)\}\\
&=n^{-1/2}\lambda^{(2a+1)/(4D)}\sum_{i=1}^n(1-M_{i,q})\,\tau(X_i\otimes\e_i)(s,t)\,.
\end{align*}
By exactly the same arguments used in the proof of Theorem~\ref{thm:process}, it follows that $\sup_{(s,t)\in[0,1]^2}|\mathfrak J_{n,q}^*(s,t)|=o_p(1)$ as $n\to\infty$. Furthermore, recall from \eqref{in1} that $\sup_{(s,t)\in[0,1]^2}|I_{n,2}(s,t)|=o_p(1)$. Since in the proof of Theorem~\ref{thm:process} we have shown that $H_n\weakconverge Z$ in $C([0,1]^2)$, therefore, in order to show \eqref{gnq}, we shall show that $(H_{n,1}^*,\ldots,H_{n,Q}^*)\weakconverge (Z_1,\ldots,Z_Q)$ in $C([0,1]^2)$. The proof of \eqref{gnq} therefore relies on the finite dimensional convergence of $(H_{n,1}^*,\ldots,H_{n,Q}^*)$ and the asymptotic tightness of the process $H_{n,q}^*$.

We first show convergence of the finite dimensional distributions and introduce the notations $\HH_n^*=(H_{n,1}^*,\ldots,H_{n,Q}^*)^\top$ and $\bZ=(Z_1,\ldots,Z_Q)^\top$. For arbitrary $L\in\mathbb{N}$, $b_1,\ldots,b_L\in\mathbb{R}$ and $\bc_1,\ldots,\bc_L\in\mathbb{R}^{Q}$, we need to prove that
\begin{align}\label{ch}
\sum_{\ell=1}^{L}\bc_\ell^{\top}\HH^*_n(s_\ell,t_\ell)\converged \sum_{\ell=1}^L\bc_\ell^{\top}\bZ(s_\ell,t_\ell)\,.
\end{align}
For $1\leq i\leq n$, let $\mathfrak H_i^*=n^{-1/2}\lambda^{(2a+1)/(4D)}\sum_{\ell=1}^{L}\sum_{q=1}^Qc_{\ell,q}(M_{i,q}-1)\,\tau(X_{i}\otimes\e_{i})(s_\ell,t_\ell)$, and note that the $\mathfrak H_i^*$'s are i.i.d.~and $\sum_{\ell=1}^{L}\bc_\ell^{\top}\HH^*_n(s_\ell,t_\ell)=\sum_{i=1}^n\mathfrak H_i^*$. Since the $M_{i,q}$'s are i.i.d., $\E(M_{i,q})=1$ and $\E|M_{i,q}-1|^2=1$, direct calculations yield that
\begin{align*}
&\var\bigg\{\sum_{\ell=1}^{L}\bc_\ell^{\top}\HH^*_n(s_\ell,t_\ell)\bigg\}=\var(\mathfrak H_1^*)=\lambda^{(2a+1)/(2D)}\var\bigg\{\sum_{\ell=1}^{L}\sum_{q=1}^Qc_{\ell,q}(M_{1,q}-1)\,\tau(X_{1}\otimes\e_{1})(s_\ell,t_\ell)\bigg\}\\
&=\lambda^{(2a+1)/(2D)}\sum_{\ell,\ell'=1}^{L}\sum_{q,q'=1}^Qc_{\ell,q}c_{\ell',q'}\E\{(M_{1,q}-1)(M_{1,q'}-1)\}\E\big\{\tau(X_{1}\otimes\e_{1})(s_\ell,t_\ell)\tau(X_{1}\otimes\e_{1})(s_{\ell'},t_{\ell'})\big\}\\
&=\lambda^{(2a+1)/(2D)}\sum_{\ell,\ell'=1}^{L}\sum_{q=1}^Qc_{\ell,q}c_{\ell',q}\E\{(M_{1,q}-1)^2\}\E\bigg\{\tau(X_{1}\otimes\e_{1})(s_\ell,t_\ell)\tau(X_{1}\otimes\e_{1})(s_{\ell'},t_{\ell'})\bigg\}\\
&=\lambda^{(2a+1)/(2D)}\sum_{\ell,\ell'=1}^{L}\sum_{q=1}^Qc_{\ell,q}c_{\ell',q}\sum_{k,j}\frac{\phi_{kj}(s_\ell,t_\ell)\phi_{kj}(s_{\ell'},t_{\ell'})}{(1+\lambda\rho_{kj})^2}\\
&=\sum_{q=1}^Q\sum_{\ell,\ell'=1}^{L}c_{\ell,q}c_{\ell',q}\,C_Z\{(s_\ell,t_\ell),(s_{\ell'},t_{\ell'})\}+o(1)\,,
\end{align*}
as $n\to\infty$. Note that $\var\big\{\sum_{\ell=1}^L\bc_\ell^{\top}\bZ(s_\ell,t_\ell)\big\}=\sum_{q=1}^Q\sum_{\ell,\ell'=1}^{L}c_{\ell,q}c_{\ell',q}\,C_Z\{(s_\ell,t_\ell),(s_{\ell'},t_{\ell'})\}$. When $\sum_{q=1}^Q\sum_{\ell,\ell'=1}^{L}c_{\ell,q}c_{\ell',q}\,C_Z\{(s_\ell,t_\ell),(s_{\ell'},t_{\ell'})\}=0$, $\sum_{\ell=1}^L\bc_\ell^{\top}\bZ(s_\ell,t_\ell)$ has a degenerate distribution with a point mass at zero, and $\sum_{\ell=1}^{L}\bc_\ell^{\top}\HH^*_n(s_\ell,t_\ell)=o_p(1)$, so that \eqref{ch} is valid. When $\sum_{q=1}^Q\sum_{\ell,\ell'=1}^{L}c_{\ell,q}c_{\ell',q}\,C_Z\{(s_\ell,t_\ell),(s_{\ell'},t_{\ell'})\}\neq0$, using arguments similar to the ones used in the proof of Theorem~\ref{thm:process}, we can show that Lindeberg's condition is satisfied, so that \eqref{ch} is valid.

For the asymptotic tightness of the $H_{n,q}^*$, note that $|1-M_{i,q}|\leq\sqrt{2}$ almost surely. Therefore, the asymptotic tightness of $H_n$ in \eqref{hnst}, proved in Section~\ref{app:thm:process}, implies the asymptotic tightness of $H_{n,q}^*$. By Theorem~1.5.4 in \cite{vaart1996}, together with the weak convergence $H_n\weakconverge Z$ proved in Section~\ref{app:thm:process}, we have that, for any $Q\geq 2$, as $n\to\infty$,
\begin{align*}
(H_n,H_{n,1}^*,\ldots,H_{n,Q}^*)\weakconverge (Z,Z_1,\ldots,Z_Q)\quad\text{ in }\{C([0,1]^2)\}^{Q+1}\,,
\end{align*}
which validates \eqref{gnq} and completes the proof.

\subsection{Proof of Theorem~\ref{thm:lrt}}\label{app:thm:lrt}

Defining $\hat\beta_\Delta=\hat\beta_{n}-\beta_*$ and observing \eqref{dell}, \eqref{dsn} and \eqref{s}, it follows that
\begin{align*}
\L_n(\beta_*)&= L_{n,\lambda}(\beta_*)- L_{n,\lambda}(\beta_*+\hat\beta_\Delta)=-\mathcal D L_{n,\lambda}(\beta_*)\hat\beta_\Delta-\frac{1}{2}\mathcal D^2 L_{n,\lambda}(\beta_*)\hat\beta_\Delta\hat\beta_\Delta\\
&=-\big\l S_{n,\lambda}(\beta_*),\hat\beta_\Delta\big\r_K-\frac{1}{2}\big\l \mathcal DS_{n,\lambda}(\beta_*)\hat\beta_\Delta,\hat\beta_\Delta\big\r_K\\
&=-\big\l S_{n,\lambda}(\beta_*),\hat\beta_\Delta\big\r_K-\frac{1}{2}\Big\l \mathcal DS_{n,\lambda}(\beta_*)\hat\beta_\Delta-\mathcal DS_{\lambda}(\beta_*)\hat\beta_\Delta,\hat\beta_\Delta\Big\r_K-\frac{1}{2}\Vert\hat\beta_\Delta\Vert_K^2\\
&=\frac{1}{2}\Vert\hat\beta_\Delta\Vert_K^2-\big\l \hat\beta_\Delta+S_{n,\lambda}(\beta_*),\hat\beta_\Delta\big\r_K-\frac{1}{2}\Big\l \mathcal DS_{n,\lambda}(\beta_*)\hat\beta_\Delta-\mathcal DS_{\lambda}(\beta_*)\hat\beta_\Delta,\hat\beta_\Delta\Big\r_K\,,
\end{align*}
where we use the fact that $\mathcal DS_{\lambda}(\beta_*)=id$. Note that
\begin{align*}
\Vert\hat\beta_\Delta\Vert_K^2&=\big\Vert\{\hat\beta_\Delta+S_{n,\lambda}(\beta_*)\}-S_{n,\lambda}(\beta_*)\big\Vert_K^2\\
&=\Vert S_{n,\lambda}(\beta_*)\Vert_K^2-2\l S_{n,\lambda}(\beta_*),\hat\beta_\Delta+S_{n,\lambda}(\beta_*)\r_K+\Vert\hat\beta_\Delta+S_{n,\lambda}(\beta_*)\Vert_K^2\,.
\end{align*}
Therefore, we deduce that
\begin{align*}
2n\,\L_n(\beta_*)&=I_{1,n}+I_{2,n}+I_{3,n}+I_{4,n}\,,
\end{align*}
where we use the notations that
\begin{align}\label{I1n}
\begin{split}
&I_{1,n}=n\Vert S_{n,\lambda}(\beta_*)\Vert_K^2\,,\\
&I_{2,n}=-2n\big\l\hat\beta_\Delta+S_{n,\lambda}(\beta_*),\hat\beta_\Delta\big\r_K\,,\\
&I_{3,n}=-n\big\l \mathcal DS_{n,\lambda}(\beta_*)\hat\beta_\Delta-\mathcal DS_{\lambda}(\beta_*)\hat\beta_\Delta,\hat\beta_\Delta\big\r_K\,\\
&I_{4,n}=-2n\big\l S_{n,\lambda}(\beta_*),\hat\beta_\Delta+S_{n,\lambda}(\beta_*)\big\r_K+n\Vert\hat\beta_\Delta+S_{n,\lambda}(\beta_*)\Vert_K^2\,.
\end{split} 
\end{align}

We now discuss the term $I_{\ell,n}$ separately, starting with $I_{1,n}$. In view of \eqref{dsn}, we have
\begin{align}\label{nsn}
&n\Vert S_{n,\lambda}(\beta_*)\Vert_K^2=n\,\bigg\Vert-\frac{1}{n}\sum_{i=1}^n\tau(X_i\otimes\e_i)+W_{\lambda}(\beta_*)\bigg\Vert_K^2=I_{1,1,n}+I_{1,2,n}+I_{1,3,n}\,,
\end{align}
where
\begin{align}\label{i13}
\begin{split}
&I_{1,1,n}=\frac{1}{n}\,\bigg\Vert\sum_{i=1}^n\tau(X_i\otimes\e_i)\bigg\Vert_K^2\,\\
&I_{1,2,n}=-2\sum_{i=1}^n\big\l\tau(X_i\otimes\e_i),W_{\lambda}(\beta_*)\big\r_K\,\\
&I_{1,3,n}=n\,\Vert W_\lambda(\beta_*)\Vert_K^2\,.
\end{split} 
\end{align}
Observing \eqref{expansion}, we have
\begin{align}\label{expansion2}
\Vert\beta\Vert_K^2=\sum_{k,\ell,k',\ell'}\frac{\l\beta,\phi_{k\ell}\r_K\,\l\beta,\phi_{k'\ell'}\r_K}{(1+\lambda\rho_{k\ell})(1+\lambda\rho_{k'\ell'})}\l\phi_{k\ell},\phi_{k'\ell'}\r_K=\sum_{k,\ell}\frac{\l\beta,\phi_{k\ell}\r_K^2}{1+\lambda\rho_{k\ell}}\,,\quad \forall \beta\in\H\,,
\end{align}
which gives
\begin{align}\label{dterm}
I_{1,1,n}&=\frac{1}{n}\bigg\Vert\sum_{i=1}^n\tau(X_i\otimes\e_i)\bigg\Vert_K^2=\frac{1}{n}\sum_{k,\ell}\frac{1}{1+\lambda\rho_{k\ell}}\bigg\l\sum_{i=1}^n\tau(X_i\otimes\e_i),\phi_{k\ell}\bigg\r_K^2\notag\\
&=\frac{1}{n}\sum_{k,\ell}\frac{1}{1+\lambda\rho_{k\ell}}\bigg\{\sum_{i=1}^n\big\l\tau(X_i\otimes\e_i),\phi_{k\ell}\big\r_K\bigg\}^2=\frac{W_n}{n}+\frac{1}{n}\sum_{i=1}^nW_{0,i}\,,
\end{align}
where $W_n=\sum_{i_1<i_2}W_{i_1i_2}$ and
\begin{align}\label{w01}
&W_{i_1i_2}=2\sum_{k,\ell}\frac{1}{1+\lambda\rho_{k\ell}}\big\l\tau(X_{i_1}\otimes\e_{i_1}),\phi_{k\ell}\big\r_K\big\l\tau(X_{i_2}\otimes\e_{i_2}),\phi_{k\ell}\big\r_K\,,\notag\\
&W_{0,i}=\sum_{k,\ell}\frac{1}{1+\lambda\rho_{k\ell}}\big\l\tau(X_i\otimes\e_i),\phi_{k\ell}\big\r_K^2\,.
\end{align}
For $W_n$ in \eqref{dterm}, we have
\begin{align*}
\sigma_{W_n}^2&:=\E(W_n^2)=\sum_{i_1<i_2}\E(W_{i_1i_2}^2)\\
&=4\sum_{i_1<i_2}\E\Bigg[\bigg\{\sum_{k,\ell}\frac{1}{1+\lambda\rho_{k\ell}}\big\l\tau(X_{i_1}\otimes\e_{i_1}),\phi_{k\ell}\big\r_K\big\l\tau(X_{i_2}\otimes\e_{i_2}),\phi_{k\ell}\big\r_K\bigg\}\\
&\qquad\times\bigg\{\sum_{k',\ell'}\frac{1}{1+\lambda\rho_{k'\ell'}}\big\l\tau(X_{i_1}\otimes\e_{i_1}),\phi_{k'\ell'}\big\r_K\big\l\tau(X_{i_2}\otimes\e_{i_2}),\phi_{k'\ell'}\big\r_K\bigg\}\Bigg]\\
&=4\sum_{i_1<i_2}\sum_{k,k',\ell,\ell'}\bigg[\frac{1}{(1+\lambda\rho_{k\ell})(1+\lambda\rho_{k'\ell'})}\E\Big\{\big\l\tau(X_{i_1}\otimes\e_{i_1}),\phi_{k\ell}\big\r_K\big\l\tau(X_{i_1}\otimes\e_{i_1}),\phi_{k'\ell'}\big\r_K\Big\}\\
&\qquad\times\E\Big\{\big\l\tau(X_{i_2}\otimes\e_{i_2}),\phi_{k\ell}\big\r_K\big\l\tau(X_{i_2}\otimes\e_{i_2}),\phi_{k'\ell'}\big\r_K\Big\}\bigg]\\
&=2n(n-1)\sum_{k,\ell}\frac{1}{(1+\lambda\rho_{k\ell})^2}=O(n^2\lambda^{-1/(2D)})\,,
\end{align*}
by Assumption~\ref{a201}.

In order to show the asymptotic normality of $W_n$, we use Proposition~3.2 in \cite{dejong1987} and show that
\begin{align*}
&H_1=\sum_{i_1<i_2}\E(W_{i_1i_2}^4)\\
&H_2=\sum_{i_1<i_2<i_3}\big\{\E(W_{i_1i_2}^2W_{i_2i_3}^2)+\E(W_{i_2i_1}^2W_{i_2i_3}^2)+\E(W_{i_3i_1}^2W_{i_3i_2}^2)\big\}\,,\\
&H_3=\sum_{i_1<i_2<i_3<i_4}\big\{\E(W_{i_1i_2}W_{i_1i_3}W_{i_4i_2}W_{i_4i_3})+\E(W_{i_1i_2}W_{i_1i_4}W_{i_3i_2}W_{i_3i_4})+\E(W_{i_1i_3}W_{i_1i_4}W_{i_2i_3}W_{i_2i_4})\big\}\,.
\end{align*}
are of order $o(\sigma_{W_n}^4)$ as $n\to\infty$. Since $\E\{\l\tau(X_{i_1}\otimes\e_{i_1}),\phi_{k\ell}\r_K\}=0$ due to Assumption~\ref{a0}, we have $\E(W_{i_1i_2}|\e_{i_1},X_{i_1})=0$ for $i_1\neq i_2$. From \eqref{e4}, we obtain $\E\{\l\tau(X_{i}\otimes\e_{i}),\phi_{k\ell}\r_K^4\}\leq c$, which implies that
\begin{align*}
H_1=\sum_{i_1<i_2}\E(W_{i_1i_2}^4)&=16\sum_{i_1<i_2}\E\Bigg[\bigg\{\sum_{k,\ell}\frac{1}{1+\lambda\rho_{k\ell}}\big\l\tau(X_{i_1}\otimes\e_{i_1}),\phi_{k\ell}\big\r_K\big\l\tau(X_{i_2}\otimes\e_{i_2}),\phi_{k\ell}\big\r_K\bigg\}^4\Bigg]\\
&\leq c\,n^2\bigg(\sum_{k,\ell}\frac{1}{1+\lambda\rho_{k\ell}}\bigg)^4\leq c\,n^2\lambda^{-2/D}\,.
\end{align*}
Since $\E(W_{i_1i_2}^2W_{i_1i_3}^2)\leq \E(W_{i_1i_2}^4)$, we have $H_2\leq cn^3\lambda^{-2/D}$. Finally, for the term $H_3$, we use \eqref{core} and obtain,
\begin{align*}
&\E(W_{i_1i_2}W_{i_1i_3}W_{i_4i_2}W_{i_4i_3})\\
&=16\,\E\Bigg[\bigg\{\sum_{k_1,\ell_1}\frac{1}{1+\lambda\rho_{k_1\ell_1}}\big\l\tau(X_{i_1}\otimes\e_{i_1}),\phi_{k_1\ell_1}\big\r_K\big\l\tau(X_{i_2}\otimes\e_{i_2}),\phi_{k_1\ell_1}\big\r_K\bigg\}\\
&\quad\times\bigg\{\sum_{k_2,\ell_2}\frac{1}{1+\lambda\rho_{k_2\ell_2}}\big\l\tau(X_{i_1}\otimes\e_{i_1}),\phi_{k_2\ell_2}\big\r_K\big\l\tau(X_{i_3}\otimes\e_{i_3}),\phi_{k_2\ell_2}\big\r_K\bigg\}\\
&\quad\times\bigg\{\sum_{k_3,\ell_3}\frac{1}{1+\lambda\rho_{k_3\ell_3}}\big\l\tau(X_{i_4}\otimes\e_{i_4}),\phi_{k_3\ell_3}\big\r_K\big\l\tau(X_{i_2}\otimes\e_{i_2}),\phi_{k_3\ell_3}\big\r_K\bigg\}\\
&\quad\times\bigg\{\sum_{k_4,\ell_4}\frac{1}{1+\lambda\rho_{k_4\ell_4}}\big\l\tau(X_{i_4}\otimes\e_{i_4}),\phi_{k_4\ell_4}\big\r_K\big\l\tau(X_{i_3}\otimes\e_{i_3}),\phi_{k_4\ell_4}\big\r_K\bigg\}\Bigg]\\
&=16\,\sum_{k_1,k_2,k_3,k_4,\ell_1,\ell_2,\ell_3,\ell_4}\delta_{k_1k_2}\,\delta_{k_1k_3}\,\delta_{k_2k_4}\,\delta_{k_3k_4}\,\delta_{\ell_1\ell_2}\,\delta_{\ell_1\ell_3}\,\delta_{\ell_2\ell_4}\,\delta_{\ell_3\ell_4}\,\prod_{j=1}^4\frac{1}{1+\lambda\rho_{k_j\ell_j}}\\
&=16\,\sum_{k,\ell}\frac{1}{(1+\lambda\rho_{k\ell})^4}\,.
\end{align*}
We therefore deduce that $H_3\leq cn^4\lambda^{-1/(2D)}$, which yields 
\begin{align*}
H_1+H_2+H_3=O(n^3\lambda^{-2/D}+n^4\lambda^{-1/(2D)})=o(\sigma_{W_n}^4)
\end{align*}
since $n\lambda^{1/D}\to \infty$. By Proposition~3.2 in \cite{dejong1987}, it follows that
\begin{align}\label{normal}
\bigg\{2\sum_{k,\ell}\frac{1}{(1+\lambda\rho_{k\ell})^2}\bigg\}^{-1/2}\,\frac{W_n}{n}\converged N(0,1)\,.
\end{align}

Next, we examine the second term $n^{-1}\sum_{i=1}^nW_{0,i}$ in \eqref{dterm}. Note that the $W_{0,i}$'s in \eqref{w01} are i.i.d.~and satisfy
\begin{align}\label{e}
\E(W_{0,i})=\sum_{k,\ell}\frac{1}{1+\lambda\rho_{k\ell}}\E\Big\{\big\l\tau(X_i\otimes\e_i),\phi_{k\ell}\big\r_K^2\Big\}=\sum_{k,\ell}\frac{1}{1+\lambda\rho_{k\ell}}\,,
\end{align}
where we used \eqref{core} in the last step. For the variance, by \eqref{e4}, $\E\{\l\tau(X_i\otimes\e_i),\phi_{k\ell}\r_K^4\}\leq c$, so that by Assumption~\ref{a201},
\begin{align}\label{vara}
\var\Big(n^{-1}\sum_{i=1}^nW_{0,i}\Big)&=n^{-1}\,\var(W_{0,i})\leq n^{-1}\,\E(W_{0,i}^2)=n^{-1}\,\E\,\bigg|\sum_{k,\ell}\frac{1}{1+\lambda\rho_{k\ell}}\big\l\tau(X_i\otimes\e_i),\phi_{k\ell}\big\r_K^2\bigg|^2\notag\\
&\leq c\,n^{-1}\,\bigg\{\sum_{k,\ell}\frac{1}{1+\lambda\rho_{k\ell}}\bigg\}^2=O(n^{-1}\lambda^{-1/D})\,.
\end{align}
Therefore, the above equation and \eqref{e} yields that
\begin{align*}
\frac{1}{n}\sum_{i=1}^nW_{0,i}=\sum_{k,\ell}\frac{1}{1+\lambda\rho_{k\ell}}+O_p(n^{-1/2}\lambda^{-1/(2D)})=O_p(\lambda^{-1/(2D)}+n^{-1/2}\lambda^{-1/(2D)})=O_p(\lambda^{-1/(2D)})\,,
\end{align*}
where we use Assumption~\ref{a201}, which yields $\E(W_{0,i})\leq c\,\lambda^{-1/{2D}}$. In view of \eqref{dterm}, since we have shown $W_n/n=O_p(\lambda^{-1/(4D)})$, we therefore deduce from the above equation that
\begin{align}\label{111}
I_{1,1,n}=\frac{1}{n}\,\bigg\Vert\sum_{i=1}^n\tau(X_i\otimes\e_i)\bigg\Vert_K^2=O_p(\lambda^{-1/(2D)})\,.
\end{align}
Moreover, since by Assumption~\ref{a201}, $\sum_{k,\ell}(1+\lambda\rho_{k\ell})^{-2}=O(\lambda^{-1/(2D)})$, and in view of \eqref{vara}, $\var\big(n^{-1}\sum_{i=1}^nW_{0,i}\big)=O(n^{-1}\lambda^{-1/D})=o(\lambda^{-1/(2D)})$ due to the assumption that $n^{-1}\lambda^{-1/(2D)}=o(1)$ in Assumption~\ref{a:rate}, combining \eqref{dterm}, \eqref{normal} and \eqref{e},
\begin{align}\label{conv}
&\bigg\{2\sum_{k,\ell}\frac{1}{(1+\lambda\rho_{k\ell})^2}\bigg\}^{-1/2}\big\{I_{1,1,n}-\E(I_{1,1,n})\big\}\notag\\
&=\bigg\{2\sum_{k,\ell}\frac{1}{(1+\lambda\rho_{k\ell})^2}\bigg\}^{-1/2}\bigg(\frac{W_n}{n}+\frac{1}{n}\sum_{i=1}^nW_{0,i}-\sum_{k,\ell}\frac{1}{1+\lambda\rho_{k\ell}}\bigg)\converged N(0,1)\,.
\end{align}

We now consider the term $I_{1,2,n}$ in \eqref{i13}. By Assumption~\ref{a201}, we obtain 
\begin{align}\label{m8}
W_\lambda(\beta_*)=\sum_{k,\ell}V(\beta_*,\phi_{k\ell})W_\lambda(\phi_{k\ell})=\lambda\sum_{k,\ell}\frac{V(\beta_*,\phi_{k\ell})\,\rho_{k\ell}\,\phi_{k\ell}}{1+\lambda\rho_{k\ell}}\,.
\end{align}
Since $J(\beta_*,\beta_*)=\sum_{k,\ell}\rho_{k\ell}V^2(\beta_*,\phi_{k\ell})<\infty$, we have $\lambda\sum_{k,\ell}\rho_{k\ell}^2\,V^2(\beta_*,\phi_{k\ell})/(1+\lambda\rho_{k\ell})^2=o(1)$ by the dominated convergence theorem. Since $\E\{\l\tau(X_i\otimes\e_i),W_{\lambda}(\beta_*)\r_K\}=0$, it follows from \eqref{tau} and Assumption~\ref{a201} that
\begin{align}\label{esquare}
\E(I_{1,2,n}^2)&=n\,\E\Big\{\l\tau(X_i\otimes\e_i),W_{\lambda}(\beta_*)\r_K\Big\}^2=n\,\E\Big\{\big\l X\otimes\e,W_{\lambda}(\beta_*)\big\r_{L^2}^2\Big\}\notag\\
&=n\,V\{W_{\lambda}(\beta_*),W_{\lambda}(\beta_*)\}=n\lambda^2V\bigg\{\sum_{k,\ell}\frac{\rho_{k\ell}V(\beta_*,\phi_{k\ell})}{1+\lambda\rho_{k\ell}}\phi_{k\ell},\sum_{k,\ell}\frac{\rho_{k\ell}V(\beta_*,\phi_{k\ell})}{1+\lambda\rho_{k\ell}}\phi_{k\ell}\bigg\}\notag\\
&=n\lambda^2\sum_{k,\ell}\frac{\rho^2_{k\ell}\,V^2(\beta_*,\phi_{k\ell})}{(1+\lambda\rho_{k\ell})^2}=o(n\lambda)\,.
\end{align}
For the term $I_{1,3,n}=n\Vert W_\lambda(\beta_*)\Vert_K^2$ in \eqref{i13}, we use \eqref{m8}, the dominated convergence theorem and the fact that $\sum_{k,\ell}\rho_{k\ell}V^2(\beta_*,\phi_{k\ell})<\infty$, and obtain
\begin{align}\label{nwlambda}
n\Vert W_\lambda(\beta_*)\Vert_K^2=n\bigg\Vert \sum_{k,\ell}\frac{\rho_{k\ell}V(\beta_*,\phi_{k\ell})}{1+\lambda\rho_{k\ell}}\phi_{k\ell}\bigg\Vert_K^2=n\lambda^2\sum_{k,\ell}\frac{\rho_{k\ell}^2\, V^2(\beta_*,\phi_{k\ell})}{1+\lambda\rho_{k\ell}}=o(n\lambda)\,.
\end{align}
Therefore, combining \eqref{nsn}, \eqref{111}, \eqref{esquare} and \eqref{nsnlambda}, we find
\begin{align}\label{nsnlambda}
n\Vert S_{n,\lambda}(\beta_*)\Vert_K^2&=O_p(\lambda^{-1/(2D)})+o(n\lambda)+o_p(n^{1/2}\lambda^{1/2})=O_p(\lambda^{-1/(2D)})\,.
\end{align}

For the term $I_{2,n}$ in \eqref{I1n}, by Lemma~\ref{lem:hatbeta} and Theorem~\ref{thm:bahadur}, we find
\begin{align*}
|I_{2,n}|&\leq 2n\,\Vert \hat\beta_\Delta+S_{n,\lambda}(\beta_*)\Vert_K\times\Vert\hat\beta_\Delta\Vert_K= O_p(nv_n)\times O_p\big(\lambda^{1/2}+n^{-1/2}\lambda^{-1/(4D)}\big)\,,
\end{align*}
where $v_n$ is defined in \eqref{vn}. For the term $I_{3,n}$ in \eqref{I1n}, note that, in view of \eqref{dsn}, for $H_n(\cdot)$ defined in \eqref{hnbeta},
\begin{align*}
&\Vert\mathcal DS_{n,\lambda}(\beta_*)\hat\beta_\Delta-\mathcal DS_{\lambda}(\beta_*)\hat\beta_\Delta\Vert_K=n^{-1/2}\Vert H_n(\hat\beta_\Delta)\Vert_K=O_p(v_n)\,,
\end{align*}
where we used \eqref{hn}. Therefore, by Lemma~\ref{lem:hatbeta},
\begin{align*}
|I_{3,n}|\leq n\,\Vert \mathcal DS_{n,\lambda}(\beta_*)\hat\beta_\Delta-\mathcal DS_{\lambda}(\beta_*)\hat\beta_\Delta\Vert_K\times\Vert\hat\beta_\Delta\Vert_K=O_p(nv_n)\times O_p\big(\lambda^{1/2}+n^{-1/2}\lambda^{-1/(4D)}\big)\,.
\end{align*}
For the term $I_{4,n}$ in \eqref{I1n}, by \eqref{nsnlambda} and Theorem~\ref{thm:bahadur},
\begin{align*}
|I_{4,n}|\leq 2n\Vert S_{n,\lambda}(\beta_*)\Vert_K\times\Vert\hat\beta_\Delta+S_{n,\lambda}(\beta_*)\Vert_K+n\Vert\hat\beta_\Delta+S_{n,\lambda}(\beta_*)\Vert_K^2\leq O_p(\lambda^{-1/(2D)}v_n+nv_n^2)\,.
\end{align*}
Combining \eqref{111}, \eqref{esquare}, \eqref{nwlambda} and the above convergence rates of $I_{2,n},I_{3,n},I_{4,n}$, it follows that $I_{1,2,n}+I_{1,3,n}+I_{2,n}+I_{3,n}+I_{4,n}=o_p(\lambda^{-1/(2D)})$. In addition, we use Assumption~\ref{a201} and Lemma~\ref{lem:sum} to obtain that $\sum_{k,\ell}(1+\lambda\rho_{k\ell})^{-1}\asymp \lambda^{-1/(2D)}$ and $\sum_{k,\ell}(1+\lambda\rho_{k\ell})^{-2}\asymp \lambda^{-1/(2D)}$. Therefore, in view of \eqref{111}, \eqref{conv},
\begin{align*}
\bigg\{\sum_{k,\ell}\frac{2}{(1+\lambda\rho_{k\ell})^2}\bigg\}^{-1/2}\bigg(2n\,\L_n(\beta_*)-\sum_{k,\ell}\frac{1}{1+\lambda\rho_{k\ell}}\bigg)\converged N(0,1)\,.
\end{align*}
Since for $u_n$ and $\sigma_n^2$ in \eqref{usigma},
\begin{align*}
&\sqrt{u_n}=\frac{\sum_{k,\ell}(1+\lambda\rho_{k\ell})^{-1}}{\{\sum_{k,\ell}(1+\lambda\rho_{k\ell})^{-2}\}^{1/2}}\,;\quad\frac{\sigma_n^2}{\sqrt{u_n}}=\frac{\big\{\sum_{k,\ell}(1+\lambda\rho_{k\ell})^{-2}\big\}^{1/2}}{\sum_{k,\ell}(1+\lambda\rho_{k\ell})^{-2}}=\bigg\{\sum_{k,\ell}\frac{1}{(1+\lambda\rho_{k\ell})^2}\bigg\}^{-1/2}\,,
\end{align*}
the proof is therefore complete.

\subsection{Proof of Corollary~\ref{cordet} and \eqref{eq:rele}}\label{app:thm:extreme}

By assumption, we have $n^{-1/2}\lambda^{-(2a+1)/(4D)}\log(n\lambda^{(2a+1)/(2D)})=o(1)$. Therefore, Corollary~\ref{cordet} is a consequence of Theorem~B.1 in \cite{dettebio} and Theorem~\ref{thm:process}. For a proof of \eqref{eq:rele}, note that $d_\infty<\Delta$,
\begin{align*}
&\lim_{n\to\infty}\P\bigg\{\hat d_\infty>\Delta+\frac{\mathcal Q_{1-\alpha}(T_\EE)}{\sqrt{n}\lambda^{(2a+1)/(4D)}}\bigg\}\\
&=\lim_{n\to\infty}\P\Big\{\sqrt{n}\lambda^{(2a+1)/(4D)}(\hat d_\infty-d_\infty)>\sqrt{n}\lambda^{(2a+1)/(4D)}(\Delta-d_\infty)+\mathcal Q_{1-\alpha}(T_\EE)\Big\}=0\,,
\end{align*}
since $\sqrt{n}\lambda^{(2a+1)/(4D)}\to\infty$ as $n\to\infty$, where $T_\EE$ is defined in \eqref{te}. If $d_\infty=\Delta$,
\begin{align*}
&\lim_{n\to\infty}\P\bigg\{\hat d_\infty>\Delta+\frac{\mathcal Q_{1-\alpha}(T_\EE)}{\sqrt{n}\lambda^{(2a+1)/(4D)}}\bigg\}=\lim_{n\to\infty}\P\Big\{\sqrt{n}\lambda^{(2a+1)/(4D)}(\hat d_\infty-d_\infty)>\mathcal Q_{1-\alpha}(T_\EE)\Big\}=\alpha\,.
\end{align*}
Under the alternative hypothesis $H_1$ in \eqref{rele}, i.e., $d_\infty>\Delta$,
\begin{align*}
&\lim_{n\to\infty}\P\bigg\{\hat d_\infty>\Delta+\frac{\mathcal Q_{1-\alpha}(T_{\EE})}{\sqrt{n}\lambda^{(2a+1)/(4D)}}\bigg\}\\
&=\lim_{n\to\infty}\P\Big\{\sqrt{n}\lambda^{(2a+1)/(4D)}(\hat d_\infty-d_\infty)>\sqrt{n}\lambda^{(2a+1)/(4D)}(\Delta-d_\infty)+\mathcal Q_{1-\alpha}(T_\EE)\Big\}=1\,.
\end{align*}

\subsection{Proof of Theorem~\ref{thm:rt:boot}}\label{app:proof:rt}

We notice that, by the continuous mapping theorem, Theorem~\ref{thm:process} and Lemma~B.3 in \cite{dette2020aos}, conditional on the data, the bootstrap statistic $\hat T_{\EE,n,q}^*$ in \eqref{hattnq} converges to $T_\EE$ in \eqref{te}, the same limit as $\sqrt{n}\lambda^{(2a+1)/(4D)}(\hat d_\infty-d_\infty)$. Hence, the assertion in Theorem~\ref{thm:rt:boot} follows from arguments similar to the ones in the proof of \eqref{eq:rele} in Section~\ref{app:thm:extreme}.

\subsection{Proof of Theorem~\ref{thm:prediction:band}}\label{app:thm:prediction:band}

By \eqref{cz}, we obtain, for the kernel $C_{Z,x_0}$ in \eqref{czx0},
\begin{align}\label{converge}
C_{Z,{x_0}}(t_1,t_2)&=\lambda^{(2a+1)/(2D)}\sum_{k,\ell}\frac{1}{(1+\lambda\rho_{k\ell})^2}\int_0^1\phi_{k\ell}(s_1,t_1)x_0(s_1)ds_1\notag\\
&\hspace{2cm}\times\int_0^1\phi_{k\ell}(s_2,t_2)x_0(s_2)ds_2+o(1)\,,
\end{align}
since it follows from the dominated convergence theorem and the Cauchy-Schwarz inequality, uniformly in $n\geq 1$,
\begin{align*}
&\lambda^{(2a+1)/(2D)}\bigg|\sum_{k,\ell}\frac{1}{(1+\lambda\rho_{k\ell})^2}\int_0^1\phi_{k\ell}(s_1,t_1)x_0(s_1)ds_1\int_0^1\phi_{k\ell}(s_2,t_2)x_0(s_2)ds_2\bigg|\\
&\leq\lambda^{(2a+1)/(2D)}\Vert x_0\Vert_{L^2}^2\sum_{k,\ell}\frac{\Vert\phi_{k\ell}\Vert_\infty^2}{(1+\lambda\rho_{k\ell})^2}\leq c\lambda^{(2a+1)/(2D)}\Vert x_0\Vert_{L^2}^2\sum_{k,\ell}\frac{(k\ell)^{2a}}{\{1+\lambda(k\ell)^{2D}\}^2}\leq c\,.
\end{align*}

By the definition of $\tau$ in \eqref{tau},
\begin{align*}
\hat\mu_{x_0}(t)-\mu_{x_0}(t)&=\int_0^1\{\hat\beta_{n}(s,t)-\beta_0(s,t)\}x_0(s)ds=\big\l\hat\beta_{n}-\beta_0,\tau(x_0\otimes\delta_t)\big\r_K\,,
\end{align*}
where $\delta_t$ is the delta function at $t\in[0,1]$. We have, 
\begin{align*}
\sqrt{n}\lambda^{(2a+1)/(4D)}\{\hat\mu_{x_0}(t)-\mu_{x_0}(t)\}=I_{1,n}(t)+I_{2,n}(t)+I_{3,n}(t)\,,
\end{align*}
where
\begin{align*}
&I_{1,n}(t)=\sqrt{n}\lambda^{(2a+1)/(4D)}\big\l\hat\beta_{n}-\beta_0+S_{n,\lambda}\beta_0,\tau(x_0\otimes\delta_t)\big\r_K\,,\\
&I_{2,n}(t)=\sqrt{n}\lambda^{(2a+1)/(4D)}\big\l W_\lambda\beta_0,\tau(x_0\otimes\delta_t)\big\r_K\,,\\
&I_{3,n}(t)=n^{-1/2}\lambda^{(2a+1)/(4D)}\sum_{i=1}^n\big\l\tau(X_i\otimes\e_i),\tau(x_0\otimes\delta_t)\big\r_K\,.
\end{align*}

Observing \eqref{tau}, Assumption~\ref{a201} and Lemma~\ref{lem:sum}, and the Cauchy-Schwarz inequality, it follows
\begin{align}\label{taudelta}
\sup_{t\in[0,1]}\Vert\tau(x_0\otimes\delta_t)\Vert_K^2&=\sup_{t\in[0,1]}\sum_{k,\ell}\frac{\l x_0\otimes\delta_t,\phi_{k\ell}\r_{L^2(T)}^2}{1+\lambda\rho_{k\ell}}\leq\sup_{t\in[0,1]}\Vert x_0\otimes\delta_t\Vert_{L^2}^2 \times\sum_{k,\ell}\frac{\Vert\phi_{k\ell}\Vert_{\infty}^2}{1+\lambda\rho_{k\ell}}\notag\\
&\leq c\sum_{k,\ell}\frac{(k\ell)^{2a}}{1+\lambda(k\ell)^{2D}}\leq c\,\lambda^{-(2a+1)/(2D)}\,.
\end{align}
Hence, by Theorem~\ref{thm:bahadur} and the Cauchy-Schwarz inequality,
\begin{align}\label{i1n}
\sup_{t\in[0,1]}|I_{1,n}(t)|&\leq\sqrt{n}\lambda^{(2a+1)/(4D)} \Vert\hat\beta_{n}-\beta_0+S_{n,\lambda}\beta_0\Vert_K\times\sup_{t\in[0,1]}\Vert\tau(x_0\otimes\delta_t)\Vert_K\notag\\
&=O_p(\sqrt{n}v_n)=o_p(1)\,.
\end{align}
where $v_n$ is defined in \eqref{vn}.

In addition, since by assumption, $\sum_{k,\ell}\rho_{k\ell}^2V^2(\beta_0,\phi_{k\ell})<\infty$, we have
\begin{align*}
\Vert W_\lambda\beta_0\Vert_K^2=\lambda^2\,\bigg\Vert\sum_{k,\ell}\frac{\rho_{k\ell}\,V(\beta_0,\phi_{k\ell})}{1+\lambda\rho_{k\ell}}\,\phi_{k\ell}\bigg\Vert_K^2=\lambda^2\sum_{k,\ell}\frac{\rho_{k\ell}^2V^2(\beta_0,\phi_{k\ell})}{1+\lambda\rho_{k\ell}}\leq c\lambda^2\,.
\end{align*}
Therefore, we obtain, for the term $I_{2,n}$,
\begin{align}\label{i2n}
\sup_{t\in[0,1]}|I_{2,n}(t)|&\leq \sqrt{n}\lambda^{(2a+1)/(4D)}\Vert W_\lambda\beta_0\Vert_K\times\sup_{t\in[0,1]}\Vert\tau(x_0\otimes\delta_t)\Vert_K\notag\\
&\leq c\sqrt{n}\Vert W_\lambda\beta_0\Vert_K=O(\sqrt{n}\lambda)=o(1)\,.
\end{align}

Finally, using the representation $\tau(X\otimes\e)=\sum_{k,\ell}(1+\lambda\rho_{k\ell})^{-1}\l X\otimes\e,\phi_{k\ell}\r_{L^2}\phi_{k\ell}$ (see equation \eqref{tk} in the proof of Lemma~\ref{lem:etk} in Section~\ref{app:aux}), we obtain
\begin{align}\label{i3n}
I_{3,n}(t)&=n^{-1/2}\lambda^{(2a+1)/(4D)}\sum_{i=1}^n\big\l\tau(X_i\otimes\e_i),\tau(x_0\otimes\delta_t)\big\r_K\notag\\
&=n^{-1/2}\lambda^{(2a+1)/(4D)}\sum_{i=1}^n\sum_{k,\ell}\frac{1}{1+\lambda\rho_{k\ell}}\l X_i\otimes\e_i,\phi_{k\ell}\r_{L^2}\l\phi_{k\ell},\tau(x_0\otimes\delta_t)\r_K\notag\\
&=n^{-1/2}\lambda^{(2a+1)/(4D)}\sum_{i=1}^n\sum_{k,\ell}\frac{1}{1+\lambda\rho_{k\ell}}\l X_i\otimes\e_i,\phi_{k\ell}\r_{L^2}\l\phi_{k\ell},x_0\otimes\delta_t\r_{L^2}\notag\\
&=n^{-1/2}\lambda^{(2a+1)/(4D)}\sum_{i=1}^n\sum_{k,\ell}\frac{\l X_i\otimes\e_i,\phi_{k\ell}\r_{L^2}}{1+\lambda\rho_{k\ell}}\int_0^1\phi_{k\ell}(s,t)x_0(s)ds\,.
\end{align}
For $1\leq i\leq n$, let
\begin{align*}
\mathfrak U_i(t)&=n^{-1/2}\lambda^{(2a+1)/(4D)}\big\l\tau(X_i\otimes\e_i),\tau(x_0\otimes\delta_t)\big\r_K\\
&=n^{-1/2}\lambda^{(2a+1)/(4D)}\sum_{k,\ell}\frac{\l X_i\otimes\e_i,\phi_{k\ell}\r_{L^2}}{1+\lambda\rho_{k\ell}}\int_0^1\phi_{k\ell}(s,t)x_0(s)ds\,,
\end{align*}
so that $I_{3,n}(t)=\sum_{i=1}^n\mathfrak U_i(t)$. Since $\E\{\l X_i\otimes\e_i,\phi_{k\ell}\r_{L^2}\}=0$ for $k,\ell\geq1$, we have $\E\{\mathfrak U_i(t)\}=0$, and observing \eqref{core}, it follows that
\begin{align}\label{cov}
&\cov\{\mathfrak U_i(t_1),\mathfrak U_i(t_2)\}\notag\\
&=\lambda^{(2a+1)/(2D)}\sum_{k,\ell,k',\ell'}\frac{1}{(1+\lambda\rho_{k\ell})(1+\lambda\rho_{k'\ell'})}\E\Big\{\l X\otimes\e,\phi_{k\ell}\r_{L^2}\,\l X\otimes\e,\phi_{k'\ell'}\r_{L^2}\Big\}\notag\\
&\qquad\times\int\phi_{k\ell}(s_1,t_1)x_0(s_1)ds_1\times\int\phi_{k\ell}(s_2,t_2)x_0(s_2)ds_2\notag\\
&=\lambda^{(2a+1)/(2D)}\sum_{k,\ell}\frac{1}{(1+\lambda\rho_{k\ell})^2}\int\phi_{k\ell}(s_1,t_1)x_0(s_1)ds_1\times\int\phi_{k\ell}(s_2,t_2)x_0(s_2)ds_2\notag\\
&=C_{Z,x_0}(t_1,t_2)+o(1)\,,\qquad\text{as } n\to\infty\,,
\end{align}
where we used \eqref{converge} in the last step.

In order to prove the weak convergence of the finite-dimensional marginal distributions of $\hat\mu_{x_0}$, by the Cram\'er-Wold device, we shall show that, for any $q\in\mathbb{N}$, $(c_1,\ldots,c_q)^{\rm T}\in\mathbb{R}^q$ and $t_1,\ldots,t_q\in[0,1]$,
\begin{align}\label{i3nvalid}
\sum_{j=1}^qc_j\hat\mu_{x_0}(t_j)\overset{d.}{\longrightarrow}\sum_{j=1}^qc_j Z_{x_0}(t_j)\,.
\end{align}
In view of \eqref{i1n} and \eqref{i2n}, we deduce that
\begin{align*}
\sum_{j=1}^qc_j\hat\mu_{x_0}(t_j)&=\sum_{j=1}^qc_jI_{3,n}(t_j)+\sum_{j=1}^qc_j\{I_{n,1}(t_j)+I_{n,2}(t_j)\}=\sum_{i=1}^n\sum_{j=1}^qc_j\mathfrak U_i(t_j)+o_p(1)\,.
\end{align*}
Observing \eqref{cov}, we have $$
\var\bigg\{\sum_{i=1}^n\sum_{j=1}^qc_j\mathfrak U_i(t_j)\bigg\}=\sum_{j_1,j_2=1}^qc_{j_1}c_{j_2}C_{Z,x_0}(t_{j_1},t_{j_2})+o(1)
$$ 
as $n\to\infty$. If $\sum_{j_1,j_2=1}^qc_{j_1}c_{j_2}C_{Z,x_0}(t_{j_1},t_{j_2})=0$, $\sum_{j=1}^qc_j Z_{x_0}(t_j)$ has a degenerate distribution with a point mass at zero, so that \eqref{i3nvalid} is a consequence of the Markov's inequality. If $\sum_{j_1,j_2=1}^qc_{j_1}c_{j_2}C_{Z,x_0}(t_{j_1},t_{j_2})\neq0$, we have $\var\big\{\sum_{i=1}^n\sum_{j=1}^qc_j\mathfrak U_i(t_j)\big\}=\sum_{j_1,j_2=1}^qc_{j_1}c_{j_2}C_{Z,x_0}(t_{j_1},t_{j_2})+o(1)=\var\big\{\sum_{j=1}^qc_j Z_{x_0}(t_j)\big\}+o(1)$. To prove \eqref{i3nvalid}, we shall check that the triangular array of random variables $\{\sum_{j=1}^qc_j\mathfrak U_i(t_j)\}_{i=1}^n$ satisfies Lindeberg's condition. Let $\Sigma_q=\sum_{j=1}^q|c_j|$. We have $\Sigma_q>0$ since $\Sigma_q=0$ indicates $\sum_{j_1,j_2=1}^qc_{j_1}c_{j_2}C_{Z,x_0}(t_{j_1},t_{j_2})=0$. For any $e>0$, by the Cauchy-Schwarz inequality,
\begin{align}\label{temp2}
&\sum_{i=1}^n\E\bigg[\bigg|\sum_{j=1}^qc_j\mathfrak U_i(t_j)\bigg|^2\times\one\bigg\{\bigg|\sum_{j=1}^qc_j\mathfrak U_i(t_j)\bigg|>e\bigg\}\bigg]\notag\\
&=\lambda^{(2a+1)/(2D)}\,\E\bigg[\bigg|\sum_{j=1}^qc_j\big\l\tau(X_i\otimes\e_i),\tau(x_0\otimes\delta_{t_j})\big\r_K\bigg|^2\times\one\bigg\{\bigg|\sum_{j=1}^qc_j\mathfrak U_i(t_j)\bigg|>e\bigg\}\bigg]\notag\\
&\leq c\,\lambda^{(2a+1)/(2D)}\,\sup_{t\in[0,1]}\E\Big\{\big|\big\l\tau(X_i\otimes\e_i),\tau(x_0\otimes\delta_{t})\big\r_K\big|^4\Big\}^{\frac{1}{2}}\times\P\bigg\{\bigg|\sum_{j=1}^qc_j\mathfrak U_i(t_j)\bigg|>e\bigg\}^{\frac{1}{2}}\,.
\end{align}

Using \eqref{taudelta} and Lemma~\ref{lem:etk} in Section~\ref{app:aux:lem}, it follows that
\begin{align}\label{4m}
&\sup_{t\in[0,1]}\E\big|\big\l\tau(X_i\otimes\e_i),\tau(x_0\otimes\delta_{t})\big\r_K\big|^4\notag\\
&\leq\E\Vert\tau(X_i\otimes\e_i)\Vert_K^4\times\sup_{t\in[0,1]}\Vert\tau(x_0\otimes\delta_{t})\Vert_K^4\leq c\,\lambda^{-(2a+2)/D}\,.
\end{align}
In addition, for some $c_0>0$,
\begin{align*}
&\sup_{t\in[0,1]}\bigg|\sum_{k,\ell}\frac{\l X \otimes\e ,\phi_{k\ell}\r_{L^2}}{1+\lambda\rho_{k\ell}}\int_0^1\phi_{k\ell}(s,t)x_0(s)ds\bigg|\\
&\leq \Vert x_0\Vert_{L^2}\times\Vert X \otimes\e\Vert_{L^2}\times\sum_{k,\ell}\frac{\Vert\phi_{k\ell}\Vert_\infty^2}{1+\lambda\rho_{k\ell}}\leq c_0\lambda^{-(2a+1)/(2D)}\Vert X\Vert_{L^2}\,\Vert\e\Vert_{L^2}\,.
\end{align*}
Hence, by arguments similar to the ones used in \eqref{11}, by taking $c_1>a(c_\e D)^{-1}$, for $c_\e>0$ in Assumption~\ref{a:x}, we find
\begin{align*}
&\P\bigg\{\bigg|\sum_{j=1}^qc_j\mathfrak U_i(t_j)\bigg|>e\bigg\}\\
&=\P\bigg\{\sup_{t\in[0,1]}\bigg|\sum_{k,\ell}\frac{\l X \otimes\e ,\phi_{k\ell}\r_{L^2}}{1+\lambda\rho_{k\ell}}\int_0^1\phi_{k\ell}(s,t)x_0(s)ds\bigg|>e\Sigma_q^{-1}\sqrt{n}\lambda^{-(2a+1)/(4D)}\bigg\}\\
&\leq\P\Big\{\Vert X\Vert_{L^2}\,\Vert\e\Vert_{L^2}>ec_0^{-1}\Sigma_q^{-1}\sqrt{n}\lambda^{(2a+1)/(4D)}\Big\}\\
&\leq\P\Big\{\Vert X\Vert_{L^2}\geq ec_1^{-1}c_0^{-1}\Sigma_q^{-1}\sqrt{n}\lambda^{(2a+1)/(4D)}/{\log(\lambda^{-1})}\Big\}+\P\Big\{\Vert\e\Vert_{L^2}\geq c_1{\log(\lambda^{-1})}\Big\}\notag\\
&\leq\exp\big\{-c_Xec_1^{-1}c_0^{-1} \Sigma_q^{-1}\sqrt{n}\lambda^{(2a+1)/(4D)}/\log(\lambda^{-1})\big\}\E\{\exp(c_X\Vert X\Vert_{L^2})\}\notag\\
&\qquad+\lambda^{c_1c_\e}\E\{\exp(c_\e\Vert\e\Vert_{L^2})\}\notag\\
&=O\big(\lambda^{c_Xec_1^{-1}c_0^{-1}\Sigma_q^{-1} \{\sqrt{n}\lambda^{(2a+1)/(4D)}/\log^2(\lambda^{-1})\}}\big)+O(\lambda^{c_1c_\e})=o(\lambda^{1/D})\,,
\end{align*}
where we used the assumption $\sqrt{n}\lambda^{(2a+1)/(4D)}/\log^2(\lambda^{-1})\to\infty$ in the last step. Therefore, combining the above result with \eqref{temp2} and \eqref{4m} yields
\begin{align*}
&\sum_{i=1}^n\E\bigg[\bigg|\sum_{j=1}^qc_j\mathfrak U_i(t_j)\bigg|^2\times\one\bigg\{\bigg|\sum_{j=1}^qc_j\mathfrak U_i(t_j)\bigg|>e\bigg\}\bigg]\leq c\,\lambda^{(2a+1)/(2D)}\,\lambda^{-(a+1)/D} \, o(\lambda^{1/(2D)})=o(1)\,.
\end{align*}
By Lindeberg's CLT,
\begin{align*}
&\sum_{j=1}^qc_j\hat{\mu}_{x_0}(t_j)=\sum_{i=1}^n\sum_{j=1}^qc_j\mathfrak U_i(t_j)+o_p(1)\converged N\bigg(0,\sum_{j_1,j_2=1}^qc_{j_1}c_{j_2}C_{Z,x_0}(t_{j_1},t_{j_2})\bigg)\overset{d.}{=}\sum_{j=1}^qc_j Z_{x_0}(t_j)\,.
\end{align*}

Next, we prove the asymptotic tightness of $\hat\mu_{x_0} $. For any $t_1,t_2\in[0,1]$, since $\E\{I_{3,n}(t_1)-I_{3,n}(t_2)\}=0$, by \eqref{core} and the Cauchy-Schwarz inequality, we find
\begin{align*}
&\E|I_{3,n}(t_1)-I_{3,n}(t_2)|^2=n\,\E|\mathfrak U_i(t_1)-\mathfrak U_i(t_2)|^2\\
&=\lambda^{(2a+1)/(2D)}\E\bigg|\sum_{k,\ell}\frac{\l X_i\otimes\e_i,\phi_{k\ell}\r_{L^2}}{1+\lambda\rho_{k\ell}}\int_0^1\{\phi_{k\ell}(s,t_1)-\phi_{k\ell}(s,t_2)\}x_0(s)ds\bigg|^2\\
&=\lambda^{(2a+1)/(2D)}\sum_{k,\ell}\frac{1}{(1+\lambda\rho_{k\ell})^2}\bigg|\int_0^1\{\phi_{k\ell}(s,t_1)-\phi_{k\ell}(s,t_2)\}x_0(s)ds\bigg|^2\\
&\leq\lambda^{(2a+1)/(2D)}\Vert x_0\Vert_{L^2}^2\sum_{k,\ell}\frac{1}{(1+\lambda\rho_{k\ell})^2}\int_0^1|\phi_{k\ell}(s,t_1)-\phi_{k\ell}(s,t_2)|^2ds\,.
\end{align*}
By the assumption in \eqref{holder}, we have, for constants $c_0$, $\vartheta$ and $b$ specified in Theorem~\ref{thm:process},
\begin{align*}
&\E|I_{3,n}(t_1)-I_{3,n}(t_2)|^2\leq c_0\lambda^{(a-b)/D}\Vert x_0\Vert_{L^2}^2|t_1-t_2|^{2\vartheta}\leq c\,\lambda^{(a-b)/D}|t_1-t_2|^{2\vartheta}\,.
\end{align*}
Moreover, in view of \eqref{i3n}, by the Cauchy-Schwarz inequality, Assumption~\ref{a201}, we deduce that
\begin{align*}
&\sup_{t\in[0,1]}|\mathfrak U_i(t)|\\
&\leq n^{-1/2}\lambda^{(2a+1)/(4D)}\sum_{k,\ell}\frac{|\l X_i\otimes\e_i,\phi_{k\ell}\r_{L^2}|}{1+\lambda\rho_{k\ell}}\times\sup_{t\in[0,1]}\bigg|\int_0^1\phi_{k\ell}(s,t)x_0(s)ds\bigg|\\
&\leq n^{-1/2}\lambda^{(2a+1)/(4D)}\Vert X_i\otimes\e_i\Vert_{L^2}\Vert x_0\Vert_{L^2}\sum_{k,\ell}\frac{\Vert\phi_{k\ell}\Vert_{L^2}}{1+\lambda\rho_{k\ell}}\times\sup_{t\in[0,1]}\bigg\{\int_0^1|\phi_{k\ell}(s,t)|^2ds\bigg\}^{1/2}\\
&\leq c\,n^{-1/2}\lambda^{(2a+1)/(4D)}(\log n)^2\sum_{k,\ell}\frac{\Vert\phi_{k\ell}\Vert_{\infty}^2}{1+\lambda\rho_{k\ell}}\\
&\leq c\,n^{-1/2}\lambda^{(2a+1)/(4D)}(\log n)^2\sum_{k,\ell}\frac{(k\ell)^{2a}}{1+\lambda(k\ell)^{2D}}\\
&\leq c\,n^{-1/2}\lambda^{-(2a+1)/(4D)}(\log n)^2\,,
\end{align*}
almost surely, where we used Lemma~\ref{lem:sum} and the fact that $\Vert X\otimes\e\Vert_{L^2}\leq c(\log n)^2$ almost surely from Section~\ref{app:thm:process}. Therefore, by arguments similar to the ones used in the proof of Theorem~\ref{thm:process}, we find that, there exists a semi-metric $d$ on $[0,1]^2$ such that, for any $e>0$,
\begin{align*}
\lim_{\delta\downarrow0}\,\limsup_{n\to\infty}\,\P\bigg\{\sup_{d\{(s_1,t_1),(s_2,t_2)\}\leq\delta}|I_{3,n}(t_1)-I_{3,n}(t_2)|>e\bigg\}=0\,.
\end{align*}
Combining the above result with \eqref{i1n} and \eqref{i2n},
\begin{align*}
\lim_{\delta\downarrow0}\,\limsup_{n\to\infty}\,\P\bigg\{\sup_{d\{(s_1,t_1),(s_2,t_2)\}\leq\delta}|\hat\mu_{x_0}(t_1)-\hat\mu_{x_0}(t_2)|>e\bigg\}=0\,.
\end{align*}
Therefore, by applying Theorems~1.5.4 and 1.5.7 in \cite{vaart1996}, we have shown that $\sqrt{n}\lambda^{(2a+1)/(4D)}\{\hat\mu_{x_0}-\mu_{x_0}\}\weakconverge Z_{x_0}$ in $C([0,1]^2)$. By the continuous mapping theorem, $\sqrt{n}\lambda^{(2a+1)/(4D)}\sup_{t\in[0,1]}|\hat\mu_{x_0}(t)-\mu_{x_0}(t)|\converged\max_{t\in[0,1]}|Z_{x_0}(t)|$.

\newpage 

\section{Auxiliary lemmas and technical details}\label{app:aux}

\subsection{Auxiliary lemmas for the proofs in Section~\ref{app:proof}}\label{app:aux:lem}

\begin{lemma}\label{lem:change}
For any $\beta_1,\beta_2\in\H$ and $x\in L^2([0,1])$, for $\tau$ defined in \eqref{tau0},
\begin{align*}
\left\l\tau\bigg[x\otimes\bigg\{\int_0^1\beta_1(s,\cdot)x(s)ds\bigg\}\bigg],\beta_2\right\r_K&=\left\l\tau\bigg[x\otimes\bigg\{\int_0^1\beta_2(s,\cdot)x(s)ds\bigg\}\bigg],\beta_1\right\r_K\\
&=\int_{[0,1]^3}\beta_1(s_1,t)\,\beta_2(s_2,t)\,x(s_1)\,x(s_2)\,ds_1\,ds_2\,dt\,.
\end{align*}
\end{lemma}
\begin{proof}
By Fubini's theorem and \eqref{tau}, direct calculation yields
\begin{align*}
&\left\l\tau\bigg[x\otimes\bigg\{\int_0^1\beta_1(s,\cdot)x(s)ds\bigg\}\bigg],\beta_2\right\r_K\\
&=\int_0^1\int_0^1x(s_2)\bigg\{\int_0^1\beta_1(s_1,t)\,x(s_1)\,ds_1\bigg\}\beta_2(s_2,t)\,ds_2\,dt\\
&=\int_{[0,1]^3}\beta_1(s_1,t)\,\beta_2(s_2,t)\,x(s_1)\,x(s_2)\,ds_1\,ds_2\,dt\\
&=\int_0^1\int_0^1x(s_1)\bigg\{\int_0^1\beta_2(s_2,t)\,x(s_2)\,ds_2\bigg\}\beta_1(s_1,t)\,ds_1\,dt\\
&=\left\l\tau\bigg[x\otimes\bigg\{\int_0^1\beta_2(s,\cdot)\,x(s)\,ds\bigg\}\bigg],\beta_1\right\r_K\,.
\end{align*}
\end{proof}

\begin{lemma}\label{lem:sum}
For $D$ in Assumption~\ref{a201}, for any $0\leq\nu<D-1/2$, for $0<\lambda<1$ and for either $r=1$ or $2$, there exist constants $c_1,c_2>0$ independent of $\lambda$ such that
\begin{align}\label{sum1}
&c_1\,\lambda^{-(2\nu+1)/(2D)}\leq\sum_{k,\ell\geq 1}\frac{(k\ell)^{2\nu}}{\{1+\lambda (k\ell)^{2D}\}^r}\leq c_2\,\lambda^{-(2\nu+1)/(2D)}\,.
\end{align}

\end{lemma}

\begin{proof}
Using change of variables, we first notice that
\begin{align*}
\int_0^{ \infty}\int_0^{ \infty}\frac{x^{2\nu}y^{2\nu}}{(1+\lambda\, x^{2D}y^{2D})^r}\,dx\,dy=\lambda^{-(2\nu+1)/(2D)}\int_{0}^{ \infty}\int_{0}^{ \infty}\frac{x^{2\nu}y^{2\nu}}{(1+x^{2D}y^{2D})^r}\,dx\,dy\,.
\end{align*}
Since $2D-2\nu>1$, we have $\int_{0}^{ \infty}\int_{0}^{ \infty}x^{2\nu}y^{2\nu}(1+x^{2D}y^{2D})^{-r}dxdy< \infty$.

Let $\mathfrak m_\lambda=\lambda^{-1/(2D)}\{\nu/(rD-\nu)\}^{1/(2D)}$. Note that the function $x^{2\nu}/(1+\lambda x^{2D})$ is increasing on $(0,\mathfrak m_\lambda)$, and is decreasing on $(\mathfrak m_\lambda, \infty)$. For any real value $x\in\mathbb{R}$, let $\lceil x\rceil$ denote the smallest integer greater than or equal to $x$, and let $\lfloor x\rfloor$ denote the largest integer smaller than or equal to $x$. Let $\one\{\cdot\}$ denote the indicator function. For the left-hand side of the inequality in \eqref{sum1}, note that
\begin{align*}
&\sum_{k,\ell\geq 1}\frac{(k\ell)^{2\nu}}{\{1+\lambda (k\ell)^{2D}\}^r}\\
&=\sum_{k,\ell\geq 1}\frac{(k\ell)^{2\nu}}{\{1+\lambda (k\ell)^{2D}\}^r}\big[\one\{k\ell< \lfloor \mathfrak m_\lambda\rfloor\}+\one\{k\ell> \lceil \mathfrak m_\lambda\rceil\}\big]\\
&\qquad+\sum_{k,\ell\geq 1}\frac{(k\ell)^{2\nu}}{\{1+\lambda (k\ell)^{2D}\}^r}\one\{\lfloor \mathfrak m_\lambda\rfloor\leq k\ell\leq \lceil \mathfrak m_\lambda\rceil\}\\
&\leq \int_0^{ \infty}\int_0^{ \infty}\frac{(xy)^{2\nu}}{\{1+\lambda (xy)^{2D}\}^r}dxdy+\sum_{k,\ell\geq 1}\frac{(k\ell)^{2\nu}}{\{1+\lambda (k\ell)^{2D}\}^r}\one\{\lfloor \mathfrak m_\lambda\rfloor\leq k\ell\leq \lceil \mathfrak m_\lambda\rceil\}\\
&\leq \int_0^{ \infty}\int_0^{ \infty}\frac{(xy)^{2\nu}}{\{1+\lambda (xy)^{2D}\}^r}dxdy+\frac{\lceil \mathfrak m_\lambda\rceil^{2\nu}}{(1+\lambda \lfloor \mathfrak m_\lambda\rfloor^{2D})^r}\sum_{k,\ell\geq 1}\one\{\lfloor \mathfrak m_\lambda\rfloor\leq 
k\ell\leq \lceil \mathfrak m_\lambda\rceil\}\\
&\leq \int_0^{ \infty}\int_0^{ \infty}\frac{(xy)^{2\nu}}{\{1+\lambda (xy)^{2D}\}^r}dxdy+\frac{2\lceil \mathfrak m_\lambda\rceil^{2\nu+1}}{(1+\lambda \lfloor \mathfrak m_\lambda\rfloor^{2D})^r}\\
&\leq \int_0^{ \infty}\int_0^{ \infty}\frac{(xy)^{2\nu}}{\{1+\lambda (xy)^{2D}\}^r}dxdy+c\,\lambda^{-(2\nu+1)/(2D)}\,,
\end{align*}
for some $c>0$ that does not depend on $\lambda$, where we used the fact that $\sum_{k,\ell\geq 1}\one\{\lfloor \mathfrak m_\lambda\rfloor\leq 
k\ell\leq \lceil \mathfrak m_\lambda\rceil\}\leq 2\lceil \mathfrak m_\lambda\rceil$. This proves the right-hand side of \eqref{sum1}.

For the left-hand side of the inequality in \eqref{sum1}, by change of variables, we find
\begin{align*}
&\sum_{k,\ell\geq 1}\frac{(k\ell)^{2\nu}}{\{1+\lambda (k\ell)^{2D}\}^r}\\
&=\sum_{k,\ell\geq 1}\frac{(k\ell)^{2\nu}}{\{1+\lambda (k\ell)^{2D}\}^r}\one\{k\ell\leq \lfloor \mathfrak m_\lambda\rfloor\}+\sum_{k,\ell\geq 1}\frac{(k\ell)^{2\nu}}{\{1+\lambda (k\ell)^{2D}\}^r}\one\{k\ell\geq \lceil \mathfrak m_\lambda\rceil\}\\
&\geq \int_0^{ \infty}\int_0^{ \infty}\frac{(xy)^{2\nu}}{\{1+\lambda (xy)^{2D}\}^r}\big[\one\{xy\leq \lfloor \mathfrak m_\lambda\rfloor\}+\one\{xy\geq \lceil \mathfrak m_\lambda\rceil\}\big]dxdy\\
&=\lambda^{-(2\nu+1)/(2D)}\int_0^{ \infty}\int_0^{ \infty}\frac{(xy)^{2\nu}}{\{1+(xy)^{2D}\}^r}\big[\one\{xy\leq \lambda^{1/(2D)}\lfloor \mathfrak m_\lambda\rfloor\}+\one\{xy\geq \lambda^{1/(2D)}\lceil \mathfrak m_\lambda\rceil\}\big]dxdy\,.
\end{align*}
Note that there exists constants $0<\tilde c_1<\tilde c_2$ independent of $\lambda$ such that $0<\tilde c_1\leq \lambda^{1/(2D)}\lfloor \mathfrak m_\lambda\rfloor$ and $\lambda^{1/(2D)}\lceil \mathfrak m_\lambda\rceil \leq \tilde c_2< \infty$. We therefore deduce from the above equation that
\begin{align*}
&\sum_{k,\ell\geq 1}\frac{(k\ell)^{2\nu}}{\{1+\lambda (k\ell)^{2D}\}^r}\geq \lambda^{-(2\nu+1)/(2D)}\int_0^{ \infty}\int_0^{ \infty}\frac{(xy)^{2\nu}}{\{1+(xy)^{2D}\}^r}\big[\one\{xy>\tilde c_2\}+\one\{xy\leq \tilde c_1\}\big]dxdy\,,
\end{align*}
which completes the proof of \eqref{sum1}.

\end{proof}

\begin{lemma}\label{lem:norm}
Under Assumptions~\ref{a1}--\ref{a:x}, for $\beta\in\H$, we have, for some constant $c_K>0$,
\begin{align*}
\Vert\beta\Vert_{L^2}\leq c_K\lambda^{-(2a+1)/(4D)}\, \Vert\beta\Vert_K\,.
\end{align*}
\end{lemma}

\begin{proof}
For any $\beta\in\H$ and the reproducing kernel $K$ in \eqref{kst2}, we have $\beta(s,t)=\l\beta,K_{(s,t)}\r_K$, from which we deduce that $|\beta(s,t)|\leq \Vert\beta\Vert_K\,\Vert K_{(s,t)}\Vert_K$, so that
\begin{align*}
\Vert\beta\Vert_{L^2}^2\leq \Vert\beta\Vert_K^2\int_0^1\int_0^1\Vert K_{(s,t)}\Vert_K^2\,ds\,dt\,,
\end{align*}
which yields
\begin{align}\label{kst}
\Vert K_{(s,t)}\Vert_K^2=\sum_{k,\ell}\frac{\l K_{(s,t)},\phi_{k\ell}\r_K^2}{1+\lambda\rho_{k\ell}}=\sum_{k,\ell}\frac{\phi_{k\ell}^2(s,t)}{1+\lambda\rho_{k\ell}}\,.
\end{align}
Therefore, we find
\begin{align*}
\int_0^1\int_0^1\Vert K_{(s,t)}\Vert_K^2\,ds\,dt&=\sum_{k,\ell}\frac{1}{1+\lambda\rho_{k\ell}}\Vert\phi_{k\ell}\Vert^2_{L^2}\leq c\sum_{k,\ell}\frac{k^{2a}\ell^{2a}}{1+\lambda\, k^{2D}\ell^{2D}}\leq c\lambda^{-(2a+1)/(2D)}\,,
\end{align*}
where we used the assumption that $D>a+1/2$ in Assumption~\ref{a201} and Lemma~\ref{lem:sum} in the last step.

\end{proof}

\begin{lemma}\label{lem:etau2}
Under Assumptions~\ref{a1}--\ref{a:x}, for any $\beta\in\H$,
\begin{align}
&\bigg\Vert\tau\bigg[x\otimes\bigg\{\int_0^1\beta(s,\cdot)x(s)ds\bigg\}\bigg]\bigg\Vert_K\leq c_1\,\lambda^{-(2a+1)/(2D)}\,\Vert\beta\Vert_K\times\Vert x\Vert_{L^2}^2\,,\label{norm}\\
&\E\bigg\Vert\tau\bigg[X_i\otimes\bigg\{\int_0^1\beta(s,\cdot)X_i(s)ds\bigg\}\bigg]\bigg\Vert_K^2\leq c_2\,\lambda^{-1/D}\,\Vert\beta\Vert_K^2\,.\label{enorm}
\end{align}
Here, $c_1,c_2>0$ are absolute constants.
\end{lemma}

\begin{proof}
By the Cauchy-Schwarz inequality,
\begin{align*}
&\bigg\Vert\tau\bigg[x\otimes\bigg\{\int_0^1\beta(s,\cdot)x(s)ds\bigg\}\bigg]\bigg\Vert_K\notag\\
&=\sup_{\Vert\gamma\Vert_K=1}\bigg|\left\l\tau\bigg[x\otimes\bigg\{\int_0^1\beta(s,\cdot)x(s)ds\bigg\}\bigg],\gamma\right\r_K\bigg|\notag\\
&=\sup_{\Vert\gamma\Vert_K=1}\bigg|\int_{[0,1]^3}\beta(s_1,t)\,\gamma(s_2,t)\,x(s_1)\,x(s_2)\,ds_1\,ds_2\,dt\bigg|\notag\\
&\leq\sup_{\Vert\gamma\Vert_K=1}\bigg\{\int_{[0,1]^3}\gamma^2(s_1,t)\,x^2(s_2)\,ds_1\,ds_2\,dt\bigg\}^{1/2}\times\bigg\{\int_{[0,1]^3}\beta^2(s_2,t)\,x^2(s_1)\,ds_1\,ds_2\,dt\bigg\}^{1/2}\notag\\
&=\sup_{\Vert\gamma\Vert_K=1}\Vert\gamma\Vert_{L^2}\times\Vert\beta\Vert_{L^2}\times\Vert x\Vert_{L^2}^2\\
&\leq c_K\,\lambda^{(2a+1)/(4D)}\,\Vert\beta\Vert_{L^2}\times\Vert x\Vert_{L^2}^2\notag\\
&\leq c_K^2\,\lambda^{(2a+1)/(2D)}\,\Vert\beta\Vert_{K}\times\Vert x\Vert_{L^2}^2\,,
\end{align*}
where we used Lemma~\ref{lem:norm}. This proves \eqref{norm}.


In order to prove \eqref{enorm}, by Lemma~\ref{lem:change}, we find
\begin{align*}
&\E\Bigg(\bigg\l\tau\bigg[X_i\otimes\bigg\{\int_0^1\phi_{k\ell}(s,\cdot)X_i(s)ds\bigg\}\bigg],\phi_{k'\ell'}\bigg\r_K^2\Bigg)\\
&=\E\bigg\{\int_{[0,1]^3} X_i(s_1)X_i(s_2)\phi_{k\ell}(s_1,t)\phi_{k'\ell'}(s_2,t)ds_1ds_2dt\bigg\}^2\\
&=\E\bigg\{\int_{[0,1]^3} X_i(s_1)X_i(s_2)x_{k\ell}(s_1)x_{k'\ell'}(s_1)ds_1ds_2\bigg\}^2\times\bigg\{\int_0^1 \eta_\ell(t)\eta_{\ell'}(t)dt\bigg\}^2\\
&=\E\Bigg[\bigg\{\int_0^1 X_i(s_1)x_{k\ell}(s_1)ds_1\bigg\}^2\times\bigg\{\int_0^1X_i(s_2)x_{k'\ell'}(s_2)ds_2\bigg\}^2\Bigg]\times\bigg\{\int_0^1 \eta_\ell(t)\eta_{\ell'}(t)dt\bigg\}^2\,.
\end{align*}
Using the Cauchy-Schwarz inequality, by Assumptions~\ref{a201} and \ref{a:x}, we deduce from the above equation that, for the constant  $c_0>0$ in \eqref{moment},
\begin{align*}
&\E\Bigg(\bigg\l\tau\bigg[X_i\otimes\bigg\{\int_0^1\phi_{k\ell}(s,\cdot)X_i(s)ds\bigg\}\bigg],\phi_{k'\ell'}\bigg\r_K^2\Bigg)\\
&\leq\Bigg[\E\bigg\{\int_0^1 X_i(s_1)x_{k\ell}(s_1)ds_1\bigg\}^4\Bigg]^{\frac{1}{2}}\Bigg[\E\bigg\{\int_0^1X_i(s_2)x_{k'\ell'}(s_2)ds_2\bigg\}^4\Bigg]^{\frac{1}{2}}\times\Vert\eta_\ell\Vert_{L^2}^2\times\Vert\eta_{\ell'}\Vert_{L^2}^2\\
&\leq c_0\,\E\bigg\{\int_0^1 X_i(s_1)x_{k\ell}(s_1)ds_1\bigg\}^2\times\E\bigg\{\int_0^1X_i(s_2)x_{k'\ell'}(s_2)ds_2\bigg\}^2\times\Vert\eta_\ell\Vert_{L^2}^2\times\Vert\eta_{\ell'}\Vert_{L^2}^2\\
&=c_0\,\bigg\{\int_{[0,1]^2} C_X(s_1,s_2)x_{k\ell}(s_1)x_{k\ell}(s_2)ds_1ds_2\bigg\}\\
&\qquad\times\bigg\{\int_{[0,1]^2} C_X(s_1,s_2)x_{k'\ell'}(s_1)x_{k'\ell'}(s_2)ds_1ds_2\bigg\}\times\Vert\eta_\ell\Vert_{L^2}^2\times\Vert\eta_{\ell'}\Vert_{L^2}^2\\
&=c_0\,\l\C_X(x_{k\ell}),x_{k\ell}\r_{L^2}\,\l\C_X(x_{k'\ell'}),x_{k'\ell'}\r_{L^2}\,\Vert\eta_\ell\Vert_{L^2}^2\,\Vert\eta_{\ell'}\Vert_{L^2}^2\\
&=c_0\,V(x_{k\ell}\otimes\eta_{\ell},x_{k\ell}\otimes\eta_{\ell})\,V(x_{k'\ell'}\otimes\eta_{\ell'},x_{k'\ell'}\otimes\eta_{\ell'})\\
&=c_0\,V(\phi_{k\ell},\phi_{k\ell})\,V(\phi_{k'\ell'},\phi_{k'\ell'})=c_0\,.
\end{align*}
In view of \eqref{expansion2}, the above equation implies that
\begin{align}\label{temp3}
&\E\,\bigg\Vert\tau\bigg[X_i\otimes\bigg\{\int_0^1\phi_{k\ell}(s,\cdot)X_i(s)ds\bigg\}\bigg]\bigg\Vert_K^2\notag\\
&=\sum_{k',\ell'}\frac{1}{1+\lambda\rho_{k'\ell'}}\,\E\Bigg(\bigg\l\tau\bigg[X_i\otimes\bigg\{\int_0^1\phi_{k\ell}(s,\cdot)X_i(s)ds\bigg\}\bigg],\phi_{k'\ell'}\bigg\r_K^2\Bigg)\notag\\
&\leq c_0\,\sum_{k',\ell'}\frac{1}{1+\lambda\rho_{k'\ell'}}\,.
\end{align}

Now, using \eqref{expansion2} once again, by Lemma~\ref{lem:change} and Cauchy-Schwarz inequality,
\begin{align*}
&\E\Bigg\Vert\tau\bigg[X_i\otimes\bigg\{\int_0^1\beta(s,\cdot)X_i(s)ds\bigg\}\bigg]\Bigg\Vert_K^2\\
&=\sum_{k,\ell}\frac{1}{1+\lambda\rho_{k\ell}}\,\E\Bigg(\bigg\l\tau\bigg[X_i\otimes\bigg\{\int_0^1\beta(s,\cdot)X_i(s)ds\bigg\}\bigg],\phi_{k\ell}\bigg\r_K^2\Bigg)\\
&=\sum_{k,\ell}\frac{1}{1+\lambda\rho_{k\ell}}\,\E\Bigg(\bigg\l\tau\bigg[X_i\otimes\bigg\{\int_0^1\phi_{k\ell}(s,\cdot)X_i(s)ds\bigg\}\bigg],\beta\bigg\r_K^2\Bigg)\\
&\leq\Vert\beta\Vert_K^2\times\sum_{k,\ell}\frac{1}{1+\lambda\rho_{k\ell}}\,\E\,\bigg\Vert\tau\bigg[X_i\otimes\bigg\{\int_0^1\phi_{k\ell}(s,\cdot)X_i(s)ds\bigg\}\bigg]\bigg\Vert_K^2\,.
\end{align*}
Combining the above result and \eqref{temp3}, by Assumption~\ref{a201} and Lemma~\ref{lem:sum}, we find
\begin{align*}
&\E\Bigg\Vert\tau\bigg[X_i\otimes\bigg\{\int_0^1\beta(s,\cdot)X_i(s)ds\bigg\}\bigg]\Bigg\Vert_K^2\leq c_0\,\Vert\beta\Vert_K^2\times\sum_{k,\ell}\frac{1}{1+\lambda\rho_{k\ell}}\times\sum_{k',\ell'}\frac{1}{1+\lambda\rho_{k'\ell'}}\\
&\leq c\,\Vert\beta\Vert_K^2\times\bigg\{\sum_{k,\ell}\frac{1}{1+\lambda({k\ell})^{2D}}\bigg\}^2\leq c\,\lambda^{-1/D}\,\Vert\beta\Vert_K^2\,,
\end{align*}
which proves \eqref{enorm}.

\end{proof}

\begin{lemma}\label{lem:l2}
Under Assumptions~\ref{a1}--\ref{a:x}, for the $x_{k\ell}$'s in Assumption~\ref{a201}, for any $\beta\in\H$ and $\breve x\in L^2([0,1])$,
\begin{align}\label{eq:core}
&\bigg\Vert\tau\bigg[\breve x\otimes\bigg\{\int_0^1\beta(s,\cdot)\breve x(s)ds\bigg\}\bigg]\bigg\Vert_K^2
\leq \Vert \breve x\Vert_{L^2}^2\,\Vert\beta\Vert_{L^2}^2\sum_{k,\ell}\frac{1}{1+\lambda\rho_{k\ell}}\,\l\breve x,x_{k\ell}\r_{L^2}^2\,\Vert\eta_\ell\Vert^2_{L^2}\,.
\end{align}

\end{lemma}

\begin{proof}
By Assumption~\ref{a201}, the representation  $\beta(s,t)=\sum_{k,\ell}V(\beta,\phi_{k\ell})\phi_{k\ell}(s,t)=\sum_{k,\ell}V(\beta,x_{k\ell}\otimes\eta_\ell) x_{k\ell}(s)\eta_\ell(t)$
holds for any $\beta\in\H$, and  the Cauchy-Schwarz inequality and Lemma~\ref{lem:change} yield
\begin{align*}
&\bigg\l\tau\bigg[\breve x\otimes\bigg\{\int_0^1\beta(s,\cdot)\breve x(s)ds\bigg\}\bigg],\phi_{k\ell}\bigg\r_K^2\\
&=\bigg|\int_{[0,1]^3}\phi_{k\ell}(s_1,t)\,\beta(s_2,t)\,\breve x(s_1)\,\breve x(s_2)\,ds_1\,ds_2\,dt\bigg|^2\\
&=\bigg|\sum_{k',\ell'}V(\beta,x_{k'\ell'}\otimes\eta_{\ell'})\int_{[0,1]^3} x_{k\ell}(s_1)\,x_{k'\ell'}(s_2)\,\breve x(s_1)\,\breve x(s_2)\,\eta_{\ell}(t)\,\eta_{\ell'}(t)ds_1\,ds_2\,dt\bigg|^2\\
&=\bigg|\int_0^1\breve x(s_1)\,x_{k\ell}(s_1)\,ds_1\bigg|^2\times\bigg|\sum_{k',\ell'}V(\beta,x_{k'\ell'}\otimes\eta_{\ell'})\int_0^1\breve x(s_2)\,x_{k'\ell'}(s_2)\,ds_2\int_0^1\eta_\ell(t)\,\eta_{\ell'}(t)\,dt\bigg|^2\\
&=\bigg|\int_0^1\breve x(s_1)\,x_{k\ell}(s_1)\,ds_1\bigg|^2\times\bigg|\int_{[0,1]^2}\beta(s_2,t)\,\breve x(s_2)\,\eta_\ell(t)\,ds_2\,dt\bigg|^2\\
&\leq \Vert\beta\Vert_{L^2}^2\,\Vert \breve x\Vert_{L^2}^2\,\Vert\eta_\ell\Vert_{L^2}^2\bigg\{\int_0^1 \breve x(s)\,x_{k\ell}(s)\,ds\bigg\}^2\\
&=\Vert\beta\Vert_{L^2}^2\,\Vert \breve x\Vert_{L^2}^2\,\l\breve x,x_{k\ell}\r_{L^2}^2\,\Vert\eta_\ell\Vert^2_{L^2}\,.
\end{align*}
Therefore, in view of \eqref{expansion2}, we deduce from the above result that
\begin{align*}
\bigg\Vert\tau\bigg[\breve x\otimes\bigg\{\int_0^1\beta(s,\cdot)\breve x(s)ds\bigg\}\bigg]\bigg\Vert_K^2
& =\sum_{k,\ell}\frac{1}{1+\lambda\rho_{k\ell}}\,\bigg\l\tau\bigg[\breve x\otimes\bigg\{\int_0^1\beta(s,\cdot)\breve x(s)ds\bigg\}\bigg],\phi_{k\ell}\bigg\r_K^2\\
&\leq\Vert\beta\Vert_{L^2}^2\,\Vert \breve x\Vert_{L^2}^2\sum_{k,\ell}\frac{1}{1+\lambda\rho_{k\ell}}\,\l\breve x,x_{k\ell}\r_{L^2}^2\,\Vert\eta_\ell\Vert^2_{L^2}\,,
\end{align*}
which proves \eqref{eq:core}. 

\end{proof}

Recall from \eqref{wx} that, for the $\{x_{k\ell}\}_{k,\ell\geq1}$ and $\{\eta_\ell\}_{\ell\geq 1}$ in Assumption~\ref{a201},
\begin{align*}
w^2(X_i)=\Vert X_i\Vert_{L^2}^2\sum_{k,\ell}\frac{1}{1+\lambda\rho_{k\ell}}\,\l X_i,x_{k\ell}\r_{L^2}^2\,\Vert\eta_\ell\Vert^2_{L^2}\,,\quad 1\leq i\leq n\,.
\end{align*}
We have the following lemma regarding the second moment of $w(X)$.

\begin{lemma}\label{lem:wx}
Under Assumptions~\ref{a1}--\ref{a:x}, we have $\E\{w^2(X_i)\}\leq c\,\lambda^{-1/(2D)}$, where
$w(X_i)$ is defined in \eqref{wx} and $ c>0$ is an absolute constant.
\end{lemma}

\begin{proof}

By Assumption~\ref{a201}, for $k,\ell\geq1$,
\begin{align*}
V(\phi_{k\ell},\phi_{k\ell})=\Vert\eta_\ell\Vert_{L^2}^2\int C_{X}(s_1,s_2)x_{k\ell}(s_1) x_{k\ell}(s_2)ds_1ds_2=1\,.
\end{align*}
Using the Cauchy-Schwarz inequality and Assumption~\ref{a:x},
\begin{align*}
\E\{w^2(X_i)\}
%
&=\sum_{k,\ell}\frac{1}{1+\lambda\rho_{k\ell}}\,\Vert\eta_\ell\Vert_{L^2}^2\E\Bigg[\Vert X_i\Vert_{L^2}^2\,\bigg\{\int_0^1 X_i(s)\,x_{k\ell}(s)\,ds\bigg\}^2\Bigg]\\
&\leq\Big(\E\Vert X_i\Vert_{L^2}^4\Big)^{1/2}\sum_{k,\ell}\frac{1}{1+\lambda\rho_{k\ell}}\,\Vert\eta_\ell\Vert_{L^2}^2\Bigg[\E\bigg\{\int_0^1 X_i(s)\,x_{k\ell}(s)\,ds\bigg\}^4\Bigg]^{1/2}\\
&\leq c\Big(\E\Vert X_i\Vert_{L^2}^4\Big)^{1/2}\sum_{k,\ell}\frac{1}{1+\lambda\rho_{k\ell}}\,\Vert\eta_\ell\Vert_{L^2}^2\,\E\bigg\{\int_0^1 X_i(s)\,x_{k\ell}(s)\,ds\bigg\}^2\\
&=c\Big(\E\Vert X_i\Vert_{L^2}^4\Big)^{1/2}\sum_{k,\ell}\frac{1}{1+\lambda\rho_{k\ell}}\,\Vert\eta_\ell\Vert_{L^2}^2\int C_{X}(s_1,s_2)\, x_{k\ell}(s_1)\, x_{k\ell}(s_2)\,ds_1\,ds_2\\
&=c\Big(\E\Vert X_i\Vert_{L^2}^4\Big)^{1/2}\sum_{k,\ell}\frac{1}{1+\lambda\rho_{k\ell}}\,V(\phi_{k\ell},\phi_{k\ell})\\
&=c\Big(\E\Vert X_i\Vert_{L^2}^4\Big)^{1/2}\sum_{k,\ell}\frac{1}{1+\lambda\rho_{k\ell}}\,.
\end{align*}
Since $\E\Vert X_i\Vert_{L^2}^4$ is finite by Assumption~\ref{a:x}, by Assumption~\ref{a201} and Lemma~\ref{lem:sum}, we deduce from the above equation that
\begin{align*}
\E\{w^2(X_i)\}\leq c\,\sum_{k,\ell}\frac{1}{1+\lambda(k\ell)^{2D}}\leq c\,\lambda^{-1/(2D)}\,.
\end{align*}

\end{proof}

\begin{lemma}\label{lem:etk}
Under Assumptions~\ref{a1}--\ref{a:x},
\begin{align*}
&\E\Vert\tau(X\otimes\e)\Vert_K^2=\sum_{k,\ell}\frac{1}{1+\lambda\rho_{k\ell}}\,,\\
&\E\Vert\tau(X\otimes\e)\Vert_K^4\leq c\,\lambda^{-1/D}\,,
\end{align*}
where $c>0$ is an absolute constant.
\end{lemma}

\begin{proof}
In view of \eqref{expansion} and \eqref{expansion2}, we have
\begin{align}\label{tk}
\begin{split}
&\tau(X\otimes\e)=\sum_{k,\ell}\frac{\l\tau(X\otimes\e),\phi_{k\ell}\r_{K}}{1+\lambda\rho_{k\ell}}\,\phi_{k\ell}=\sum_{k,\ell}\frac{\l X\otimes\e,\phi_{k\ell}\r_{L^2}}{1+\lambda\rho_{k\ell}}\,\phi_{k\ell}\,,\\
&\Vert\tau(X\otimes\e)\Vert_K^2=\sum_{k,\ell}\frac{\l\tau(X\otimes\e),\phi_{k\ell}\r_{K}^2}{1+\lambda\rho_{k\ell}}=\sum_{k,\ell}\frac{\l X\otimes\e,\phi_{k\ell}\r_{L^2}^2}{1+\lambda\rho_{k\ell}}\,.
\end{split}
\end{align}
Recall from Assumption~\ref{a0} that $C_\e(s_1,s_2)=\E\{\epsilon(s_1)\epsilon(s_2)|X\}=\delta(s_1,s_2)$ (we have assumed $\sigma_\e^2=1$ without loss of generality; see the beginning of Section~\ref{app:thm:minimax}). Observing the definition of $V$ in \eqref{v} and $\tau$ in
\eqref{tau}, we find, for $k,k',\ell,\ell'\geq 1$,
\begin{align}\label{core}
&\E\Big\{\big\l \tau(X\otimes\e),\phi_{k\ell}\big\r_{K}\,\big\l \tau(X\otimes\e),\phi_{k'\ell'}\big\r_{K}\Big\}\notag\\
&=\E\Big\{\big\l X\otimes\e,\phi_{k\ell}\big\r_{L^2}\,\big\l X\otimes\e,\phi_{k'\ell'}\big\r_{L^2}\Big\}\notag\\
&=\E\Bigg[\bigg\{\int_{[0,1]^2}X(s)\epsilon(t)\phi_{k\ell}(s,t)\,ds\,dt\bigg\}\times\bigg\{\int_{[0,1]^2}X(s')\epsilon(t')\phi_{k'\ell'}(s',t')\,ds'\,dt'\bigg\}\Bigg]\notag\\
&=\E\bigg[\int_{[0,1]^4}\E\{\e(t)\e(t')|X\}\,X(s)X(s')\,\phi_{k\ell}(s,t)\,\phi_{k'\ell'}(s',t')\,ds\,ds'\,dt\,dt'\bigg]\notag\\
&=\int_{[0,1]^4}C_\epsilon(t,t')\,C_X(s,s')\,\phi_{k\ell}(s,t)\,\phi_{k'\ell'}(s',t')\,ds\,ds'\,dt\,dt'\notag\\
&=\int_{[0,1]^3}C_X(s,s')\,\phi_{k\ell}(s,t)\,\phi_{k'\ell'}(s',t)\,ds\,ds'\,dt\notag\\
&=V(\phi_{k\ell},\phi_{k'\ell'})=\delta_{kk'}\,\delta_{\ell\ell'}\,,
\end{align}
where we used Assumption~\ref{a201} in the last step. From the above equation we have obtained
\begin{align*}
\E\Big(\big\l X\otimes\e,\phi_{k\ell}\big\r_{L^2}^2\Big)&=1\,.
\end{align*}
In view of \eqref{tk},
\begin{align*}
\E\Vert\tau(X\otimes\e)\Vert_K^2=\sum_{k,\ell}\frac{\E\big(\l X\otimes\e,\phi_{k\ell}\r_{L^2}^2\big)}{1+\lambda\rho_{k\ell}}=\sum_{k,\ell}\frac{1}{1+\lambda\rho_{k\ell}}\,.
\end{align*}

To show the order of $\E\Vert\tau(X\otimes\e)\Vert_K^4$ we use  Assumptions~\ref{a201} and \ref{a:x}
to obtain 
\begin{align}\label{e4}
&\E\big(\l X\otimes\e,\phi_{k\ell}\r_{L^2}^4\big)=\E\bigg\{\int_0^1\int_0^1X(s)\epsilon(t)\phi_{k\ell}(s,t)\,ds\,dt\bigg\}^4\notag\\
&=\E\Bigg[\bigg\{\int_0^1X(s)x_{k\ell}(s)ds\bigg\}^4\int_{[0,1]^4}\E\{\e(t_1)\e(t_2)\e(t_3)\e(t_4)|X\}\eta_\ell(t_1)\eta_\ell(t_2)\eta_\ell(t_3)\eta_\ell(t_4)dt_1dt_2dt_3dt_4\Bigg]\notag\\
&\leq \bigg(\int_{[0,1]^4}\Big[\E\{\e(t_1)\e(t_2)\e(t_3)\e(t_4)|X\}\Big]^2dt_1dt_2dt_3dt_4\bigg)^{1/2}\times\Vert\eta_\ell\Vert_{L^2}^4\times\Bigg[\E\bigg\{\int_0^1X(s)x_{k\ell}(s)ds\bigg\}^2\Bigg]^2\notag\\
&= c\,\Bigg[\E\bigg\{\int_0^1\int_0^1C_X(s_1,s_2)x_{k\ell}(s_1)x_{k\ell}(s_2)ds_1ds_2\bigg\}^2\Bigg]^2=c\,\{V(x_{k\ell},x_{k\ell})\}^2=c\,.
\end{align}
Therefore, we deduce that
\begin{align*}
\E\Vert\tau(X\otimes\e)\Vert_K^4&=\sum_{k_1,\ell_1,k_2,\ell_2}\frac{1}{(1+\lambda\rho_{k_1\ell_1})(1+\lambda\rho_{k_2\ell_2})}\E\Big(\l X\otimes\e,\phi_{k_1\ell_1}\r_{L^2}^2\l X\otimes\e,\phi_{k_2\ell_2}\r_{L^2}^2\Big)\\
&\leq\Bigg[\sum_{k,\ell}\frac{1}{1+\lambda\rho_{k\ell}}\Big\{\E\big(\l X\otimes\e,\phi_{k\ell}\r_{L^2}^4\big)\Big\}^{1/2}\Bigg]^2\leq c\,\bigg\{\sum_{k,\ell}\frac{1}{1+\lambda\rho_{k\ell}}\bigg\}^2\leq c\lambda^{-1/D}\,.
\end{align*}

\end{proof}

The following lemma is a modified version of Lemma~A.1 in \cite{kley2016}, which we use to prove Theorem~\ref{thm:process}.
\begin{lemma}\label{lem:kley}
For any non-decreasing, convex function $\Psi:\mathbb{R}^+\to\mathbb{R}^+$ with $\Psi(0)=0$ and for any real-valued random variable $Z$, let $\Vert Z\Vert_\Psi=\inf\{c>0:\E\{\Psi(|Z|/c)\}\leq 1\}$ denote the Orlicz norm. Let $\{H(s,t):(s,t)\in[0,1]^2\}$ be a separable stochastic process with $\Vert H(s_1,t_1)-H(s_2,t_2)\Vert_\Psi\leq c\,d\{(s_1,t_1),(s_2,t_2)\}$ for any $(s_1,t_1),(s_2,t_2)\in[0,1]^2$ with $d\{(s_1,t_1),(s_2,t_2)\}\geq \overline\eta/2\geq0$ and for some constant $c>0$. Let $\D(w,d)$ denote the packing number of the metric space $([0,1]^2,d)$. Then, for any $\delta>0$, $\eta>\overline\eta$, there exists a random variable $S$ and a constant $K>0$ such that
\begin{align*}
\sup_{d\{(s_1,t_1),(s_2,t_2)\}\leq\delta}|H(s_1,t_1)-H(s_2,t_2)|\leq S+2\sup_{\substack{d\{(s_1,t_1),(s_2,t_2)\}\leq\overline\eta\\(s_1,t_1)\in[0,1]^2}}|H(s_1,t_1)-H(s_2,t_2)|
\end{align*}
and
\begin{align*}
\Vert S\Vert_\Psi\leq K\bigg[\int_{\overline\eta/2}^\eta\Psi^{(-1)}\{\D(\e,d)\}d\e+(\delta+2\overline\eta)\,\Psi^{(-1)}\{\D^2(\eta,d)\}\bigg]\,,
\end{align*}
where $\Psi^{(-1)}$ is the inverse function of $\Psi$, and the set $[0,1]^2$ contains at most $\D(\overline\eta,d)$ points. 

\end{lemma}


\subsection{An example for Assumption~\ref{a201}}\label{app:cond:prop1.3}\label{app:aux:cond}


In this section we provide a concrete example that satisfies Assumption~\ref{a201}. We use the cosine basis of $L^2([0,1])$ defined by
\begin{align}\label{fourier}
%
\eta_1(t) \equiv 1~,~~
\eta_\ell(t) = \sqrt{2}\cos\{(\ell-1)\pi t\}
~~~ \ell=  2,3, \ldots  
\end{align}
Now, the derivatives of these functions are orthogonal with respect to the $L^2$-inner product $\l\cdot,\cdot\r_{L^2}$, that is, for any integer $\theta\geq 0$ and $\ell,\ell'\geq2$, $\l\eta_{\ell}^{(\theta)},\eta_{\ell'}^{(\theta)}\r_{L^2}=\delta_{\ell\ell'}\,\Vert\eta_{\ell}^{(\theta)}\Vert_{L^2}^2=\delta_{\ell\ell'}(\ell-1)^{2\theta}\pi^{2\theta}$; for any $\theta\geq0$ and $\ell\geq2$, $\l\eta_{1}^{(\theta)},\eta_{\ell}^{(\theta)}\r_{L^2}=0$.
Given $\{\eta_{\ell}\}_{\ell\geq1}$, the  functions $x_{k\ell}$ 
are defined as the solution of a series of integral-differential equations whose parameters depend on $\Vert\eta_\ell^{(\theta)}\Vert_{L^2}^2$, such that \eqref{diag} is satisfied.
In particular, we have the following proposition, which provides an example of an eigen-system that satisfies Assumption~\ref{a201}
and is proved in Section~\ref{app:proof:prop1.3}
using  the theory of integro-differential equations.

\begin{proposition}\label{prop:1.3} 
For each $\ell\geq 1$, let $\{(\rho_{k\ell},\tilde x_{k\ell})\}_{k\geq1}$ denote the eigenvalue-eigenfunction pairs of the following integro-differential equations with boundary conditions:
\begin{align}\label{id}
\hspace{-0.4em}
\left\{
\begin{aligned}
&\displaystyle\rho_\ell\int_{0}^1C_X(s,s')\tilde x(s')ds'=(-1)^{m}\tilde x^{(2m)}(s)+\sum_{\theta=0}^{m-1}{m\choose\theta}(-1)^{\theta}\{(\ell-1)\pi\}^{2m-2\theta}\tilde x^{(2\theta)}(s)\,,\\
&\tilde x^{(\theta)}(0)=\tilde x^{(\theta)}(1)=0\,,\qquad\text{ for }m\leq\theta\leq 2m-1\,.
\end{aligned}\right.
\end{align}
For the functions $\eta_\ell$ defined in \eqref{fourier}, let $x_{k\ell}=\l\C_X(\tilde x_{k\ell}),\tilde x_{k\ell}\r^{-1/2}_{L^2}\,\tilde x_{k\ell}$ and $\phi_{k\ell}=x_{k\ell}\otimes\eta_\ell$, where the operator $\C_X$ is defined by \eqref{m1a}. Suppose Condition~{\rm\ref{sudosy}} below
is satisfied for some constant $r\geq 0$.
Then, the pairs $(\rho_{k\ell},\phi_{k\ell})_{k,\ell}$ satisfy Assumption~\ref{a201} with $D=m+r+1$ and $a=r+1$.
\end{proposition}

For a constant $r\geq0$, we now state as follows Condition~\ref{sudosy} for Proposition~\ref{prop:1.3}, which was proposed in \cite{shang2015}. Let $\Omega_+=\{(s,t)\in[0,1]:s>t\}$, $\Omega_-=\{(s,t)\in[0,1]:s<t\}$ and let ${\rm cl}(A)$ denote the closure of $A\subset[0,1]^2$. Recall that $C_X$ defined in \eqref{m2a} is the covariance function of $X$. 

\begin{newcondition}{B($r$)}\label{sudosy}

Suppose that there exists a constant $r\geq0$ such that one of the following two assumptions is satisfied: (i) $r=0$; (ii) $r\geq 1$, and for any $j=0,1,\ldots,r-1$, $C_X^{(j,0)}(0,t)=0$, for any $0\leq t\leq 1$. Assume $C_X$ satisfies the following {\it pseudo SY conditions} of order $r$:
\begin{enumerate}[label=(\arabic*),nolistsep]
\item $L(s_1,s_2):=C_X^{(r,r)}(s_1,s_2)$ is continuous on $[0,1]^2$. All the partial derivatives of $L(s_1,s_2)$ up to order $2m+2r+2$ are continuous on $\Omega_+\cup\Omega_-$, and continuously extendable to ${\rm cl}(\Omega_+)$ and ${\rm cl}(\Omega_+)$.

\item $a(s):=L_-^{(1,0)}(s,s)-L_+^{(1,0)}(s,s)$ has a positive lower bound for any $s\in[0,1]$, where $L_-^{(1,0)}$ and $L_+^{(1,0)}$ are two different extensions of $L^{(1,0)}$ to $[0,1]^2$ that are continuous on ${\rm cl}(\Omega_-)$ and ${\rm cl}(\Omega_+)$, respectively.

\item $L^{(2,0)}_+(s_1,s_2)$ is bounded over $[0,1]^2$, where $L_+^{(2,0)}$ is the extension of $L^{(2,0)}$ to $[0,1]^2$ that is continuous on ${\rm cl}(\Omega_+)$. 
\end{enumerate}
In addition, assume that the integro-differential equations with the boundary conditions in Proposition~\ref{prop:1.3} are \emph{regular} in the the sense of \cite{birk1908} (see Definition~\ref{def:regular} below).

\end{newcondition}

\begin{definition}[Regular boundary conditions of even order; \citealp{birk1908}]\label{def:regular}
Consider the linear differential equation of order $2k$ in $\phi$: $\phi^{(2k)}(x)+\sum_{\ell=0}^{2k-2}p_\ell(x)\phi^{(\ell)}(x)+\gamma\phi(x)=0$ on an interval $[\mathfrak a,\mathfrak b]$ with $2k$ linear homogeneous boundary conditions $W_i(\phi)=0$ in $\phi(\mathfrak a),\phi'(\mathfrak a),\phi^{(k-1)}(\mathfrak a),\ldots,\phi(\mathfrak b),\phi'(\mathfrak b),\phi^{(k-1)}(\mathfrak b)$, for $1\leq i\leq 2k$. Applying the linear transformation on the $W_i$'s to obtain the normalized boundary conditions in the form $W_i(\phi)=W_{i\mathfrak a}(\phi)+W_{i\mathfrak b}(\phi)=0$, where $W_{i\mathfrak a}(\phi)=a_{i}\phi^{(j_i)}(\mathfrak a)+\sum_{\ell=0}^{j_i-1}a_{i\ell}\phi^{(\ell)}(\mathfrak a)$, $W_{i\mathfrak b}(\phi)=b_{i}\phi^{(j_i)}(\mathfrak b)+\sum_{\ell=0}^{j_i-1}b_{i\ell}\phi^{(\ell)}(\mathfrak b)$, and where $j_1\geq\ldots\geq j_n$ are such that no successive three of them are equal. Define $\zeta_0,\zeta_1,\zeta_2\in\mathbb{C}$ through the identity
\begin{align*}
\zeta_0+\zeta_1s+\frac{\zeta_2}{s}\equiv
\left|\begin{matrix}
a_1w_1^{j_1} & \cdots& a_1w_{k-1}^{j_1}&(a_1+b_1s)w_k^{j_1}&(a_1+\frac{b_1}{s})w_{k+1}^{j_1}&b_1w_{k+2}^{j_1}&\cdots&b_1w_n^{j_1}\\
a_2w_1^{j_2} & \cdots& a_2w_{k-1}^{j_2}&(a_2+b_2s)w_k^{j_2}&(a_2+\frac{b_2}{s})w_{k+1}^{j_2}&b_2w_{k+2}^{j_2}&\cdots&b_2w_n^{j_2}\\
\cdots & \cdots& \cdots&\cdots&\cdots&\cdots&\cdots&\cdots\\
a_nw_1^{j_n} & \cdots& a_nw_{k-1}^{j_n}&(a_n+b_ns)w_k^{j_n}&(a_n+\frac{b_n}{s})w_{k+1}^{j_2}&b_nw_{k+2}^{j_n}&\cdots&b_nw_n^{j_n}\\
\end{matrix}\right|\,,
\end{align*}
where the $w_i$'s are the $2k$-th root of unity ordered according to ${\rm Re}(\rho w_1)<{\rm Re}(\rho w_2)<\cdots<{\rm Re}(\rho w_n)$ and $\rho=\gamma^{1/(2k)}$. Then, the boundary conditions $W_1(\phi),\ldots,W_{2k}(\phi)$ are \emph{regular} if $\zeta_1\neq0$ and $\zeta_2\neq 0$.

\end{definition}

\subsection{Proof of Proposition~\ref{prop:1.3}}\label{app:proof:prop1.3}

We follow the proof of Proposition~2.2 in \cite{shang2015} and we assume that $a(\cdot)$ in Condition~\ref{sudosy} satisfies $a\equiv 1$ without loss of generality. For the integro-differential equation \eqref{id}, taking $\ell=1$ yields
\begin{align*}
\begin{cases}
\displaystyle\rho_{1}\int_{0}^1C_X(s_1,s_2)\,x(s_2)\,ds_2=(-1)^{m}\,\tilde x^{(2\theta)}(s_1)\,,\\
\tilde x^{(\theta)}(0)=\tilde x^{(\theta)}(1)=0\,,\hspace{1cm} \text{ for }m\leq\theta\leq 2m-1\,.
\end{cases}
\end{align*}
This case was proved by \cite{shang2015}. We therefore focus on the case where $\ell\geq2$ in equation \eqref{id}, which is equivalent to, for $\ell\geq2$,
\begin{align}\label{int0}
\begin{cases}
\displaystyle\rho_{\ell}\int_{0}^1C_X(s_1,s_2)\,x(s_2)\,ds_2=\sum_{\theta=0}^m{m\choose\theta}(-1)^{\theta}\{(\ell-1)\pi\}^{2m-2\theta}\,\tilde x^{(2\theta)}(s_1)\,,\\
\tilde x^{(\theta)}(0)=\tilde x^{(\theta)}(1)=0\,,\hspace{1cm} \text{ for }m\leq\theta\leq 2m-1\,.
\end{cases}
\end{align}
By virtue of simple presentation, without loss of generality, we change the subscript of $\rho_\ell$ to $\ell-1$ in \eqref{int0} to write, for $\ell\geq 1$,
\begin{align}\label{int}
\begin{cases}
\displaystyle\rho_{\ell}\int_{0}^1C_X(s_1,s_2)\,x(s_2)\,ds_2=\sum_{\theta=0}^m{m\choose\theta}(-1)^{\theta}(\ell\pi)^{2m-2\theta}\,\tilde x^{(2\theta)}(s_1)\,,\\
\tilde x^{(\theta)}(0)=\tilde x^{(\theta)}(1)=0\,,\hspace{1cm} \text{ for }m\leq\theta\leq 2m-1\,.
\end{cases}
\end{align}
In the sequel we show the results in Proposition~\ref{prop:1.3} based on \eqref{int}. Let $N(s_1,s_2)=L_+^{(2,0)}(s_1,s_2)$ for $L_+$ in Section~\ref{app:cond:prop1.3} and let $M(s_1,s_2)$ denote its reciprocal kernel such that the following reciprocal property (\citealp{tam1927}) is satisfied
\begin{align}\label{repro}
M(s_1,s_2)+N(s_1,s_2)=\int_0^1M(s_1,\xi)N(\xi,s_2)d\xi=\int_0^1N(s_1,\xi)M(\xi,s_2)d\xi\,.
\end{align}


For $0\leq\theta\leq m$, let $c_\theta={m\choose\theta}(-1)^{\theta}\pi^{2m-2\theta}$. For $r$ in Condition~\ref{sudosy}, we have \eqref{int} is equivalent to the following equation
\begin{align}\label{integro2}
\begin{cases}
&\displaystyle\sum_{\theta=0}^mc_\theta\,\ell^{2m-2\theta}\psi^{(2\theta+r)}(s_1)=\rho_\ell \int_{0}^1C_X(s_1,s_2)\,\psi^{(r)}(s_2)\,ds_2\,,\\
&\psi^{(\upsilon)}(0)=\psi^{(\upsilon)}(1)=0\,,\text{ for }m+r\leq \upsilon\leq 2m+r-1\,,\\
&\psi^{(\upsilon)}(1)=0\,,\text{ for }0\leq \upsilon\leq r-1\,.
\end{cases}
\end{align}
That is, $\tilde x=\psi^{(r)}$ being the solution to \eqref{int} is equivalent to $\psi$ being the solution to \eqref{integro2}. From the first equation in \eqref{integro2}, by integration by parts we find
\begin{align*}
\sum_{\theta=0}^mc_\theta\,\ell^{2m-2\theta}\psi^{(2\theta+r)}(s_1)=(-1)^r\rho_\ell\int_{0}^1C_X^{(0,r)}(s_1,s_2)\,\psi(s_2)\,ds_2\,,
\end{align*}
due to the assumption that $C_X^{(0,\upsilon)}(s,0)=1$ for $0\leq\upsilon\leq r-1$, and that $\psi^{(\upsilon)}(1)=0$, for $0\leq \upsilon\leq r-1$ in \eqref{integro2}. Taking partial derivatives of the above equation yields that, for $0\leq\upsilon\leq r$,
\begin{align*}
&\sum_{\theta=0}^mc_\theta\,\ell^{2m-2\theta}\psi^{(2\theta+r+\upsilon)}(s_1)=(-1)^r\rho_\ell \int_{0}^1C_X^{(\upsilon,r)}(s_1,s_2)\,\psi(s_2)\,ds_2\,,\\
&\sum_{\theta=0}^mc_\theta\,\ell^{2m-2\theta}\psi^{(2\theta+2r+1)}(s_1)=(-1)^r\rho_\ell \int_{0}^1L^{(1,0)}(s_1,s_2)\,\psi(s_2)\,ds_2\,,\\
&\sum_{\theta=0}^mc_\theta\,\ell^{2m-2\theta}\psi^{(2\theta+2r+2)}(s_1)=(-1)^{r+1}\rho_\ell\,\psi(s_1)+(-1)^r\rho_\ell \int_{0}^1L_+^{(2,0)}(s_1,s_2)\,\psi(s_2)\,ds_2\,,
\end{align*}
due to the fact that $a(s)=1$ and $\int_{0}^1L^{(1,0)}(s_1,s_2)\psi(s_2)\,ds_2=\int_{s_1}^1L_-^{(1,0)}(s_1,s_2)\psi(s_2)ds_2+\int_0^{s_1}L_+^{(1,0)}(s_1,s_2)\psi(s_2)ds_2$. Hence we find that \eqref{integro2} is equivalent to the following boundary value problem
\begin{align}\label{integro3}
\begin{cases}
&\displaystyle \sum_{\theta=0}^mc_\theta\,\ell^{2m-2\theta}\psi^{(2\theta+2r+2)}(s_1)+(-1)^r\rho_\ell\bigg\{\psi(s_1)-\int_{0}^1L_+^{(2,0)}(s_1,s_2)\psi(s_2)ds_2\bigg\}=0\,,\\
&\psi^{(\upsilon)}(0)=\psi^{(\upsilon)}(1)=0\,,\qquad \text{ for }m+r\leq \upsilon\leq 2m+r-1\,,\\
&\psi^{(\upsilon)}(1)=0\,,\qquad\qquad\qquad\  \text{ for }0\leq \upsilon\leq r-1\,,\\
&\displaystyle \sum_{\theta=0}^mc_\theta\,\ell^{2m-2\theta}\psi^{(2\theta+r+\upsilon)}(0)=(-1)^{r}\rho_\ell \int_0^1C_X^{(\upsilon,r)}(0,s)\psi(s)ds\,,\ \  0\leq \upsilon\leq r+1\,.
\end{cases}
\end{align}

Recall from the paragraph above \eqref{repro} that $N(s_1,s_2)=L_+^{(2,0)}(s_1,s_2)$. For the first equation in \eqref{integro3}, using the reciprocal property in \eqref{repro},
\begin{align*}
&(-1)^{r+1}\sum_{\theta=0}^mc_\theta\,\ell^{2m-2\theta}\int_0^1M(\xi,s)\,\psi^{(2\theta+2r+2)}(s)ds\\
&=\rho_\ell \bigg\{\int_0^1M(\xi,s)\psi(s)ds-\int_0^1\int_0^1M(\xi,s_1)N(s_1,s_2)\psi(s_2)ds_1ds_2\bigg\}\\
&=\rho_\ell\bigg[\int_0^1M(\xi,s)\psi(s)ds-\int_0^1\{M(\xi,s)+N(\xi,s)\}\psi(s)ds\bigg]\\
&=-\rho_\ell\breve\int_0^1N(\xi,s)\psi(s)ds=-\rho _\ell \int_{0}^1L_+^{(2,0)}(\xi,s)\psi(s)\,ds\,.
\end{align*}
Combining the above equation with the first equation of \eqref{integro3} yields
\begin{align}\label{s10}
\rho_\ell\psi(s_1)&=(-1)^{r+1}\sum_{\theta=0}^mc_\theta\,\ell^{2m-2\theta}\left\{\psi^{(2\theta+2r+2)}(s_1)-\int_0^1M(s_1,s_2)\,\psi^{(2\theta+2r+2)}(s_2)ds_2\right\}\,.
\end{align}
Combining \eqref{s10} with the first equation of \eqref{integro3} and using the reciprocal property in \eqref{repro}, we find
\begin{align*}
&(-1)^{r+1}\sum_{\theta=0}^mc_\theta\,\ell^{2m-2\theta}\psi^{(2\theta+2r+2)}(s_1)\\
&=\rho_\ell\psi(s_1)-\int_0^1N(s_1,s_2)\\
&\quad\times\left[(-1)^{r+1}\sum_{\theta=0}^mc_\theta\,\ell^{2m-2\theta}\left\{\psi^{(2\theta+2r+2)}(s_2)-\int_0^1M(s_2,\xi)\,\psi^{(2\theta+2r+2)}(\xi)d\xi\right\}\right]ds_2\\
&=\rho_\ell\psi(s_1)-(-1)^{r+1}\sum_{\theta_1=0}^mc_\theta\,\ell^{2m-2\theta}\int_0^1N(s_1,s_2)\psi^{(2\theta+2r+2)}(s_2)ds_2\\
&\quad+(-1)^{r+1}\sum_{\theta=0}^mc_\theta\,\ell^{2m-2\theta}\int_0^1\int_0^1N(s_1,s_2)M(s_2,\xi)\psi^{(2\theta+2r+2)}(\xi)d\xi ds_2\\
&=\rho_\ell\psi(s_1)-(-1)^{r+1}\sum_{\theta=0}^mc_\theta\,\ell^{2m-2\theta}\int_0^1N(s_1,s_2)\psi^{(2\theta+2r+2)}(s_2)ds_2\\
&\quad+(-1)^{r+1}\sum_{\theta=0}^mc_\theta\,\ell^{2m-2\theta}\int_0^1N(s_1,\xi)\psi^{(2\theta+2r+2)}(\xi)d\xi\\
&\quad+(-1)^{r+1}\sum_{\theta=0}^mc_\theta\,\ell^{2m-2\theta}\int_0^1M(s_1,\xi)\psi^{(2\theta+2r+2)}(\xi)d\xi\\
&=\rho_\ell\psi(s_1)+(-1)^{r+1}\sum_{\theta=0}^mc_\theta\,\ell^{2m-2\theta}\int_0^1M(s_1,s_2)\psi^{(2\theta+2r+2)}(s_2)ds_2\,.
\end{align*}

Using integration by parts, we deduce from the above equation that
\begin{align}\label{temp}
&(-1)^{r+1}\sum_{\theta=0}^mc_\theta\,\ell^{2m-2\theta}\psi^{(2\theta+2r+2)}(s_1)\notag\\
&=\rho_\ell\psi(s_1)+(-1)^{r+1}\sum_{\theta=0}^mc_\theta\,\ell^{2m-2\theta}\int_0^1M^{(0,2\theta+2r+2)}(s_1,s_2)\,\psi(s_2)ds_2\notag\\
&+\sum_{\theta=0}^m\sum_{j=1}^{2\theta+2r+2}(-1)^{r+j}c_\theta\,\ell^{2m-2\theta}\Big\{M^{(0,j-1)}(s_1,1)\psi^{(2\theta+2r+2-j)}(1)-M^{(0,j-1)}(s_1,0)\psi^{(2\theta+2r+2-j)}(0)\Big\}\notag\\
&=\rho_\ell\psi(s_1)+ L_\ell\psi(s_1)\,,
\end{align}
if we define
\begin{align}\label{Lell}
L_\ell\psi(s_1)&=(-1)^{r+1}\sum_{\theta=0}^mc_\theta\,\ell^{2m-2\theta}\int_0^1M^{(0,2\theta+2r+2)}(s_1,s_2)\,\psi(s_2)ds_2+\sum_{\theta=0}^m\sum_{j=1}^{2\theta+2r+2}(-1)^{r+j}c_\theta\,\ell^{2m-2\theta}\notag\\
&\quad\times\Big\{M^{(0,j-1)}(s_1,1)\,\psi^{(2\theta+2r+2-j)}(1)-M^{(0,j-1)}(s_1,0)\,\psi^{(2\theta+2r+2-j)}(0)\Big\}\,.
\end{align}

For the last boundary value condition in \eqref{integro3}, by \eqref{s10} and integration by parts, we find that, for $0\leq \upsilon\leq r+1$,
\begin{align*}
&(-1)^{r}\sum_{\theta=0}^mc_\theta\,\ell^{2m-2\theta}\psi^{(2\theta+r+\upsilon)}(0)\\
&=\int_0^1C_X^{(r,\upsilon)}(s_1,0)\Bigg[(-1)^{r+1}\sum_{\theta=0}^mc_\theta\,\ell^{2m-2\theta}\bigg\{\psi^{(2\theta+2r+2)}(s_1)-\int_0^1M(s_1,s_2)\,\psi^{(2\theta+2r+2)}(s_2)ds_2\bigg\}\Bigg]ds_1\\
&=(-1)^{r+1}\sum_{\theta=0}^mc_\theta\,\ell^{2m-2\theta}\bigg\{\int_0^1C_X^{(r,\upsilon)}(s_1,0)\,\psi^{(2\theta+2r+2)}(s_1)ds_1\\
&\qquad-\int_0^1\int_0^1C_X^{(r,\upsilon)}(s_1,0)M(s_1,s_2)\,\psi^{(2\theta+2r+2)}(s_2)ds_2ds_1\bigg\}\\
&=(-1)^{r+1}\sum_{\theta=0}^mc_\theta\,\ell^{2m-2\theta}\Bigg[\int_0^1C_X^{(2\theta+3r+2,\upsilon)}(s_1,0)\,\psi(s_1)ds_1\\
&\qquad-\int_0^1\int_0^1C_X^{(r,\upsilon)}(s_1,0)M^{(0,2\theta+2r+2)}(s_1,s_2)\,\psi(s_2)ds_2ds_1\\
&\qquad+\sum_{j=1}^{2\theta+2r+2}(-1)^{j-1}\bigg\{C_X^{(r+j-1,\upsilon)}(1,0)\,\psi^{(2\theta+2r+2-j)}(1)-C_X^{(r+j-1,\upsilon)}(0,0)\,\psi^{(2\theta+2r+2-j)}(0)\\
&\qquad-\psi^{(2\theta+2r+2-j)}(1)\int_0^1C_X^{(r,\upsilon)}(s_1,0)M^{(0,j-1)}(s_1,1)ds_1\\
&\qquad+\psi^{(2\theta+2r+2-j)}(0)\int_0^1C_X^{(r,\upsilon)}(s_1,0)M^{(0,j-1)}(s_1,0)ds_1\,\bigg\}\Bigg]\,.
\end{align*}
For $0\leq \upsilon\leq r+1$ and $1\leq j\leq 2m+2r+2$, if we denote
\begin{align}\label{ab}
&A_{\ell,\upsilon}(s)=\sum_{\theta=0}^mc_\theta\,\ell^{2m-2\theta}\left\{C_X^{(2\theta+3r+2,\upsilon)}(s,0)-\int_0^1C_X^{(r,\upsilon)}(\xi,0)M^{(0,2\theta+2r+2)}(\xi,s)d\xi\right\}\,;\notag\\
&a_{j,\upsilon}=(-1)^{j}\bigg\{C_X^{(r+j-1,\upsilon)}(0,0)-\int_0^1C_X^{(r,\upsilon)}(s,0)M^{(0,j-1)}(s,0)ds\bigg\}\,;\notag\\
&b_{j,\upsilon}=(-1)^{j+1}\bigg\{C_X^{(r+j-1,\upsilon)}(1,0)-\int_0^1C_X^{(r,\upsilon)}(s,0)M^{(0,j-1)}(s,1)ds\bigg\}\,,
\end{align}
we have that for $0\leq \upsilon\leq r+1$,
\begin{align*}
&\sum_{\theta=0}^m\sum_{j=1}^{2\theta+2r+2}c_\theta\,\ell^{2m-2\theta}\Big\{a_{j,\upsilon}\,\psi^{(2\theta+2r+2-j)}(0)+b_{j,\upsilon}\,\psi^{(2\theta+2r+2-j)}(1)\Big\}\\
&\hspace{3cm}+\sum_{\theta=0}^mc_\theta\,\ell^{2m-2\theta}\,\psi^{(2\theta+r+\upsilon)}(0)+\int A_{\ell,\upsilon}(s)\psi(s)ds=0\,.
\end{align*}

By assumption, we have $C_X^{(j,0)}(0,s)\equiv0$ for $0\leq j\leq r-1$, so that $C_X^{(q,\upsilon)}(s,0)=C_X^{(\upsilon,q)}(0,s)\equiv0$ for $0\leq \upsilon\leq r-1$ and $1\leq q\leq 2m+3r+2$. Hence $a_{j,\upsilon}=b_{j,\upsilon}=0$ and $A_{\ell,\upsilon}(s)\equiv0$, for $0\leq j\leq r-1$ and $1\leq \upsilon\leq 2\theta+2r+2$. Therefore, in view of \eqref{temp}, from the above calculations, if we let $D=m+r+1$, we find that \eqref{int}, \eqref{integro2} and \eqref{integro3} are equivalent, and are equivalent to the following boundary value problem
\begin{align}\label{integro4}
\left\{\begin{array}{l}
\displaystyle\psi^{(2D)}(s)+(-1)^{m}\sum_{\theta=0}^{m-1}c_\theta\,\ell^{2m-2\theta}\psi^{(2D-2\theta)}(s)+(-1)^{m+r}\rho_\ell\psi(s)+L_\ell\psi(s)=0\,,\\
\psi^{(\upsilon)}(0)=0\,,\qquad\text{ for }m+r\leq \upsilon\leq 2m+r-1\,,\\
\psi^{(\upsilon)}(1)=0\,,\qquad\text{ for }0\leq \upsilon\leq r-1\text{ and }m+r\leq \upsilon\leq 2m+r-1\,,\\
\displaystyle\sum_{\theta=0}^m\sum_{j=1}^{2\theta+2r+2}c_\theta\,\ell^{2m-2\theta}\Big\{a_{j,\upsilon}\,\psi^{(2\theta+2r+2-j)}(0)+b_{j,\upsilon}\,\psi^{(2\theta+2r+2-j)}(1)\Big\}\\
\displaystyle\qquad+\sum_{\theta=0}^mc_\theta\,\ell^{2m-2\theta}\psi^{(2\theta+r+\upsilon)}(0)+\int A_{\ell,\upsilon}(s)\psi(s)ds=0\,,\quad \text{ for } \upsilon=r,r+1\,.
\end{array}\right.
\end{align}
The \emph{auxiliary problem} (cf.~\citealp{tamarkin1928}, p.~459) of \eqref{integro4} is
\begin{align}\label{integro5} 
\begin{cases}
&\displaystyle\psi^{(2D)}(s)+(-1)^{m}\sum_{\theta=0}^{m-1}c_\theta\,\ell^{2m-2\theta}\psi^{(2D-2\theta)}(s)+(-1)^{m+r}\breve\rho_\ell\psi(s)=0\,,\\
&\psi^{(\upsilon)}(0)=0\,,\qquad \text{ for }m+r\leq \upsilon\leq 2m+r-1\,,\\
&\psi^{(\upsilon)}(1)=0\,,\qquad \text{ for }0\leq \upsilon\leq r-1\text{ and }m+r\leq \upsilon\leq 2m+r-1\,,\\
&\displaystyle\sum_{\theta=0}^m\sum_{j=1}^{2\theta+2r+2}c_\theta\,\ell^{2m-2\theta}\Big\{a_{j,\upsilon}\,\psi^{(2\theta+2r+2-j)}(0)+b_{j,\upsilon}\,\psi^{(2\theta+2r+2-j)}(1)\Big\}\\
&\displaystyle\qquad+\sum_{\theta=0}^mc_\theta\,\ell^{2m-2\theta}\psi^{(2\theta+r+\upsilon)}(0)=0\,,\quad \text{ for } \upsilon=r,r+1\,.
\end{cases}
\end{align}
The above boundary value problem in \eqref{integro5} is a linear differential equation of the order of $2m+2r+2=2D$ in $\psi$, with $2D$ linear homogeneous conditions on the $\psi^{(j)}(0)$ and $\psi^{(j)}(1)$, for $0\leq j\leq 2m+2r+1$; furthermore, the coefficients of the odd-order derivatives of $\psi$ in the first equation of \eqref{integro5} are all zero. The \emph{characteristic value} (see \citealp{birk1908}) of \eqref{integro5} is $(-1)^{m+r}\breve\rho_\ell$. Letting $\breve\rho_\ell=(-1)^{D+1}\ell^{\,2D}\breve\varrho^{2D}$, we have the first equation in \eqref{integro5} is
\begin{align}\label{diffcore}
\psi^{(2D)}(s)+\sum_{\theta=1}^{m}c_{m-\theta}\,\ell^{2\theta}\psi^{(2D-2\theta)}(s)+\ell^{\,2D}\breve\varrho^{2D}\,\psi(s)=0\,,
\end{align}
Let $\tilde\psi(s)=\psi(s/\ell)$, so that $\tilde\psi^{(\upsilon)}(s)=\ell^{-\upsilon}\psi^{(\upsilon)}(s/\ell)$, for $0\leq\upsilon\leq 2D$. The key of this proof is that $\psi(s)$ being the solution to \eqref{diffcore} is equivalent to $\tilde\psi(s)$ being the solution to the following ordinary differential equation corresponding to the characteristic value $\varrho^{2D}$ that is independent of $\ell$:
\begin{align}\label{rescale}
\tilde\psi^{(2D)}(s)+\sum_{\theta=1}^{m}c_{m-\theta}\,\tilde\psi^{(2D-2\theta)}(s)+\breve\varrho^{2D}\,\tilde\psi(s)=0\,.
\end{align}
In view of \eqref{integro5}, together with the boundary conditions, $\tilde\psi$ is the solution of the following boundary value problem.
\begin{align}\label{a21}
\begin{cases}
&\displaystyle\tilde\psi^{(2D)}(s)+\sum_{\theta=1}^{m}c_{m-\theta}\,\tilde\psi^{(2D-2\theta)}(s)+\breve\varrho^{2D}\,\tilde\psi(s)=0\,,\\
&\tilde\psi^{(\upsilon)}(0)=0\,,\qquad \text{ for }m+r\leq \upsilon\leq 2m+r-1\,,\\
&\tilde\psi^{(\upsilon)}(\ell)=0\,,\qquad \text{ for }0\leq \upsilon\leq r-1\text{ and }m+r\leq \upsilon\leq 2m+r-1\,,\\
&\displaystyle\sum_{\theta=0}^m\sum_{j=1}^{2\theta+2r+2}c_\theta\,\ell^{2D-j}\Big\{a_{j,\upsilon}\,\tilde\psi^{(2\theta+2r+2-j)}(0)+b_{j,\upsilon}\,\tilde\psi^{(2\theta+2r+2-j)}(\ell)\Big\}\\
&\displaystyle\qquad+\ell^{2m+r+\upsilon}\sum_{\theta=0}^mc_\theta\tilde\psi^{(2\theta+r+\upsilon)}(0)=0\,,\qquad \text{ for } \upsilon=r,r+1\,.
\end{cases}
\end{align}
Rearranging the last two boundary conditions in \eqref{a21} yields that, for $\upsilon=r+1,r$ and for $a_{j,\upsilon},b_{j,\upsilon}$ in \eqref{ab},
\begin{align*}
&\sum_{\theta=0}^m\sum_{j=0}^{2\theta+2r+1}c_{\theta}\ell^{2m-2\theta+j}\Big\{a_{2\theta+2r+2-j,\upsilon}\tilde\psi^{(j)}(0)+b_{2\theta+2r+2-j,\upsilon}\tilde\psi^{(j)}(\ell)\Big\}+\ell^{2m+r+\upsilon}\sum_{\theta=0}^mc_{\theta}\tilde\psi^{(2\theta+r+\upsilon)}(0)\\
&=\ell^{2D-1}\Bigg[\sum_{j=0}^{2m+2r+1}\Big\{\tilde a_{j,\upsilon,\ell}\,\tilde\psi^{(j)}(0)+\tilde b_{j,\upsilon,\ell}\,\tilde\psi^{(j)}(\ell)\Big\}+\ell^{\upsilon-r-1}\sum_{\theta=0}^mc_{\theta}\,\tilde\psi^{(2\theta+r+\upsilon)}(0)\Bigg]=0\,.
\end{align*}
For convenience of presentation, let $a_{j,\upsilon}=0$, for $-2m+1\leq j\leq 0$. For $\upsilon=r,r+1$ and $0\leq j\leq 2m+2r+1$, denote
\begin{align}\label{tildeab}
&\tilde a_{j,\upsilon,\ell}=\sum_{\theta=0}^m\ell^{j-2\theta-2r-1}c_{\theta}\,a_{2\theta+2r+2-j,\upsilon}\,,\qquad\tilde b_{j,\upsilon,\ell}=\sum_{\theta=0}^m\ell^{j-2\theta-2r-1}c_{\theta}\,b_{2\theta+2r+2-j,\upsilon}\,.
\end{align}
Let 
\begin{align*}
&\tilde W_1(\tilde\psi)=\sum_{j=0}^{2m+2r+1}\Big\{\tilde a_{j,r+1,\ell}\,\tilde\psi^{(j)}(0)+\tilde b_{j,r+1,\ell}\,\tilde\psi^{(j)}(\ell)\Big\}+\sum_{\theta=0}^mc_{\theta}\tilde\psi^{(2\theta+2r+1)}(0)\\
&\tilde W_2(\tilde\psi)=\sum_{j=0}^{2m+2r+1}\Big\{\tilde a_{j,r,\ell}\,\tilde\psi^{(j)}(0)+\tilde b_{j,r,\ell}\,\tilde\psi^{(j)}(\ell)\Big\}+\ell^{-1}\sum_{\theta=0}^mc_{\theta}\tilde\psi^{(2\theta+2r)}(0)
\end{align*}
In view of \eqref{integro5}, for the $\tilde a_{j,\upsilon,\ell}$'s and $\tilde b_{j,\upsilon,\ell}$'s in \eqref{tildeab}, $\tilde\psi$ satisfies the following differential equation with normalized boundary conditions (\citealp{birk1908}, p.~382) is
\begin{align}\label{tildepsi}
\begin{cases}
&\displaystyle\tilde\psi^{(2D)}(s)+\sum_{\theta=1}^{m}c_\theta\,\tilde\psi^{(2D-2\theta)}(s)+\breve\varrho^{2D}\,\tilde\psi(s)=0\\
&\displaystyle \tilde W_1(\tilde\psi)=\sum_{j=0}^{2m+2r+1}\Big\{\tilde a_{j,r+1,\ell}\,\tilde\psi^{(j)}(0)+\tilde b_{j,r+1,\ell}\,\tilde\psi^{(j)}(\ell)\Big\}+\sum_{\theta=0}^mc_{\theta}\tilde\psi^{(2\theta+2r+1)}(0)=0\,,\\
&\displaystyle\tilde W_2(\tilde\psi)=\sum_{j=0}^{2m+2r+1}\Big\{\tilde a_{j,r,\ell}\,\tilde\psi^{(j)}(0)+\tilde b_{j,r,\ell}\,\tilde\psi^{(j)}(\ell)\Big\}+\ell^{-1}\sum_{\theta=0}^mc_{\theta}\tilde\psi^{(2\theta+2r)}(0)=0\,,\\
&\tilde\psi^{(\upsilon)}(0)=0\,,\quad\tilde\psi^{(\upsilon)}(\ell)=0\,,\quad \text{ for }m+r\leq \upsilon\leq 2m+r-1\,,\\
&\tilde\psi^{(\upsilon)}(\ell)=0\,,\quad \text{ for }0\leq \upsilon\leq r-1\,.
\end{cases}
\end{align}
Let $\tilde W_3(\tilde\psi),\ldots,\tilde W_{2D}(\tilde\psi)$ denote left hand side of the boundary conditions in the rest of the last two lines of the above equations. Now, $\psi$ being the solution to \eqref{integro5} is equivalent to $\tilde\psi$ being the solution to \eqref{tildepsi}.

Next, we draw the conclusion of the growing rate of the characteristic value $\breve\varrho^{2D}$ in \eqref{tildepsi}. For $\breve\varrho\in\mathbb{C}$, let $\tilde w_{1},\ldots,\tilde w_{2D}$ denote the roots of $w^{2D}+1=0$, whose subscript is ordered according to
\begin{align*}
{\rm Re}(\breve\varrho\,\tilde w_{1})\leq{\rm Re}(\breve\varrho\,\tilde w_{2})\leq\cdots\leq{\rm Re}(\breve\varrho\,\tilde w_{2D})\,.
\end{align*}
By \cite{birk1908a} and Theorem III' in \cite{stone}, for any $\varrho\in\mathbb{C}$, \eqref{rescale} has $2D$ linear independent analytic solutions $\breve\psi_{\ell,1},\ldots,\breve\psi_{\ell,2D}$ in the form of
\begin{align}
\breve\psi^{(\upsilon)}_{\ell,j}(s)=(\breve\varrho\tilde w_{j})^{\upsilon}\exp(\breve\varrho\tilde w_{j}s)\bigg\{1+\sum_{q=1}^{M-1}\frac{B_{q,\upsilon,\ell}(s)}{(\breve\varrho\tilde w_{j})^q}+\frac{E_{j,\upsilon,\ell}(s,\breve\varrho)}{\breve\varrho^M}\bigg\}\,,
\end{align}
for some uniformly bounded functions $B_{q,\upsilon,\ell}$ and $E_{j,\upsilon,\ell}$. The condition that $\varrho^{2D}$ is the characteristic value of \eqref{tildepsi} is that
\begin{align}\label{tildeDelta}
\tilde\Delta\equiv\left|\begin{matrix}
\tilde W_1(\breve\psi_{\ell,1}) & \tilde W_1(\breve\psi_{\ell,2}) & \cdots & \tilde W_1(\breve\psi_{\ell,2D})\\
\tilde W_2(\breve\psi_{\ell,1}) & \tilde W_2(\breve \psi_{\ell,2}) & \cdots & \tilde W_2(\breve \psi_{\ell,2D})\\
%
%
\cdots & \cdots & \cdots & \cdots\\
\tilde W_{2D}(\breve\psi_{\ell,1}) & \tilde W_{2D}(\breve\psi_{\ell,2}) & \cdots & \tilde W_{2D}(\breve\psi_{\ell,2D})
\end{matrix}\right|=0\,.
\end{align}
In order to analyse the condition in \eqref{tildeDelta}, \cite{birk1908} introduced the definition of regular boundary conditions; see \cite{birk1908}, p.~382 and Definition~\ref{def:regular} in Section~\ref{app:cond:prop1.3}. For the $\tilde a_{j,\upsilon,\ell}$'s and $\tilde b_{j,\upsilon,\ell}$'s in \eqref{tildeab}, the boundary conditions in \eqref{tildepsi} is regular if $\zeta_{0,\ell},\zeta_{1,\ell},\zeta_{2,\ell}$ defined through the following equation according to the boundary conditions in \eqref{tildepsi} is such that $\zeta_{1,\ell}\,\zeta_{2,\ell}\neq 0$ for any $\ell\geq 1$:
\newcommand{\colvec}[2][.8]{%
  \scalebox{#1}{%
    \renewcommand{\arraystretch}{.8}%
    $\begin{bmatrix}#2\end{bmatrix}$%
  }
}
\begin{align*}
&\zeta_{0,\ell}+\zeta_{1,\ell}s+\frac{\zeta_{2,\ell}}{s}\equiv\\
&{\rm det}\colvec[.59]{
(\tilde a_{j,r+1,\ell}+c_0)\tilde w_1^{2D-1} & \cdots& (\tilde a_{j,r+1,\ell}+c_0)\tilde w_{D-1}^{2D-1}\, &\, \{(\tilde a_{j,r+1,\ell}+c_0)+\tilde b_{j,r+1,\ell}s\}\tilde w_D^{2D-1}\, &\, \{(\tilde a_{j,r+1,\ell}+c_0)+\frac{\tilde b_{j,r+1,\ell}}{s}\}\tilde w_{D+1}^{2D-1}\ &\ \tilde b_{j,r+1,\ell}\tilde w_{D+2}^{2D-1}&\cdots&\tilde b_{j,r+1,\ell}\tilde w_{2D}^{2D-1}\\
\tilde a_{j,r,\ell}\tilde w_1^{2D} & \cdots& \tilde a_{j,r,\ell}\tilde w_{D-1}^{2D}&(\tilde a_{j,r,\ell}+\tilde b_{j,r,\ell}s)\tilde w_D^{2D}&(\tilde a_{j,r,\ell}+\frac{\tilde b_{j,r,\ell}}{s})\tilde w_{D+1}^{2D}&\tilde b_{j,r,\ell}\tilde w_{D+2}^{2D}&\cdots&\tilde b_{j,r,\ell}\tilde w_{2D}^{2D}\\
\tilde w_1^{2m+r-1} & \cdots& \tilde w_{D-1}^{2m+r-1}&\tilde w_D^{2m+r-1}&\tilde w_{D+1}^{2m+r-1}&0&\cdots&0\\
0&\cdots&0&s\tilde w_D^{2m+r-1}&\frac{1}{s}\tilde w_D^{2m+r-1}&\tilde w_{D+2}^{2m+r-1}&\cdots&\tilde w_{2D}^{2m+r-1}\\
\tilde w_1^{2m+r-2} & \cdots& \tilde w_{D-1}^{2m+r-2}&\tilde w_D^{2m+r-2}&\tilde w_{D+1}^{2m+r-2}&0&\cdots&0\\
0&\cdots&0&s\tilde w_D^{2m+r-2}&\frac{1}{s}\tilde w_D^{2m+r-2}&\tilde w_{D+2}^{2m+r-2}&\cdots&\tilde w_{2D}^{2m+r-2}\\
\cdots & \cdots& \cdots&\cdots&\cdots&\cdots&\cdots&\cdots\\
\tilde w_1^{m+r} & \cdots& \tilde w_{D-1}^{m+r}&\tilde w_D^{m+r}&\tilde w_{D+1}^{m+r}&0&\cdots&0\\
0&\cdots&0&s\tilde w_D^{m+r}&\frac{1}{s}\tilde w_{D+1}^{m+r}&\tilde w_{D+2}^{m+r}&\cdots&\tilde w_{2D}^{m+r}\\
0 & \cdots& 0&s\tilde w_D^{r-1}&\frac{1}{s}\tilde w_{D+1}^{r-1}&\tilde w_{D+2}^{r-1}&\cdots&b_n\tilde w_{2D}^{r-1}\\
\cdots & \cdots& \cdots&\cdots&\cdots&\cdots&\cdots&\cdots\\
0 & \cdots& 0&s\tilde w_D^{0}&\frac{1}{s}\tilde w_{D+1}^{0}&\tilde w_{D+2}^{0}&\cdots&b_n\tilde w_{2D}^{0}
}\,.
\end{align*}
Then, letting $\i=\sqrt{-1}$, by the theorem in \cite{birk1908}, p.~383, we find that the eigenvalue $\varrho_k$ in \eqref{tildepsi} is of the following form:
\begin{align*}
\breve\varrho_k=\pm\frac{2k\pi\i}{\tilde w_{D}}+\frac{1}{\tilde w_{D}}\log\bigg(-\frac{\zeta_{1,\ell}}{\zeta_{2,\ell}}\bigg)+\sum_{j=1}^{M_0-1}\frac{e_{j,\ell}}{\breve\varrho^j}+\frac{E_{\ell}(\breve\varrho)}{\breve\varrho^{M_0}}\,,
\end{align*}
where $e_{j,\ell}(s)$ and $E_{2,\ell}$ are the coefficients of the equal or higher order terms of $\breve\varrho^{-1}$ from the determinant $\tilde\Delta$ in \eqref{tildeDelta}, and $|e_{j,\ell}|\leq c_1$ and $|E_{\ell}|\leq c_2$ uniformly in $k,\ell\geq 1$ and $\breve\varrho$. Moreover, $c_1\leq|\zeta_{1,\ell}/\zeta_{2,\ell}|\leq c_2$ for some $c_1,c_2>0$. Therefore, $(-1)^{D+1}\breve\varrho_k^{2D}\asymp k^{2D}$. In conclusion, the eigenvalue of \eqref{integro5} is $\breve\rho_{k\ell}=(-1)^{D+1}\ell^{\,2D}\breve\varrho_k^{2D}\asymp(k\ell)^{2D}$.

Suppose $\varrho^{2D}$ is a characteristic value of \eqref{integro4} and suppose $\psi(s)=\tilde\psi(\ell s)$ is the solution to the problem \eqref{integro4} corresponding to $\varrho^{2D}$. We rewrite the first equation in \eqref{integro4} that $\tilde\psi$ satisfies. For the $L_\ell$ in \eqref{integro4} defined in \eqref{Lell}, substituting $\psi(s)$ by $\tilde\psi(\ell s)$ yields
\begin{align*}
&L_\ell\psi(s_1)=\sum_{\theta=0}^m\ell^{-2\theta-2r-2}c_{\theta}\int_0^1 M^{(0,2\theta+2r+2)}(s_1,s_2)\,\tilde\psi(\ell s_2)ds_2+\sum_{\theta=0}^m\sum_{j=1}^{2\theta+2r+2}(-1)^{m+r+j}\ell^{-j}c_{\theta}\\
&\hspace{2cm}\times\Big\{M^{(0,j-1)}(s_1,1)\,\tilde\psi^{(2\theta+2r+2-j)}(\ell)-M^{(0,j-1)}(s_1,0)\,\tilde\psi^{(2\theta+2r+2-j)}(0)\Big\}\\
&=\int_0^1\mathfrak H_\ell(s,\xi)\,\tilde\psi(\ell\xi)d\xi+\sum_{\theta=0}^m\sum_{j=1}^{2\theta+2r+2}\Big\{\mathfrak A_{j,\theta,\ell}(s)\tilde\psi^{(2\theta+2r+2-j)}(\ell)-\mathfrak B_{j,\theta,\ell}(s)\tilde\psi^{(2\theta+2r+2-j)}(0)\Big\}\,,
\end{align*}
where, for $0\leq \theta\leq m$ and $1\leq j\leq 2\theta+2r+2$,
\begin{align*}
&\mathfrak H_\ell(s_1,s_2)=\sum_{\theta=0}^m\ell^{-2\theta-2r-2}c_{\theta}M^{(0,2\theta+2r+2)}(s_1,s_2)\,,\\
&\mathfrak A_{j,\theta,\ell}(s)=(-1)^{m+r+j}\ell^{-j}c_{\theta}\,M^{(0,j-1)}(s,1)\,,\\
&\mathfrak B_{j,\theta,\ell}(s)=(-1)^{m+r+j}\ell^{-j}c_{\theta}\,M^{(0,j-1)}(s,0)\,.
\end{align*}
We have $|\mathfrak H_\ell(s_1,s_2)|\leq c\ell^{-2r-2}$ uniformly in $s_1,s_2\in[0,1]$; $|\mathfrak A_{j,\theta,\ell}(s)|,|\mathfrak B_{j,\theta,\ell}(s)|\leq c\ell^{-1}$ uniformly in $s\in[0,1]$, for $0\leq \theta\leq m$ and $1\leq j\leq 2\theta+2r+2$. Therefore, for $A_{\ell,\upsilon}$ is as in \eqref{ab}, we have that $\tilde\psi(s)$ is the solution to the following equations
\begin{align}\label{integro10}
\begin{cases}
&\displaystyle\tilde\psi^{(2D)}(s)+\sum_{\theta=1}^{m}c_\theta\tilde\psi^{(2D-2\theta)}(s)+\varrho^{2D}\,\tilde\psi(s)+\int_0^1\mathfrak H_\ell(s,\xi)\,\tilde\psi(\ell\xi)d\xi\\
&\hspace{1cm}\displaystyle+\sum_{\theta=0}^m\sum_{j=1}^{2\theta+2r+2}\Big\{\mathfrak A_{j,\theta,\ell}(s)\tilde\psi^{(2\theta+2r+2-j)}(\ell)-\mathfrak B_{j,\theta,\ell}(s)\tilde\psi^{(2\theta+2r+2-j)}(0)\Big\}=0\,,\\
&\tilde\psi^{(\upsilon)}(0)=0\,,\qquad\text{ for }m+r\leq \upsilon\leq 2m+r-1\,,\\
&\tilde\psi^{(\upsilon)}(\ell)=0\,,\qquad\text{ for }0\leq \upsilon\leq r-1\text{ and }m+r\leq \upsilon\leq 2m+r-1\,,\\
&\displaystyle\sum_{\theta=0}^m\sum_{j=1}^{2\theta+2r+2}\ell^{2D-j}c_{\theta}\Big\{a_{j,\upsilon}\,\tilde\psi^{(2\theta+2r+2-j)}(0)+b_{j,\upsilon}\,\tilde\psi^{(2\theta+2r+2-j)}(\ell)\Big\}\\
&\displaystyle\qquad+\sum_{\theta=0}^m\ell^{2m+r+\upsilon}c_{\theta}\,\tilde\psi^{(2\theta+r+\upsilon)}(0)+\int A_{\ell,\upsilon}(s)\tilde\psi(s)ds=0\,,\quad \text{ for } \upsilon=r,r+1\,.
\end{cases}
\end{align}
Following the proof of Theorem 7 in \cite{tam1927}, the characteristic value $\varrho_k$ of \eqref{integro10} and the characteristic value $\breve\varrho_k$ of \eqref{tildepsi} have the same growing rate uniformly in $\ell$, so that $\rho_{k\ell}\asymp(k\ell)^{2D}$.


%
%

For the order of $\Vert x_{k\ell}\Vert_\infty$, 
Let $\tilde\gamma_{k\ell}=\rho_{k\ell}^{1/(2D)}\in\mathbb{R}$ and $\gamma_{k\ell}=\i\tilde\gamma_{k\ell}\exp\{\pi\i/(2D)\}$. We therefore have $\gamma_{k\ell}^{2k}=(-1)^{m+r}\rho_{k\ell}$, and $\Re(\gamma_{k\ell})=-\tilde\gamma_{k\ell}\sin\{\pi/(2D)\}$, and $\Im(\gamma_{k\ell})=\tilde\gamma_{k\ell}\cos\{\pi/(2D)\}$. 
Let $w_{k\ell,1},\ldots,w_{k\ell,2D}\in\big\{\exp\{(2\nu-1)\pi\i/(2D)\}\big\}_{1\leq\nu\leq2D}$ denote the solution of $w^{2D}+1=0$, of which the subscripts are assigned according to the order
\begin{align*}
{\rm Re}(\ell\gamma_{k\ell}\, w_{k\ell,1})\leq{\rm Re}(\ell\gamma_{k\ell}\, w_{k\ell,2})\leq\cdots\leq{\rm Re}(\ell\gamma_{k\ell}\, w_{k\ell,2D})\,,
\end{align*}
so that ${\rm Re}(\gamma_{k\ell}\, w_{k\ell,j})<0$ when $1\leq j\leq D$, and ${\rm Re}(\gamma_{k\ell}\, w_{k\ell,j})\geq0$ when $D+1\leq \nu\leq 2D$. Following \cite{tamarkin1928}, pp.~467--469, the solution $\psi(s)$ corresponding to the eigenvalue $\rho_{k\ell}$ in \eqref{integro4} is such that, for $0\leq \upsilon\leq 2D-1$,
\begin{align}\label{solution}
\psi_{k\ell}^{(\upsilon)}(s)&=\gamma_{k\ell}^{\upsilon}\bigg[\sum_{j=1}^{D}\exp(\gamma_{k\ell}\,w_{k\ell,j}s)[Q_{\ell,j}w_{k\ell,j}^{\upsilon}]+\sum_{j=D+1}^{2D}\exp\{\gamma_{k\ell}\,w_{k\ell,j}(s-1)\}[Q_{\ell,j}w_{k\ell,j}^{\upsilon}]\bigg]\,,
\end{align}
where for $z\in\mathbb{C}$, $[z]$ is such that $|[z]-z|=O(k^{-1}+\ell^{-1})$, and $Q_{\ell,1},\ldots,Q_{\ell,2D}$ are real-valued constants that does not depend on $k$ and are bounded, and at least one of these $2D$ constants is non-zero. Without loss of generality we may assume that $\Im(w_{k\ell,D})<0$ and $\Im(w_{k\ell,D+1})>0$, so that $\gamma_{k\ell}\, w_{k\ell,D}=-\i\tilde\gamma_{k\ell}$ and $\gamma_{k\ell}\, w_{k\ell,D+1}=\i\tilde\gamma_{k\ell}$; when $j\neq D$ and $2D$, $|{\rm Re}(\gamma_{k\ell}\, w_{k\ell,j})|\geq\tilde\gamma_{k\ell}\sin(\pi/D)$. Now, since $s\in[0,1]$, when $1\leq\nu\leq D-1$, $|\exp(\gamma_{k\ell}\,w_{k\ell,j}s)|=\exp\{\Re(\gamma_{k\ell}\,w_{k\ell,j})s\}\leq 1$; when $D+2\leq\nu\leq 2D$, $|\exp\{\gamma_{k\ell}\,w_{k\ell,j}(s-1)\}|=\exp\{\Re(\gamma_{k\ell}\,w_{k\ell,j})(s-1)\}\leq 1$; $|\exp(\gamma_{k\ell}\,w_{k\ell,D}s)|=|\exp\{\gamma_{k\ell}\,w_{k\ell,D+1}(s-1)\}|=1$. Therefore, from \eqref{solution}, we find that, for $0\leq\theta\leq m$,
\begin{align}\label{sup}
&\sup_{s\in[0,1]}\big|\psi_{k\ell}^{(r+\theta)}(s)\big|\notag\\
&\leq|\gamma_{k\ell}|^{r+\theta}\sup_{s\in[0,1]}\bigg[\sum_{ j =1}^{D}|\exp(\gamma_{k\ell}\,w_{k\ell, j }s)|\,|Q_{\ell, j }|+\sum_{ j =D+1}^{2D}|\exp\{\gamma_{k\ell}\,w_{k\ell, j }(s-1)\}|\,|Q_{\ell, j }|+O(k^{-1})\bigg]\notag\\
&\leq |\gamma_{k\ell}|^{r+\theta}\bigg\{\sup_{\ell\geq1}\sum_{ j =1}^{2D}|Q_{\ell, j }|+O(k^{-1})\bigg\}\asymp (\ell k)^{r+\theta}\,.
\end{align}
Let
\begin{align*}
Z_{k\ell,\theta,j}(s)=
\begin{cases}
\exp(\gamma_{k\ell}\,w_{k\ell, j }s)[Q_{\ell, j }w_{k\ell, j }^{r+\theta}]\,,\hspace{2cm} \text{for } 1\leq j \leq D\,,\\
\exp\{\gamma_{k\ell}\,w_{k\ell, j }(s-1)\}[Q_{\ell, j }w_{k\ell, j }^{r+\theta}]\,,\hspace{0.85cm} \text{for } D+1\leq j \leq 2D\,,\\
\end{cases}
\end{align*}
so that $\psi_{k\ell}^{(r+\theta)}(s)=\gamma_{k\ell}^{r+\theta}\sum_{j=1}^{2D}Z_{k\ell,\theta,j}(s)$. Note that for $1\leq j_1\leq D-1$ and for $1\leq j_2\leq 2D$,
\begin{align}\label{t21}
\big|Z_{k\ell,\theta,j_1}(s)\,\overline{Z_{k\ell,\theta,j_2}(s)}\big|&\leq\big|\exp(\gamma_{k\ell}\,w_{k\ell, j_1 }s)\big|\times|[Q_{\ell, j_1 }]|\times|[Q_{\ell, j_2 }]|+O(k^{-1})\notag\\
&=\exp\{\Re(\gamma_{k\ell}\,w_{k\ell,j_1})s\}\times|Q_{\ell, j_1 }Q_{\ell, j_2 }|+O(k^{-1})\notag\\
&\leq\exp\{-\tilde\gamma_{k\ell}\sin(\pi/D)s\}\times|Q_{\ell, j_1 }Q_{\ell, j_2 }|+O(k^{-1})\notag\\
&\leq c_{\ell}\exp(-c_{\ell}'ks)\,,
\end{align}
for $c_\ell,c_\ell'>0$. Hence, we deduce from the above equation that $\int_0^1\big|Z_{k\ell,\theta,j_1}(s)\,\overline{Z_{k\ell,\theta,j_2}(s)}\big|ds\lesssim k^{-1}$. Likewise, for $D+2\leq j_1\leq 2D$ and for any $1\leq j_2\leq 2D$,
\begin{align}\label{t22}
\big|Z_{k\ell,\theta,j_1}(s)\,\overline{Z_{k\ell,\theta,j_2}(s)}\big|&\leq\big|\exp\{\gamma_{k\ell}\,w_{k\ell, j_1 }(s-1)\}\big|\times|[Q_{\ell, j_1 }]|\times|[Q_{\ell, j_2 }]|+O(k^{-1})\notag\\
&=\exp\{\Re(\gamma_{k\ell}\,w_{k\ell,j_1})(s-1)\}\times|Q_{\ell, j_1 }Q_{\ell, j_2 }|+O(k^{-1})\notag\\
&\leq\exp\{\tilde\gamma_{k\ell}\sin(\pi/D)(s-1)\}\times|Q_{\ell, j_1 }Q_{\ell, j_2 }|+O(k^{-1})\notag\\
&\leq c_{\ell}\exp\{-c_{\ell}'k(s-1)\}\,,
\end{align}
for $c_\ell,c_\ell'>0$. Hence, we find that $\int_0^1\big|Z_{k\ell,\theta,j_1}(s)\,\overline{Z_{k\ell,\theta,j_2}(s)}\big|ds\lesssim k^{-1}$. Recall that $\gamma_{k\ell}\, w_{k\ell,D}=-\i\tilde\gamma_{k\ell}$ and $\gamma_{k\ell}\, w_{k\ell,D+1}=\i\tilde\gamma_{k\ell}$. We have
\begin{align*}
&\big|Z_{k\ell,\theta,D}(s)\big|^2=\exp\{2\Re(\gamma_{k\ell}\,w_{k\ell,D})s\}\times|Q_{\ell,D}|^2+O(k^{-1})=|Q_{\ell,D}|^2+O(k^{-1})\,.\\
&\big|Z_{k\ell,\theta,D+1}(s)\big|^2=\exp\{2\Re(\gamma_{k\ell}\,w_{k\ell,D+1})(s-1)\}\times|Q_{\ell,D+1}|^2+O(k^{-1})=|Q_{\ell,D+1}|^2+O(k^{-1})\,.
\end{align*}
Therefore,
\begin{align*}
\Vert\psi_{k\ell}^{(r+\theta)}\Vert_{L^2}^2&=|\gamma_{k\ell}|^{2r+2\theta}\sum_{j_1=1}^{2D}\sum_{j_2=1}^{2D}\int Z_{k\ell,\theta,j_1}(s)\,\overline{Z_{k\ell,\theta,j_2}(s)}ds\\
&=|\gamma_{k\ell}|^{2r+2\theta}\bigg\{\int\big|Z_{k\ell,\theta,D}(s)\big|^2ds+\int\big|{Z_{k\ell,\theta,D+1}(s)}\big|^2ds+O(k^{-1})\bigg\}\\
&=|\gamma_{k\ell}|^{2r+2\theta}\Big\{|Q_{\ell,D}|^2+|Q_{\ell,D+1}|^2+O(k^{-1})\Big\}\,.
\end{align*}
Now we show that in the above equation $|Q_{\ell,D}|^2+|Q_{\ell,D+1}|^2\neq 0$, which can be proved by contradiction. Suppose $Q_{\ell,D}=Q_{\ell,D+1}=0$. Then, in view of the boundary condition in \eqref{integro5}, for $m+r\leq \upsilon_1\leq 2m+r-1$, and for $\upsilon_2$ such that $0\leq \upsilon_2\leq r-1$ or $m+r\leq \upsilon_2\leq 2m+r-1$,
\begin{align*}
&0=\psi_{k\ell}^{(\upsilon_1)}(0)=\gamma_{k\ell}^{\upsilon_1}\bigg\{\sum_{j=1}^{D-1}[Q_{\ell,j}w_{k\ell,j}^{\upsilon_1}]+\sum_{j=D+2}^{2D}\exp(-\gamma_{k\ell}\,w_{k\ell,j})[Q_{\ell,j}w_{k\ell,j}^{\upsilon_1}]+O(k^{-1})\bigg\}\,,\\
&0=\psi_{k\ell}^{(\upsilon_2)}(1)=\gamma_{k\ell}^{\upsilon_2}\bigg\{\sum_{j=1}^{D-1}\exp(\gamma_{k\ell}\,w_{k\ell,j})[Q_{\ell,j}w_{k\ell,j}^{\upsilon_2}]+\sum_{j=D+2}^{2D}[Q_{\ell,j}w_{k\ell,j}^{\upsilon_2}]+O(k^{-1})\bigg\}\,.
\end{align*}
Following the arguments in \cite{shang2015b}, pp.~7--9, letting $k\to\infty$ yields $Q_{\ell,j}=0$ for all $1\leq j\leq 2D$, which contradicts with the fact that at least one of these $2D$ constants is nonzero. Therefore, we deduce that $|Q_{\ell,D}|^2+|Q_{\ell,D+1}|^2\neq 0$, so that $\Vert\psi_{k\ell}^{(r+\theta)}\Vert_{L^2}^2\asymp |\gamma_{k\ell}|^{2r+2\theta}\asymp (\ell k)^{2r+2\theta}$, for $0\leq\theta\leq m$. From this we deduce that
\begin{align*}
\big\l\C_X\psi_{k\ell}^{(r)},\psi_{k\ell}^{(r)}\big\r_{L^2}=\rho_{k\ell}^{-1}\sum_{\theta=0}^m{m\choose\theta}\Vert\eta_\ell^{(m-\theta)}\Vert_{L^2}^2\big\Vert\psi_{k\ell}^{(r+\theta)}\big\Vert_{L^2}^2\asymp (k\ell)^{-2}\,.
\end{align*}
Recall that
$x_{k\ell}=\l\C_X(\psi_{k\ell}^{(r)}),\psi_{k\ell}^{(r)}\r^{-1/2}_{L^2}\,\psi_{k\ell}^{(r)}$. Hence, from the above equation and \eqref{sup} we conclude that $\Vert x_{k\ell}\Vert_\infty=\l\C_X(\psi_{k\ell}^{(r)}),\psi_{k\ell}^{(r)}\r^{-1/2}_{L^2}\,\Vert\psi_{k\ell}^{(r)}\Vert_\infty\lesssim (k\ell)^{r+1}$.

Since the cosine series $\{\eta_\ell\}_{\ell\geq 1}$ is a complete basis of $L^2([0,1])$ (Theorem~2.4.18 in \citealp{hsing2015}), by the argument similar to the ones in \cite{shang2015b} p.~7, we have $\{\phi_{k\ell}\}_{k,\ell\geq1}$ is complete in $\H$, and any $\beta\in\H$ admits the Fourier expansion $\beta=\sum_{k,\ell}V(\beta,\phi_{k\ell})\phi_{k\ell}$.

In order to show \eqref{diag}, due to Assumption~\ref{a1}, $V_X(x_1,x_2)\equiv\l\C_Xx_1,x_2\r_{L^2}$ defines an inner product, for $x_1,x_2\in L^2([0,1])$. For each $\ell\geq1$, we may orthonormalize the $\breve x_{k\ell}$'s, $k\geq1$, w.r.t.~$V_X$ to obtain the $x_{k\ell}$'s. That is, $V_X(x_{k\ell},x_{k'\ell})=\delta_{kk'}$, for $k,k'\geq1$. Recall that $\phi_{k\ell}(s,t)=x_{k\ell}(s)\eta_\ell(t)$, $k,\ell\geq1$. For $V$ in \eqref{v},
\begin{align*}
V(\phi_{k\ell},\phi_{k'\ell'})&=\int_{s_1,s_2\in[0,1];\,t\in[0,1]}C_X(s_1,s_2)\big\{x_{k\ell}(s_1)\eta_\ell(t)\big\}\big\{x_{k'\ell'}(s_2)\eta_{\ell'}(t)\big\}\,ds_1ds_2dt\\
&=\int_{[0,1]^2}C_X(s_1,s_2)x_{j\ell}(s_1)x_{k\ell'}(s_2)ds_1ds_2\times\int_0^1\eta_\ell(t)\eta_{\ell'}(t)\,dt\\
&=\big\l\C_X(x_{k\ell}),x_{k\ell'}\big\r_{L^2}\times\l\eta_\ell,\eta_{\ell'}\r_{L^2}=V_X(x_{k\ell},x_{k'\ell})\,\delta_{\ell\ell'}=\delta_{kk'}\,\delta_{\ell\ell'}\,.
\end{align*}
For $J$ in \eqref{jm} and the $\eta_\ell$ in \eqref{fourier}, we have $\l\eta_\ell^{(\theta)},\eta_{\ell'}^{(\theta)}\r_{L^2}=0$ for $0\leq\theta\leq m$ and $\ell\geq1$, so that
\begin{align*}
J(\phi_{k\ell},\phi_{k'\ell'})&=J( x_{k\ell}\otimes\eta_\ell, x_{k'\ell'}\otimes\eta_{\ell'})=\sum_{\theta=0}^m{m\choose \theta}\big\l x_{k\ell}^{(\theta)}, x_{k'\ell'}^{(\theta)}\big\r_{L^2}\big\l\eta_\ell^{(m-\theta)},\eta_{\ell'}^{(m-\theta)}\big\r_{L^2}\\
&=\sum_{\theta=0}^m{m\choose \theta}\big\l x_{k\ell}^{(\theta)}, x_{k'\ell'}^{(\theta)}\big\r_{L^2}\big\l\eta_\ell^{(m-\theta)},\eta_{\ell'}^{(m-\theta)}\big\r_{L^2}\\
&=\delta_{\ell\ell'}\sum_{\theta=0}^m{m\choose\theta}\big\l x_{k\ell}^{(\theta)}, x_{k'\ell'}^{(\theta)}\big\r_{L^2}\Vert\eta_\ell^{(m-\theta)}\Vert_{L^2}^2\,.
\end{align*}
Using integration by parts and the boundary conditions, in view of \eqref{int}, we deduce from the above equation that
\begin{align*}
J(\phi_{k\ell},\phi_{k'\ell'})&=\delta_{\ell\ell'}\sum_{\theta=0}^m{m\choose\theta}(-1)^{\theta}\,\big\l x_{k\ell}^{(2\theta)}, x_{k'\ell'}\big\r_{L^2}\,\Vert\eta_\ell^{(m-\theta)}\Vert_{L^2}^2\\
&=\delta_{\ell\ell'}\left\l  \sum_{\theta=0}^m{m\choose\theta}(-1)^{\theta}\,\Vert\eta_\ell^{(m-\theta)}\Vert_{L^2}^2\,x_{k'\ell'}^{(2\theta)}\,,x_{k'\ell'}\right\r_{L^2}\\
&=\delta_{\ell\ell'}\left\l\rho_{k\ell}\int_{0}^1C_X(s_1,s_2)\,x_{k\ell'}(s_2)\,ds_2\,,x_{k'\ell'}\right\r_{L^2}\\
&=\delta_{\ell\ell'}\,\rho_{k\ell}\,V_X(x_{j\ell},x_{k\ell'})=\rho_{k\ell}\,\delta_{kk'}\,\delta_{\ell\ell'}\,,
\end{align*}
which completes the proof.

{\centering

}

\end{document}